\documentclass[reqno,a4paper]{amsart} 
\usepackage{amsaddr}
\usepackage{natbib}
\usepackage[utf8]{inputenc}

\title[Error estimates for DeepONets]{Error estimates for DeepONets: A deep learning framework in infinite dimensions
}
\author{Samuel Lanthaler}
\thanks{\emph{E-Mail}: samuel.lanthaler@math.ethz.ch }

\address[S. Lanthaler]{
Seminar for Applied Mathematics (SAM), \\
  Eidgen\"ossische Technische Hochschule Z\"urich (ETHZ), \\
  R\"amistrasse 101, 8092 Z\"urich, Switzerland
  }

\author{Siddhartha Mishra}
\address[S. Mishra]{
Seminar for Applied Mathematics (SAM), \\
  Eidgen\"ossische Technische Hochschule Z\"urich (ETHZ), \\
  R\"amistrasse 101, 8092 Z\"urich, Switzerland}

\author{George Em Karniadakis}
\address[G.E. Karniadakis]{
Division of Applied Mathematics and School of Engineering, 
\\
Brown University, \\
Providence, RI 02912, USA
}

\date{\today}

\usepackage{amsthm,amssymb}
\usepackage{mathtools,thmtools}
\usepackage{bm,esint}
\usepackage{color}
\usepackage{float,graphicx,subcaption}
\usepackage{cancel}
\usepackage[shortlabels]{enumitem}
\usepackage{mathrsfs}  
\usepackage{bbm}
\usepackage{euscript}
\usepackage{hyperref}
\usepackage{cleveref}
\usepackage{upgreek}

\usepackage{todonotes}

\newcommand{\size}{\mathrm{size}}

\newcommand{\depth}{\mathrm{depth}}

\newcommand{\unif}{\mathrm{Unif}}
\newcommand{\Unif}{\unif}
\renewcommand{\Re}{\mathrm{Re}}
\renewcommand{\Im}{\mathrm{Im}}

\newcommand{\fb}{\bm{\mathrm{e}}}

\renewcommand{\tilde}{\widetilde}
\renewcommand{\hat}{\widehat}

\newcommand{\Id}{\mathrm{Id}}

\newcommand{\opt}{\mathrm{opt}}

\DeclareMathOperator*{\essinf}{ess\,inf}
\DeclareMathOperator*{\esssup}{ess\,sup}
\DeclareMathOperator*{\argmin}{argmin}
\newcommand{\erfc}{\mathrm{erfc}}

\newcommand{\Tr}{{\mathrm{Tr}}}

\newcommand{\Err}{\widehat{\mathscr{E}}}

\newcommand{\Fourier}{\mathrm{Fourier}}
\newcommand{\proj}{\mathrm{Proj}}
\newcommand{\Proj}{\mathrm{Proj}}
\newcommand{\Span}{\mathrm{span}}

\newcommand{\im}{\mathrm{Im}}
\newcommand{\embeds}{{\hookrightarrow}}

\renewcommand{\bar}{\overline}

\newcommand{\explain}[2]{\overset{\mathclap{\underset{\downarrow}{#2}}}{#1}}

\newcommand{\slot}{{\,\cdot\,}}

\newcommand{\supp}{\mathrm{supp}}

\newcommand{\T}{\mathbb{T}}
\newcommand{\R}{\mathbb{R}}
\newcommand{\C}{\mathbb{C}}
\newcommand{\E}{\mathbb{E}}
\newcommand{\N}{\mathbb{N}}

\newcommand{\Z}{\mathbb{Z}}
\newcommand{\Prob}{\mathrm{Prob}}

\newcommand{\G}{\mathcal{G}}

\newcommand{\Lip}{\mathrm{Lip}}

\newcommand{\BV}{\mathrm{BV}}
\newcommand{\Normal}{\mathcal{N}}

\renewcommand{\P}{\mathcal{P}}
\newcommand{\cA}{\mathcal{A}}

\newcommand{\cC}{\mathcal{C}}
\newcommand{\cD}{\mathcal{D}}
\newcommand{\cE}{\mathcal{E}}
\newcommand{\cF}{\mathcal{F}}
\newcommand{\cG}{\mathcal{G}}

\newcommand{\cI}{\mathcal{I}}
\newcommand{\cJ}{\mathcal{J}}
\newcommand{\cK}{\mathcal{K}}
\newcommand{\cL}{\mathcal{L}}
\newcommand{\cM}{\mathcal{M}}
\newcommand{\cN}{\mathcal{N}}

\newcommand{\cP}{\mathcal{P}}

\newcommand{\cR}{\mathcal{R}}
\newcommand{\cS}{\mathcal{S}}
\newcommand{\cT}{\mathcal{T}}

\newcommand{\cY}{\mathcal{Y}}

\newcommand{\tE}{\tilde{\cE}}
\newcommand{\tF}{\tilde{\cF}}
\newcommand{\tG}{\tilde{\cG}}
\newcommand{\tN}{\tilde{\cN}}
\newcommand{\tR}{\tilde{\cR}}
\newcommand{\tP}{\tilde{\cP}}

\newcommand{\hL}{\widehat{\cL}}
\newcommand{\hN}{\widehat{\cN}}

\newcommand{\eps}{\epsilon}
\newcommand{\e}{\eps}

\newcommand{\define}{\textbf}

\renewcommand{\hat}{\widehat}

\newcommand{\set}[2]{{\left\{ #1 \,\middle|\, #2 \right\}}}

\newcommand{\shrink}{\mathrm{shrink}}

\newcommand{\tr}{\tau}
\newcommand{\br}{\beta}

\newcommand{\enum}{{\upkappa}}

\newcommand{\map}{\EuScript{L}}

\declaretheoremstyle[
  headfont=\normalfont\bfseries\itshape,
  numbered=unless unique,
  bodyfont=\normalfont,
  spaceabove=1em plus 0.75em minus 0.25em,
  spacebelow=1em plus 0.75em minus 0.25em,
  qed={},
]{deflt}

\theoremstyle{deflt}
\newtheorem{theorem}{Theorem}[section]
\newtheorem{assumption}[theorem]{Assumption}
\newtheorem{setup}[theorem]{Setup}
\newtheorem{example}[theorem]{Example}
\newtheorem{remark}[theorem]{Remark}

\newtheorem{definition}[theorem]{Definition}
\newtheorem{lemma}[theorem]{Lemma}
\newtheorem{proposition}[theorem]{Proposition}
\newtheorem{corollary}[theorem]{Corollary}

\numberwithin{equation}{section}
\numberwithin{theorem}{section}

\newtheorem*{claim}{Claim}

\usepackage{tikz,tikz-cd}
\tikzset{every picture/.style={line width=0.75pt}} 
\usetikzlibrary{decorations.pathmorphing} 

\newcommand{\rev}[1]{{\color{black} #1}}

\begin{document}

\maketitle
\begin{abstract}
DeepONets have recently been proposed as a framework for learning nonlinear operators mapping between infinite dimensional Banach spaces. We analyze DeepONets and prove estimates on the resulting approximation and generalization errors. In particular, we extend the universal approximation property of DeepONets to include measurable mappings in non-compact spaces. By a decomposition of the error into encoding, approximation and reconstruction errors, we prove both lower and upper bounds on the total error, relating it to the spectral decay properties of the covariance operators, associated with the underlying measures. We derive almost optimal error bounds with very general affine reconstructors and with random sensor locations as well as bounds on the generalization error, using covering number arguments. 

We illustrate our general framework with four prototypical examples of nonlinear operators, namely those arising in a nonlinear forced ODE, an elliptic PDE with variable coefficients and nonlinear parabolic and hyperbolic PDEs. \rev{
While the approximation of \emph{arbitrary} Lipschitz operators by DeepONets to accuracy $\epsilon$ is argued to suffer from a ``curse of dimensionality'' (requiring a neural networks of exponential size in $1/\epsilon$), in contrast, for all the above \emph{concrete} examples of interest, we rigorously prove that DeepONets can \emph{break this curse of dimensionality} (achieving accuracy $\epsilon$ with neural networks of size that can grow algebraically in $1/\epsilon$)}. Thus, we demonstrate the efficient approximation of a potentially large class of operators with this machine learning framework.

\end{abstract}

\section{Introduction}
\emph{Deep neural networks} \cite[]{DLbook} have been very successfully used for a diverse range of regression and classification learning tasks in science and engineering in recent years \cite[]{DL-nat}. These include image and text classification, computer vision, text and speech recognition, natural language processing, autonomous systems and robotics, game intelligence and protein folding \cite[]{deepfold}.

As deep neural networks are \emph{universal approximators}, i.e., they can approximate any continuous (even measurable) finite-dimensional function to arbitrary accuracy \cite[]{BAR1,HOR1,Cy1,tpchen}, it is natural to use them as ansatz spaces for the solutions of partial differential equations (PDEs). They have been used for solving high-dimensional parabolic PDEs by emulating explicit representations such as the Feynman-Kac formula as in \cite{HEJ1,E1,Jent1} and references therein, and as \emph{physics informed neural networks} (PINNs) for solving both forward problems \cite[]{KAR1,KAR2,KAR3,MM1,MM3}, as well as inverse problems \cite[]{KAR2,KAR4,MM2,lu2021} for a variety of linear and non-linear PDEs. 

Deep neural networks are also being widely used in the context of \emph{many query} problems for PDEs, such as  uncertainty quantification (UQ) \cite[see e.g.][]{UQ1,UQ2,UQ3}, optimal control (design), deterministic and Bayesian inverse problems \cite[]{Inverse1,Inverse2} and PDE constrained optimization \cite[]{Optim1,LMRP1}. In such \emph{many query} problems, the inputs are functions such as the initial and boundary data, source terms and/or coefficients in the underlying differential operators. The outputs are either the solution field (in space-time or at fixed time instances) or possibly \emph{observables} (functionals of the solution field). Thus, the input to output map is, in general, a (possibly) non-linear \emph{operator}, mapping one function space to another. 

Currently, it is standard to approximate the underlying input function with a finite, but possibly very high-dimensional, \emph{parametric} representation. Similarly, the resulting output function is approximated by a finite dimensional representation, for instance, values on a grid or coefficients of a suitable basis. Thus, the underlying \emph{operator, mapping infinite dimensional spaces,} is approximated by a function that maps a finite but high dimensional input spaces into another finite-dimensional output space. Consequently, this finite-dimensional map for the resulting \emph{parametric PDE} can be \emph{learned} with \emph{standard} deep neural networks, as for elliptic and parabolic PDEs in \cite{SchwabZech2019,OSZ2019, OSZ2020, Kuty}, for transport PDEs \cite[]{PP1} and for hyperbolic and related PDEs \cite[][and references therein]{DRM1,UQ3,LMRP1}.

However, this finite dimensional parametrization of the underlying infinite dimensional problem is subject to the inherent and non-vanishing error both at the input end, due to the finite dimensional representation as well as at the output end, on account of numerical errors at finite resolution. More fundamentally, a parametric representation requires explicit knowledge of the underlying measure on input space such that a finite dimensional approximation of inputs can be performed. Such explicit knowledge may not always be available. Finally, the parametric approach does not cover a large number of situations where the underlying physics, in the form of governing PDEs, may not even be known explicitly, yet large amounts of (possibly noisy) data for the input-output mapping is available. It is not obvious how such a learning task can be performed with standard neural networks.

Hence, \emph{operator learning}, i.e. learning nonlinear operators mapping one infinite-dimensional Banach space to another, from data, is increasingly being investigated in the contexts of PDEs and possibly other fields. 
One research direction has focused largely on operators which can be expressed as solution operators to a suitable PDE/ODE; examples of this approach include the identification of individual terms of the underlying differential equation from data, expressed in terms of non-local integral operators \citep[and references therein]{OpRegression1,OpRegression3}, the identification of suitable closure models for turbulent flows \citep[and references therein]{closure1,closure2} or the discovery of the governing equations of an underlying dynamical system, expressed in terms of an ODE (or PDE), \cite[][and references therein]{Sindy}. In a different research direction, the aim is to use deep neural networks to directly learn the \emph{underlying (solution--)operator}, itself. Several frameworks have been proposed for this task; we refer to \cite{li2020multipole} and \cite{li2020neural} for graph kernel operators, \cite{STU3} for a recent approach based on principal component analysis, and \cite{STU2} and references therein on \emph{Fourier neural operators}.
A different approach was proposed by \cite{ChenChen1995}, where they presented a neural network architecture, termed as \emph{operator nets}, to approximate a non-linear operator $\G: K \to K'$, where $K, K'$ are compact subsets of infinite dimensional Banach spaces, $K\subset C(D)$, $K' \subset C(U)$ with $D, \, U$ compact domains in $\R^d$, $\R^n$, respectively. Then, an operator net can be formulated in terms of two shallow, i.e., one hidden layer, neural networks. The first is the so-called \emph{branch net} $\bm{\br}(u) = (\br_1(u), \dots, \br_p(u))$, defined for $1 \leq k \leq p$ as,
\begin{equation}
    \label{eq:bnet1}
    \br_k(u)
=
\sum_{i=1}^\ell c_k^i \sigma\left( \sum_{j=1}^m \xi_{ki}^j u(x_j) + \theta_k^i \right).
\end{equation}
Here, $\{x_j\}_{1 \leq j \leq m} \subset D$, are the so-called \emph{sensors} and $c^i_k,\xi^j_{ki}$ are weights and $\theta_k^i$ are biases of the neural network.  

The second neural network is the so-called \emph{trunk net} $\bm{\tr}(y) = (\tr_1(y),\dots, \tr_p(y))$, defined as, 
\begin{equation}
    \label{eq:tnet1}
    \tr_k(y) 
=
\sigma(w_k \cdot y + \zeta_k), \quad 1 \leq k \leq p,
\end{equation}
for any $y \in U$ and with weights $w_k$ and biases $\zeta_k$. Here, $\sigma$ is a non-linear activation function in the branch net \eqref{eq:bnet1} and (a possibly different one) in the trunk net \eqref{eq:tnet1}. The branch and trunk nets are then combined to approximate the underlying non-linear operator in the \emph{operator net}
\begin{equation}
    \label{eq:donet1}
\G(u)(y) \approx \sum_{k=1}^p \br_k(u) \tr_k(y), 
\quad u \in K, \; y \in U.
\end{equation}
More recently, \cite{deeponets} replace the shallow branch and trunk nets in the operator net \eqref{eq:donet1} with deep neural networks to propose   \emph{deep operator nets} (\emph{DeepONets} in short), which are expected to be more expressive than shallow operator nets and have already been successfully applied to a variety of problems with differential equations. These include learning linear and non-linear dynamical systems and reaction-diffusion PDEs with source terms \cite[]{deeponets}, learning the PDEs governing electro-convection \cite[]{dnet2}, Navier-Stokes equations in hypersonics with chemistry \cite[]{dnet3} and the dynamics of bubble growth \cite[]{dnet4}, among others. A simple example that illustrates DeepONets \eqref{eq:donet1} and their ability to learn an operator efficiently is included in Appendix \ref{app:deeponet-linear} (cf. Figure \ref{fig:quadrature}).

Why are DeepONets able to approximate operators mapping infinite dimensional spaces, efficiently? A first answer to this question lies in a remarkable \emph{universal approximation theorem} for the operator network \eqref{eq:donet1} first proved by \cite{ChenChen1995}, and extended to DeepONets by \cite{deeponets}, where it is shown that as long as the underlying operator $\G$ is \emph{continuous} and maps a \emph{compact subset} of the infinite-dimensional space into another Banach space, there always exists an operator network of the form \eqref{eq:donet1}, that approximates $\G$ to arbitrary precision, i.e. to any given error tolerance. However, the assumptions on continuity and in particular, compactness of the input space, in the universal approximation theorem do not cover most examples of practical interest, such as many of the operators considered by \cite{deeponets}. Moreover, this universal approximation property does not provide any explicit information on the computational complexity of the operator network, i.e. no explicit knowledge of the number of sensors $m$, number of branch and trunk nets $p$ and the sizes (number of weights and biases) as well as  depths (for DeepONets) of these neural networks  can be inferred from the universal approximation property. 

Given the infinite dimensional setting, it could easily happen that the computational complexity of the DeepONet for attaining a given tolerance $\epsilon$ scales \emph{exponentially} in $1/\epsilon$. In fact, we will provide a heuristic argument strongly suggesting that such exponential scaling can not be overcome in the approximation of \emph{general} Lipschitz continuous operators (cp. Remark \ref{rem:cod}; and Thm. 2.2 of \cite{mhaskar1997neural} for related work on rigorous lower bounds). This (worst-case) scaling will be referred to as the \emph{curse of dimensionality} and can severely inhibit the efficiency of DeepONets at realistic learning tasks. Although numerical experiments presented in \cite{deeponets, dnet3,dnet2,dnet4} strongly indicate that DeepONets may not suffer from this curse of dimensionality for many cases of interest, no rigorous results to this end are available currently.  Moreover, no rigorous results on the DeepONet \emph{generalization error}, i.e., the error due to finite sampling of the input space, are available currently.

The above considerations motivate our current paper where we seek to provide rigorous and explicit bounds on the error incurred by DeepONets in approximating nonlinear operators on Banach spaces. As a first step, we extend the universal approximation theorem from continuous to \emph{measurable} operators, while removing the compactness requirements of \cite{ChenChen1995}. Next, using a very natural decomposition of DeepONets (cf Figure \ref{fig:2}) into an \emph{encoder} that maps the infinite-dimensional input space into a finite-dimensional space, an \emph{approximator} neural network that maps one finite-dimensional space into another and a trunk net induced affine \emph{reconstructor}, that maps a finite dimensional space into the infinite dimensional output space, we decompose the total DeepONet approximation error in terms of the resulting \emph{encoding, approximation and reconstruction errors} and estimate each part separately. This allows us to derive rigorous \emph{upper} as well as \emph{lower} bounds on the DeepONet error, under very general hypotheses on the underlying nonlinear operator and underlying measures on input Hilbert spaces. In particular, optimal bounds on the encoding and reconstruction errors stem from a careful analysis of the eigensystem for the covariance operators, associated with the underlying input measure. 

A similar error decomposition has been employed in \cite{STU3}, to analyze an operator learning architecture combining principal components analysis (PCA) autoencoders for the encoding and reconstruction with a neural network for the non-linear approximation step. In particular, the authors derive a quantitative error estimate for the \emph{empirical} PCA autoencoder, which is based on a \emph{finite number} of input/output samples $(u_j, \G(u_j))$, $j=1,\dots,N$. However, we need significant additional efforts to translate "PCA" based ideas into quantitative error and complexity estimates for the point-evaluation encoder and the neural network reconstruction of DeepONets (even in the limit of infinite data). Another key distinction of the present work with \cite{STU3} is a detailed discussion of the \emph{efficiency} of the DeepONet approximation, providing quantitative error and complexity bounds not only for the encoding and reconstruction steps, but also for the approximator network.

In addition to our analysis of the encoding, approximation and reconstruction errors, 
we illustrate these abstract error estimates with four prototypical differential equations, namely a nonlinear ODE with a forcing term, a linear elliptic PDE with variable diffusion coefficients, a semi-linear parabolic PDE (Allen-Cahn equation) and a quasi-linear hyperbolic PDE (scalar conservation law), thus covering a wide spectrum of differential equations with different types of inputs and different levels of Sobolev regularity of the resulting solutions. For each of these four problems, we rigorously prove that the underlying operators possess additional structure, through which DeepONets can achieve an approximation accuracy $\epsilon$ with a size that scales only algebraically in $1/\epsilon$, i.e. \emph{DeepONets can break the curse of dimensionality}, associated with the approximation of the infinite-dimensional input-to-output map. Thus providing the first rigorous proofs of their possible efficiency at operator approximation. An appropriate notion of the curse of dimensionality in the DeepONet context will be given in Definition \ref{def:cod} (cp. also Remark \ref{rem:cod}). Finally, we also provide a rigorous bound for the \emph{generalization error} of DeepONets and show that, despite the underlying infinite dimensional setting, the estimate on generalization error scales (asymptotically) as $1/\sqrt{N}$ with $N$ being the number of training samples (up to log terms), which is consistent with the standard finite dimensional bound with statistical learning theory techniques. 

The rest of the paper is organized as follows: In section \ref{sec:2}, we formulate the underlying operator learning problem and introduce DeepONets. The abstract error estimates are presented in section \ref{sec:3} and are illustrated on four concrete model problems in section \ref{sec:4}. Sections \ref{sec:3} and \ref{sec:4} focus on quantitative bounds for the \emph{best-approximation error} that is achievable, in principle, by the given DeepONet architecture; additional (generalization) errors due to the availability of only a finite number of training samples, are discussed in section \ref{sec:5}, where estimates on the DeepONet generalization error are derived. The proofs of our theoretical results are presented in the appendix. Other error sources, e.g. due to (imperfect) training algorithms such as stochastic gradient descent, errors due to uncertain and noisy data, or errors due to a mismatch between the training and evaluation data, will not be discussed in the present work. The analysis of such errors represent avenues for extensive future work. Moreover, for simplicity of the exposition, the results of the present work are formulated for neural networks with \emph{ReLU activation function}; this particular choice of activation function is, however, not essential to reach the main conclusions.

\section{Deep Operator Networks}
\label{sec:2}
Our main aim in this section is to follow \cite{deeponets} and introduce DeepONets, i.e., deep version of the shallow operator network \eqref{eq:donet1} for approximating operators. To this end, we start with a brief recapitulation of what a neural network is.
\subsection{Neural Networks.}
\label{sec:NN}
Let $\R^{d_{in}}$ and $\R^{d_{out}}$ denote the input and output spaces, respectively. Given any input vector $z \in \R^{d_{in}}$, a feedforward neural network (also termed as a multi-layer perceptron), transforms it to an output through layers of units (neurons) consisting of either affine-linear maps between units (in successive layers) or scalar non-linear activation functions within units \cite{DLbook}, resulting in the representation,
\begin{equation}
\label{eq:ann1}
\map_{\theta}(y) = C_K \circ\sigma \circ C_{K-1}\ldots \ldots \ldots \circ\sigma \circ C_2 \circ \sigma \circ C_1(y).
\end{equation} 
Here, $\circ$ refers to the composition of functions and $\sigma$ is a scalar (non-linear) activation function. A large variety of activation functions have been considered in the machine learning literature \cite{DLbook}, including adaptive activation functions in \cite{adaptive_activation}. Popular choices for the activation function $\sigma$ in \eqref{eq:ann1} include the sigmoid function, the $\tanh$ function and the \emph{ReLU} function defined by,
\begin{equation}
\label{eq:relu}
\sigma(z) = \max(z,0).
\end{equation}
In the present work, we will only consider neural networks with ReLU activation function, i.e., the term ``neural network'' should be understood synonymous with ``ReLU neural network''.

For any $1 \leq k \leq K$, we define
\begin{equation}
\label{eq:C}
C_k z_k = W_k z_k + b_k, \quad \text{for} ~ W_k \in \R^{d_{k+1} \times d_k}, z_k \in \R^{d_k}, b_k \in \R^{d_{k+1}}.
\end{equation}
For consistency of notation, we set $d_1 = d_{in}$ and $d_{K+1} = d_{out}$. 

Thus in the terminology of machine learning, the neural network \eqref{eq:ann1} consists of an input layer, an output layer and $(K-1)$ hidden layers for some $1 < K \in \N$. The $k$-th hidden layer (with $d_{k+1}$ neurons) is given an input vector $z_k \in \R^{d_k}$ and transforms it first by an affine linear map $C_k$ \eqref{eq:C} and then by a nonlinear (component wise) activation $\sigma$. A straightforward addition shows that our network contains $\left(d_{in} + d_{out} + \sum\limits_{k=2}^{K} d_k\right)$ neurons. 
We also denote, 
\begin{equation}
\label{eq:theta}
\theta = \{W_k, b_k\},
\end{equation} 
to be the concatenated set of (tunable) weights and biases for our network. It is straightforward to check that $\theta \in \Theta \subset \R^M$ with
\begin{equation}
\label{eq:ns}
M = \sum_{k=1}^{K} (d_k +1) d_{k+1}.
\end{equation}
We also introduce the following nomenclature for a deep neural network $\map$,
\begin{equation}
    \label{eq:nom}
    \size(\map_\theta):= \Vert \theta \Vert_{\ell^0}, \quad \depth(\map_\theta)= K-1,
\end{equation}
with $\Vert \theta \Vert_{\ell^0} = \#\{\theta_k \ne 0\}$ denoting the total number of \emph{non-zero} tuning parameters (weights and biases) of the neural network and $K-1$ being the number of hidden layers of the network. Henceforth, the explicit $\theta$-dependence is suppressed for notational convenience and we denote the neural network \eqref{eq:ann1} as $\map$.

\subsection{DeepONets}

A DeepONet, as proposed in \cite{deeponets} is a deep neural network extension of the operator network \eqref{eq:donet1}. Roughly speaking, the shallow branch and trunk nets in \eqref{eq:bnet1} and \eqref{eq:tnet1} are replaced by deep neural networks of the form \eqref{eq:ann1}. However, we present a slightly more general form of DeepONets in this paper, as compared to the DeepONets of \cite{deeponets}. To this end, we recall that $D \subset \R^d$ and $U \subset \R^n$ are compact domains (e.g. with Lipschitz boundary) and introduce the following \emph{operators} (cp. Figure \ref{fig:2}):
\begin{itemize}
    \item {\bf Encoder.} Given a set of \emph{sensor} points $x_j \in D$, for $1 \leq j \leq m$, we define the linear mapping,
    \begin{align} \label{eq:encoder}
\cE: C(D) \to \R^m, \quad \cE(u) = (u(x_1), \dots, u(x_m)),
\end{align}
as the \emph{encoder} mapping. Note that the encoder $\cE$ is well-defined as one can evaluate continuous functions pointwise.
\item {\bf Approximator.} Given the above sensor points $\{x_j\}$ for $1 \leq j \leq m$, the \emph{approximator} is a deep neural network of the form \eqref{eq:ann1} and defined as,
\begin{align} \label{eq:approximator}
\cA: \R^m \to \R^p, \; \{u_j\}_{j=1}^m \mapsto \{\cA_k\}_{k=1}^p,
\end{align}
Note that $d_{in}=m$ and $d_{out} = p$ in the approximator neural network $\cA$ of form \eqref{eq:ann1}. Given the encoder and approximator, we define the \emph{branch net}  $\bm{\br}: C(D) \to \R^p$   as the composition $\bm{\br}(u) = \cA\circ \cE(u)$.
\item {\bf Reconstructor.} First, we denote a \emph{trunk net}  $\bm{\tr}$ as a neural network
\[
\bm{\tr}: \R^n \to \R^{p+1}, \; y = (y_1, \dots, y_n) \mapsto \{\tr_k(y)\}_{k=0}^p,
\]
with each $\tr_k$ of the form \eqref{eq:ann1}, with $d_{in}=n$ and $d_{out}=1$ and for any $y \in U \subset \R^n$. 

Then, we define a $\bm{\tr}$-induced \emph{reconstructor} as \begin{align} \label{eq:reconstruction}
\cR = \cR_{\bm{\tr}}: \R^p \to C(U), 
\quad 
\cR_{\bm{\tr}}(\alpha_k) := \tr_0(y) + \sum_{k=1}^p \alpha_k \tr_k(y).
\end{align}

Henceforth for notational convenience, we will suppress the $\bm{\tr}$-dependence of the $\bm{\tr}$-induced reconstructor and simply label it as $\cR$. Note that the reconstructor is well-defined as the activation function $\sigma$ in \eqref{eq:ann1} is at least continuous.  
\end{itemize}

\begin{figure}
\includegraphics[width=\textwidth]{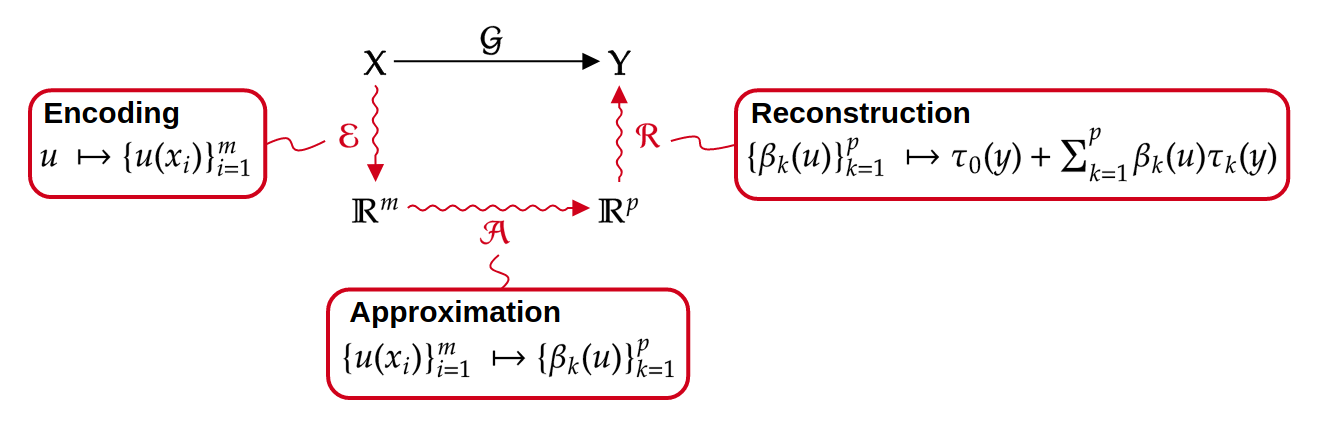}
\caption{Schematic illustration of the decomposition of a DeepONet into the encoder $\cE$, approximator $\cA$ and reconstructor $\cR$.}
\label{fig:2}
\end{figure}

Given the above ingredients, we combine them into a DeepONet as,
\begin{equation}
    \label{eq:donet}
\cN: C(D) \to C(U), \quad \cN(u)= 
(\cR \circ \cA \circ \cE)(u).
\end{equation}
I.e. a DeepONet is composed of three components:
\begin{enumerate}
\item Encoding: The encoder mapping $\cE: C(D) \to \R^m$, $u \mapsto \{u(x_j)\}_{j=1}^m$,
\item Approximation: The encoded (finite-dimensional) data is approximated by a neural network mapping $\cA: \R^m \to \R^p$,
\item Reconstruction: The result is decoded by $\cR: \R^p \to C(U)$, $\{\cA_k\}_{k=1}^p \mapsto \tr_0 + \sum_{k=1}^p \cA_k \tr_k$, with $\bm{\tr}$ being the trunk net. 
\end{enumerate}

A graphical depiction of the constituent parts of a DeepONet is shown in figure \ref{fig:2}. 
The only difference between our version of the DeepONet \eqref{eq:donet} and the version presented in the recent paper \cite{deeponets}, lies in the fact that we use a more general affine reconstruction step. In other words, setting $\tr_0(y) \equiv \tr_0$, for some $\tr_0 \in \R$ in \eqref{eq:reconstruction} recovers the DeepONet of \cite{deeponets}. We remark in passing that although the above formulation assumes a mapping between (scalar) functions, $\cN: C(D) \to C(U)$, all results in this work extend trivially to the more general case of DeepONet approximations for \emph{systems} $\cN: C(D;\R^{d_u}) \to C(U;\R^{d_v})$. For clarity of the exposition and simplicity of notation, we will focus on the case $d_u=d_v=1$, in the following.

We recall that the DeepONet \eqref{eq:donet} contains parameters corresponding to the weights and biases of the approximator neural network $\cA$ and the trunk net $\bm{\tr}$, that need to be tuned (\emph{trained}) such that the DeepONet \eqref{eq:donet} approximates the underlying operator $\G: X \to Y$. To this end, we need to define a distance between $\G$ and the DeepONet $\cN$. A natural way to do this, is to fix a probability measure $\mu \in \P(X)$, and to consider the following error, measured in the $L^2(\mu)$-\emph{norm}:
\begin{align} \label{eq:approxerr}
\Err
=
\left(
\int\limits_{X} \int\limits_{U}
\left|
\G(u)(y) 
-
\cN(u)(y)
\right|^2 
\, dy
\, d\mu(u)
\right)^{1/2},
\end{align}
with $\cN$ being the DeepONet \eqref{eq:donet}. Note that we have replaced the function spaces $C(D)$ and $C(U)$ by more general function spaces $X$ and $Y$, for which we will assume that there exists an embedding $X \embeds L^2(D)$, $Y\embeds L^2(U)$. In particular, for the error \eqref{eq:approxerr} to be well-defined, it suffices that \begin{itemize}
\item there exists a Borel set $A \subset X$, such that $\mu(A) = 1$, and $A \subset C(D)$ so that $\cE(u) = (u(x_1),\dots,u(x_m))$ is well-defined on $A$, 
\item The mapping
\[
\G: X \to Y,
\]
maps given data $u\in X$ to a $L^2(U)$ function $v(y) = \G(u)(y)$ defined on $U$, and $\G\in L^2(\mu) := L^2(\mu; \Vert \slot \Vert_{L^2(U)})$, in the sense that 
\[
\int_{X} \Vert \G(u) \Vert_{L^2(U)}^2 \, d\mu(u) 
< \infty.
\]
\end{itemize}

We formalize these concepts with the following definition:

\begin{definition}[Data for DeepONet approximation] 
\label{def:data}
Let $D \subset \R^d$, $U \subset \R^n$ be bounded domains. Let $X$, $Y$ be separable Banach spaces with a continuous embedding $\iota: \, X\embeds L^2(D)$ and $\bar{\iota}: \, Y\embeds L^2(U)$. We call $\mu$, $\G$ \define{data for the DeepONet approximation problem}, provided $\mu \in \P_2(X)$ is a Borel probability measure on $X$, there exists a Borel set $A \subset X$, such that $\mu(A) = 1$ and $A$ consists of continuous functions, and $\G: X \to Y$ is a Borel measurable mapping, such that $\G \in L^2(\mu)$, i.e. $\int_X \Vert \G(u) \Vert_{L^2(U)}^2 \, d\mu(u) < \infty$. Here, $\P_2(X)$ is the set of probability measures with finite second moments $\int_X \Vert u \Vert_X^2 \, d\mu(u) < \infty$.
\end{definition}

\begin{remark}
The setting considered in Definition \ref{def:data} corresponds to a ``perfect data setting'', in which the input/output pairs $(u,\G(u))$, $u\sim \mu$ are provided exactly, i.e. in the absence of measurement noise and uncertainty. Furthermore, the main focus of this work will be on the problem of finding complexity bounds on the best-approximation provided by the DeepONet architecture \eqref{eq:encoder}--\eqref{eq:reconstruction}. Concerning the \emph{finite data setting}, i.e. when only a finite number of input/output pairs $(u_1,\G(u_1)), \dots, (u_N,\G(u_N))$, $u_j \sim \mu$ are available, first results on the generalization error will be presented in Section \ref{sec:5}.
\end{remark}

In the framework of nonlinear operators that arise in differential equations, the Banach spaces $X$ and $Y$ will be function spaces on $D$ and $U$, respectively; a typical example is $X$, $Y = H^s(D)$ for some $s\ge 0$, where $H^s(D)$ denotes the $L^2$-based Sobolev space on $D$. The embeddings $\iota: \, X \embeds L^2(D)$, $\bar{\iota}: \, Y \embeds L^2(U)$ will thus be canonical, and henceforth, we will identify $X \simeq \iota(X)$, $Y\simeq \bar{\iota}(Y)$ as subsets of $L^2(D)$ and $L^2(U)$, respectively.

\begin{remark}
A technical difficulty associated with DeepONets arises due to the specific form of the encoder \eqref{eq:encoder}, which is defined via point-wise evaluations. In principle, the components of this encoder could easily be replaced by more general functionals of $u$, and in fact, this might be more natural in certain settings (e.g. to model physical measurements; or for mathematical reasons, see Section \ref{sec:scl}). For our general discussion, we will focus instead on encoders of the particular form \eqref{eq:encoder}. The main reason for this choice is the possibility for direct comparison with the numerical experiments of \cite{deeponets}, which are based on the DeepONet architecture \eqref{eq:encoder}--\eqref{eq:reconstruction}. Furthermore, fixing a particular choice will allow us to analyse the encoding error associated with $\cE$ in great detail in Section \ref{sec:encoding} (see Section \ref{sec:over-enc} for an overview).
\end{remark}

Since the point-wise encoder $\cE(u)= (u(x_1),\dots, u(x_m))$ is not well-defined on spaces such as $L^2(D)$, we first show that \eqref{eq:approxerr} is nevertheless well-defined. From Lemma \ref{lem:app1}, we infer that if $A\subset X$ is a Borel measurable set such that $A \subset C(D)$, then $\cN = \cR\circ \cA \circ \cE: A \to L^2(U)$ is measurable, and possesses a measurable extension $\cR\circ \cA \circ \bar{\cE}: L^2(D) \to L^2(U)$. Clearly, since $\mu(A) = 1$, we then have
\begin{align*}
\int_{X} \int_{U}
\big|
\G(u)(y) 
-
&(\cR \circ \cA \circ \bar{\cE})(u)(y)
\big|^2 
\, dy
\, d\mu(u)
\\
&=
\int_{A} \int_{U}
\left|
\G(u)(y) 
-
(\cR \circ \cA \circ \cE)(u)(y)
\right|^2 
\, dy
\, d\mu(u),
\end{align*}
for any extension $\bar{\cE}$. This allows us to define the error \eqref{eq:approxerr} uniquely and allows to formulate the following precise definition of DeepONets,
\begin{definition}[DeepONet]
Let $\mu$, $\G$ be given data for the DeepONet approximation problem (see Definition \ref{def:data}). A \define{DeepONet} $\cN$, approximating the nonlinear operator $\G$, is a mapping $\cN: C(D) \to L^2(U)$ of the form $\cN = \cR \circ \cA \circ \cE$, where $\cE: (X,\Vert \slot \Vert_X) \to (\R^m, \Vert \slot \Vert_{\ell^2})$ denotes the encoder given by \eqref{eq:encoder}, $\cA: (\R^m, \Vert \slot \Vert_{\ell^2}) \to (\R^p, \Vert \slot \Vert_{\ell^2})$ denotes the approximator network \eqref{eq:approximator}, and $\cR: (\R^p, \Vert \slot \Vert_{\ell^2}) \to (L^2(U), \Vert \slot \Vert_{L^2(U)})$ denotes the reconstruction of the form \eqref{eq:reconstruction}, induced by the trunk net $\bm{\tr}$. 
\end{definition}
\section{Error bounds for DeepONets}
\label{sec:3}

Our aim in this section is to derive bounds on the error \eqref{eq:approxerr} incurred by the DeepONet \eqref{eq:donet} in approximating the underlying nonlinear operator $\G$.

\subsection{A universal approximation theorem} 
As a first step in showing that the DeepONet error \eqref{eq:approxerr} can be small, we have the following \emph{universal approximation theorem}, that generalizes the universal approximation property of \cite{ChenChen1995} to significantly more general nonlinear operators,
\begin{theorem} \label{thm:uni}
Let $\mu \in \P(C(D))$ be a probability measure on $C(D)$. Let $\G: C(D) \to L^2(D)$ be a Borel measurable mapping, with $\G\in L^2(\mu)$, then for every $\epsilon > 0$, there exists an operator network $\cN = \cR \circ \cA \circ \cE$, such that
\[
\Vert \G - \cN \Vert_{L^2(\mu)}
=
\left(
\int_X \Vert \G(u) - \cN(u) \Vert_{L^2(U)}^2 \, d\mu(u) \right)^{1/2} < \epsilon.
\]
\end{theorem}

The proof of this theorem is based on an application of Lusin's Theorem to approximate measurable maps by continuous maps on compact subsets and then using the universal approximation theorem of \cite{ChenChen1995}. It is presented in detail in Appendix \ref{app:pf31}. 

\begin{remark}
The universal approximation theorem of \cite{ChenChen1995} states that DeepONets can approximate continuous operators $\G$ uniformly over compact subsets $K\subset X$. In contrast, the above theorem removes both the compactness constraint, as well as the continuity assumption on $\G$ and paves the way for the theorem to be applied in realistic settings, for instance in the approximation of nonlinear operators that arise when considering differential equations such as those of \cite{deeponets} and later in this paper. However, this extension comes at the expense of considering a weaker distance ($L^2(\mu \otimes dy)$ vs. $L^\infty(\mu \otimes dy)$) than in \cite{ChenChen1995}. In practice, it is indeed the $L^2$-distance that is minimized during the training process. Moreover, the above theorem also allows us to consider cases of practical interest where $\mu$ is supported on an \emph{unbounded} subset, as is e.g. the case when $\mu$ is a non-degenerate Gaussian measure. Indeed, in most of the numerical examples in \cite{deeponets}, the underlying measure $\mu$ is a Gaussian measure given by the law of a Gaussian random field.
\end{remark}
The universal approximation theorem \ref{thm:uni} shows that for any given tolerance $\epsilon$, there exists a DeepONet of the form \eqref{eq:donet} such that the resulting approximation error \eqref{eq:approxerr} is smaller than this tolerance. However, this theorem does not provide any explicit information about the number of sensors $m$, the number of branch and trunk net outputs $p$ or the hyperparameters of the approximator neural network $\cA$ and the trunk net $\bm{\tr}$. As discussed in the introduction, these numbers specify the complexity of a DeepONet and we would like to obtain explicit bounds (information) on the computational complexity of a DeepONet for achieving a given error tolerance and ascertain whether DeepONets are efficient at approximating a given nonlinear operator $\G$. In practice, we are thus interested in deriving \emph{quantitative} error and complexity bounds for the DeepONet approximation of operators. This will be the focus of the remainder of the present section.

\subsection{Overview of quantitative error bounds}
\label{sec:overview}

We will first provide an overview of the main results on quantitative error bounds derived in the present work.  An extended discussion of these results can be found in the following subsections, which include detailed derivations and proofs.

\subsubsection{Error decomposition and the curse of dimensionality}

Given the decomposition of the DeepONet \eqref{eq:donet} into an encoder $\cE$, approximator $\cA$ and reconstructor $\cR$, it is natural to expect that the total error \eqref{eq:approxerr} also decomposes into errors associated with them. For a given encoder $\cE$ and reconstructor $\cR$, we can define (approximate) inverses $\cD$ (the \emph{decoder}) and $\cP$ (the \emph{projector}), which are required to satisfy the following relations exactly
\[
\cE \circ \cD = \Id: \, {\R^m \to \R^m}, \quad
\cP \circ \cR = \Id: \, {\R^p \to \R^p},
\]
and should satisfy
\[
\cD \circ \cE \approx \Id: \, {X\to X},
\quad
\cR \circ \cP \approx \Id: \, {Y\to Y}.
\]
We note that $\cD$ and $\cP$ are not necessarily unique, and need to be chosen. All mappings are illustrated in the following diagram:
\[
\begin{tikzcd}
L^2(D) 
\arrow[r, "\G"] 
\arrow[d, shift right=0ex, "\cE" left, bend right=30] 
& L^2(U) 
\arrow[d, shift right=0ex, "\cP" left, bend right=30] 
\\
\R^m 
\arrow[u, shift right=0ex, "\cD" right, bend right=30] 
\arrow[r, "\cA"]   
& \R^p 
\arrow[u, shift right=0ex, "\cR" right, bend right=30]
\end{tikzcd}
\]
Given choices for the decoder $\cD$ and the projector $\cP$, we can now define the \define{encoding error} $\Err_{\cE}$, the \define{approximation error} $\Err_{\cA}$, and the \define{reconstruction error} $\Err_{\cR}$, respectively, as follows:
\begin{align}
\Err_{\cE}
&:=
\left(\int_{X}
\Vert \cD \circ \cE(u) - u \Vert_{X}^2 
\, d\mu(u)\right)^{\frac{1}{2}}, 
\label{eq:encoding}
\\
\Err_{\cA}
&:=
\left(\int_{\R^m}
\Vert \cA(\bm{u}) - \cP \circ \G \circ \cD(\bm{u}) \Vert_{\ell^2(\R^p)}^2 
\, d(\cE_\#\mu)(\bm{u})\right)^{\frac{1}{2}}
\label{eq:approximation}
\\
\Err_{\cR}
&:= 
\left(\int_{L^2(U)}
\Vert \cR \circ \cP(u) - u \Vert_{L^2(U)}^2 
\, d(\G_\#\mu)(u)\right)^{\frac{1}{2}}.
\label{eq:decoding}
\end{align}
We could also have written these errors as $\Err_{\cE} = \left\Vert \cD \circ \cE - \Id \right\Vert_{L^2(\mu)}$, $\Err_{\cA} = \Vert \cA - \cP \circ \G \circ \cD \Vert_{L^2(\cE_\#\mu)}$, and $\Err_{\cR} = \left\Vert \cR \circ \cP - \Id \right\Vert_{L^2(\G_\#\mu)}$.
Intuitively, the encoding and reconstruction errors, $\Err_{\cE}$ and $\Err_{\cR}$, measure the loss of information by the DeepONet's finite-dimensional encoding of the underlying infinite-dimensional spaces; these error sources are weighted by the input measure $\mu$ on the input side, and the push-forward measure $\G_\#\mu$ on the output side, respectively. The approximation error $\Err_{\cA}$ measures the error due to the approximation by the approximator network $\cA: \R^m \to \R^p$ of the ``encoded/projected'' operator, $G = \cP\circ \G \circ \cD: \R^{m} \to \R^{p}$. 

To state our main result on the error decomposition of $\Err$ in terms of $\Err_{\cE}$, $\Err_{\cA}$ and $\Err_{\cR}$, we first recall the following notation for a mapping $\cF: X \to Y$ between arbitrary Banach spaces $X,Y$:
\[
\Lip_\alpha(\cF: X\to Y) 
:=
\sup_{u,u'\in X}
\frac{\Vert \cF(u) - \cF(u') \Vert_{Y}}{\Vert u - u' \Vert_{X}^\alpha},
\quad
\Lip(\cF) := \Lip_1(\cF).
\]
We then have the following error estimate, whose proof can be found in Appendix \ref{app:errdec}:

\begin{theorem} 
\label{thm:err-upper-bound}
Consider the setting of Definition \ref{def:data}. Let the nonlinear operator $\G: A \subset X \to Y$ be $\alpha$-H\"older continuous (or Lipschitz continuous if $\alpha = 1$), where $X\embeds L^2(D)$, $Y\embeds L^2(U)$. Choose an arbitrary encoder $\cE: C(D) \to \R^m$, approximator $\cA: \R^m \to \R^p$ and an arbitrary reconstruction $\cR: \R^p \to L^2(U)$, of the form \eqref{eq:reconstruction}. Then the error \eqref{eq:approxerr} associated with the DeepONet $\cN = \cR \circ \cA \circ \cE$ satisfies the following upper bound,
\begin{equation}
    \label{eq:donetbd}
    \Err 
    \leq 
    \Lip_\alpha(\G) \Lip(\cR\circ \cP) (\Err_{\cE})^\alpha +  \Lip(\cR) \Err_{\cA} + \Err_{\cR}.
\end{equation}
\end{theorem}
The last theorem shows that the total error $\Err$ is indeed controlled by $\Err_{\cE}$, $\Err_{\cA}$ and $\Err_{\cR}$. 
Furthermore, this theorem provides us with a clear strategy for estimating the DeepONet error \eqref{eq:approxerr} for a concrete operator $\G$ of interest: 
\begin{itemize}
\item First, we will bound the encoding/reconstruction errors, providing suitable estimates for 
\[
\cD \circ \cE \approx \Id, 
\quad
\cR \circ \cP \approx \Id.
\]
In this step, we need to choose  $\cD$, $\cE$, $\cR$, $\cP$ in order to minimize the resulting encoding and reconstruction errors. 
\item In a second step, we estimate the approximation error 
\[
\Vert \cA - \cP\circ \G \circ \cD \Vert_{L^2(\cE_\#\mu)},
\]
for fixed projector $\cP$ and decoder $\cD$. The second step boils down to the conventional approximation of a function $G: \R^m \to \R^p$ by neural networks.
\end{itemize}
Our goal will be to analyze the total error $\Err$ in terms of this decomposition, with the aim of showing the \emph{efficiency} of the DeepONet approximation for a wide range of operators of interest. To this end, we will first need to discuss a suitable notion of ``efficiency'', which is motivated by the following remark.

\begin{remark} \label{rem:cod}
As shown above, the error introduced by the approximation step in the DeepONet decomposition is naturally related to the error in the approximation of a high-dimensional mapping $G = \cP \circ \G \circ \cD: \R^m \to \R^p$ by the neural network $\cA: \R^m \to \R^p$. The relevant function $G$ can be thought of as a finite-dimensional projection of the operator $\G$. In particular, $G$ inherits regularity properties of $\G$ such as Lipschitz continuity. As shown in \cite[Theorem 1]{yarotskynew}, the approximation of a \emph{general} Lipschitz continuous function to accuracy $\sim \epsilon$, requires a ReLU network of size $\gtrsim \epsilon^{-m/2}$, and hence suffers from the \emph{curse of dimensionality} in high dimensions, $m\gg 1$. In the context of DeepONets, we recall that $m$ is the number of sensors used in the encoding step $u \mapsto \cE(u) = (u(x_1),\dots, u(x_m))$ of the DeepONet architecture. Achieving a small error of order $\sim \epsilon$ in this encoding step requires $m = m(\epsilon)$ to depend on $\epsilon$, with $m(\epsilon) \to \infty$ as $\epsilon \to 0$. Therefore, general neural network approximation results indicate that the required DeepONet complexity for the approximation $\cN \approx \G$ of an \emph{arbitrary} Lipschitz continuous operator $\G$ to a given accuracy $\epsilon$ requires at least $\size(\cN) \gtrsim \epsilon^{-m(\epsilon)/2}$. In particular, this scaling is faster than \emph{any algebraic rate} in $\epsilon^{-1}$. This connection between the curse of dimensionality, as commonly understood in the literature \cite[]{cod1,CDS2011,cod2}, and this worse-than-algebraic asymptotic growth of DeepONet size in $\epsilon^{-1}$ provides an appropriate notion of the curse of dimensionality in the infinite-dimensional DeepONet context.
\end{remark}
Given Remark \ref{rem:cod} above, we can say that ``DeepONets break the curse of dimensionality'' in the approximation of a given operator $\G$, if, for any accuracy $\epsilon > 0$, there exists a DeepONet which achieves an approximation error $\Err\le \epsilon$ (cp. \eqref{eq:approxerr}), with a complexity which scales \emph{at most algebraically} in $\epsilon^{-1}$. 
This key concept with respect to computational complexity and efficiency of DeepONet approximation is made precise below:
\begin{definition}
[Curse of Dimensionality for DeepONets]
\label{def:cod}
 Given a DeepONet $\cN$ \eqref{eq:donet}, we define the \emph{size} of the DeepONet as the sum of the sizes of the approximator neural network $\cA$ and the trunk net $\bm{\tr}$, i.e. $\size(\cN) = \size(\cA) + \size(\bm{\tr})$ (cp. \eqref{eq:nom}). For a given tolerance $\epsilon > 0$, let $\cN_{\epsilon}$ be a DeepONet such that the error $\Err$ \eqref{eq:approxerr} is less than $\epsilon$, and
\begin{equation}
    \label{eq:cod1}
    \size\left(\cN_{\epsilon}\right) \sim {\mathcal O}\left(\epsilon^{-\vartheta_{\epsilon}}\right),
\end{equation}
for some $\vartheta_{\epsilon} \geq 0$. Note that the universal approximation theorem \ref{thm:uni} guarantees the existence of such a $\cN_{\epsilon}$ and $\vartheta_{\epsilon}$ for every $\epsilon > 0$.

The DeepONet approximation of a nonlinear operator $\G$, with underlying measure $\mu$ (check from Definition \ref{def:data}) is said to incur a \emph{curse of dimensionality}, if
\begin{equation}
    \label{eq:cod}
    \lim\limits_{\epsilon \to 0} \vartheta_{\epsilon} = +\infty.
\end{equation}
On the other hand, the DeepONet approximation is said to \emph{break the curse of dimensionality} if there exist DeepONets $\cN_{\epsilon}$ such,
\begin{equation}
    \label{eq:bcod}
    \lim\limits_{\epsilon \to 0} \vartheta_{\epsilon} = \overline{\vartheta} < +\infty.
\end{equation}
\end{definition}

This definition emphasizes the fundamental role played by bounds on the size of the DeepONet for obtaining a certain level of error tolerance. We provide such explicit bounds later in this paper. In the following subsections, we survey our main results on the reconstruction error $\Err_{\cR}$, the encoding error $\Err_{\cE}$ and the approximation error $\Err_{\cA}$.

\subsubsection{Reconstruction error}
\label{sec:over-rec}

The reconstruction error $\Err_{\cR}$ \eqref{eq:reconstruction} is intimately related to the eigenfunctions and eigenvalues of the covariance operator $\Gamma_{\G_\#\mu}$ of the push-forward measure $\G_\#\mu$,
\begin{align} \label{eq:covGmu}
\Gamma_{\G_\#\mu} 
=
\int_{Y} (v-\E[v])\otimes (v-\E[v]) \, d(\G_\#\mu)(v),
\end{align}
where $\E[v] = \int_{Y} v \, d(\G_\#\mu)(v)$ denotes the mean of $\G_\#\mu$. General results on the relation between the eigenstructure of covariance operators and optimal projections onto finite-dimensional linear and affine subspace are presented in Section \ref{sec:projection}, below. Based on these results, it will be shown that 
\begin{enumerate}
\item for any affine reconstruction $\cR: \R^p \to L^2(U)$, there exists a (unique) optimal projection $\cP: L^2(U) \to \R^p$, and that this optimal $\cP$ is itself \emph{affine} (cp. Lemma \ref{lem:opt-proj}), 
\item among all affine reconstructions $\cR: \R^p \to L^2(U)$, there exists an optimal choice $\cR_\opt: \R^p \to L^2(U)$, which achieves the minimum reconstruction error: $\Err_{\cR_\opt} = \min_{\cR} \Err_{\cR}$ (cp. Theorem \ref{thm:rec-opt}).
\end{enumerate}
As a consequence of the analysis of the optimal reconstruction, we can then derive the following \emph{lower bound} on the reconstruction and total approximation errors:
\begin{theorem} \label{thm:lowerbd}
Consider the setting of Definition \ref{def:data}, let $\G: X\to L^2(U)$ be an operator. Let $\cN = \cR\circ \cA \circ \cE: X \to L^2(U)$ be an arbitrary DeepONet approximation of $\cG$, with encoder $\cE: C(D)\to \R^m$, approximator $\cA: \R^m\to \R^p$ and reconstruction $\cR: \R^p \to L^2(U)$. Let $\cP: L^2(U) \to \R^p$ be the optimal affine projection associated with $\cR$. Then the total error $\Err$ and the reconstruction error $\Err_{\cR}$ can be estimated \emph{from below} by
\begin{align} \label{eq:donetlbd}
\sqrt{\sum_{k>p} \lambda_k} \le \Err_{\cR} \le \Err,
\end{align}
in terms of the eigenvalues $\lambda_1\ge \lambda_2 \ge \dots$ of the covariance operator $\Gamma_{\G_\#\mu}$ associated with the push-forward measure $\G_\#\mu$.
\end{theorem}
Theorem \ref{thm:lowerbd} provides a definite, a priori limitation on the best error which can be achieved by a DeepONet approximation. The proof of this theorem can be found on page \pageref{pf:lowerbd}. 

The next goal is to provide \emph{upper} bounds on $\Err_{\cR}$ for a suitably chosen trunk net reconstruction $\cR$ \eqref{eq:reconstruction}. As mentioned in point (2) above, in principle, there exists a provably \emph{optimal} choice $\cR_\opt$ among all affine reconstructions. However, given that the trunk net basis functions $\tr_0,\dots, \tr_p$ are represented by \emph{neural networks}, this optimal choice $\cR_\opt$ cannot in general be represented exactly by the trunk net reconstruction $\cR$, leading to an additional contribution to the reconstruction error, which depends on how well the eigenfunctions of the covariance operator can be approximated by neural networks (cp. Proposition \ref{prop:err-rec-approx}). 

As will be discussed in Section \ref{sec:decayspec}, due to the distortion by $\G$, the eigenstructure of the push-forward measure $\G_\#\mu$ can be very different from that of $\mu$. In fact, even if $\mu$ has exponentially decaying spectrum of the covariance operator $\Gamma_\mu$, the push-foward under $\G$ can destroy such high rates of decay of the eigenvalues of $\Gamma_{\G_\#\mu}$ (cp. Proposition \ref{prop:cexample}). Thus, with the exception of \emph{linear} operators (cp. Proposition \ref{prop:lin}), the eigenstructure of $\G_\#\mu$ may to depend in a very complicated way on both $\G$ and $\mu$, for \emph{non-linear} $\G$, making it difficult to analyze $\Gamma_{\G_\#\mu}$.

As it can be very difficult to obtain the necessary information on the eigenfunctions needed to \emph{quantify} their approximability by the trunk net $\bm{\tr}$, we will discuss an alternative way to obtain estimates on $\Err_{\cR}$, in Section \ref{sec:decoding}. This alternative relies on a comparison principle with a given (non-optimal) reconstructor $\tR: \R^p \to L^2(U)$ (cp. Lemma \ref{lem:err-rec-comp}). As one concrete application, this general comparison principle is applied to the reconstructor $\tR = \cR_\Fourier$, obtained by expansion in the standard Fourier basis. This allows us to derive the following quantitative upper reconstruction error and complexity estimate, which depends only on the \emph{average smoothness} of the output functions $\G(u) \in L^2(U)$ (stated in the periodic setting, for simplicity):
\begin{theorem}
\label{thm:smoothness-NN}
If $\G$ defines a Lipschitz mapping $\G: X \to H^s(\T^n)$, for some $s>0$, with
\[
\int_{X} \Vert \G(u) \Vert_{H^s}^2 \, d\mu(u) \le M < \infty,
\]
then there exists a constant $C=C(n,s,M)>0$, such that for any $p\in \N$, there exists a trunk net $\bm{\tr}: \R^n \to \R^p$ (with bias term $\tr_0\equiv 0$), with 
\begin{gather} \label{eq:smoothness-NN-size}
\begin{gathered}
\size(\bm{\tr}) \le Cp(1+\log(p)^2), \\
\depth(\bm{\tr}) \le C(1+\log(p)^2),
\end{gathered}
\end{gather}
and such that the associated reconstruction $\cR: \R^p\to L^2(\T^n)$, 
$\cR(\alpha) = \sum_{k=1}^p \alpha_k \tr_k$ satisfies
\begin{align} \label{eq:smoothness-NN}
\Err_{\cR} \le C p^{-s/n}.
\end{align}
Furthermore, the reconstruction $\cR$ and the associated optimal projection $\cP$ satisfy $\Lip(\cR), \, \Lip(\cP) \le 2$.
\end{theorem}

Theorem \ref{thm:smoothness-NN} follows from the discussion in Section \ref{sec:over-rec-smooth}. Thus, \eqref{eq:smoothness-NN} provides us with a quantitative algebraic rate of decay for the reconstruction error \eqref{eq:decoding} as long as the size of the trunk net scales as in \eqref{eq:smoothness-NN-size} and the nonlinear operator $\G$ maps onto the Sobolev space $H^s$. Such nonlinear operators arise frequently in PDEs as we will see in Section \ref{sec:4}. For further details and extended analysis of the reconstruction error, we refer to Section \ref{sec:over-rec}.

\subsubsection{Encoding error}
\label{sec:over-enc}

Next, we aim to bound the encoding error \eqref{eq:encoding}, associated with the DeepONet \eqref{eq:donet}. Full details and an extended discussion will be given in Section \ref{sec:encoding}. 

We observe from \eqref{eq:encoding} that the encoding error \emph{does not depend} on the nonlinear operator $\G$, but only depends on the underlying probability measure $\mu$. Following the architecture used by \cite{deeponets}, we have fixed the form of the encoder $\cE$ \eqref{eq:encoder} to be the point-wise evaluation of the input functions at \emph{sensors}, i.e. points $\{x_j\} \subset D$ with $1 \leq j \leq m$. Thus, key objectives of our analysis are to determine suitable choices of sensors for a fixed $m$, as well as to find the appropriate form of a decoder $\cD$ in order to minimize the encoding error \eqref{eq:encoding}. To this end, we start with a result that provides a lower bound on the encoding error:
\begin{theorem}
\label{thm:enlb}
Let $\mu$ be a probability measure on $X = L^2(D)$ with $\int_X \Vert u \Vert_{L^2_x}^2 \, d\mu(u) < \infty$ and $\int_X u \, d\mu(u) = 0$. If $\cE: X \to \R^m$ and $\cD: \R^m \to X$ are any encoder/decoder pair with a \emph{linear} decoder $\cD$, then
\[
\int_X 
\Vert \cD \circ \cE - \Id \Vert_{L^2_x}^2 
\, 
d\mu
\ge 
\sum_{k>m} \lambda_k.
\]
Here, $\lambda_k$ refers the $k$-th eigenvalue of the covariance operator $\Gamma_\mu$ associated with the measure $\mu$. 
In particular, we then have the lower bound
\begin{equation}
    \label{eq:enlb}
\Err_{\cE} \ge \sqrt{\sum_{k>m} \lambda_k}.
\end{equation}
\end{theorem}
The proof is presented in Appendix \ref{app:pf37}. The bound \eqref{eq:enlb} provides a lower bound on the encoding error and connects this error to the spectral decay of the underlying covariance operator, at least for linear decoders. 

Our next aim is to derive upper bounds on the encoding error. To illustrate our main ideas, we will restrict our discussion to the case $X=L^2(D)$. The results are readily extended to more general spaces, such as Sobolev spaces $X = H^s(D)$ for $s>0$. We fix a probability measure $\mu$ on $L^2(D)$, and write the covariance operator $\Gamma_\mu$ as an eigenfunction decomposition,
\[
\Gamma_\mu = \sum_{\ell=1}^\infty \lambda_\ell (\phi_\ell \otimes \phi_\ell),
\]
where $\lambda_1 \ge \lambda_2 \ge \dots$ are the decreasing eigenvalues, and such that the $\phi_\ell$ are an orthonormal basis of $L^2(D)$. We will assume that all $\phi_\ell$ are continuous functions, so that point-wise evaluation of $\phi_\ell$ makes sense. In this case, it can be shown (cp. Lemma \ref{lem:aliasing-general}) that the encoding error is composed of the optimal lower bound \eqref{eq:enlb}, and an additional \emph{aliasing} contribution, due to the encoding in terms of point-evaluations:
\[
(\Err_{\cE})^2 = \sum_{\ell>m} \lambda_\ell + (\Err_{\mathrm{aliasing}})^2 .
\]
An explicit expression for $\Err_{\mathrm{aliasing}}$ can be given (cp. \eqref{eq:aliaserr}), but finding sensor locations $x_1,\dots, x_m$ which minimize this aliasing error contribution appears to be very difficult, in general. We therefore propose to replace the sensors $x_1,\dots, x_m$ by \emph{random} sensor locations $X_1, \dots, X_M$ (iid, uniformly distributed over the domain $D$), and study the corresponding \emph{random encoder},
\[
\cE(u) = (u(X_1), \dots, u(X_M)).
\]
Surprisingly, it can be shown that this random encoder can be close to optimal, as is made precise by the following theorem:
\begin{theorem}
\label{thm:random-enc}
If the eigenbasis of the uncentered covariance operator $\overline{\Gamma}$ \eqref{eq:covb}, associated with the underlying measure $\mu$, is bounded in $L^\infty$, then there exists a constant $C\ge 1$, depending only on $\sup_{\ell\in \N} \Vert \phi_\ell\Vert_{L^\infty}$ and $|D|$, such that the encoding error \eqref{eq:encoding} with $M$ random sensors satisfies (for almost all $M \in \N$):
\[
\Err_{\cE}(X_1,\dots, X_M)
\le C \sqrt{\sum_{\ell > M/C\log(M)} \lambda_\ell},
\quad
(\text{as } M \to \infty),
\]
with probability $1$ in the iid random sensors $X_1,X_2,X_3,\dots \sim \Unif(D)$. 
\end{theorem}

Thus, for an underlying measure $\mu$ whose covariance operator has a bounded eigenbasis, then as the number of randomly chosen sensors increases, the resulting encoding error goes to zero almost surely. The rate of decay only depends on the spectral decay of the covariance operator. Moreover, given the lower bound \eqref{eq:enlb} on the encoding error, we have the surprising result that \emph{randomly chosen sensor points} lead to an optimal (up to a $\log$) decay of the encoding error \eqref{eq:encoding}, corresponding to the DeepONet \eqref{eq:donet}. We refer the interested reader to the discussion leading to Lemma \ref{lem:aliasing-gen}, page \pageref{lem:aliasing-gen}, for precise details of the derivation of Theorem \ref{thm:random-enc}.

In addition to the above general results, we also consider the specific case of an input measure $\mu \in \cP(L^2(\T^d))$, which is given as the \emph{law} of a random field of the form
\[
u(x;Y) = \bar{u}(x) + \sum_{k\in \Z^d} Y_k \alpha_k \fb_k(x),
\]
where $\{\fb_k(x)\}_{k\in \Z^d}$ denotes the trigonometric basis (cp. Appendix \ref{app:Fourier}), $Y_k \in [-1,1]$ are centered random variables, and the coefficients $\alpha_k\ge 0$ satisfy a decay of the form $\alpha_k \lesssim \exp(-\ell |k|_\infty)$, for all $k\in \Z^d$, for a fixed ``length-scale'' $\ell>0$. In this case, we study the encoder $\cE$ obtained by evaluation at sensor locations on an \emph{equidistant grid} on $\T^d$. Given this setting, we show that there exists a decoder $\cD$, such that the corresponding encoding error $\Err_{\cE}$ can be estimated by an exponential upper bound, $\Err_{\cE} \lesssim \exp(-c\ell m^{1/d})$ (cp. Theorem \ref{thm:holomorphic-encoding}). We refer to Section \ref{sec:parametrized-mu} for the precise details.

\subsubsection{Approximation error}
\label{sec:over-approx}
Given a particular choice of encoder/decoder and reconstruction/projection pairs $(\cE,\cD)$ and $(\cR,\cP)$, the approximation error $\Err_{\cA}$ \eqref{eq:approximation} for the approximator $\cA: \R^m \to \R^p$ in the DeepONet \eqref{eq:donet} is a measure for the non-commutativity of the following diagram:
\[
\begin{tikzcd}
L^2(D) 
\arrow[r, "\G"] 
\arrow[d, shift right=0ex, "\cE" left, bend right=30, dashrightarrow] 
& L^2(U) 
\arrow[d, shift right=0ex, "\cP" left] 
\\
\R^m 
\arrow[u, shift right=0ex, "\cD" right] 
\arrow[r, rightsquigarrow, "\cA"]   
& \R^p 
\arrow[u, shift right=0ex, "\cR" right, dashrightarrow, bend right=30]
\end{tikzcd}
\]
I.e., it measures the error in the approximation $\cA \approx G = \cP\circ\G\circ \cD$. Thus, bounding the approximation error $\Err_{\cA}$ can be viewed as a special instance of the general problem of the neural network approximation of a high-dimensional mapping $G: \R^m \to \R^p$. As already pointed out in Remark \ref{rem:cod}, relying on ``mild'' regularity properties of $G$ (or $\G$), such as Lipschitz continuity, leads to complexity bounds which suffer from the curse of dimensionality (cp. Definition \ref{def:cod}, and Section \ref{sec:regular-high-dim}, below). 

Given this possible curse of dimensionality in bounding the approximation error \eqref{eq:approximation} for Lipschitz continuous maps, we seek to find a class of nonlinear operators $\cG$ for which this curse of dimensionality can be avoided:
\begin{itemize}
\item One possible class is the class of \emph{holomorphic mappings} $[-1,1]^\N \to V$, $\bm{y} \mapsto \cF(\bm{y})$, with $V$ an arbitrary Banach space: Such operators have been shown in recent papers  to be efficiently approximated by ReLU neural networks, \emph{breaking the curse of dimensionality}.
\item Another strategy relies on the use of additional structure of the underlying operator $\G$, which is \emph{not captured} by its smoothness properties. In this direction, we show that many PDE operators may possess such internal structure, making them not only amenable to approximation by classical numerical methods, but also enabling DeepONets to \emph{break the curse of dimensionality}.
\end{itemize}

The first approach relying on holomorphy is discussed at an abstract level in Section \ref{sec:be-holomorphy}, where we apply the main results of \cite{SchwabZech2019,OSZ2019,OSZ2020} to DeepONets. The discussion of this specific parametrized setting ultimately leads to Theorem \ref{thm:holomorphic-approx}, which provides quantitative estimates on the approximation error for the DeepONet approximation of holomorphic operators. This abstract result is applied to two concrete examples of holomorphic operators in Sections \ref{sec:pendulum} and \ref{sec:elliptic}. 

In contrast to the holomorphic case, generally applicable results which rely on internal structure of $\G$ \emph{other than smoothness}, appear to be much more difficult to state at an abstract level; therefore, in this case, we instead focus on a case-by-case discussion of the approximation error for concrete operators of interest, which we defer to Section \ref{sec:parabolic}, for a parabolic PDE, and in Section \ref{sec:scl}, for a hyperbolic PDE.

The remaining subsections of the present Section \ref{sec:3} provide the details as well as an extended discussion of the error associated with projections onto linear and affine subspaces (Section \ref{sec:projection}), bounds on the reconstruction error (Section \ref{sec:decoding}), bounds on the encoding error (Section \ref{sec:encoding}), and bounds on the approximation error (Section \ref{sec:approximation}).

\subsection{On the error due to projections of Hilbert spaces onto linear and affine subspaces.}
\label{sec:projection}
We start with the observation that the encoder $\cE$ is a linear mapping from $X$ to $\R^m$. As long as we choose the decoder $\cD$ to also be a linear mapping from $\R^m$ to $X$, we see that the encoding error $\Err_{\cE}$ \eqref{eq:encoding} can be bounded from below by the following \emph{projection error}:
\begin{equation}
    \label{eq:eeproj}
\Err_\Proj(\hat{V};\mu) 
=
\int_{X}
\inf_{\hat{v} \in \hat{V}}
\Vert v - \hat{v} \Vert^2_{X} \, d\mu(v),
\end{equation}
where $\hat{V} = \im(\cD)$. 

Hence, we need to study general properties of the projection error $\Err_\Proj(\hat{V};\nu)$ onto a finite-dimensional linear subspace $\hat{V} \subset H$, for an arbitrary Hilbert space $H$, and given a probability measure $\nu \in \P(H)$ with finite second moment $\int_H \Vert v \Vert^2_H \, d\nu(v)$. This study results in the following theorem that characterizes \emph{optimal finite-dimensional subspaces} of the Hilbert space $H$,
\begin{theorem} \label{thm:opt-linear}
Let $\nu \in \P_2(H)$ be a probability measure on a separable Hilbert space $H$. For any $p\in \N$, there exists an optimal $p$-dimensional subspace $V_p \subset H$, such that 
\[
\Err_{\Proj}(V_p;\nu) = 
\inf_{\substack{\hat{V}\subset V,\\ \dim(\hat{V})=p}} \Err_{\Proj}(\hat{V}; \nu).
\]
Furthermore, we can characterize the set of optimal subspaces $\hat{V}\subset H$ as follows:
Let $\overline{\lambda_1} > \overline{\lambda_2} > \dots$ denote the distinct eigenvalues of the operator 
\begin{align} \label{eq:covb}
\overline{\Gamma} = \int_{H} (v \otimes v) \, d\nu(v).
\end{align}
Let $E_k = \set{ v \in H }{ \overline{\Gamma} v = \overline{\lambda}_k v}$, $k\in \N$ denote the corresponding eigenspaces. Choose $n\in \N$, such that 
\[
\sum_{k=1}^{n-1} \dim(E_k) 
< p 
\le \sum_{k=1}^n \dim(E_k).
\]
A $p$-dimensional subspace $\hat{V} \subset H$ is optimal for $\Err_{\Proj}(\hat{V};\nu)$, if and only if,
\[
\bigoplus_{k=1}^{n-1} E_k
\subset
\hat{V}
\subset
\bigoplus_{k=1}^n E_k.
\]
For any $n\in \N$, there exists a \emph{unique} optimal subspace $V_{p_n}\subset H$, of dimension $p_n = \sum_{k=1}^n\dim(E_k)$. For any optimal subspace $\hat{V}\subset H$, the resulting projection error is given by 
\begin{equation}
    \label{eq:projerr1}
\Err_{\proj}(\hat{V})
=
\sum_{j>p} \lambda_j,
\end{equation}
where 
\[
\underbrace{
\lambda_1 = \dots = \lambda_{n_1}
}_{=\overline{\lambda}_1} 
> 
\underbrace{
\lambda_{n_1+1} = \dots = \lambda_{n_1 + n_2}
}_{=\overline{\lambda}_2} 
> 
\dots,
\]
are the eigenvalues of $\overline{\Gamma}$ repeated according to multiplicity, with $n_j = \dim(E_j)$.
\end{theorem}
The proof of this theorem is based on a series of highly technical lemmas and is presented in detail in Appendix \ref{app:pf32}. The essence of the above theorem is the connection between optimal linear subspaces (that minimize projection errors) of a Hilbert space and the eigensystem of its (uncentered) covariance operator \eqref{eq:covb}. Thus given \eqref{eq:projerr1}, the study of projection errors with respect to finite-dimensional linear subspaces requires a careful investigation into the decay of the eigenvalues of the operator \eqref{eq:covb} and will be instrumental in providing bounds on the encoding error \eqref{eq:encoding} and enable us to identify suitable sensors $\{x_j\}$ for defining the encoder $\cE$.

\begin{remark}
We would like to point out that the main observations of Theorem \ref{thm:opt-linear}, and in particular, the important identity \eqref{eq:projerr1}  for the minimal projection error have previously been observed in \cite{STU3}. In the finite-dimensional case, the underlying ideas are well-known in principal component analysis. While the basic ideas are not new, we nevertheless include Theorem \ref{thm:opt-linear} in the present work for completeness, due to its central importance to our discussion.
\end{remark}

Similarly, we observe that the trunk-net induced reconstructor $\cR$ \eqref{eq:reconstruction} is an affine mapping between $\R^p$ and the output Banach space $Y$. As will be shown in Lemma \ref{lem:opt-proj}, the reconstruction error \eqref{eq:decoding} for any $\cR$ can be bounded from below by the error with respect to projection onto affine subspaces of the output Hilbert space. We formalize this notion below.

Given a separable Hilbert space $H$, let $\hat{V}_0$ now denote an affine subspace of the form
\[
\hat{V}_0
=
\set{
\hat{v} = \hat{v}_0 + \sum_{j=1}^p \alpha_j \hat{v}_j
}{
\alpha_1, \dots, \alpha_p \in \R
}.
\]
for $\hat{v}_0, \dots, \hat{v}_p \in H$. Note that for \emph{any} $\hat{v}'\in \hat{V}_0$, the set 
\[
\hat{V}_0 - \hat{v}' = \set{ \hat{v} - \hat{v}' }{ \hat{v}\in \hat{V} },
\]
is a vector space $\hat{V}$, spanned by $\hat{v}_1, \dots, \hat{v}_p$. It is easy to see that the vector space $\hat{V}$, associated with the affine space $\hat{V}_0$ is unique and only depends on $\hat{V}_0$, and not on a particular choice of the $\hat{v}_0, \dots, \hat{v}_p$. 

The following theorem provides a complete characterization of finite-dimensional optimal affine subspaces of $H$ and the resulting projection error.
\begin{theorem} \label{thm:opt-proj}
Let $H$ be a separable Hilbert space and $\nu\in \P_2(H)$ be a probability measure with finite second moment. Let $p\in \N$. Let $\hat{V}_0$ be an affine subspace with associated vector space $\hat{V}$ such that $\dim(\hat{V}) = p$. Then there exists a unique element $\hat{v}_0 \in \hat{V}_0$ such that 
\[
\Vert \E[v] -\hat{v}_0 \Vert 
=
\inf_{\hat{v}\in \hat{V}_0}
\Vert  \E[v] -\hat{v} \Vert,
\]
and the projection error given by,
\begin{equation}
    \label{eq:perr}
\Err_{\Proj}(\hat{V}_0)
=
\int_H 
\inf_{\hat{v}\in \hat{V}_0}
\Vert 
{v} - \hat{v}
\Vert^2
\, 
d\nu(v)
\end{equation}
can be written as,
\begin{align} \label{eq:affine}
\Err_{\Proj}(\hat{V})
=
\Vert \E[v]  - \hat{v}_0 \Vert^2
+
\int_H
\inf_{\hat{v} \in \hat{V}}
\Vert 
v - \E[v] - \hat{v}
\Vert^2
\, d\nu(v),
\end{align}
Furthermore, the affine space $\hat{V}_0$ is a minimizer of $\Err_{\Proj}(\hat{V}_0)$ (in the class of affine subspaces of dimension $p$), if and only if,
\[
\hat{v}_0 = \E[v],
\quad
\text{and}
\quad
\bigoplus_{j=1}^{n-1} E_j 
\subset
\hat{V}
\subset 
\bigoplus_{j=1}^n E_j.
\]
Here $E_j$ denote the eigenspaces of the covariance operator 
\begin{equation}
\label{eq:cov}
\Gamma = \int_H (v-\E[v])\otimes (v-\E[v]) \, d\nu(v),
\end{equation}
associated with the distinct eigenvalues $\overline{\lambda}_1 > \overline{\lambda}_2 > \dots$, and $n\in N$ is chosen such that
\[
\sum_{j=1}^{n-1} \dim(E_j) < p \le \sum_{j=1}^n \dim(E_j).
\]
In this case, the projection error is given by 
\[
\Err_{\proj}(\hat{V}_0)
=
\sum_{j>p} \lambda_j,
\]
where 
\[
\underbrace{
\lambda_1 = \dots = \lambda_{n_1}
}_{=\overline{\lambda}_1} 
> 
\underbrace{
\lambda_{n_1+1} = \dots = \lambda_{n_1 + n_2}
}_{=\overline{\lambda}_2} 
> 
\dots,
\]
are the eigenvalues of $\Gamma$ repeated according to multiplicity, with $n_j = \dim(E_j)$.
\end{theorem}
The above theorem is proved in Appendix \ref{app:pf-opt-proj}. It is the extension of Theorem \ref{thm:opt-linear} to affine subspaces of a Hilbert space. It serves to relate the projection error to the decay of eigenvalues of the associated covariance operator \eqref{eq:cov} and will be the key to proving bounds on the reconstruction error \eqref{eq:decoding} and in the identification of the optimal trunk network for the DeepONet \eqref{eq:donet}. 
\subsection{Bounds on the reconstruction error \eqref{eq:decoding}}
\label{sec:decoding}
In this section, we will apply results from the previous sub-section to bound the reconstruction error $\Err_{\cR}$ \eqref{eq:decoding}. We start by recalling that the reconstructor $\cR$ \eqref{eq:reconstruction} in the DeepONet \eqref{eq:donet} is affine. The following lemma, whose proof is provided in Appendix \ref{app:pf34}, identifies the optimal projector $\cP$ for a given reconstructor $\cR$.
\begin{lemma} \label{lem:opt-proj}
Let $\cR = \cR_{\bm{\tr}}: \R^p \to L^2(U)$ be an affine reconstructor of the form \eqref{eq:reconstruction}, for $\tr_0 \in L^2(U)$ and linearly independent $\tr_1,\dots,\tr_p\in L^2(U)$. Let $\nu\in \P_2(L^2(U))$ be a probability measure with finite second moments. Then, the reconstruction error \eqref{eq:decoding} is minimized in the class of Borel measurable projectors $\cP: L^2(U) \to \R^p$, for
\begin{align} \label{eq:opt-proj}
\cP(u) := 
\left(
\langle u - \tr_0, \tr^\ast_1 \rangle,
\dots,
\langle u - \tr_0, \tr^\ast_p \rangle
\right),
\end{align}
where $\tr^\ast_1,\dots, \tr^\ast_p \in \Span(\tr_1,\dots, \tr_p)$ denotes the dual basis of $\tr_1, \dots, \tr_p$, i.e. such that 
\[
\langle \tr_\ell, \tr^\ast_k \rangle
=
\delta_{\ell k},
\quad
\forall \, k,\ell \in \{1,\dots, p\}.
\]
In this case, we have
\begin{align} \label{eq:RP-opt}
\cR\circ \cP(u) 
=
\tr_0^\perp 
+
\sum_{k=1}^p
\langle u, \tr^\ast_k\rangle \tr_k
\end{align}
where 
\[
\tr_0^\perp
=
\tr_0
-
\sum_{k=1}^p 
\langle \tr_0, \tr^\ast_k \rangle \tr_k,
\]
is the projection of $\tr_0$ onto the orthogonal complement of $\Span(\tr_1,\dots, \tr_p) \subset L^2(U)$.
\end{lemma}
Next, we can directly apply Theorem \ref{thm:opt-proj} to identify an \emph{optimal} reconstructor $\cR$, and apply Lemma \ref{lem:opt-proj} to identify the associated \emph{optimal} projector $\cP$, in the following theorem,
\begin{theorem} \label{thm:rec-opt}
Denote $\nu := \G_\#\mu$. If the reconstruction $\cR = \cR_{\bm{\tr}}: \R^p \to L^2(U)$ is fixed to be of the form \eqref{eq:reconstruction} with $\tr_j\in L^2(U)$ for $j=0,\dots, p$, 
then the reconstruction error \eqref{eq:decoding}, in the class of affine reconstructions $\cR$ and \emph{arbitrary measurable} projections $\cP: L^2(U) \to \R^p$, is minimized by the choice
\begin{equation}
    \label{eq:orec}
\cR_\opt = \cR_{\hat{\bm{\tr}}}(\alpha_1, \dots, \alpha_p)
=
\hat{\tr}_0 +  \sum_{j=1}^p \alpha_j \hat{\tr}_j,
\end{equation}
where 
\[
\hat{\tr}_0 = \E_{\nu}[v]= \int_{L^2(U)} v \, d\nu,
\]
is the mean and $\hat{\tr}_j$, $j=1,\dots, p$ are $p$ eigenvectors of the covariance operator 
\[
\Gamma = \int_Y (v-\E_{\nu}[v])\otimes(v-\E_{\nu}[v]) \, d\nu(v),
\]
corresponding to the $p$ largest eigenvalues $\lambda_1 \ge \lambda_2 \ge \dots \ge \lambda_p \ge \dots$. Furthermore, the optimal reconstruction error satisfies the lower bound
\begin{align}\label{eq:recerr-lower}
\Err_{\cR_\opt} \ge \sqrt{\sum_{k > p} \lambda_k},
\end{align}
in terms of the spectrum $\lambda_1 \ge \lambda_2 \ge \dots $ of $\Gamma$ (eigenvalues repeated according to their multiplicity). Given $\cR_\opt$, the corresponding optimal measurable projection $\cP: L^2(U) \to \R^p$ is affine and is given by the orthogonal projection
\begin{equation}
    \label{eq:oproj}
\cP(v) 
=
\left( \langle (v - \hat{\tr}_0), \hat{\tr}_1 \rangle, \dots, \langle (v - \hat{\tr}_0), \hat{\tr}_p \rangle \right).
\end{equation}
\end{theorem}

Given the existence of an \emph{optimal} affine reconstructor \eqref{eq:orec} and the associated optimal projector \eqref{eq:oproj}, the proof of Theorem \ref{thm:lowerbd} is now straightforward.

\begin{proof}[Proof of Theorem \ref{thm:lowerbd}]
\label{pf:lowerbd}
Let $\cN = \cR\circ \cA \circ \cE$ be a DeepONet approximation of $\G$. We aim to show that 
\[
\sqrt{\sum_{k>p} \lambda_k } \le \Err_{\cR} \le \Err,
\]
where $\lambda_1\ge \lambda_2\ge \dots$ denote the eigenvalues of $\Gamma_{\G_\#\mu}$.

By Lemma \ref{lem:opt-proj}, the optimal projector $\cP: L^2(U) \to \R^p$ for a the given reconstructor $\cR$ is such that $\cR\circ \cP = \Pi_{V_0}: L^2(U) \to L^2(U)$ is the orthogonal projection onto the affine subspace $V_0 = \im(\cR) \subset L^2(U)$. To prove the lower bound $\Err_{\cR} \le \Err$, we observe that for any $u\in X$, we have
\begin{align*}
\Vert \G(u) - \cN(u) \Vert_{L^2(U)}
&=
\Vert \G(u) - \cR \circ \cA \circ \cE(u) \Vert_{L^2(U)}
\\
&\ge 
\inf_{v \in V_0} \Vert \G(u) - v \Vert_{L^2(U)}
\\
&=
\Vert \G(u) - \Pi_{V_0} \G(u) \Vert_{L^2(U)}
\\
&= 
\Vert \G(u) - \cR \circ \cP \circ \G(u) \Vert_{L^2(U)},
\end{align*}
and hence,
\begin{align*}
\Err
=
\Vert \G - \cR \circ \cA \circ \cE \Vert_{L^2(\mu)}
\ge 
\Vert \G - \cR \circ \cP \circ \G \Vert_{L^2(\mu)}
=
\Vert \Id - \cR \circ \cP \Vert_{L^2(\G_\#\mu)}
=
\Err_{\cR}.
\end{align*}
Furthermore, by Theorem \ref{thm:rec-opt}, we thus have
\begin{align}\label{eq:lbd2}
\sqrt{\sum_{k>p} \lambda_k} = \Err_{\cR_\opt} \le \Err_{\cR} \le \Err,
\end{align}
where $\lambda_1\ge \lambda_2 \ge \dots$ denote the eigenvalues of the covariance operator $\Gamma_{\G_\#\mu}$ of the push-forward measure $\G_\#\mu$. This proves the claim.
\end{proof}

The lower bound in \eqref{eq:lbd2} is fundamental, as it reveals that the spectral decay rate for the  operator $\Gamma_{\G_\#\mu}$ of the push-forward measure essentially determines how low the approximation error of DeepONets can be for a given output dimension $p$ of the trunk nets. 

After establishing the lower bound on the reconstruction error, we seek to derive upper bounds on this error. We observe that the optimal reconstructor is given by eigenfunctions of the covariance operator $\Gamma$ \eqref{eq:cov}. In general, these eigenfunctions are not neural networks of the form \eqref{eq:ann1}. However, given the fact that neural networks are universal approximators of functions on finite dimensional spaces, one can expect that the trunk nets in \eqref{eq:reconstruction} will approximate these underlying eigenfunctions to high accuracy. This is indeed established in the following,
\begin{proposition} \label{prop:err-rec-approx}
Let $\nu = \G_\#\mu \in \P_2(Y)$ be a probability measure with finite second moments. Write the covariance operator in the form 
\[
\Gamma_{\nu} = \sum_{k=1}^\infty \lambda_k (\phi_k \otimes \phi_k),
\]
with $\lambda_1 \ge \lambda_2 \ge \dots$ and orthonormal eigenbasis $\phi_k$. Let $\hat{\tr}_k$, $k=0,1,\dots, p$, denote the optimal choice for an affine reconstruction $\cR_{\opt}$, as in Theorem \ref{thm:rec-opt}, more precisely, let $\hat{\tr}_0 = \E_\nu[v]$, $\hat{\tr}_k = \phi_k$ for $k=1,\dots, p$.  Let $\tr_0, \tr_1, \dots, \tr_p$ be the trunk net functions of an arbitrary DeepONet \eqref{eq:donet}. The reconstruction error for the reconstruction $\cR = \cR_{\bm{\tr}}$:
\[
\cR(\alpha_1, \dots, \alpha_p) = \tr_0 + \sum_{k=1}^p \alpha_k \tr_k,
\]
with corresponding (unique optimal) projection $\cP$ \eqref{eq:opt-proj} satisfies
\begin{align} \label{eq:err-rec-approx}
\Err_{\cR}
\le
\sqrt{ 1 + \Tr(\Gamma_{\nu}) } \max_{k=0,\dots, p} \Vert \hat{\tr}_k - \tr_k \Vert_{L^2_y}
+ 
\sqrt{\sum_{k > p} \lambda_k}.
\end{align}
\end{proposition}
This proposition is proved in Appendix \ref{app:pf35}. The estimate \eqref{eq:err-rec-approx} shows that an upper bound on the reconstruction error has two contributions, one of them arises from the decay rate of the eigenvalues of the covariance operator associated with the push-forward measure $\nu = \G_\#\mu$ and does not depend on the underlying neural networks. On the other hand, the second contribution to \eqref{eq:err-rec-approx} depends on the choice and approximation properties of the trunk net in the DeepONet \eqref{eq:donet}. 

Thus, to bound the reconstruction error, we need some explicit information about the optimal reconstruction in terms of the eigensystem of the covariance operator. However in practice, one does not have access to the form of the nonlinear operator $\G$, but only to measurements of $\G$ on a finite number of training samples. Hence, it may not be possible to determine what the optimal reconstructor for a particular nonlinear operator $\G$ is. Therefore, we have the following lemma that compares the reconstruction error of the DeepONet \eqref{eq:donet} with the reconstruction error that arises from another, possibly non-optimal, choice of affine reconstructor. 
\begin{lemma} \label{lem:err-rec-comp}
Let $\nu \in \P_2(Y)$ be a probability measure with finite second moment. Let $\bm{\tr} = (0, \tr_1, \dots, \tr_p)$ denote the trunk net functions of a DeepONet \emph{without bias} ($\tr_0\equiv 0$), and with associated reconstruction $\cR = \cR_{\bm{\tr}}$:
\[
\cR(\alpha_1,\dots, \alpha_p)
=
\sum_{k=1}^p \alpha_k \tr_k.
\]
Let $\tilde{\bm{\tr}} = ( 0, \tilde{\tr}_1, \dots, \tilde{\tr}_p)$ denote the basis functions for a reconstruction ${\tR} = \tR_{\tilde{\bm{\tr}}}: \R^p \to Y$ without bias ($\tilde{\tr}_0 \equiv 0$), of the form
\[
\tR(\alpha_1,\dots, \alpha_p)
=
\sum_{k=1}^p \alpha_k \tilde{\tr}_k.
\]
Assume that the functions $\tilde{\tr}_1,\dots, \tilde{\tr}_p\in Y$ are \emph{orthonormal}. Let $\cP,\tP: Y \to \R^p$ denote the corresponding projection mappings \eqref{eq:opt-proj}, and let $\Err_{\cR}$, $\Err_{\tR}$ denote the reconstruction errors of $(\cR,\cP)$ and $(\tR,\tP)$, respectively. Let $\epsilon \in (0,1/2)$, $p\ge 1$ be given. If 
\begin{align} \label{eq:tdiff-ass}
\max_{k=1,\dots, p} \Vert \tr_k - \tilde{\tr}_k \Vert_{L^2}
\le 
\frac{\epsilon}{p^{3/2}},
\end{align}
then we have
\begin{align} \label{eq:err-rec-comp}
\Err_{\cR}
\le
\Err_{\tR}
+ 
C \epsilon.
\end{align}
where $C\ge 1$ depends only on $\int_{L^2} \Vert u \Vert^2 \, d\nu(u)$. Furthermore, we have the estimate 
\begin{align}
\Lip(\cP), \; \Lip(\cR)
\le 2,
\end{align}
for the Lipschitz constant of the projection $\cP$ (cp. \eqref{eq:opt-proj}) and $\cR$, respectively, where 
\begin{align*}
\Lip(\cP) &= \Lip\left(\cP:(L^2(U), \Vert \slot \Vert_{L^2})\to (\R^p, \Vert \slot \Vert_{\ell^2}) \right), 
\\
\Lip(\cR) &= \Lip\left(\cR:(\R^p, \Vert \slot \Vert_{\ell^2})\to (L^2(U), \Vert \slot \Vert_{L^2}) \right).
\end{align*}
\end{lemma}
The proof of this lemma is provided in Appendix \ref{app:pf306} and we would like to point out that for simplicity of exposition, we have set the \emph{biases} $\tr_0 = \tilde{\tr}_0 \equiv 0$ in the above. One can readily incorporate these bias terms to derive an analogous version of the bound \eqref{eq:err-rec-comp}.  

We illustrate the comparison principle elucidated in Lemma \ref{lem:err-rec-comp} with an example. To this end, we set the target space as the $n$-dimensional torus, i.e. $U=[0,2\pi]^n$ (or $U = \T^n$). With respect to this $U$, a possible \emph{canonical} reconstructor is given by the $n$-dimensional Fourier reconstruction;
\begin{equation}
    \label{eq:recF}
\cR_{\Fourier}(\alpha_1, \dots, \alpha_p)
=
\sum_{j=1}^p \alpha_j \fb_j(x),
\end{equation}
where we use the notation introduced in Appendix \ref{app:Fourier}. With respect to this Fourier reconstructor, we prove in Appendix \ref{app:pf36} the following estimate on approximation by trunk neural networks,
\begin{lemma} \label{lem:Fourier-NN}
Let $n, p\in \N$, and consider the Fourier reconstruction $\cR_{\Fourier}$ on $[0,2\pi]^n \simeq \T^n$. There exists a constant $C>0$, independent of $p$, such that for any $\epsilon \in (0,1/2)$, there exists a trunk net $\bm{\tr}: \R^n \to \R^p$, with 
\[
\begin{gathered}
\size(\bm{\tr}) \le C p(1 + \log(\epsilon^{-1}p)^2), 
\\
\depth(\bm{\tr}) \le C (1 + \log(\epsilon^{-1}p)^2),
\end{gathered}
\]
and such that 
\begin{align} \label{eq:Fourier-NN}
p^{3/2} \max_{j=1,\dots, p} \Vert \tr_j - \fb_j \Vert_{L^2([0,2\pi]^n)}
\le \epsilon,
\end{align}
where $\fb_1,\dots, \fb_p$ denote the first $p$ elements of the Fourier basis.

Furthermore, the Lipschitz norm of $\cR = \cR_{\bm{\tr}}$ and of the linear projection $\cP: L^2(U) \to \R^p$ associated with the $\bm{\tr}$-induced reconstruction $\cR_{\bm{\tr}}$ via \eqref{eq:opt-proj} can be estimated by
\[
\Lip
\left(
\cR: 
(\R^p, \Vert \slot \Vert_{\ell^2})
\to 
(L^2(U), \Vert \slot \Vert_{L^2}) 
\right)
\le 2,
\]
and
\[
\Lip
\left(
\cP: 
(L^2(U), \Vert \slot \Vert_{L^2}) 
\to 
(\R^p, \Vert \slot \Vert_{\ell^2})
\right)
\le 2.
\]
\end{lemma}
Now combining the results of the above Lemma with the estimate \eqref{eq:err-rec-comp} yields the following result,
\begin{lemma} \label{lem:Fourier-trunk}
Let $n\in \N$ and fix $\nu \in \P_2(L^2(\T^n))$. There exists a constant $C>0$, depending only on $n$ and $\int_{L^2(\T^n)} \Vert u \Vert^2 \, d\nu(u)$, such that for any $\epsilon \in (0,1/2)$, there exists a trunk net $\bm{\tr}: \R^n \to \R^p$, with 
\[
\begin{gathered}
\size(\bm{\tr}) \le C (1 +p \log(\epsilon^{-1} p)^2), 
\\
\depth(\bm{\tr}) \le C (1 +\log(\epsilon^{-1} p)^2),
\end{gathered}
\]
and such that the reconstruction 
\[
\cR: \R^p \to L^2(\T^n),
\quad
\cR(\alpha_1, \dots, \alpha_p)
=
\sum_{k=1}^p \alpha_k \tr_k,
\]
satisfies
\begin{equation}
    \label{eq:errF}
\Err_{\cR}
\le
\Err_{\cR_{\Fourier}}
+ 
C\epsilon.
\end{equation}
Furthermore, $\cR$ and the associated projection $\cP: (L^2(U), \Vert \slot \Vert_{L^2} \to (\R^p, \Vert \slot \Vert_{\ell^2})$ (cp. \eqref{eq:opt-proj}) satisfy $\Lip(\cR), \, \Lip(\cP) \le 2$.
\end{lemma}
The significance of bound \eqref{eq:errF} lies in the fact that it reduces the problem of estimating the reconstruction error \eqref{eq:decoding} for a DeepONet to estimating the reconstruction error for a Fourier reconstructor \eqref{eq:recF}, which might be much easier to derive in concrete examples. Our next goal will be to obtain further insight into the decay of the eigenvalues of $\Gamma_{\G_\#\mu}$ (in terms of $\G$ and $\mu$), and -- in view of the explicit complexity estimate provided by Lemma \ref{lem:Fourier-trunk} -- to estimate $\Err_{\cR_{\Fourier}}$.

\subsubsection{On the decay of spectrum for the push-forward measure}
\label{sec:decayspec}

From the bound \eqref{eq:donetbd}, it is clear that the decay of eigenvalues of the covariance operator, associated with the push forward measure $\G_\#{\mu}$ plays a crucial role in estimating the total error \eqref{eq:approxerr}, and in particular, the reconstruction error \eqref{eq:decoding}. 

If $\G$ is at least Lipschitz continuous, then we can write
\begin{align*}
\Tr(\Gamma_{\G_\#\mu})
&=
\int_{Y} \left\Vert v - \E_{\G_\#\mu}[v] \right\Vert^2_{L^2(U)} \, d(\G_\#\mu)(v)
\\
&=
\inf_{\hat{v}\in {L^2(U)}} \int_{L^2(U)} \Vert v - \hat{v} \Vert^2_{L^2(U)} \, d(\G_\#\mu)(v)
\\
&=
\inf_{\hat{v}\in {L^2(U)}} \int_{X} \Vert \G(u) - \hat{v} \Vert^2_{L^2(U)} \, d\mu(u).
\end{align*}
Making the particular (sub-optimal) choice $\hat{v} := \G(\E_\mu[u]) \in L^2(U)$ and utilizing the Lipschitz continuity of $\G$, we can estimate the last expression as
\begin{align*}
&\le
\int_{X} \left\Vert \G(u) - \G(\E_\mu[u]) \right\Vert^2_{L^2(U)} \, d\mu(u)
\\
&\le
\Lip(\G)^2 \int_{X} \left\Vert u - \E_\mu[u] \right\Vert^2_X \, d\mu(u)
\\
&=
\Lip(\G)^2 \Tr\left({\Gamma}_{\mu}\right).
\end{align*}
Thus, $\Gamma_{\G_\#\mu}$ is a trace-class operator for Lipschitz continuous $\G$, implying that 
\[
\sum_{k>p} \lambda_k \to 0, \quad (\text{as } p\to \infty),
\]
However, the efficiency of DeepONets in approximating the nonlinear operator $\G$ relies on the \emph{precise rate} of this spectral decay. In particular, an exponential decay of the eigenvalues would facilitate efficient approximation by DeepONets. 

Clearly, the spectral decay of $\Gamma_{\G_\#\mu}$ depends on both $\G$ and $\mu$ (possibly in a complicated manner). If the eigenvalues of $\Gamma_\mu$ decay rapidly, e.g. exponentially, one might hope that the same is true for the eigenvalues of $\Gamma_{\G_\#\mu}$, under relatively mild conditions on $\G$. The following lemma, proved in Appendix \ref{app:pf307} shows that this is unfortunately not the case under just the assumption that the operator $\G$ is only Lipschitz continuous,
\begin{proposition}
\label{prop:cexample}
Let $\mu \in \P_2(X)$ be any non-degenerate Gaussian measure; in particular, the spectrum of $\Gamma_\mu$ may have arbitrarily fast spectral decay.  Given any sequence $(\gamma_k)_{k\in \N}$ such that 
\[
\gamma_k \ge 0, \quad \forall \, k\in \N, \qquad \sum_{k=1}^\infty k^2\gamma_k < \infty,
\]
there exists a Lipschitz continuous map $\G: X \to L^2(U)$, such that the spectrum of covariance operator $\Gamma_{\G_\#\mu}$ of the push-forward measure $\G_\#\mu$ is given by $(\gamma_k)_{k\in \N}$.
\end{proposition}
Thus, the above proposition clearly shows that, in general, a nonlinear operator $\G$ can possibly destroy high rates of spectral decay for the covariance operator associated with a push-forward measure, even if the eigenvalues of the covariance operator associated with the underlying measure $\mu$, decay exponentially rapidly.

However, there are special cases where one can indeed obtain fast rates of spectral decay for the covariance operator associated with a push-forward measure. We identify two of these special cases, with wide ranging applicability, below. 
\subsubsection{Reconstruction error for linear operators $\G: X \to Y$}
\label{sec:linspec}

If the operator $\G: X \to Y$ is a bounded \emph{linear} operator, then the spectrum of the push-forward measure $\G_\#\mu$ can be bounded above by the spectrum of $\mu$. More precisely, we have

\begin{proposition}
\label{prop:lin}
Let $X$ be a separable Hilbert space. Let $\mu \in \P(X)$ be a probability measure with finite second moment $\int_X \Vert u \Vert_X^2 \, d\mu(u) < \infty$. Let $\lambda_1 \ge \lambda_2 \ge \dots$ denote the eigenvalues of the covariance operator $\Gamma_\mu$ of $\mu$, repeated according to multiplicity. If $\G: X \to L^2(U)$ is a bounded linear operator, then for any $p\in \N$, there exists a $p$-dimensional affine subspace $W_0 \subset L^2(U)$, such that 
\begin{align} \label{eq:projj}
\int_{L^2(U)} \inf_{w\in W_0} \Vert w - u \Vert_{L^2(U)}^2 \, d(\G_\#\mu)(u)
\le
\Vert \G \Vert^2 
\,
\sum_{k>p} \lambda_k.
\end{align}
Here $\Vert \G \Vert$ ($= \Lip(\G)$) denotes the operator norm of $\G: X\to L^2(U)$. If the affine reconstruction/projection pair is chosen such that $\cR \circ \cP: L^2(U) \to L^2(U)$ is the orthogonal projection onto $W_0$, then the reconstruction error $\Err_{\cR}$ is bounded by
\[
\Err_{\cR} \le \Vert \G \Vert \sqrt{\sum_{k>p} \lambda_k}.
\]
\end{proposition}
The proof of this proposition is presented in Appendix \ref{app:pf37a}.
\subsubsection{Reconstruction error for operators with \emph{smooth} image}
\label{sec:over-rec-smooth}
Another class of operators for which we can readily estimate the spectral decay of the covariance operator associated with the push-forward measure are those operators that map into smoother (more regular) subspaces of the target space. 

As a concrete example, we set $U = \T^n$ as the periodic torus. Let $N\in \N$. We denote by $P_N: L^2(\T^n) \to L^2(\T^n)$ the orthogonal Fourier projection onto the Fourier basis 
\[
P_N u 
=
\sum_{|{k}|_\infty\le N} \hat{u}_{{k}} \fb_{{k}}({x}),
\]
where the sum is over ${k} = (k_1,\dots, k_n) \in \Z^n$, such that 
\[
|{k}|_\infty := \max_{j=1,\dots, n} |k_j| \le N.
\]
Note that the size of this set is $|\set{{k}\in \Z^n}{|{k}|_\infty \le N}|=(2N+1)^n$. Hence, with $p$ degrees of freedom, we can represent $P_N$ for
\begin{align} \label{eq:dof}
p \ge (2N+1)^n 
\quad
\Rightarrow
\quad 
N \le \left\lfloor \frac{(p^{1/n} -1)}{2} \right\rfloor \le p^{1/n}.
\end{align}

Given $p \in \N$, let $\cP: L^2(\T^n) \to \R^{p}$ be a mapping encoding all Fourier coefficients with $|{k}|_\infty \le N$, where $N \le p^{1/n}$ is the largest integer satisfying \eqref{eq:dof}. Let $\cR_{\Fourier}: \R^p \to L^2(\T^n)$ denote the corresponding Fourier reconstruction \eqref{eq:recF}, so that $\cR_{\Fourier}\circ \cP: L^2(\T^n) \to L^2(\T^N)$ satisfies $\cR_{\Fourier}\circ \cP = P_N$. It is well-known that for $u\in H^s(\T^n)$, we have
\[
\Vert P_N u - u \Vert_{L^2(\T^n)}
\le 
\frac{1}{N^s} \Vert u \Vert_{H^s(\T^n)}.
\]
Hence the resulting reconstruction error $\Err_{\cR_{Fourier}}$ is given by,
\begin{align*}
\int_Y \Vert \cR_{\Fourier} \circ \cP - \Id \Vert^2_{L^2(\T)} \, d(\G_\#\mu(u))
&\le
\frac{1}{N^{2s}}
\int_Y \Vert u \Vert^2_{H^s(\T)} \, d(\G_\#\mu(u))
\\
&=
\frac{1}{N^{2s}}
\int_X \Vert \G(u) \Vert^2_{H^s(\T)} \, d\mu(u).
\end{align*}
This elementary calculation leads to the following result,
\begin{proposition} \label{prop:rec-smoothness}
If $\G$ defines a Lipschitz mapping $\G: X \to H^s(\T^n)$, for some $s>0$, with
\[
\int_{X} \Vert \G(u) \Vert_{H^s}^2 \, d\mu(u) \le M < \infty,
\]
then we have the following estimate on the reconstruction error:
\begin{align} \label{eq:rec-smoothness}
\Err_{\cR_{\Fourier}} 
\le CM p^{-s/n},
\end{align}
where $C = C(n,s)>0$ depends on $n$, $s$, but is independent of $p$.
\end{proposition}
Combining \eqref{eq:rec-smoothness} and \eqref{eq:errF}, and setting $\epsilon =  p^{-s/n}$ in \eqref{eq:errF} immediately implies Theorem \ref{thm:smoothness-NN}, which has already been stated in the overview Section \ref{sec:over-rec}. Theorem \ref{thm:smoothness-NN} provides a complexity and error estimate for the trunk net approximation under the assumptions of the previous proposition.

\subsection{Bounds on the encoding error \eqref{eq:encoding}}
\label{sec:encoding}

Our aim in this section is to bound the encoding error \eqref{eq:encoding}, associated with the DeepONet \eqref{eq:donet}. This error \emph{does not depend} on the nonlinear operator $\G$, but only on the underlying probability measure $\mu$. Key objectives of our analysis are to determine suitable choices of sensors for a fixed $m$ as well as to find the appropriate form of a decoder $\cD$ in order to minimize the encoding error \eqref{eq:encoding}. We first recall the lower bound of Theorem \ref{thm:enlb}, on the encoding error:
\begin{align} \label{eq:enlb2}
(\Err_{\cE})^2
=
\int_{X} \Vert \cD\circ \cE(u) - u \Vert_{L^2(D)}^2 \, d\mu(u)
 \ge \sum_{k>m}\lambda_k,
\end{align}
if $\cD: \R^m \to X$ is a \emph{linear} decoder, $\mu$ is mean-zero and $\lambda_k$ refers to the $k$-th eigenvalue of the covariance operator $\Gamma_{\mu} = \int_X u \otimes u \, d\mu(u)$. The proof presented in Appendix \ref{app:pf37} relies on Theorem \ref{thm:opt-linear} and from this proof, we can readily see that the restriction on the zero mean of the measure $\mu$ can be relaxed by using an affine decoder and Theorem \ref{thm:opt-proj}. We note in passing that the above bound \eqref{eq:enlb2} in fact holds for any encoder (not necessarily linear) as long as the decoder is linear.

Our next aim is to derive \emph{upper} bounds on the encoding error. To illustrate our main ideas, we will restrict our discussion to the case $X=L^2(D)$. The results are readily extended to more general spaces, such as Sobolev spaces $X = H^s(D)$ for $s>0$. We fix a probability measure $\mu$ on $L^2(D)$, and write the covariance operator $\Gamma$ as an eigenfunction decomposition 
\[
\Gamma = \sum_{\ell=1}^\infty \lambda_\ell (\phi_\ell \otimes \phi_\ell),
\]
where $\lambda_1 \ge \lambda_2 \ge \dots$ are the decreasing eigenvalues, and such that the $\phi_\ell$ are an orthonormal basis of $L^2(D)$. We will assume that all $\phi_\ell$ are continuous functions, so that point-wise evaluation of $\phi_\ell$ makes sense. Note that since the $\phi_\ell$ are orthonormal in $L^2(D)$, they are also linearly independent as elements in $C(D)$.

Assume now that for some $M\in \N$, there exist sensors $X_1, \dots, X_M\in D$, such that the matrix $\Phi_M \in \R^{m \times M}$  with entries 
\begin{align}\label{eq:PhiM}
[\Phi_M]_{ij} = [\phi_i(X_j)], \quad \text{for $i=1,\dots, m$, $j=1,\dots, M$,}
\end{align}
has full rank, i.e.
\begin{align} \label{eq:nonsing}
\det(\Phi_M \Phi_M^T) \ne 0. 
\end{align}
Then we can define a ``\emph{projection}'' onto $\phi_1, \dots, \phi_m$ by 
\begin{align} \label{eq:encode-general}
u(x) 
\overset{\cE}{\to} (u(X_1), \dots, u(X_m)) 
\overset{\cD}{\to} \sum_{j,k=1}^m [\Phi_M^\dagger]_{kj} u(X_j) \phi_k(x),
\end{align}
where 
\begin{align} \label{eq:pseudoinverse}
\Phi_M^\dagger := 
\left(\Phi_M \Phi_M^T\right)^{-1} \Phi_M,
\end{align}
denotes the pseudo-inverse of $\Phi_M \in \R^{m\times M}$. Written in somewhat simpler matrix/vector-multiplication notation, we might write this projection as
\[
u(x) 
\mapsto 
\left\langle
\Phi_M^\dagger u(\bm{X}),   \bm{\varphi}(x) 
\right\rangle_{\ell^2(\R^m)},
\]
where $\bm{X} := (X_1,\dots, X_M)$, $u(\bm{X}) = (u(X_1), \dots, u(X_M))$ and $\bm{\varphi}(x) = (\phi_1(x), \dots, \phi_m(x))$.

Note that for $u(x) = \phi_{i}(x)$, $1\le i\le m$, we have
\[
\sum_{j=1}^M  [\Phi_M]_{\ell,j} \phi_i(X_j) 
=
\sum_{j=1}^M  [\Phi_M]_{\ell,j} [\Phi_M]_{i,j}
=
[\Phi_M\Phi_M^T]_{\ell,i}.
\]
Hence
\begin{align*}
\phi_i(x) 
&\mapsto 
 \sum_{k,\ell=1}^m \sum_{j=1}^M  [\Phi_M]_{\ell,j} \phi_i(X_j) \left[(\Phi_M \Phi_M^T)^{-1}\right]_{k,\ell}  \phi_k(x)
 \\
 &=
 \sum_{k,\ell=1}^m  [\Phi_M\Phi_M^T]_{\ell,i} \left[(\Phi_M \Phi_M^T)^{-1}\right]_{k,\ell}  \phi_k(x)
 \\
 &=
 \sum_{k=1}^m \delta_{ik}  \phi_k(x)
 =
 \phi_i(x).
\end{align*}
So the map \eqref{eq:encode-general} clearly provides a projection onto $\Span(\phi_j;\; j=1,\dots, m)$. 

Next, we have the following Lemma, provided in Appendix \ref{app:pf38}, which characterizes the encoding error \eqref{eq:encoding}, associated with the encoder/decoder pair given by \eqref{eq:encode-general}.
\begin{lemma} 
\label{lem:aliasing-general}
Let $\mu \in \P_2(L^2(D))$ be a measure with finite second moments, i.e. such that $\int_{L^2(D)} \Vert u \Vert_{L^2}^2 \, d\mu(u) < \infty$.
Let $\Phi_M$ be given by \eqref{eq:PhiM}, and assume that the non-singularity condition \eqref{eq:nonsing} holds. Then the encoding error $\Err_{\cE}$ for the pair $\cE$, $\cD$, defined by \eqref{eq:encode-general} can be written as 
\begin{align*}
(\Err_{\cE})^2
&=
(\Err_{\mathrm{aliasing}})^2+ (\Err_{\perp})^2,
\end{align*}
where
\begin{align*}
(\Err_\perp)^2
&=
\int_X \Vert P^\perp_m u \Vert^2_{L^2} \, d\mu(u)
=
\sum_{\ell > m} \lambda_\ell,
\end{align*}
with $P^\perp_m: L^2 \to L^2$ the orthogonal projection onto the orthogonal complement of $\Span(\phi_1, \dots, \phi_m)$, where $\phi_1,\phi_2,\dots$ denote a basis of orthonormal eigenfunctions of the covariance operator $\Gamma_\mu$ of $\mu$, with corresponding eigenvalues $\lambda_1 \ge \lambda_2 \ge \dots $, and
\begin{equation}
    \label{eq:aliaserr}
(\Err_{\mathrm{aliasing}})^2
=
\sum_{\ell > m} \lambda_\ell \Vert \Phi_M^\dagger \phi_\ell (\bm{X}) \Vert_{\ell^2}^2.
\end{equation}
\end{lemma}
Thus, the above Lemma \ref{lem:aliasing-general} provides us a strategy to bound the encoding error as long as the non-singularity condition \eqref{eq:nonsing} on the matrix \eqref{eq:PhiM} is satisfied. Moreover, one component of the encoding error is completely specified in terms of the spectral decay of the associated covariance operator. However, the other component measures the error due to \emph{aliasing}, with the notation being motivated by an example from Fourier analysis, that is considered in the following subsection. Hence, bounding the aliasing error and checking the validity of the non-singularity condition \eqref{eq:nonsing} require us to specify the location of sensors. 

Given the form of the matrix \eqref{eq:PhiM}, it would make sense to relate the sensor locations with the eigenfunctions of the covariance operator $\Gamma$. However in general, we may not have any information on these eigenfunctions $\phi_1, \dots, \phi_m$, and hence, finding suitable sensors $X_1, \dots, X_M$ satisfying the non-singularity condition \eqref{eq:nonsing} might be a very difficult task. Instead, we propose to choose them iid randomly in $D$, allowing for $M\ge m$, and study the corresponding random encoder 
\[
\cE(u) = (u(X_1),\dots, u(X_M)),
\]
with associated decoder $\cD$ given by \eqref{eq:encode-general}. Note that the decoder $\cD$ is merely used in the analysis of the encoding error, but its explicit form is not needed, when DeepONets are used in practice. In fact, \emph{the use of random sensors and the simplicity of the resulting encoder could constitute one of the key benefits of the DeepONets}.

 More precisely, in the following we fix a probability space $(\Omega,\Prob)$, and a sequence of iid random variables $\{X_k\}_{k\in \N}$:
\[
\omega \mapsto 
(X_1(\omega), X_2(\omega), X_3(\omega), \dots),
\quad (\omega \in \Omega),
\]
such that $X_k \sim \Unif(D)$ for all $k\in \N$. For any $M\in \N$, and $\omega\in \Omega$, we can then study the corresponding encoder $\cE: C(D)\to \R^M$,
\[
\cE(u;X_1(\omega),\dots, X_M(\omega)) = (u(X_1(\omega)),\dots, u(X_M(\omega))).
\]
As is common in probability theory, we will usually suppress the argument $\omega$, and write, e.g., $X_1$ instead of $X_1(\omega)$.

Hence, the matrix $\Phi_M = [\phi_i(X_j)]_{i,j} \in \R^{m\times M}$ is a random matrix, depending on $X_1, \dots, X_M$. In order for the resulting decoder in \eqref{eq:encode-general} to be well-defined, we need to show that there is a non-zero probability that $\Phi_M\Phi_M^T$ is non-singular, for sufficiently large $M$. 

Moreover, by Lemma \ref{lem:aliasing-general}, the aliasing error \eqref{eq:aliaserr} depends on 
\[
\Vert \Phi_M^\dagger \phi_\ell (\bm{X}) \Vert_{\ell^2}^2
\le 
\Vert \Phi_M^\dagger \Vert_{\ell^2\to \ell^2}^2 \Vert \phi_\ell(\bm{X}) \Vert_{\ell^2}^2.
\]
We note that 
\[
\Vert \Phi_M^\dagger \Vert_{\ell^2\to \ell^2}^2
= 
\frac{|D|}{M \sigma_{\mathrm{min}}\left( \frac{|D|}M \Phi_M \Phi_M^T \right)},
\]
can be written in terms of the smallest singular value, i.e. $\sigma_{\mathrm{min}}\left( \frac{|D|}M \Phi_M \Phi_M^T \right)$, of the rescaled matrix $\frac{|D|}M\Phi_M \Phi_M^T$ (the reason for introducing the  rescaling $|D|/M$ will be explained below). Hence, to bound the aliasing error and the overall encoding error, we will need to provide a lower bound (with high probability) on this smallest singular value. We investigate these two issues in the following. 

First, we note that for iid random sensors $X_1,\dots, X_M$, by our definition of $\Phi_M$, we have
\[
\left[\frac{|D|}M \Phi_M \Phi_M^T\right]_{k,\ell}
=
\frac{|D|}M
\sum_{j=1}^M \phi_k(X_j) \phi_\ell(X_j),
\]
and the last sum can be interpreted as a Monte-Carlo estimate of 
\[
|D|\,\E[ \phi_k(X) \phi_\ell(X) ]
=
\int_D \phi_k(x) \phi_\ell(x) \, dx
=
\delta_{k\ell}.
\]
For such Monte-Carlo estimates, we can rely on well-known bounds from probability theory and have the following lemma on the bounds for the smallest singular value:
\begin{lemma}
\label{lem:singb}
For any $m\in \N$, denote
\[
\omega_m := \max_{k \le m} \Vert \phi_k \Vert_{L^\infty}.
\]
Let $X_1, \dots, X_M \sim \mathrm{Unif}(D)$ be iid random variables. Define $\Phi_M$ by \eqref{eq:PhiM}.  Then we have 

\begin{equation}
    \label{eq:prob1}
\Prob\left[
\sigma_{\mathrm{min}}\left(\frac{|D|}M \Phi_M \Phi_M^T \right) < 1-\frac{1}{\sqrt{2}}
\right]
\le 
2m^2 \exp\left(
-\left(\frac{M}{|D| \omega_m^2 m } \right)^2 
\right).
\end{equation}
\end{lemma}
This lemma is proved using the well-know Hoeffding's inequality and the proof is presented in Appendix \ref{app:pf39}. A direct application of the bound \eqref{eq:prob1} allows us to bound the aliasing error \eqref{eq:aliaserr} in the following,
\begin{lemma} \label{lem:aliasing-gen}
Let $X_1, \dots, X_M \sim \mathrm{Unif}(D)$ be iid random variables, uniform on $D$. Let $\mu \in \P(L^2(D))$ be a probability measure concentrated on continuous functions. Let $\lambda_1,\lambda_2,\dots,$ denote the eigenvalues of the covariance operator $\overline{\Gamma}_\mu$ \eqref{eq:covb} of $\mu$, with associated orthonormal eigenbasis $\phi_1,\phi_2,\dots$. The aliasing error \eqref{eq:aliaserr}, resulting from this random choice of sensors, is bounded by
\[
\int \Vert \cD \circ \cE(P_m^\perp u) \Vert^2_{L^2} \, d\mu(u)
\le
\frac{1}{\sigma_{\mathrm{min}}\left(\frac{|D|}{M}\Phi_M \Phi_M^T\right)}
\frac{|D|}{M}\sum_{j=1}^M
\left(
\sum_{\ell > m} \lambda_\ell \vert \phi_\ell(X_j) \vert^2
\right).
\]
Denote $\omega_m := \max_{k\le m} \Vert \phi_k \Vert_{L^\infty}$. Then, with probability
\[
\Prob
\ge 
1 - 
2m^2 \exp\left(
-\left(\frac{M}{ |D|\omega_m^2 m } \right)^2 
\right),
\]
we have 
\[
\int \Vert \cD \circ \cE(P_m^\perp u) \Vert^2_{L^2} \, d\mu(u)
\le
\frac{\sqrt{2}|D|}{\sqrt{2}-1}
\sum_{\ell > m} \lambda_\ell \Vert \phi_{\ell} \Vert_{L^\infty}^2.
\]
Moreover, let $\kappa\in \N$, and define $M = M(m) = \lceil\kappa|D| m \omega_m^2 \log(m)\rceil$. Then, with probability 
\[
\Prob \ge 
1 - \frac{2}{m^{\kappa - 2}},
\]
it holds that for $M$ iid uniformly chosen random sensors $X_1,\dots, X_M\in D$, the encoder 
\[
\cE: C(D) \to \R^M,
\quad
u(x) \mapsto (u(X_1), \dots, u(X_M)),
\]
possesses a decoder $\cD$ given by \eqref{eq:encode-general}, such that
\begin{equation}
\label{eq:enerr1}
(\Err_{\cE})^2
\le 
\frac{\sqrt{2}|D|}{\sqrt{2}-1} \sum_{\ell>m} \lambda_\ell (1 + \Vert \phi_\ell \Vert_{L^\infty}^2).
\end{equation}
\end{lemma}
A direct consequence of this lemma is Theorem \ref{thm:random-enc}, stated in the overview Section \ref{sec:over-enc}, and whose proof is provided in appendix \ref{app:pf301}. Theorem \ref{thm:random-enc} shows the remarkable result that \emph{randomly chosen sensor points} can lead to an optimal (up to a $\log$) decay of the encoding error \eqref{eq:encoding}, corresponding to the DeepONet \eqref{eq:donet}. 
\subsubsection{Examples}
In this section, we seek to illustrate the estimates on the encoding error for concrete prototypical examples. We recall that the encoding error \eqref{eq:encoding} is independent of the operator $\G$, depending only on the probability measure $\mu\in \P(X)$. We assume that $X = L^2(D)$ in the following. Assuming furthermore that $\mu$ has finite second moments $\int_{L^2(D)} \Vert u \Vert^2 \, d\mu(u) < \infty$, then it is well-known \cite{STU} that by the Karhunen-Lo\`eve expansion, we can write $\mu$ as the law of a random variable $u = u(\slot; \bm{Z})$, of the form
\begin{align} \label{eq:karhunen-loeve}
u(\slot; \bm{Z})
=
\bar{u} + \sum_{\ell=1}^\infty \sqrt{\lambda_\ell} Z_\ell \phi_\ell,
\quad
\bm{Z} = (Z_1,Z_2,\dots),
\end{align}
where $\bar{u}\in L^2(D)$ is the mean, $\phi_1,\phi_2,\dots \in L^2(D)$ are an orthonormal basis consisting of eigenfunctions of the covariance operator $\Gamma_\mu$ of $\mu$, $\lambda_1\ge \lambda_2 \ge \dots$ denote the corresponding eigenvalues, and $Z_1,Z_2,\dots$ are real-valued random variables satisfying
\[
\E[Z_\ell] = 0, 
\quad
\E[Z_k Z_{\ell}] = \delta_{k\ell},
\quad
\forall \, k,\ell \in \N.
\]
In practice, the probability measure $\mu$ is often specified as the law of an expansion of the form \eqref{eq:karhunen-loeve} \cite{STU} (see e.g. examples 5 and 6 of \cite{deeponets}). This provides a very convenient method to sample from the measure $\mu$, which is defined on an infinite dimensional space $X$.

A particularly important class of measures on infinite-dimensional spaces are the so-called \emph{Gaussian measures} \cite{STU}. A Gaussian measure $\mu\in \P(L^2(D))$ is uniquely characterized by its mean $\overline{u} = \int_{L^2(D)} u \, d\mu(u) \in L^2(D)$, and its covariance operator $\Gamma = \int_{L^2(D)} (u\otimes u) \, d\mu(u)$, which may, e.g., be expressed in terms of a covariance integral kernel $k(x,x')\in L^2(D\times D)$, and through which the covariance operator $\Gamma$ is defined by integration against $k(x,x')$:
\begin{equation}
    \label{eq:covker}
\Gamma : L^2(D) \to L^2(D), 
\quad
u(x) \mapsto 
\int_{D} k(x,x') u(x') \, dx'.
\end{equation}
\subsubsection{Encoding error for a particular Gaussian measure.} 
For this concrete example, we consider a Gaussian measure that was used in the context of DeepONets in \cite{deeponets}. For definiteness and simplicity of exposition, we consider the one-dimensional periodic case by setting $D = \T = [0,2\pi]$ and consider a Gaussian measure $\mu$, defined on it with the periodization,
\begin{align} \label{eq:quadraticexp-kernel}
k_p(x,x') := \sum_{h \in 2\pi \Z} 
\exp\left(
\frac{-|x-x'-h|^2}{2\ell^2}
\right).
\end{align}
of the frequently used covariance kernel
\[
k(x,x') = \exp\left(\frac{-|x-x'|^2}{2\ell^2}\right).
\]
We have the following result, proved in Appendix \ref{app:pf310} on the encoder and the encoding error for this Gaussian measure,
\begin{lemma} \label{lem:encoding-quadraticexp}
Let $\mu$ be given by the law of the Gaussian process with covariance kernel $k_p(x,x')$ \eqref{eq:quadraticexp-kernel}. Let 
\[
x_j = \frac{2\pi(j-1)}{m},
\]
denote equidistant points on $[0,2\pi]$ for $m = 2K+1$, $K\in \N$. Define the \emph{pseudo-spectral} encoder $\cE: L^2(\T) \to \R^m$ by $\cE(u) = (u(x_1), \dots, u(x_m))$, then a decoder $\cD: \R^m \to L^2(\T)$ is given via the discrete Fourier transform:
\begin{equation}
    \label{eq:fdecode1}
\cD(u_1, \dots, u_m)(x)
=
\sum_{k=-K}^{K}
\widehat{u}_k e^{ik x},
\end{equation}
where
\[
\widehat{u}_k 
:= 
\frac{1}{m} \sum_{j=1}^{m}
u_j e^{-2\pi i jk}.
\]
The encoding error \eqref{eq:encoding} for the resulting DeepONet \eqref{eq:donet} satisfies,
\begin{align}
\label{eq:enerrg}
\Err_{\cE} &\le 
\sqrt{
2\sum_{|k|>\lfloor m/2 \rfloor} \lambda_k
}
\le 
\sqrt{4\pi \, \erfc\left( \frac{\lfloor m/2 \rfloor \ell}{\sqrt{2}} \right)}.
\end{align}
Here $\lambda_k = \sqrt{2\pi} \, \ell \exp(-(\ell k)^2/2)$ are the eigenvalues of the covariance operator \eqref{eq:covker} for the Gaussian measure and $\erfc$ denotes the complementary error function, i.e.
\[
\erfc(x) := \frac{2}{\sqrt{\pi}} \int_x^\infty e^{-t^2} \, dt
\]

\end{lemma}
By an asymptotic expansion of $\erfc(x)$, the above lemma implies that there exists a constant $C>0$, such that
\begin{equation}
\label{eq:enerrg1}
\Err_{\cE}
\le 
C \exp\left(-(\lfloor m/2 \rfloor \ell)^2/4 \right)
\lesssim
C \exp\left(-\gamma m^2 \ell^2 \right), \quad \forall \, \gamma < \frac{1}{16},
\end{equation}
i.e., the encoding error decays \emph{super-exponentially} in this case.

Thus, we have a very fast decay of the encoding error with a \emph{pseudo-spectral encoder} for this Gaussian measure. In view of the earlier discussion in this subsection, it is natural to examine what happens if the encoder was a \emph{random encoder}, i.e. based on pointwise evaluation at uniformly distributed random points in $[0,2\pi]$. Given Lemma \ref{lem:aliasing-gen}, one would expect a similar super-exponential decay (modulo a logarithmic correction). This is indeed the case as shown in the following Lemma (proved in appendix \ref{app:pf311}),
\begin{lemma}
\label{lem:311}
Let $\mu$ be a Gaussian measure on the one-dimensional periodic torus $D = \T =  [0,2\pi]$, characterized by the covariance kernel \eqref{eq:quadraticexp-kernel}. Let $X_1,\cdots,X_M$ be uniformly distributed random sensors on $D$. There exists a constant $\gamma >0$, such that with probability $1$, the encoding error \eqref{eq:encoding} corresponding to the random encoder (pointwise evaluations at random sensors) is bounded by,
\begin{equation}
\label{eq:ren1}
\Err_{\cE}
\lesssim
\exp\left(-
\frac{\gamma M^2 \ell^2}{\log(M)^2}
\right).
\end{equation}
\end{lemma}
Comparing the bounds \eqref{eq:enerrg1} and \eqref{eq:ren1}, we see that the random encoder also decays super-exponentially (up to a log). However, as remarked before, no information about the underlying measure is used in defining the random encoder. 
\subsubsection{Encoding error for a parametrized measure.}
\label{sec:parametrized-mu}
As a second example, we define the underlying measure $\mu \in \P(L^2(D))$, as the \emph{law} of its Karhunen-Loeve expansion \eqref{eq:karhunen-loeve} with the following ansatz on the resulting random field,
\begin{align} \label{eq:expansion}
u(x;Y) = \bar{u}(x) + \sum_{\ell = 1}^\infty Y_\ell \alpha_\ell \psi_\ell(x),
\end{align}
with $\alpha_\ell > 0$, $\bar{u}, \, \psi_\ell\in L^2(D)$ are such that $\sum_{\ell=1}^\infty \alpha_\ell \Vert \psi_\ell\Vert_{L^2} < \infty$ is bounded, and where $Y_\ell \in [-1,1]$ are mean-zero random variables, distributed according to some measure $d\rho(\bm{y})$ on $\bm{y}\in [-1,1]^{\N}$. In this case, the series in \eqref{eq:expansion} converges uniformly in $L^2(D)$ for any $\bm{y} = (y_j)_{j\in \N}\in [-1,1]^\N$, and 
\[
\Vert u(\slot;\bm{y}) \Vert_{L^2} \le \Vert \bar{u} \Vert_{L^2} + \sum_{\ell=1}^\infty \alpha_\ell \Vert \psi_\ell \Vert_{L^2}.
\]
For the sake of definiteness, we shall only discuss a prototypical case, where the (spatial) domain $D$ is either periodic, i.e. $D = \T^d$, or rectangular\footnote{Generalization to a more general domain of the form $D = \prod_{i=1}^d [\alpha_i,\beta_i]$, for $\alpha_i<\beta_i$ is straight-forward.} $D = [0,2\pi]^d$, and the expansion functions are given by the trigonometric basis $\{\fb_k\}_{k \in \Z^d}$, indexed by $k\in \Z^d$ (check notation from Appendix \ref{app:Fourier}). The underlying ideas readily extend to more general choices of basis functions $\psi_\ell$. To be definite, we will assume that the random field $u({x};\bm{y})$ is expanded as 
\begin{align} \label{eq:fourier-random}
u({x};\bm{y}) = \overline{u}({x}) +  \sum_{{k}\in \Z^d} y_{{k}} \alpha_{{k}} \fb_{{k}}({x}),
\end{align}
where $\overline{u} \in C^\infty(D)$, and such that there exist constants $C_\alpha>0$, $\ell>0$, such that 
\begin{align} \label{eq:law-decay}
|\alpha_k|
\le 
C_\alpha \exp(-\ell|k|_\infty), \quad \forall \, k\in \Z^d.
\end{align}
We furthermore assume that the $Y_k \in [-1,1]$, $k\in \Z^d$, are centered random variables, implying that $\E[u] = \overline{u}$. We let $\mu \in \P(L^2(D))$ denote the law of the random variable $u(\slot; (Y_k)_{k\in \Z^d})$. By the assumed decay \eqref{eq:law-decay}, we have $\supp(\mu) \subset C^\infty(D)$. We remark that one can also readily consider the setup where the coefficients $\alpha_k$ decay at an algebraic rate. Note that the expansion \eqref{eq:fourier-random} appears to be similar to that of the Karhunen-Loeve expansion of the Gaussian measure, considered in the last section. The main difference lies in the fact that $Y_k$ are no longer assumed to be normally distributed, nor necessarily iid. 

Given $\mu$ as the law of random field \eqref{eq:fourier-random} as the underlying measure, we need to construct a suitable encoder and decoder and then estimate the resulting encoding error \eqref{eq:encoding}. To this end, we adapt the pseudo-spectral encoder and the resulting discrete Fourier transform based decoder \eqref{eq:fdecode1} to the current setup. For multi-indices $i = (i_1, \dots, i_d) \in \{0,\dots, 2N\}^d$, let 
\[
x_{i}
:=
\frac{
2\pi \bm{i}
}{
2N+1
}
=
\left(
\frac{2\pi i_1}{2N+1}, \frac{2\pi i_2}{2N+1}, \dots, \frac{2\pi i_d}{2N+1} 
\right),
\]
denote the $(2N+1)^d$ points on $[0,2\pi]^d$ on an equidistant cartesian grid with grid size $2\pi /(2N+1)$. In the following, we will denote by $\cI_N$ the index set 
\[
\cI_N
:=
\set{
i = (i_1, \dots, i_d)
}{
i_r \in \{0,\dots, 2N\}, \; \forall r=1,\dots, d
}.
\]
Define the encoder $\cE: C(D) \to \R^{m}$, for $m=(2N+1)^d=|\cI_N|$, by 
\begin{align} \label{eq:holomorphic-encoder}
\cE(u) 
=
(u(x_{i}))_{i \in \cI_N}.
\end{align}

To construct a suitable decoder $\cD: \R^{m}\to L^2(D)$ corresponding to the encoder $\cE$ above, we will first recover (an approximation of) the coefficients $y_k$ from the encoded values $\cE(u(\slot;\bm{y}))$. This can be achieved by the following sequence of mappings:

\begin{enumerate}
\item Subtract the mean $\overline{u}$: Define 
\begin{align} \label{eq:map-mean}
\cM: \R^m \to \R^m, 
\quad
(u_i)_{i\in \cI_N}
\mapsto
(u_{i} - \overline{u}(x_{i}))_{i\in \cI_N}.
\end{align}
\item Discrete Fourier transform: Define
\begin{align} \label{eq:map-fourier}
\cF\cT: \R^m \to \R^m,
\quad
(\tilde{u}_{i})_{i\in \cI_N}
\mapsto
\left(
\frac{(2\pi)^d}{|\cI_N|}
\sum_{i\in \cI_N}
\tilde{u}_{i}\, \fb_{k}(x_{i})
\right)_{k \in \cK_N},
\end{align}
where the set of Fourier wavenumbers $k \in \cK_N$ is given by 
\begin{align*} 
\cK_N = 
\set{
k = (k_1,\dots, k_d)
}{
k \in \Z^d, \; -N \le k_j \le N \; \forall j
}.
\end{align*}
Note that we have $\cF\cT(\fb_{k}) = \fb_{k}$ for all $k\in \cK_N$.
\item Approximation of $(y_1,\dots, y_m)$: Given the discrete Fourier coefficients $\hat{u}_{k} = \cF\cT(\tilde{u})_{k}$, $k\in \cK_N$, we define (with $\cJ = \Z^d$)
\begin{gather} \textbf{\label{eq:map-coeff}}
\left\{
\begin{aligned}
&\cY: \; \R^m \to [-1,1]^{\cJ},
\\
&\left(\hat{u}_{k} \right)_{k\in \cK_N}
\mapsto 
(\hat{Y}_j)_{j\in \cJ} := \left(\shrink\left( \tilde{Y}_j \right)\right)_{j\in \cJ}, 
\end{aligned}
\right.
\end{gather}
where 
\begin{align} \label{eq:Ytilde}
\tilde{Y}_j := 
\begin{cases}
\hat{u}_{k} / \alpha_{{k}}
, & (j = k \in \cK_N), \\
0, &(\text{otherwise}),
\end{cases}
\end{align}
and the shrink-operator $\shrink: \R \to [-1,1]$ ensures that $\hat{Y}_k \in [-1,1]$ for all $k$:
\[
\shrink(Y) 
=
\begin{cases}
Y, & |Y|\le 1, \\
Y/|Y|, & |Y|>1.
\end{cases}
\]
\end{enumerate}

\noindent
Finally, the decoder $\cD: \R^m \to L^2(\T^d)$ is defined by the composition 
\begin{align} \label{eq:holomorphic-decoder}
\cD(u_{i})
:=
u(\slot, \hat{Y}(u_{i})),
\end{align}
where $u(\slot; Y)$ is given by \eqref{eq:fourier-random} and
\begin{align} \label{eq:Yhat}
\hat{Y}(u_{i}) := (\cY\circ\cF\cT\circ \cM)(u_{i}).
\end{align}
It is easy to see that the main difference between the above encoder/decoder pair and the pseudo-spectral encoder/decoder of \eqref{eq:fdecode1} is the non-linear shrink operation in \eqref{eq:map-coeff}, which we introduce to ensure that $\hat{Y} \in [-1,1]^\N$. As a consequence of this observation, we find that the usual error estimates for pseudo-spectral methods imply similar error estimates for the encoding error $\Vert \cD \circ \cE(u) - u \Vert_{L^2}$. For instance, we have the following proposition (proved in Appendix \ref{app:pf312}),
\begin{proposition} \label{prop:ps-spectral}
Let $s>d/2$, and assume that $u\in H^s(D)$. There exists a constant $C=C(s,d)>0$, such that the encoder/decoder pair $(\cE,\cD)$ defined by \eqref{eq:holomorphic-encoder} and \eqref{eq:holomorphic-decoder} satisfy the estimate 
\begin{equation}
\label{eq:enerrsobo}
\Vert \cD\circ \cE(u) - u \Vert_{L^2}
\le
C N^{-s} \Vert u \Vert_{H^s}.
\end{equation}
\end{proposition}
On the other hand, if the coefficients $\alpha_k$ decay exponentially as in the random field \eqref{eq:expansion}, one can expect exponential decay rates for the encoding error as in the following Theorem,
\begin{theorem} \label{thm:holomorphic-encoding}
Let $\mu \in \P(L^2(D))$, with $D = \T^d$ or $D = [0,2\pi]^d$, denote the law of the random field $u(\slot;Y)$ defined by \eqref{eq:fourier-random}, with random variables $Y = (Y_j)_{j\in \cJ} \in [-1,1]^\cJ$, $\cJ = \Z^d$, and with $\alpha_k$ satisfying the decay assumption \eqref{eq:law-decay}. Given $N\in \N$, consider the encoder/decoder pair $(\cE,\cD)$ based on the discrete Fourier transformation on a regular grid with grid size $m = (2N+1)^d$ on $D$. Then there exists constants $C, c > 0$, independent of $m$, such that the encoding error $\Err_{\cE}$ for the encoder/decoder pair $\cE$, $\cD$ defined by \eqref{eq:holomorphic-encoder} and \eqref{eq:holomorphic-decoder}, can be bounded by
\begin{align} \label{eq:holomorphic-encoding}
\Err_{\cE} \le C \exp(-c\ell m^{1/d}).
\end{align}
Furthermore, if $\G: L^2(D)\to L^2(U)$ is an operator, and $\cF: [-1,1]^\cJ \to L^2(U)$, $\bm{y}\mapsto \cF(\bm{y})$ is defined by 
\[
\cF(\bm{y}) := \G(u(\slot;\bm{y})),
\]
then we have the identity
\begin{align} \label{eq:GD-FY}
\G \circ \cD\left( (u_{i})_{i\in \cI_N} \right)
=
\cF(\hat{Y}\left( (u_{i})_{i\in \cI_N} \right),
\quad
\forall \,  (u_{i})_{i\in \cI_N} \in \R^{\cI_N} \simeq \R^m.
\end{align}
\end{theorem}
This theorem, proved in Appendix \ref{app:pf313}, shows that the encoding error for a very general form of the underlying measure $\mu$, decays exponentially in the number of sensors and suggests that DeepONets will have a small encoding error with a few sensors. The map $\bm{u} \mapsto \hat{Y}(\bm{u})$ defined by \eqref{eq:Yhat}, plays a key role in defining the decoder \eqref{eq:holomorphic-decoder} as well as the action of operator $\G$ on the decoder. It turns out that this map can be efficiently approximated by a neural network of moderate size as follows:
\begin{lemma} \label{lem:Yhat-NN}
Let $N\in \N$, and denote $m := (2N+1)^d = |\cK_N|$. Let $\kappa: \{1,\dots, m\} \to \cK_N$ be a bijection. There exists a constant $C>0$, independent of $N$, $m$, such that for every $N$ there exists a ReLU neural network $\cN: \R^m \to \R^m$, with 
\[
\size(\cN) \le C(1+m\log(m)), 
\quad
\depth(\cN) \le C(1+\log(m)),
\]
and such that $\cN(\bm{u}) = (\hat{Y}_{\kappa(1)}(\bm{u}),\dots, \hat{Y}_{\kappa(m)}(\bm{u}))$, for all $\bm{u}\in \R^m$.
\end{lemma}
The proof is based on a simple observation that,
$\hat{Y} = \cY \circ \cF \cT \circ \cM$, with
\begin{enumerate}
\item $\cM$ an affine mapping, introducing a bias,
\item $\cF\cT$ a linear mapping, implementing the discrete Fourier transform
\item $\cY$ a linear scaling followed by a shrink operation.
\end{enumerate}
The map $\cM$ can evidently be represented by a neural network of $\depth = \mathcal{O}(1)$, and $\size = \mathcal{O}(m)$. The discrete Fourier transform can be efficiently computed using the fast Fourier transform (FFT) algorithm in $\mathcal{O}(N^d \log(N)) = \mathcal{O}(m\log(m))$ operations. We note that each step of this recursive algorithm is linear, and requires $\mathcal{O}(m)$ multiplications and $\mathcal{O}(\log(m))$ recursive steps to compute the FFT. Each step in the recursion can be represented exactly by a finite number $\mathcal{O}(1)$ of ReLU neural network layers of size $\mathcal{O}(m)$. The $\mathcal{O}(\log(m))$ steps in the recursion can thus be represented by a composition of $\mathcal{O}(\log(m))$ neural network layers. Thus, the whole algorithm can be represented by a neural network of $\size = \mathcal{O}(m\log(m))$ and depth $\mathcal{O}(\log(m))$.  Finally, the linear scaling step can clearly be represented by a neural network of $\depth=\mathcal{O}(1)$ and $\size = \mathcal{O}(m)$, and the shrink operation can be written in the form
\[
\shrink(Y) 
=
1 - \max(0,2-\max(0,1+Y))
=
1-\sigma( 2-\sigma(1+Y)),
\]
where $\sigma(x) = \max(0,x)$ denotes the ReLU activation function. Hence, $\cY$ can be represented by a neural network of $\size = \mathcal{O}(m)$ and $\depth = \mathcal{O}(1)$. Combining these three steps, we conclude that also the composition $\hat{Y}$ can be represented by a neural network of $\size = \mathcal{O}(m \log(m))$ and $\depth=\mathcal{O}(\log(m))$. Similarly, any smooth activation function such as sigmoid or $\tanh$ can be used to define a neural network of the same size to approximate $\hat{Y}$.
\subsection{Bounds on the approximation error \eqref{eq:approximation}}
\label{sec:approximation}

Given a particular choice of encoder/decoder and reconstruction/projection pairs $(\cE,\cD)$ and $(\cR,\cP)$, the approximation error $\Err_{\cA}$ \eqref{eq:approximation} for the approximator $\cA: \R^m \to \R^p$ in the DeepONet \eqref{eq:donet} is a measure for the non-commutativity of the following diagram:
\[
\begin{tikzcd}
L^2(D) 
\arrow[r, "\G"] 
\arrow[d, shift right=0ex, "\cE" left, bend right=30, dashrightarrow] 
& L^2(U) 
\arrow[d, shift right=0ex, "\cP" left] 
\\
\R^m 
\arrow[u, shift right=0ex, "\cD" right] 
\arrow[r, rightsquigarrow, "\cA"]   
& \R^p 
\arrow[u, shift right=0ex, "\cR" right, dashrightarrow, bend right=30]
\end{tikzcd}
\]
i.e., it measures the error in the approximation $\cA \approx \cP\circ\G\circ \cD$. Thus, bounding the approximation error $\Err_{\cA}$ can be viewed as a special instance of the general problem of the neural network approximation of high-dimensional mappings $\R^m \to \R^p$. Our aim in this section is to review some of the available results on neural network approximation in a finite-dimensional setting and relate it to the problem of deriving bounds on the approximation error \eqref{eq:approximation} for the DeepONet \eqref{eq:donet}. We start by considering the neural network approximation for regular high-dimensional mappings.
\subsubsection{Regular high-dimensional mappings}
\label{sec:regular-high-dim}

Evidently, the neural network approximation of a mapping $G: \R^m \to \R^p$, $x \mapsto G(x) = (G_1(x), \dots, G_p(x))$ can be carried out by independently approximating each of the components $G_j: \R^m \to \R$ by a neural network $\cA_j: \R^m \to \R$, and then combining these individual approximations to a single neural network $\cA: \R^m \to \R^p$, $x\mapsto \cA(x) = (\cA_1(x), \dots, \cA_p(x))$, with
\[
\size(\cA) 
= 
\sum_{j=1}^p \size(\cA_j)
\le
p \max_{j=1,\dots, p} \size(\cA_j)
.
\]
The approximation of $G_j: \R^m \to \R$ (over a bounded domain $K \subset \R^m$) by neural networks is a fundamental approximation theoretic problem. One approach for deriving general approximation results relies on the Sobolev regularity of $G_j$. As a prototype of the available results in this direction, we cite the following result due to Yarotsky \cite{Yarotsky2017}, for $G_j \in W^{k,\infty}([0,1]^m)$,
\begin{theorem}[{\cite[Theorem 1]{Yarotsky2017}}] \label{thm:yarotsky}
Let $m$, $k \in \N$ be given. There exists a constant $C = C(m,k)>0$, such that for any $\e\in (0,1)$ and $G_j: [0,1]^m \to \R$ with $\Vert G_j \Vert_{W^{k,\infty}} \le 1$, there exists a ReLU neural network $\cA_j: [0,1]^m \to \R$, with
\[
\depth(\cA_j) \le C (1+\log(\e^{-1})),
\quad
\size(\cA_j) \le C \e^{-m/k} (1+\log(\e^{-1})),
\]
such that 
\[
\Vert G_j(x) - \cA_j(x) \Vert_{L^\infty([0,1]^m)} \le \epsilon.
\]
\end{theorem}

In the particular case of a Lipschitz mapping $G: \R^m\to \R^p$, the upper limit on the required size of the approximating neural network $\cA$ is thus $\mathcal{O}(p\e^{-m})$. In particular, this size scales exponentially in the input dimension (=number of DeepONet sensors) $m$. By Definition \ref{def:cod}, this constitutes a \emph{curse of dimensionality} as the number of sensors $m$ need to grow as $\e \to 0$ from bounds such as \eqref{eq:enerrsobo} on the encoding error \eqref{eq:encoding}. Thus, the above Theorem of Yarotsky may not suffice to find optimal sizes of the approximator network $\cA$ in the DeepONet \eqref{eq:donet}. 
\subsubsection{Holomorphic infinite-dimensional mappings} \label{sec:be-holomorphy}
Given this possible curse of dimensionality in bounding the approximation error \eqref{eq:approximation} for Lipschitz continuous maps, we seek to find a class of nonlinear operators $\cG$, mapping infinite-dimensional Banach spaces, for which this curse of dimensionality can be avoided. One possible class is the class of \emph{holomorphic mappings} $\bm{y} \mapsto \cF(\bm{y})$, which has been shown in recent papers \cite{SchwabZech2019,OSZ2019,OSZ2020} to be efficiently approximated by ReLU neural networks, \emph{breaking the curse of dimensionality}. For instance, the class of mappings considered in \cite{OSZ2020} are infinite dimensional mappings 
\begin{align*} 
\cF: [-1,1]^\N \to V,
\quad 
\bm{y} = (y_j)_{j\in \N} 
\to 
\cF(\bm{y}),
\end{align*}
where $V$ is a Banach space. To simplify notation, we shall replace $\N$ by an arbitrary countable index set $\cJ$ and consider mappings
\begin{align} \label{eq:param-F}
\cF: [-1,1]^{\cJ} \to V,
\quad 
\bm{y} = (y_j)_{j\in \cJ} 
\to 
\cF(\bm{y}),
\end{align}
in the following. In this context, we require the following definition \cite{OSZ2020}:

\begin{definition}[{$(b,\epsilon)$-admissibility}] \label{def:be-admissibility}
Let $V$ be a Banach space. Let $\bm{b} = (b_j)_{j\in \N}$ be a given sequence of monotonically decreasing positive reals $b_j>0$ such that $\bm{b}\in \ell^p(\N)$ for some $p\in (0,1]$. Let $\kappa: \N \to \cJ$ be an enumeration of the index set $\cJ$. A poly-radius $\bm{\rho} = (\rho_j)_{j\in \cJ} \in (1,\infty)^\cJ$ is called $(\bm{b},\e;\kappa)$-admissible for some $\e > 0$, if 
\[
\sum_{j\in \N} b_j (\rho_{\kappa(j)}-1) \le \e.
\]
\end{definition}
We further recall that for a radius $\rho>1$, the Bernstein ellipse $\cE_{\rho} \subset \C$ is defined by
\begin{align} \label{eq:bernstein-ellipse}
\cE_{\rho} 
:=
\set{
\frac{z+z^{-1}}{2}
}{
0\le |z| < \rho
}.
\end{align}
We define holomorphy, following \cite[Definition 3.3]{OSZ2020} as:
\begin{definition}[{$(b,\epsilon)$-holomorphy}] \label{def:be-holomorphy}
Let $V$ be a Banach space. Let $\kappa: \N \to \cJ$ be an enumeration of the index set $\cJ$. A continuous mapping $\cF: [-1,1]^\cJ \to V$ is called \define{$(\bm{b},\epsilon,\kappa)$-holomorphic}, if there exists a constant $C = C(\cF)$, such that the following holds: For every $(\bm{b},\epsilon,\kappa)$-admissible $\bm{\rho}$ with $\cE_{\bm{\rho}} := \prod_{j\in \cJ} \cE_{\rho_j}\subset \C^\cJ$, there exists an extension $\tF: \cE_{\bm{\rho}} \to V_{\C}$, such that 
\begin{align} \label{eq:be-holomorphy}
\bm{z} \mapsto \tF(\bm{z}) \text{ is holomorphic} 
\end{align}
as a function of each $z_j\in E_{\rho_j}$, $j\in \cJ$, and such that 
\begin{align} \label{eq:be-boundedness}
\sup_{\bm{z}\in E_{\bm{\rho}}} \left\Vert \tF(\bm{z}) \right\Vert_{V_{\C}} \le C(\cF).
\end{align}
Here $V_\C$ denotes the complexification of the (real) Banach space $V$.
\end{definition}
We remark that such \emph{holomorphic operators} arise naturally in the context of elliptic and parabolic PDEs, for instance as the data to solution map of diffusion equations with random coefficients \cite{SchwabZech2019} and references therein. We see further examples of these operators in the next section. 

We can now state the following result which follows from \cite[Theorem 4.11]{OSZ2020} (our statement here is closer to the formulation of \cite[Theorem 3.9]{SchwabZech2019}):
\begin{proposition} \label{prop:be-expansion}
Let $V$ be a Banach space. Let $\cF: [-1,1]^\cJ \to V$ be a $(\bm{b},\epsilon,\kappa)$-holomorphic map for some $\bm{b}\in \ell^q(\N)$ and $q\in (0,1)$, and an enumeration $\kappa: \N \to \cJ$. Then there exists a constant $C>0$, such that for every $N\in \N$, there exists an index set 
\[
\Lambda_N
\subset
\set{
\bm{\nu} = (\nu_1,\nu_2,\dots) \in \textstyle{\prod}_{j\in \cJ} \N_0
}{
\nu_j\ne 0 \; \text{for finitely many $j\in \cJ$}
},
\]
with $|\Lambda_N|=N$, a finite set of coefficients $\{c_{\bm{\nu}}\}_{\bm{\nu}\in \Lambda_N} \subset V$, and a ReLU network $\cN: \R^N \to \R^{\Lambda_N}$, $y\mapsto \{\cN_{\bm{\nu}}(y)\}_{\bm{\nu}\in \Lambda_N}$ with
\begin{gather*}
\size(\cN) \le C(1+N\log(N) \log\log(N)), 
\\
\depth(\cN) \le C(1+\log(N) \log\log(N)),
\end{gather*}
and such that
\begin{align} \label{eq:be-expansion}
\sup_{\bm{y}\in [-1,1]^\cJ} 
\left\Vert
\cF(\bm{y}) - \sum_{\bm{\nu}\in \Lambda_N} c_{\bm{\nu}} \cN_{\bm{\nu}}(y_{\kappa(1)},\dots, y_{\kappa(N)})
\right\Vert_{V} 
\le C N^{1-1/q}.
\end{align}
\end{proposition}
The following corollary of the above proposition, proved in Appendix \ref{app:pf314}, enables us to apply the result of the theorem to the specific structure of the approximator neural network $\cA$ and the resulting approximation error \eqref{eq:approximation} for our DeepONet \eqref{eq:donet}.
\begin{corollary} \label{cor:holomorphic-NN}
Let $V$ be a Banach space. Let $\cF: [-1,1]^\cJ \to V$ be a $(\bm{b},\epsilon,\kappa)$-holomorphic map for some $\bm{b}\in \ell^q(\N)$ and $q\in (0,1)$, where $\kappa: \N\to \cJ$ is an enumeration of $\cJ$. In particular, it is assumed that $\{b_{j}\}_{j\in \N}$ is a monotonically decreasing sequence. If $\cP: V \to \R^p$ is a continuous linear mapping, then there exists a constant $C>0$, such that for every $m\in \N$, there exists a ReLU network $\cN: \R^m \to \R^p$, with
\begin{gather*}
\size(\cN) \le C(1+pm\log(m) \log\log(m)), 
\\
\depth(\cN) \le C(1+\log(m) \log\log(m)),
\end{gather*}
and such that
\[
\sup_{\bm{y}\in [-1,1]^\cJ} \Vert \cP \circ \cF(\bm{y}) - \cN(y_{\kappa(1)},\dots, y_{\kappa(m)})\Vert_{\ell^2(\R^p)} \le C \Vert \cP \Vert \,  m^{-s}, 
\]
where $s := q^{-1} - 1 > 0$ and $\Vert \cP \Vert = \Vert \cP \Vert_{V\to \ell^2}$ denotes the operator norm.
\end{corollary}

With the above setup, we recall that maps of the form \eqref{eq:param-F} arise naturally from general nonlinear operators $\G: X \to L^2(U)$, $u\mapsto \G(u)$, when the probability measure $\mu \in \P(X)$ is given in the parametrized form \eqref{eq:expansion}. As in the subsection \ref{sec:parametrized-mu}, we will focus on the case where either $D=\T^d$ is periodic, or $D = [0,2\pi]^d$ is a rectangular domain and an expansion \eqref{eq:expansion} in terms of the standard Fourier basis $\{\fb_k\}_{k\in \Z^d}$. If the underlying measure $\mu$ can be written as the law of a random field in the parametrized form \eqref{eq:expansion} with exponentially decaying coefficients \eqref{eq:law-decay}, then we can define a ``parametrized version'' of the operator $\G$ by the following mapping $\cF: [-1,1]^\cJ \to L^2(U)$, $\cJ = \Z^d$, defined by 
\begin{align} \label{eq:G-parametrized}
\cF(\bm{y}) := \G\left(u(\slot; \bm{y})\right),
\end{align}
for $\bm{y} = (y_k)_{k\in \Z^d}\in [-1,1]^\cJ$.

To show how the neural network approximation of such $\cF: [-1,1]^\cJ \to L^2(U)$ relates to the approximator network $\cA: \R^m \to \R^p$, we recall that the encoder/decoder pair $(\cE,\cD)$ of section \ref{sec:parametrized-mu} was constructed in terms of an approximate sensor data-to-parameter mapping $\bm{u} := (u_{i})_{i\in \cI_N} \mapsto \hat{Y}(\bm{u})$ such that (cf. Theorem \ref{thm:holomorphic-encoding})
\begin{align} \label{eq:GD-FY1}
(\cG\circ \cD)(\bm{u}) = \cF(\hat{Y}(\bm{u})),
\quad
\forall \, \bm{u} \in \R^m.
\end{align}
Let $\cP: L^2(U) \to \R^p$ be the projection mapping corresponding to a reconstruction $\cR: \R^p \to L^2(U)$, and let the neural network $\cN$ be constructed as in Corollary \ref{cor:holomorphic-NN}, corresponding to an enumeration $\kappa: \N \to \Z^d$, which enumerates $k\in \Z^d$ with increasing $|k|_\infty$, i.e., such that $j \mapsto |\kappa(j)|_\infty$ is monotonically increasing. Then, it is straightforward to see that for $m=(2N+1)^d$, the Fourier wavenumbers $\kappa(1), \dots, \kappa(m) \in \Z^d$ correspond precisely to the Fourier wavenumbers in
\[
\cK_N = 
\set{k\in \Z^d}{|k|_\infty \le N}.
\]
In this case, we define
\begin{align} \label{eq:holomorphic-approx}
\cA(\bm{u}) := \cN(\hat{Y}_{\kappa(1)}(\bm{u}), \dots, \hat{Y}_{\kappa(m)}(\bm{u})),
\end{align}
which we shall also denote more compactly as $\cA(\bm{u}) = \cN(\hat{Y}_\kappa(\bm{u}))$, where $\hat{Y}_\kappa = (\hat{Y}_{\kappa(1)}, \dots, \hat{Y}_{\kappa(m)})$. We now note the following theorem, proved in Appendix \ref{app:pf315}, on the approximation of $\cA$:
\begin{theorem} \label{thm:holomorphic-approx}
Let $\G: X \to L^2(U)$ be a non-linear operator. Assume that the parametrized mapping $\cF$ given by \eqref{eq:G-parametrized}, defines a $(\bm{b},\epsilon,\kappa)$-holomorphic mapping $\cF: [-1,1]^\cJ \to L^2(U)$, with $\bm{b}\in \ell^q(\N)$ and $\kappa: \N \to \cJ$ an enumeration. Assume that $\mu\in \cP(X)$ is given as the law of the random field \eqref{eq:expansion}. Let the encoder/decoder pair be constructed as in section \ref{sec:parametrized-mu}, so that \eqref{eq:GD-FY1} holds. Given an affine reconstruction $\cR: \R^p\to L^2(U)$, let $\cP: L^2(U) \to \R^p$ denote the corresponding optimal linear projection \eqref{eq:opt-proj}. Then given $k\in\N$, there exists a constant $C_k>0$, independent of $m$, $p$, such that the approximator $\cA: \R^m \to \R^p$ defined by \eqref{eq:holomorphic-approx} can be represented by a neural network with
\[
\begin{gathered}
\size(\cA) 
\le
C_k(1+mp \log(m) \log\log(m)), 
\\
\depth(\cA) \le C_k(1+m\log(m)\log\log(m)).
\end{gathered}
\]
and such that the approximation error $\Err_{\cA}$ can be estimated by
\[
\Err_{\cA} 
\le
C_k \Vert \cP \Vert \, m^{-k},
\]
where $\Vert \cP \Vert =  \Vert \cP \Vert_{L^2(U) \to \R^p}$ is the operator norm of $\cP$.
\end{theorem}
Thus, the above theorem shows that an approximator neural network $\cA$ of log-linear size in the product of the number of sensors $m$ and number of trunk nets $p$ can lead to very small approximation error if the underlying operator $\G$ yields a holomorphic reduction \eqref{eq:G-parametrized}. This will allow us to overcome the curse of dimensionality for the approximation error \eqref{eq:approximation} for the DeepONet \eqref{eq:donet} in many cases. 

\section{Error bounds on DeepONets in concrete examples.}
\label{sec:4}

In the last section, we decomposed the error \eqref{eq:approxerr} that a DeepONet \eqref{eq:donet} incurs in approximating a nonlinear operator $\G: X \to Y$, with an underlying measure $\mu \in \P(X)$, into the encoding error \eqref{eq:encoding}, the reconstruction error \eqref{eq:decoding} and approximation error \eqref{eq:approximation}. We provided explicit bounds on each of these three errors. In particular, the encoding error was estimated in terms of the spectral decay of the underlying covariance operator of the measure $\mu$ and it was shown that under a boundedness assumption on the eigenfunctions, even a random choice of sensor points provided an optimal encoding error (modulo a log term). More judicious choices of sensor points, for instance with the pseudo-spectral encoder \eqref{eq:fdecode1}, allowed us to recover optimal (up to constants) encoding error. Similarly, we showed that the reconstruction error \eqref{eq:decoding} relies on the spectral decay properties of the covariance operator with respect to the push-forward measure $\G_\#{\mu}$ and can be bounded by using the smoothness of the operator $\G$ as in \eqref{eq:smoothness-NN}. Finally, estimating the approximation error \eqref{eq:approximation} boils down to a neural network approximation of finite, but high, dimensional mappings and one can use either Sobolev regularity or if available, holomorphy, of the map $\cP \circ \cG \circ \cD$ to bound this component. The above discussion provides the following workflow to bound the DeepONet error \eqref{eq:approxerr} in concrete cases, i.e., for concrete instances of the operator $\G$ and the underlying measure $\mu$:
\begin{itemize}
    \item For a given measure $\mu$, estimate the spectral decay rate for the associated covariance operator. Use random sensors in the case of no further information about the measure to obtain almost optimal bounds on the encoding error. If more information is available, one can use bespoke sensor points to obtain optimal bounds on the encoding error.
    \item For a given operator $\G$, use smoothness of the operator to estimate the reconstruction error \eqref{eq:decoding} as in \eqref{eq:smoothness-NN}. 
    \item For the approximation error \eqref{eq:approximation}, use the regularity, in particular possible holomorphy, of the operator $\G$ to estimate this error.
\end{itemize}
The simplest examples for $\G$ correspond to those cases where it is  a \emph{bounded linear operator}. In this case, the above workflow is carried out in Appendix \ref{app:deeponet-linear}, where we show that DeepONet approximation of these general linear operators depends on the spectral decay of the underlying measure $\mu$ and on the approximation property of the trunk net $\tau$. However for nonlinear operators $\G$, one has to carry out the above workflow in each specific case. To this end, we will illustrate this workflow for four concrete examples of \emph{nonlinear operators}, that are chosen to represent different types of differential equations, namely, a nonlinear ODE, an elliptic PDE, a nonlinear parabolic and a nonlinear hyperbolic PDE. Within each class, we choose a concrete example that is widely agreed as a prototype for this class of problems. 
\subsection{A nonlinear ODE: Gravity pendulum with external force.}
\label{sec:pendulum}
\subsubsection{Problem formulation.}
We consider the following nonlinear ODE system, already considered in the context of approximation by DeepONets in \cite{deeponets}:
\begin{gather} \label{eq:pendulum0}
\left\{
\begin{aligned}
\frac{dv_1}{dt} &= v_2,
\\
\frac{dv_2}{dt} &= -\gamma\sin(v_1) + u(t).
\end{aligned}
\right.
\end{gather}
with  initial condition $v(0) = 0$ and where $\gamma>0$ is a parameter. Let us denote $v = (v_1, v_2)$, 
\[
g(v) := 
\begin{pmatrix}
v_2 \\
-\gamma\sin(v_1)
\end{pmatrix},
\quad
U(t) := 
\begin{pmatrix}
0 \\
u(t)
\end{pmatrix},
\]
so that equation \eqref{eq:pendulum0} can be written in the form
\begin{align} \label{eq:pendulum}
\frac{dv}{dt} = g(v) + U, \quad v(0) = 0.
\end{align}
In \eqref{eq:pendulum}, $v_1,v_2$ are the angle and angular velocity of the pendulum and the constant $\gamma$
denotes a frequency parameter. The dynamics of the pendulum is driven by an external force $u= u(t)$. It is straightforward to see that for each $r > 0$, there exists a constant $C_r>0$, such that
\begin{align} \label{eq:gderiv}
\Vert g^{(r)} \Vert_{L^\infty(\R)} \le C_r,
\end{align}
where $g^{r}$ denotes the $r$-th derivative.

With the external force $u$ as the input, the output of the system is the solution vector $v(t)$ and the underlying nonlinear operator is given by $\G: L^2([0,T]) \to L^2([0,T])$, $u \mapsto \cG(u) = v$. The following Lemma, proved in Appendix \ref{app:pf41}, provides a precise characterization of this operator.
\begin{lemma}
\label{lem:pend1}
There exists a constant $C = C(\Vert g^{(1)}\Vert_{L^\infty}, T)>0$, such that for any two $u,u'\in L^2([0,T])$, we have
\[
\Vert \G(u) - \G(u') \Vert_{L^2([0,T])}
\le
C\Vert u-u' \Vert_{L^2([0,T])}.
\]
In particular, $\G: L^2([0,T]) \to L^2([0,T])$, mapping $u(t) \to v(t)$, with $v$ being the solution of the ODE \eqref{eq:pendulum}, is Lipschitz continuous.
\end{lemma}
Next, in order to define the data for the DeepONet approximation (see Definition \ref{def:data}), we need to specify an underlying measure $\mu \in \P(L^2([0,T]))$. Following the discussion in the previous section and for the sake of definiteness, we choose a parametrized measure $\mu$, as considered in section \ref{sec:parametrized-mu}, as a law of a random field $u$, that can be expanded in the form 
\begin{align} \label{eq:pendulum-law}
u(t; Y) 
=
\sum_{k\in \Z} Y_k \alpha_k \fb_k\left(\frac{2\pi t}{T}\right),
\quad
t \in [0,T],
\end{align}
where $\fb_k(x)$, $k\in \Z$, denotes the one-dimensional standard Fourier basis on $[0,2\pi]$ (with notation of Appendix \ref{app:Fourier}) and the coefficients $\alpha_k\ge 0$ decay to zero. We will assume that there exist constants $C_\alpha,\ell > 0$, such that 
\[
\alpha_k \le C_\alpha \exp(-|k|\ell).
\]
Furthermore, we assume that the $\{Y_k\}_{k\in \Z}$ are iid random variables on $[-1,1]$.

With the above data for the DeepONet approximation problem, we will provide explicit bounds on the total error \eqref{eq:approxerr} for a DeepONet \eqref{eq:donet} approximating the operator $\G$. Following the workflow outlined above, we proceed to bound the following sources of error.
\subsubsection{Bounds on the encoding error \eqref{eq:encoding}.}
Given the underlying measure $\mu$ defined as the law of the random field \eqref{eq:pendulum-law}, we choose the encoder-decoder pair as used in section \ref{sec:parametrized-mu}, i.e., the encoder is the pointwise evaluation \eqref{eq:holomorphic-encoder} on equidistant points $t_1,\dots, t_m$ on $[0,T]$ and the corresponding decoder is given by \eqref{eq:holomorphic-decoder}. A direct application of Theorem \ref{thm:holomorphic-encoding} yields the following bound on the encoding error,
\begin{proposition}\label{prop:pendulum-enc}
Let $\mu\in \P(L^2([0,T]))$ denote the law of the random field $u(\slot;Y)$ defined by \eqref{eq:pendulum-law}. Given $N \in \N$ consider the encoder/decoder pair $(\cE,\cD)$ with $m = 2N+1$ grid points and given by \eqref{eq:holomorphic-encoder} and \eqref{eq:holomorphic-decoder}, respectively. Then, there exists a constant $C>0$, independent of $m$, such that the encoding error $\Err_{\cE}$ \eqref{eq:encoding} can be bounded by 
\[
\Err_{\cE} \le C \exp(-\ell \lfloor m/2 \rfloor).
\]
Furthermore, denoting by $\cF: [-1,1]^\Z \to L^2([0,T])$, $\bm{y}\mapsto \cF(\bm{y})$, the mapping
\begin{align}\label{eq:F-rep}
\cF(\bm{y}) := \G(u(\slot; \bm{y})),
\end{align}
we have the identity $\G \circ \cD(\bm{u}) = \cF(\hat{Y}(\bm{u}))$, for all $\bm{u} \in \R^m$.
\end{proposition}
Moreover, as in Lemma \ref{lem:Yhat-NN}, $\hat{Y}$ in \eqref{eq:F-rep} can be represented by a neural network, in the sense that $\hat{Y}_k \equiv 0$, for $|k|>N$, and there exists a neural network $\cN$ with 
\[
\size(\cN) = \mathcal{O}(m\log(m)), 
\quad
\depth(\cN) = \mathcal{O}(\log(m)),
\]
and $\cN(\bm{u}) = (\hat{Y}_{-N}(\bm{u}),\dots, \hat{Y}_0(\bm{u}), \dots, \hat{Y}_N(\bm{u}))$, for all $\bm{u}\in \R^m$.
\subsubsection{Bounds on the reconstruction error \eqref{eq:decoding}}
Following our program outlined above, we will bound the reconstruction error \eqref{eq:decoding} for a DeepONet approximation the operator $\G$ for the forced pendulum by appealing to the smoothness of the image $\im(\G)$ of $\G$. To this end, we have the following lemma, proved in Appendix \ref{app:pf42},
\begin{lemma}
\label{lem:pend2}
Let $T>0$, and consider the solution $v(t)$ of \eqref{eq:pendulum} for $t\in [0,T]$, where $g(v)$ satisfies $L^\infty$-bound \eqref{eq:gderiv} for all $k\in \N$, $k\ge 1$. Then for any $k\in \N$, there exists a constant $A_k>0$ (possibly depending on $g$ and $T$, in addition to $k$, but independent of $u$), such that
\[
\Vert v^{(k)} \Vert_{L^\infty} 
\le 
A_k \left( 1 + \Vert u \Vert_{H^k}^{k}\right).
\]
Here 
\[
\Vert u \Vert_{H^k} 
:= 
\sum_{\ell=0}^k \Vert u^{(\ell)} \Vert_{L^2([0,T])}.
\]
\end{lemma}

Given the desired smoothness of the image of the operator $\G$, we need to find a suitable reconstructor \eqref{eq:reconstruction}. To this end, we will use Legendre polynomials to build our reconstructor and have the following result, proved in Appendix \ref{app:pf43}, 
\begin{lemma} \label{lem:Legendre-rec}
If $\tilde{\tr}_k$, $k=1,\dots,p$ are the first $p$ Legendre polynomials, then the reconstruction error $\Err_{\tilde{\cR}}$ for the reconstruction mapping
\[
\alpha = (\alpha_1, \dots, \alpha_p)
\mapsto 
\tilde{\cR}(\alpha) 
=
\sum_{k=1}^p \alpha_k \tilde{\tr}_k,
\]
induced by the trunk net $\tilde{\bm{\tr}} = (0,\tilde{\tr}_1,\dots, \tilde{\tr}_p)$, satisfies 
\[
\Err_{\tilde{\cR}}
\le 
\frac{C}{p^{k}}
\left(
\int_{X} (1+\Vert u \Vert_{H^{k}([0,T])})^{2k} \, d\mu(u) 
\right)^{1/2}.
\]
For some constant $C = C(k,T) > 0$. In particular, if $\mu \in \P(L^2(D))$ is concentrated on $H^k(D)$ and $\int_{L^2(D)} \Vert u \Vert_{H^k}^{2k} \, d\mu(u) < \infty$, then there exists a constant $C = C(T,k,\mu)>0$ independent of $p$, such that
\[
\Err_{\tilde{\cR}}
\le
Cp^{-k}.
\]
\end{lemma}

From \cite[Proposition 2.10]{OSZ2019}, we have the following result; for any $p\in \N$, $\delta \in (0,1)$, there exists a ReLU neural network $\bm{L}_\delta: [0,T] \to \R^p$, $t\mapsto \bm{L}_\delta(t) = (L_{1,\delta}(t),\dots, L_{p,\delta}(t))$, which approximates the first $p$ Legendre polynomials $L_1(t),\dots, L_p(t)$ with
\[
\max_{j=1,\dots, p}
\Vert L_{j} - L_{j,\delta} \Vert_{L^\infty} \le \delta,
\]
and for a constant $C>0$, independent of $\delta$ and $p$, it holds
\[
\size(\bm{L}_\delta)
\le 
Cp^3 + p^2 \log(\delta^{-1}).
\quad
\depth(\bm{L}_\delta)
\le 
C(1+\log(p))(p + \log(\delta^{-1})),
\]
We will leverage the above neural network to find a suitable trunk net for the DeepONet approximation of the operator $\G$ for the forced pendulum. To this end, let $\tR$ denote the reconstruction
\[
\tR(\alpha) = \sum_{k=1}^p \alpha_k \tilde{\tr}_k,
\]
where $\tilde{\tr}_k = L_k(t)$ denotes the $k$-th Legendre polynomial. By Lemma \ref{lem:Legendre-rec}, for any $r\in \N$, there exists $C >0$, depending on $r$ but independent of $p$, such that we have $\Err_{\tR} \le C p^{-r}$, for all $p \in \N$.
Choosing $\bm{\tr}=\bm{L}_{\delta}$ as above with $\delta = p^{-r-3/2}$, it follows that for any $p\in \N$, there exists a trunk net $\bm{\tr}$ with 
\[
p^{3/2} \max_{k=1,\dots, p} \Vert \tr_k - \tilde{\tr}_k\Vert_{L^2([0,T])} \le p^{-r},
\]
and
\[
\size(\bm{\tr}) = \mathcal{O}(p^3).
\quad
\depth(\bm{\tr}) = \mathcal{O}(p\log(p)).
\]
From Lemma \ref{lem:err-rec-comp}, it follows that for the reconstruction $\cR: \R^p \to L^2([0,T])$, 
\[
\cR(\alpha) = \sum_{k=1}^p \alpha_k \tr_k,
\]
induced by the trunk net $\bm{\tr} = (0,\tr_1,\dots, \tr_p)$, we have
\begin{align*}
\Err_{\cR} 
&\le 
\Err_{\tR} + p^{3/2}\max_{k=1,\dots, p} \Vert \tr_k - \tilde{\tr}_k \Vert_{L^2([0,T])}
\le C p^{-r}.
\end{align*}

We summarize these observations in the following proposition,
\begin{proposition} \label{prop:pendulum-rec}
Let $\mu$ denote the law of the random field $u(\slot;Y)$ given by \eqref{eq:pendulum-law}. For any $r\in \N$, there exists a constant $C >0$, depending on $T$, $\mu$ and $r$, but independent of $p$, such that for any $p\in \N$, there exists a trunk neural network $\bm{\tr}: [0,T] \to \R^{p+1}$, $s \mapsto (\tr_0(s),\dots, \tr_p(s))$, with
\[
\size(\bm{\tr}) \le C (1+p^3),
\quad
\depth(\bm{\tr}) \le C(1+p\log(p))),
\]
such that the reconstruction error $\Err_{\cR}$ with corresponding projection $\cP$, satisfies the bound
\begin{equation}
    \label{eq:pendulum-rec-err}
\Err_{\cR} \le C p^{-r}, \quad (p\in \N).
\end{equation}
Furthermore, for any $p\in \N$, the reconstruction $\cR$ and the projection satisfy the uniform bound $\Lip(\cR), \, \Lip(\cP) \le 2$, where 
\begin{align*}
\Lip(\cR) &= 
\Lip
\left(
\cR:
(\R^p, \Vert \slot \Vert_{\ell^2})
\to 
(L^2([0,T]),\Vert \slot \Vert_{L^2})
\right)
\\
\Lip(\cP) &= 
\Lip
\left(
\cP: (L^2([0,T]),\Vert \slot \Vert_{L^2})
\to (\R^p, \Vert \slot \Vert_{\ell^2})
\right).
\end{align*}
\end{proposition}
Thus, we observe from \eqref{eq:pendulum-rec-err} that the reconstruction error for a DeepONet can be made very small for a moderate number of trunk nets. Moreover, the corresponding size of the trunk net is moderate in the number of trunk nets. 
\subsubsection{Bounds on the approximation error \eqref{eq:approximation}}
Following our workflow, we will derive bounds on the approximation error \eqref{eq:approximation} for a DeepONet \eqref{eq:donet} approximating the operator $\G$ for the forced gravity pendulum by showing that the corresponding operator is holomorphic in the sense of Definition \ref{def:be-holomorphy}. 

To this end, we assume that 
\begin{align} \label{eq:pendulum-param}
u(t) = u(t;\bm{y}) = \sum_{k=1}^\infty y_k \alpha_k \fb_k(t), \quad (\bm{y} = (y_j)_{j\in \N}),
\end{align}
can be expanded by the Fourier basis functions $\fb_k \in L^2([0,T])$, such that 
\begin{align}
\alpha_k\Vert \fb_k\Vert_{L^\infty}
\le 
C_\alpha e^{-|k|\ell} \le 1,
\quad \forall \,  k \in \Z.
\end{align}
Letting $\kappa: \N \to \Z$ denote the enumeration of appendix \ref{app:Fourier}, with $j \mapsto |\kappa(j)|$ increasing, we define a monotonically decreasing sequence $(b_j)_{j\in \N}$ by
\begin{align} \label{eq:pendulum-bk}
b_j := C_\alpha e^{-|\kappa(j)| \ell}.
\end{align}
We will show that that the mapping
\[
\cF: 
[-1,1]^\N \to L^2([0,T]),
\quad
\bm{y} \mapsto v(\bm{y}),
\]
where $v({\bm{y}})$ solves the gravity pendulum equation \eqref{eq:pendulum} with forcing $U(t; \bm{y}) = (0,u(t;\bm{y}))$ can be extended to a holomorphic mapping on suitable (admissible) poly-ellipses $E_{\bm{\rho}} = \prod_{j=1}^\infty E_{\rho_j} \subset \C^\N$, 
\[
\cF:
\prod_{j=1}^\infty E_{\rho_j} \to L^2_{\C}([0,T]),
\quad
\bm{z} \mapsto v(\bm{z}),
\]
where $\bm{\rho} = (\rho_j)_{j\in \N}$, $\rho_j > 1$ for all $j\in \N$. For $\rho > 1$, $E_{\rho}\subset \C$ denotes the (interior of the) Bernstein ellipse (cp. \eqref{eq:bernstein-ellipse}). It is straightforward to see that,
\begin{align}
E_{\rho}
\subset 
\set{z \in \C
}{
|\Re(z)| < \rho, \, |\Im(z)| < \rho-1
}
.
\end{align}

To prove the existence of an analytic continuation to suitable poly-ellipses $E_{\bm{\rho}}$, our first goal is to show that there exists a $\delta > 0$, such that the ODE-system \eqref{eq:pendulum} has a complex-valued solution $v: [0,T]\to \C^2$ for any forcing $U: [0,T]\to \C^2$ with $|\Im(U(t))| \le \delta$. To this end, we first note that the complex-valued ODE \eqref{eq:pendulum} is equivalent to the following system for $v(t) = v_r(t) + i v_i(t)$, with $v_r = (v_{r,1},v_{r,2}),v_i = (v_{i,1}, v_{i,2}): [0,T]\to \R^2$:
\begin{gather} \label{eq:pendulum-cplx}
\left\{
\begin{aligned}
\frac{d v_r}{dt} &= g_r(v_r,v_i) + \Re(U), \quad v_r(0) = 0,\\
\frac{d v_i}{dt} &= g_i(v_r,v_i) + \Im(U), \quad v_i(0) = 0,
\end{aligned}
\right.
\end{gather}
where $U(t) = (0,u(t))$, $u:[0,T]\to \C$ and 
\[
g_r(v_r,v_i) = 
\begin{pmatrix}
v_{r,2} \\
-\gamma \sin(v_{r,1})\cosh(v_{i,1})
\end{pmatrix}
,
\quad
g_i(v_r,v_i) = 
\begin{pmatrix}
v_{i,2} \\
-\gamma \cos(v_{r,1})\sinh(v_{i,1})
\end{pmatrix}.
\]

We note that the second equation of \eqref{eq:pendulum-cplx} implies that
\begin{align*}
|v_i(t)| 
&\le 
\int_0^t |g_i(v_r(s),v_i(s))| \, ds
+
\int_0^t |\Im(U(s))| \, ds
\\
&\le
\int_0^t 
\left(
|v_i(s)| + \gamma|\sinh(v_{i,1}(s))|
\right)
\, ds
+
\int_0^t |\Im(u(s))| \, ds.
\end{align*}
Furthermore, we have
\[
|\sinh(x)|
=
\left|
\int_0^x \cosh(\xi) \, d\xi
\right|
\le
|x| \sup_{|\xi|\le |x|} \cosh(\xi)
\le
|x| e^{|x|},
\]
for any $x \in \R$. We thus conclude that 
\begin{align} \label{eq:pendulum-cplx-estimate}
|v_i(t)|
\le
\int_0^t  |v_i(s)| \left(1+\gamma e^{|v_i(s)|}\right) \, ds
+
\int_0^t |\Im(u(s))| \, ds,
\end{align}
for all $t\in [0,T]$.

The above estimate paves the way for the following lemma, proved in Appendix \ref{app:pf44} on the boundedness of the imaginary part of the solution of the complex extension of the forced gravity pendulum \eqref{eq:pendulum-cplx}:
\begin{lemma} \label{lem:pendulum-impart}
Let $u(t)\in L^\infty([0,T])$. Assume that the solution of \eqref{eq:pendulum-cplx} exists on $[0,T_0]$ for some $0<T_0 \le T$. There exists a constant $\delta > 0$, such that if $\sup_{s \in [0,T]} |\Im(u(s))| \le \delta$, then $\sup_{s \in [0,T_0]} |v_i(s)| \le 1$.
\end{lemma}
This lemma allows us to prove the following lemma (see the detailed proof in Appendix \ref{app:pf45}), which establishes global solutions for the complex extension of the forced pendulum \eqref{eq:pendulum-cplx},
\begin{lemma} \label{lem:pendulum-cplx-existence}
If $\delta$ is chosen as in Lemma \ref{lem:pendulum-impart}, and if $\Vert \Im(u)\Vert_{L^\infty([0,T])} \le \delta$, then the maximal existence interval of the solution of the ODE system \eqref{eq:pendulum-cplx} contains $[0,T]$, and $\sup_{t\in [0,T]} |v_i(t)|\le 1$.
\end{lemma}
Now assume that $u(t) = u(t;\bm{z})$ is parametrized as in \eqref{eq:pendulum-param}. If $\bm{z} \in E_{\bm{\rho}}$ belongs to a poly-ellipse in $\C^\N$, then clearly
\begin{align} \label{eq:pendulum-adm}
\Vert \Im(u) \Vert_{L^\infty}
\le
\sum_{k=1}^\infty |\Im(z_k)| \underbrace{\alpha_k \Vert \fb_k\Vert_{L^\infty}}_{\le b_{j} \, (k=\kappa(j))}
\le
\sum_{j=1}^\infty (\rho_{\kappa(j)} - 1) b_j.
\end{align}
We recall (cp. Definition \ref{def:be-admissibility}), that a sequence $\bm{\rho} = (\rho_j)_{j\in \N}$, with $\rho_j > 1$ for all $j\in \N$ is called {$(\bm{b},\delta,\kappa)$-admissible}, if 
\begin{align} \label{eq:pendulum-admissible}
\sum_{j=1}^\infty b_j (\rho_{\kappa(j)} - 1) < \delta,
\end{align}
where $b_j$ is defined by \eqref{eq:pendulum-bk}, and where $\delta > 0$ is chosen as in Lemma \ref{lem:pendulum-cplx-existence}. It is now clear from Lemma \ref{lem:pendulum-cplx-existence}, that for any $\bm{z}\in E_{\bm{\rho}}$ with $(\bm{b},\delta,\kappa)$-\emph{admissible} $\bm{\rho}$, the solution $v(\bm{z})\in C([0,T];\C^2)$ of \eqref{eq:pendulum-cplx} is well-defined. Hence, we can define a mapping 
\begin{align} \label{eq:pendulum-holomorphic}
\cF: E_{\bm{\rho}} = \prod_{j=1}^\infty E_{\rho_j} \to L^2_{\C}([0,T]), 
\quad
\cF(\bm{z}) := v(\bm{z}),
\end{align}
where $v(\bm{z})$ is the solution of \eqref{eq:pendulum-cplx} with forcing $U(t) = (u(t),0)$, where $u(t) = u(t;\bm{z})$ is given by \eqref{eq:pendulum-param}. This allows us to prove in Appendix \ref{app:pf46} the following lemma on the holomorphy of the map $\cF$ \eqref{eq:pendulum-holomorphic},

\begin{lemma} \label{lem:pendulum-differentiable}
The mapping $\cF$ defined by \eqref{eq:pendulum-holomorphic} is $(\bm{b},\delta,\kappa)$-holomorphic, according to Definition \ref{def:be-holomorphy}; i.e., for each index $j\in \N$, the componentwise mapping 
\[
E_{\rho_j} \mapsto L^2_{\C}([0,T]), 
\quad 
z_j \mapsto \cF(\bm{z}),
\]
where the $z_k \in E_{\rho_k}$ for $k\ne j$ are held fixed, is (complex-)differentiable. Moreover, there exists a constant $C>0$, such that for any admissible $\bm{\rho}$, we have
\[
\sup_{\bm{z} \in E_{\bm{\rho}}} \Vert \cF(\bm{z}) \Vert_{L^2_\C([0,T])} \le C.
\]
\end{lemma}
As a consequence of Lemma \ref{lem:pendulum-differentiable}, we can now state the following approximation result, which follows from the general approximation result for $(\bm{b},\e,\kappa)$-holomorphic mappings, Theorem \ref{thm:holomorphic-approx}:
\begin{proposition} \label{prop:pendulum-approx}
Let $(\cE,\cD)$ denote the encoder/decoder pair \eqref{eq:holomorphic-encoder}, \eqref{eq:holomorphic-decoder} with $m$ sensors, let $(\cR,\cP)$ denote the reconstruction/projection pair, constructed in Proposition \ref{prop:pendulum-rec}, for a given $p\in \N$. For any $k \in \N$, there exists a constant $C>0$, depending on $k$, the final time $T$ and the probability measure $\mu$, but independent of $m$ and $p$, such that there exists a neural approximator network $\cA: \R^m \to \R^p$ with 
\begin{gather*}
\size(\cA) \le C(1+pm\log(m) \log\log(m)),
\\
\depth(\cA) \le C(1+\log(m) \log\log(m)),
\end{gather*}
and such that the approximation error
\[
\Err_{\cA}
=
\left(
\int_{\R^m}
\Vert \cA(\bm{u}) - \cP \circ \G \circ \cD(\bm{u}) \Vert^2_{\ell^2} \, d(\cE_\#\mu)(\bm{u})
\right)^{1/2},
\]
can be estimated by 
\[
\Err_{\cA} 
\le 
C m^{-k}.
\]
\end{proposition}
The proof follows from a direct application of Theorem \ref{thm:holomorphic-approx}, with the observation that $\Vert \cP \Vert = \Lip(\cP) \le 2$ in Proposition \ref{prop:pendulum-rec} is bounded independently of $m$ and $p$.
\subsubsection{Bounds on the DeepONet approximation error \eqref{eq:approxerr}}
We combine propositions \ref{prop:pendulum-enc}, \ref{prop:pendulum-rec}, \ref{prop:pendulum-approx} to state the following theorem on the DeepONet error \eqref{eq:approxerr} for the forced gravity pendulum;
\begin{theorem} \label{thm:pendulum-error}
Consider the DeepONet approximation problem for the gravity pendulum \eqref{eq:pendulum0}, where the forcing $u(t)$ is distributed according to a probability measure $\mu \in \P(L^2([0,T]))$ given as the law of the random field \eqref{eq:pendulum-law}. For any $k,r\in \N$, there exists a constant $C = C(k,r)>0$, and a constant $c>0$, independent of $m$, $p$, such that for any $m, \, p \in \N$, there exists a DeepONet \eqref{eq:donet} with trunk net $\bm{\tr}$ and branch net $\bm{\br}$, such that 
\begin{gather*}
\size(\bm{\tr}) \le C(1+p^3),
\quad
\depth(\bm{\tr}) \le C(1+p\log(p)),
\end{gather*}
and
\begin{gather*}
\size(\bm{\br}) \le C(1+pm \log(m) \log\log( m)), 
\\
\depth(\bm{\br}) \le C(1+ \log(m) \log\log(m)),
\end{gather*}
and such that the DeepONet approximation error \eqref{eq:approxerr} is bounded by 
\begin{equation}
\label{eq:pendulum-err}
\Err 
\le 
C e^{-c\ell m}
+
C m^{-k} 
+
C p^{-r}.
\end{equation}
\end{theorem}
\begin{remark}
\label{rem:pendulum-comp}
Theorem \ref{thm:pendulum-error} guarantees that for $\epsilon > 0$, a DeepONet approximation error of $\Err \sim \epsilon$ can be achieved provided that $m \gtrsim \max(\ell^{-1} \log(\epsilon^{-1}), \epsilon^{-1/k})$ and $p \gtrsim \epsilon^{-1/r}$. As long as the intuitively obvious restriction that $m\ell \gg 1$ is satisfied, i.e., that the sensors can resolve the typical length scale $\ell$, these requirements can be achieved provided that $m \gtrsim \epsilon^{-1/k}$ and $p \gtrsim \epsilon^{-1/r}$.
In this case, an error $\Err \lesssim \epsilon$ can be achieved by a DeepONet with size of the order of
\begin{equation}
\label{eq:pendulum-comp}
\size(\bm{\tr},\bm{\br}) \sim \epsilon^{-1/r}\left(\epsilon^{-2/r} + \epsilon^{-1/k} \log(\epsilon^{-1})\log\log(\epsilon^{-1})\right).
\end{equation}
As the $k,r\in \N$ were arbitrary, this shows that the required size only grows \emph{sub-algebraically} with $\epsilon^{-1}\to \infty$, i.e. the required size scales asymptotically $\ll\epsilon^{-1/s}$, as $\epsilon\to 0$, for any $s>0$. Given Definition \ref{def:cod}, this clearly implies that the DeepONet approximation for this problem does not suffer from the curse of dimensionality.

We however point out that the implied constants in these asymptotic estimates depend on $k,r$. In particular, these constants might deteriorate as $k,r\to \infty$.  On the other hand, we can fix $k,r$ to be large and conclude from the complexity estimate \eqref{eq:pendulum-comp} that the complexity of the DeepONet in this case can only grow algebraically with the error tolerance $\epsilon$. Given Definition \ref{def:cod} (cp. equation \eqref{eq:bcod}), this suffices to prove that a DeepONet can approximate the operator $\G$ in the forced gravity pendulum problem with the measure $\mu$ \eqref{eq:pendulum-law} by \emph{breaking the curse of dimensionality}.
\end{remark}
\begin{remark}
In \cite{deeponets}, the authors used a Gaussian measure, with a covariance kernel \eqref{eq:quadraticexp-kernel}, as the underlying measure $\mu$ for the forced gravity pendulum. An estimate, analogous to \eqref{eq:pendulum-err} can be proved in this case, but with a super-exponential decay of the encoding error with respect to the number of sensors $m$, given the bound \eqref{eq:enerrg1}. However, the overall complexity still has sub-algebraic asymptotic growth in $\e^{-1}$, as the other error terms scale exactly as in \eqref{eq:pendulum-comp}.

A different underlying measure $\mu$ results from choosing an algebraic decay of the coefficients $\alpha_k$ in \eqref{eq:pendulum-law}. Clearly, the bound \eqref{eq:pendulum-err} will also hold in this case but with an algebraic decay in the encoding error. Nevertheless, the total complexity of the problem still scales algebraically (polynomially) with the error tolerance $\epsilon$ and shows that the resulting DeepONet breaks the curse of dimensionality also in this case.

\end{remark}

\begin{remark}
The estimates on the sizes of the trunk and branch nets rely on expressivity results for ReLU deep neural networks in approximating holomorphic functions \cite{OSZ2019}. These in turn, depend on the results of \cite{Yarotsky2017} for approximation of functions using ReLU networks. These approximation results may not necessarily be optimal and might suggest networks of larger size as well as depth, than what is needed in practice \cite{deeponets}. Moreover, the constants in complexity estimate \eqref{eq:pendulum-comp} depend on the final time $T$ and can grow exponentially in $T$ (with a familiar argument based on the Gr\"onwall's inquality). 
\end{remark}

\subsection{An elliptic PDE: Multi-d diffusion with variable coefficients.}
\label{sec:elliptic}
\subsubsection{Problem formulation.} 
We will consider the following very popular model problem for elliptic PDEs with unknown diffusion coefficients \cite{CDS2011} and references therein. For the sake of definiteness and simplicity, we shall assume a periodic domain $D = \T^d$ in the following. We consider an elliptic PDE with variable coefficients $a$:
\begin{align} \label{eq:random-elliptic}
-\nabla \cdot (a(x) \nabla u(x)) = f(x),
\end{align}
for $u\in H^1(D)$ with suitable boundary conditions, and for fixed $f \in H^{-1}(D)$. 

We also fix a probability measure $\mu$ on the coefficients $a$ on $L^2(D)$, such that $\supp(\mu) \subset L^\infty(D)$. To ensure coercivity of the problem \eqref{eq:random-elliptic}, we will assume that
\begin{align} \label{eq:cond-on-mu}
\mu(\set{a \in L^2(D)}{0 < \lambda(a) \le \Lambda(a) < \infty}) = 1,
\end{align}
where 
\begin{align}
\lambda(a) &:= \essinf_{x\in D} a(x),
\\
\Lambda(a) &:= \esssup_{x \in D} a(x)
\end{align}
denote the essential infimum and supremum of $a$, respectively. To ensure uniqueness of solutions to \eqref{eq:random-elliptic}, we require that $u \in H^1_0(\T^d)$ have zero mean, i.e., that $\int_{\T^d} u(x) \, dx = 0$.  We note that the condition \eqref{eq:cond-on-mu} on $\mu$ is for example satisfied by the log-Gaussian measures commonly employed in hydrology \cite[]{charrier}.

We note that the variable coefficient $a$ can model the rock permeability in a Darcy flow with $u$ modeling the pressure and $f$ a source/injection term. Similarly, $a$ may model variable conductivity in a medium with $u$ modeling the temperature. In many applications, one is interested in inferring the solution $u$ for a given coefficient $a$ as the input. Thus, the \emph{nonlinear} operator $\cG$ maps the input coefficient $a$ into the solution field $u$ of the PDE \eqref{eq:random-elliptic}. The well-definedness of this operator is given in the following Lemma (proved in appendix \ref{app:pf47}),
\begin{lemma}
\label{lem:elliptic1}
Assume that the coefficients $a,a' \in C(\T^d)$ satisfy the uniform coercivity assumption
\[
0< \lambda \le \inf_{x\in \T^d} a(x), \; \inf_{x\in \T^d} a'(x).
\]
Let $u,u'\in H^1_0(\T^d)$ denote the solution to \eqref{eq:random-elliptic} with coefficients $a,a'$, respectively, and with the same right-hand side $f\in L^2(\T^d)$. There exists a constant $C = C(\lambda, \T^d,\Vert f \Vert_{L^2(\T^d)})$, such that
\[
\Vert u -u' \Vert_{L^2(\T^d)}
\le
C\Vert a -a'\Vert_{L^\infty(\T^d)}.
\]
In particular the operator $\G: X \to L^2(\T^d)$, $a \mapsto u$, is Lipschitz continuous for any set $X \subset C(\T^d)$, and satisfying the coercivity assumption.
\end{lemma}
By Sobolev embedding, Lemma \ref{lem:elliptic1} ensures the well-definedness, and Lipschitz continuity, of the operator $\G: H^s(\T^d) \to L^2(\T^d)$ for any $s > d/2$. We complete the data for a DeepONet approximation problem (Definition \ref{def:data}) by specifying an underlying measure $\mu$. Following standard practice \cite{CDS2011} and references therein, and aided by the fact that we enforce periodic boundary conditions, the underlying measure $\mu$ is the law of a random field $a$, that is expanded in terms of the Fourier basis. More precisely, we assume that we can write $a$ in the following form
\begin{align} \label{eq:law-elliptic}
a(x,Y) 
=
\overline{a}(x) 
+
\sum_{{k}\in \Z^d}
\alpha_{k} Y_k \fb_k(x),
\end{align}
with notation from Appendix \ref{app:Fourier}, and where for simplicity $\bar{a}(x) \equiv 1$ is assumed to be constant. Furthermore, we will consider the case of smooth coefficients $x \mapsto a(x;Y)$, which is ensured by requiring that there exist constants $C_\alpha>0$ and $\ell > 1$, such that  
\begin{align} \label{eq:law-decay1}
|\alpha_k|
\le
C_\alpha \exp(-\ell |k|_\infty),
\quad
\forall \; k \in \Z^d.
\end{align}
Let us define $\bm{b} = (b_1,b_2,\dots)\in \ell^1(\N)$ by 
\begin{align} \label{eq:elliptic-b}
b_j := C_\alpha \exp(-\ell |\kappa(j)|_\infty),
\end{align}
where $\kappa: \N \to \Z^d$ is the enumeration for the standard Fourier basis, defined in appendix \ref{app:Fourier}. Note that by assumption on the enumeration $\kappa$, we have that $b_1\ge b_2 \ge \dots$ is a monotonically decreasing sequence. In the following, we will assume throughout that $\Vert \bm{b} \Vert_{\ell^1} < 1$, ensuring the uniform coercivity condition $\lambda(a) \ge \lambda$ for some $\lambda>0$ and all random coefficients $a = a(\slot; Y)$ in \eqref{eq:random-elliptic}. In \eqref{eq:law-decay1}, the parameter $\ell>0$ can be interpreted as the correlation length scale of the random coefficients. We furthermore assume that the $Y_j \in [-1,1]$ are centered random variables, implying that $\E[a] = \overline{a}$. We let $\mu \in \P(L^2(\T^d))$ denote the law of the random coefficient \eqref{eq:law-elliptic}. By the assumed decay \eqref{eq:law-decay1}, we have $\supp(\mu) \subset C^\infty(\T^d)$. 

Given the above setup, our aim is to find a DeepONet \eqref{eq:donet} which approximates the underlying $\cG$, corresponding to the elliptic PDE \eqref{eq:random-elliptic} efficiently. To this end, we will follow the program outlined at the beginning of this section, and bound the encoding, reconstruction and approximation errors, separately in the following sections. 
\subsubsection{Bounds on the encoding error \eqref{eq:encoding}}
\label{sec:elliptic-encoding}
We are almost in the setup that was already considered in section \ref{sec:parametrized-mu}, with the only difference that we now consider $X = H^s(\T^d)$ for a fixed $s>d/2$, instead of $X=L^2(\T^d)$. We can again consider the Fourier based encoder/decoder pair $(\cE,\cD)$ given by \eqref{eq:holomorphic-encoder} and \eqref{eq:holomorphic-decoder}, respectively. Applying a straightforward extension of Theorem \ref{thm:holomorphic-encoding} to $H^s(\T^d)$, we observe that due to the exponential decay of the (Fourier) coefficients $\alpha_k$, the exponential decay of the pseudo-spectral projection continues to hold also in the $H^s(\T^d)$-norm, yielding the following error estimate for the encoding error $\Err_{\cE}$: 

\begin{proposition} \label{prop:elliptic-encoding}
Given $m \in \N$, let $(\cE, \cD)$ denote the Fourier based encoder/decoder pair $(\cE,\cD)$ given by \eqref{eq:holomorphic-encoder} and \eqref{eq:holomorphic-decoder}, respectively.
There exists a constant $C>0$, depending on $C_\alpha, \,\ell>0$ and on $s>d/2$,  but independent of $m$, and a universal constant $c>0$, independent of $m$ such that 
\begin{align} \label{eq:elliptic-encoding}
\Err_{\cE}
=
\left(
\int_{X}
\Vert 
\cD\circ \cE(u)
- u
\Vert_{H^s(\T^d)}^2
\, d\mu(u)
\right)^{1/2}
\le
C \exp(-c\ell m^{1/d}).
\end{align}
\end{proposition}
\subsubsection{Bounds on the reconstruction error \eqref{eq:decoding}}
We follow the program outlined at the beginning of this section and bound the reconstruction error via smoothness of the image of the operator $\G$ for the elliptic PDE \eqref{eq:random-elliptic}. To this end, we have the following Lemma (proved in Appendix \ref{app:pf48}),
\begin{lemma}
\label{lem:elliptic2}
Let $k\in \N$. Let $u$ be a solution of \eqref{eq:random-elliptic}, with coefficient $a\in C^\infty(\T^d)$, right-hand side $f \in H^k(\T^d)$ and $\lambda = \min_{x\in \T^d} a(x) > 0$. Then for any $k\in \N$, there exists a constant $C>0$, depending only on $k$ and $\lambda$, such that
\[
\Vert u \Vert^2_{H^{k+1}}
\le
C \Vert f \Vert_{H^k}^2 \left( 1 + \Vert a \Vert^{2k}_{C^k} \right).
\]
\end{lemma}
Given the above smoothness estimate, one can directly apply Proposition \ref{prop:rec-smoothness} and Theorem \ref{thm:smoothness-NN} to obtain the following bound on the reconstruction error,
\begin{proposition} \label{prop:elliptic-reconstruction}
Let $\mu \in \P(L^2(\T^d))$ be a probability measure with $\supp(\mu) \subset C^k(\T^d) \cap L^2(\T^d)$. Assume that there exists $\lambda_0>0$, such that $\mu(\lambda(a) \ge \lambda_0) = 1$, and that $\int_{L^2} \Vert a \Vert_{C^k(\T^d)}^{2k} \, d\mu(a) < \infty$. Define an operator $\G: C(\T^d) \to L^2(\T^d)$ by $a \mapsto u = \G(a)$, where $u$ is the solution of \eqref{eq:random-elliptic} with a smooth right-hand side $f\in C^k(\T^d)$. Then there exists a constant $C(k,\Vert f \Vert_{C^k},\mu)>0$, depending only on $\Vert f \Vert_{C^k}$, $k$, and $\mu$, such that for any $p\in \N$, there exists a trunk net $\bm{\tr}: U \subset \R^n \to \R^p$, $y \mapsto \bm{\tr}(y) = (0,\tr_1(y),\dots, \tr_p(y))$ with
\[
\size(\bm{\tr}) \le Cp(1+\log(p)^2), 
\quad
\depth(\bm{\tr}) \le C(1+\log(p)^2),
\]
such that the corresponding reconstruction 
\[
\cR:\R^p \to L^2(\T^d) \simeq L^2([0,2\pi]^d),
\quad
\cR(\alpha) := \sum_{j=1}^p \alpha_j \tr_j,
\]
 satisfies the following reconstruction error bound:
\begin{align} \label{eq:elliptic-rec}
\Err_{\cR} 
\le C(k,\Vert f \Vert_{C^k},\mu) \, p^{-k/d}.
\end{align}
Furthermore, the reconstruction $\cR$ and the associated projection $\cP: L^2(\T^d) \to \R^p$ given by \eqref{eq:opt-proj} satisfy $\Lip(\cR), \, \Lip(\cP) \le 2$.
\end{proposition}
\subsubsection{Bounds on the approximation error \eqref{eq:approximation}}
For choices of the encoder $\cE$ in Proposition \ref{prop:elliptic-encoding}, and the reconstructor $\cR$ in Proposition \ref{prop:elliptic-reconstruction}, the approximation error with approximator network $\cA$ is given by 
\[
\int_{\R^m} \Vert \cA(a_{i}) - \cP \circ \G \circ \cD(a_{i}) \Vert^2 \, d(\cE_\#\mu)(a_{i}),
\]
where the decoder $\cD$ is given by \eqref{eq:holomorphic-decoder}, and the projector $\cP$ has a Lipschitz constant bounded by $\Lip(\cP)\le 2$. 

From \cite[Theorem 1.3]{CDS2011} (see also \cite[Example 2.2]{SchwabZech2019} for a more detailed discussion relevant to the present setting), it follows that the mapping 
\[
\cF: [-1,1]^\N \to L^2(\T^d),
\quad
Y \mapsto \G(a(\slot;Y)),
\]
is $(\bm{b}, \epsilon, \kappa)$-holomorphic according to our Definition \ref{def:be-holomorphy}, with $\bm{b}$ defined by \eqref{eq:elliptic-b}, provided that $\epsilon < 1 - \Vert \bm{b}\Vert_{\ell^1}$, and where $\kappa: \N \to \Z^d$ denotes the enumeration of the standard Fourier basis (cp. appendix \ref{app:Fourier}). Following the discussion in section \ref{sec:be-holomorphy}, such $\cF$ can be efficiently approximated by neural networks. In particular, we can directly apply Theorem \ref{eq:holomorphic-approx},  together with the observation that $\Vert \cP \Vert = \Lip(\cP) \le 2$ is bounded independently of $m$ and $p$, to conclude the following bound on the approximation error \eqref{eq:approximation},
\begin{proposition} \label{prop:elliptic-approx}
Let the operator $\G$ be defined as mapping the coefficient $a$ to the solution $u$ of the elliptic PDE \eqref{eq:random-elliptic} and the measure $\mu$  be the law of the random field \eqref{eq:law-elliptic}. Let the encoder $\cE: C(\T^d) \to \R^m$ be given by \eqref{eq:holomorphic-encoder} and the decoder $\cD: \R^m \to L^2(\T^d)$ be given by \eqref{eq:holomorphic-decoder}. Let the reconstruction/projection pair $(\cR,\cP)$ be given as in Proposition \ref{prop:elliptic-reconstruction} for given $p\in \N$. Then for any $k\in \N$, there exists a constant $C>0$, depending on $k$, but independent of the trunk net size $p$ and number of sensors $m$, such that for any $m, \, p\in \N$, there exists an approximator network $\cA$, with 
\begin{gather}
\begin{aligned}
\mathrm{size}(\cA) &\le C \left(1 + pm \log(m) \log\log(m)\right)
,
\\
\mathrm{depth}(\cA) &\le C \left( 1 + \log(m) \log\log(m)\right)
\end{aligned}
\end{gather}
and approximation error
\begin{align} \label{eq:elliptic-approx}
\cE_\cA
\le
C m^{-k}.
\end{align}
\end{proposition}
\subsubsection{Bounds on the DeepONet approximation error \eqref{eq:approxerr}}
Finally, combining the results of Propositions \ref{prop:elliptic-encoding}, \ref{prop:elliptic-reconstruction}, \ref{prop:elliptic-approx}, we conclude that 

\begin{theorem}
For any $k,r \in \N$, there exists a constant $C>0$, such that for any $m,p\in \N$, there exists a DeepONet $\cN = \cR \circ \cA \circ \cE$ with $m$ sensors, a trunk net $\bm{\tr} = (0,\tr_1, \dots, \tr_p)$ with $p$ outputs and branch net $\bm{\br} = (0,\br_1,\dots, \br_p)$, such that 
\begin{gather*}
\size(\bm{\br}) \le C(1+pm\log(m)\log\log(m)), \\
\depth(\bm{\br}) \le C(1+\log(m) \log\log(m)), 
\end{gather*}
and
\begin{gather*}
\size(\bm{\tr}) \le C(1 + p\log(p)^2) \\
\depth(\bm{\tr}) \le C(1+\log(p)^2)
\end{gather*}
such that the DeepONet approximation error \eqref{eq:approxerr} satisfies
\begin{equation}
\label{eq:elliptic-err}
\Err
\le
C e^{-c\ell m^{\frac{1}{d}}}
+
C m^{-k}
+
C p^{-r}.
\end{equation}
\end{theorem}
As the bound \eqref{eq:elliptic-err} is very similar to the bound \eqref{eq:pendulum-err}, we can directly apply the discussion in Remark \ref{rem:pendulum-comp} to derive the following complexity estimate 
\begin{equation}
\label{eq:elliptic-comp}
\size(\bm{\tr},\bm{\br}) \sim \epsilon^{-1/r}\left(\epsilon^{-2/r} + \epsilon^{-1/k} \log(\epsilon^{-1})\log\log(\epsilon^{-1})\right).
\end{equation}
To derive \eqref{eq:elliptic-comp}, we choose $m \sim \epsilon^{-1/k}$ sensors, and we assume that we are in the regime, where $e^{-c\ell m^{1/d}} \ll m^{-k} \sim \epsilon$. This clearly requires that $m \gg \ell^{-d} |\log(\epsilon)|^d$, i.e., that the $m$ sensors resolve scales of length $\ell>0$ in the $d$-dimensional domain $D=[0,2\pi]^d$. This is a reasonable assumption in low dimensions $d\in \{1,2,3\}$ of the domain $D$, where this is practically feasible.

As the $k,r\in \N$ were arbitrary, this shows that the required size only grows \emph{sub-algebraically} with $\epsilon^{-1}\to \infty$. Thus from Definition \ref{def:cod}, we can conclude that there exists a DeepONet \eqref{eq:donet}, which breaks the curse of dimensionality in approximating the nonlinear operator $\G$ mapping the coefficient $a$ to the solution $u$ of the elliptic PDE \eqref{eq:random-elliptic}.

\begin{remark}
We emphasize that the above operator provides a mapping $\G: C(\T^d) \to L^2(\T^d)$ in \emph{low spatial dimensions $d\in \{1,2,3\}$}, in which case we show that $\G$ can be efficiently approximated by a DeepONet $\cN$. Thus, the ``curse of dimensionality'' here refers to the approximation problem for the high-dimensional (in fact, $\infty$-dimensional) input-to-output mapping $\G$, in accordance with our Definition \ref{def:cod} (cp. also Remark \ref{rem:cod}). Whether DeepONets also provide an efficient approximation for similar operators with large \emph{spatial dimension $d \gg 1$} will be the subject of future research.
\end{remark}

\subsection{Nonlinear parabolic PDE: A reaction-diffusion equation.}
\label{sec:parabolic}
As a prototype for nonlinear reaction-diffusion parabolic type PDEs, we consider the following version of the well-known Allen-Cahn equation which arises in the study of phase transitions in materials such as alloys, 
\begin{gather} \label{eq:allen-cahn}
\left\{
\begin{gathered}
\frac{\partial v}{\partial t} = \Delta v + f(v),
\\
v(t=0) = u,
\end{gathered}
\right.
\end{gather}
where the non-linearity is given by $f(v) = v - v^3$. For the sake of simplicity, the Allen-Cahn PDE \eqref{eq:allen-cahn} is supplemented with periodic boundary conditions on the space-time domain $[0,T]\times \T^d$. For initial data $u = u(x)$ drawn from a probability measure $\mu \in \P(L^2(\T^d))$, our aim in this section  will be to consider the DeepONet approximation of the data-to-solution mapping $\G: L^2(\T^d) \to L^2(\T^d)$, $u \mapsto \G(u) := v(t=T)$, where $v$ solves \eqref{eq:allen-cahn}. 

As a first step, we need to show that the operator $\G$ is well-defined. To this end, we recall some well-known existence and boundedness results for the Allen-Cahn equation \eqref{eq:allen-cahn},
\begin{theorem}[see e.g. { \cite[Cor. 1]{Yang2018}}] \label{thm:ac-boundedness}
Let $v(x,t)$ solve \eqref{eq:allen-cahn}, with initial data $u$. If the initial data satisfies $\Vert u \Vert_{L^\infty_x} \le 1$, then the solution of \eqref{eq:allen-cahn} satisfies $\Vert v(t) \Vert_{L^\infty_x} \le 1$ for all $t\in [0,T]$.
\end{theorem}
Note that the Allen-Cahn equation \eqref{eq:allen-cahn} is the $L^2$-gradient flow of a Ginzburg-Landau energy functional with a double-well potential. Normalizing the wells, it makes sense to consider the initial data such that $-1 \leq u(x) \leq 1$ and the above theorem guarantees that the maximum principle holds. Using standard parabolic regularity theory, in appendix \ref{app:pf49}, we prove the following regularity result for the Allen-Cahn equation,
\begin{theorem} \label{thm:ac-regularity}
There exists an increasing function $\eta: [0,\infty) \to [0,\infty)$, $s \mapsto \eta(s)$, with the following property: If $u\in C^{4,\alpha}(\T^d)$, $\alpha \in (0,1)$, is initial data for the Allen-Cahn equation \eqref{eq:allen-cahn} with $\alpha$-H\"older continuous 4th derivatives, and such that $\Vert u \Vert_{L^\infty_x} \le 1$, then the solution $v$ of \eqref{eq:allen-cahn} has H\"older continuous partial derivatives 
\[
\frac{\partial^{k+\ell}v}{\partial t^k \partial x_{i_1}\dots \partial x_{i_\ell}},
\quad
\forall\, i_1,\dots, i_\ell \in \{1,\dots, d\}, \; 2k+\ell \le 4,
\]
and $\Vert v \Vert_{C^{(2,4)}([0,T]\times \T^d)}$ defined by 
\begin{align} \label{eq:mixed-derivatives}
\Vert v \Vert_{C^{(2,4)}([0,T]\times \T^d)}
:=
\max_{x\in \T^d} \Vert v(\slot, x) \Vert_{C^2([0,T])}
+
\max_{t\in [0,T]} \Vert v(t,\slot) \Vert_{C^4(\T^d)},
\end{align}
can be bounded from above
\begin{align}\label{eq:allen-cahn-estimate}
\Vert v \Vert_{C^{(2,4)}([0,T]\times \T^d)}
\le
\eta\left(
\Vert u \Vert_{C^{4,\alpha}(\T^d)}
\right).
\end{align}
\end{theorem}
As a corollary of the above theorems, we can show in Appendix \ref{app:pf410}, the following Lipschitz continuity of the solution mapping at time $T$: $u \mapsto v(T)$.

\begin{corollary} \label{cor:ac-continuity}
Let $\alpha \in (0,1)$. Let $u,u'\in C^{4,\alpha}(\T^d)$ be such that $\Vert u\Vert_{L^\infty}, \Vert u' \Vert_{L^\infty} \le 1$. Let $v,v'\in C^{(2,4)}(\T)$ denote the solution of \eqref{eq:allen-cahn} with initial data $u,u'$, respectively. There exists a constant $C = C(T)>0$, such that
\[
\Vert v(T) - v'(T) \Vert_{L^2(\T^d)}
\le
C\Vert u - u' \Vert_{L^2(\T^d)}.
\]
\end{corollary}
It follows from Corollary \ref{cor:ac-continuity} that the mapping $C^{4,\alpha}(\T^d) \cap B^\infty_1 \to L^2(\T^d)$, $u \mapsto v(T)$, where $B^\infty_1 = \set{u \in C^{4,\alpha}(\T^d)}{\Vert u \Vert_{L^\infty}\le 1}$ admits a unique Lipschitz continuous extension 
\[
\G: L^2(\T^d) \cap B^\infty_1 \to L^2(\T^d), \quad u \mapsto \G(u) = v(T),
\]
with $\Lip(\G) \le C = e^{4T}$. In the following, we will discuss the approximation of this mapping $\G$ by DeepONets \eqref{eq:donet}. To this end, we assume that the underlying measure $\mu$ satisfies
\begin{align}
\supp(\mu) 
\subset 
\set{
u \in L^2(\T^d)
}{
\Vert u \Vert_{L^\infty} \le 1
}
\subset L^2_x \cap L^\infty_x.
\end{align}

Motivated by the a priori estimates of Theorems \ref{thm:ac-boundedness} and \ref{thm:ac-regularity}, we shall assume that the initial measure $\mu$ is the law of a random field $u$ of the form 
\begin{align} \label{eq:mu-allen-cahn}
u(x) 
=
\sum_{k\in \Z^d} \alpha_k Y_k \fb_k(x),
\end{align}
in terms of the Fourier basis $\fb_k$ (cf. Appendix \ref{app:Fourier}) and with $Y_k \in [-1,1]$ are iid random variables. To ensure sufficient smoothness of the solutions, we further assume that 
\begin{align} \label{eq:ac-coeff-decay}
|\alpha_k| \le C\exp\left(-\ell|k|_\infty\right), \quad \forall \, k\in \Z^d,
\end{align}
for some $C>0$, and length scale $\ell > 0$; Note that \eqref{eq:ac-coeff-decay} in particular implies that $u \in C^{4,\alpha}(\T^d)$ for some $\alpha > 0$. Furthermore, we shall assume that 
\begin{align} \label{eq:ac-boundedness}
\sum_{k\in \Z^d} |\alpha_k| \le 1,
\end{align}
so that $|u(x)| \le 1$ for all $x \in \T^d$.
\subsubsection{Bounds on the encoding error \eqref{eq:encoding}}
With measure $\mu$ defined by \eqref{eq:mu-allen-cahn}, we can readily apply Theorem \ref{thm:holomorphic-encoding} to obtain the following bounds on the encoding error:
\begin{proposition} \label{prop:ac-encoding-err}
Let $d\in \{2,3\}$. Let $\mu \in \P(L^2(\T^d))$ denote the law of the random field \eqref{eq:mu-allen-cahn} with coefficients $\alpha_k$ satisfying the decay and boundedness assumptions \eqref{eq:ac-coeff-decay}, \eqref{eq:ac-boundedness}. For $N \in \N$, let $x_i$, $i=1,\dots, m=(2N+1)^d$ be an enumeration of the grid points of a regular cartesian grid on $\T^d$. Define the encoder $\cE: C(\T^d) \to \R^m$ by 
\[
\cE(u) = (u(x_1), \dots, u(x_m)),
\]
and define the corresponding decoder $\cD: \R^m \to L^2(\T^d)$ by Fourier interpolation onto Fourier modes $|k|_\infty \le N$. Then the encoding error for the encoder/decoder pair $(\cE,\cD)$ can be bounded by \[
\Err_{\cE} 
\le
C \exp(-c\ell m^{1/d}),
\]
for some constants $C,c>0$, independent of $N$.
\end{proposition}
\subsubsection{Bounds on the reconstruction error \eqref{eq:decoding}}
To bound the reconstruction error, we recall that for $u\in C^{4,\alpha}(\T^d)$, we have $\G(u) = v(T) \in C^{4}(\T^d)$, by Theorem \ref{thm:ac-regularity}, with 
\[
\Vert \G(u) \Vert_{C^{4}(\T^d)}
\le
\eta\left(
\Vert u \Vert_{C^{4,\alpha}(\T^d)}
\right)
\le
\eta \left(
\sup_{u\in \supp(\mu)}\Vert u \Vert_{C^{4,\alpha}(\T^d)}
\right)
< \infty,
\]
uniformly bounded from above. It follows that $\Vert \G(u) \Vert_{H^4} \le M$ is uniformly bounded, and by \eqref{eq:rec-smoothness}, the reconstruction error for the Fourier reconstructor/projector pair onto the standard Fourier basis $\fb_1,\dots, \fb_p$, satisfies
\[
\Err_{\cR_{\Fourier}} \le C p^{-4/d}.
\]
From Lemma \ref{lem:Fourier-NN}, we obtain:
\begin{proposition} \label{prop:ac-rec-err}
There exists a constant $C>0$, independent of $p$, such that for any $p\in \N$, there exists a trunk net $\bm{\tr}: U = \T^d \subset \R^d \to \R^p$, with 
\[
\begin{gathered}
\size(\bm{\tr}) \le C p(1 +\log(p)^2), 
\\
\depth(\bm{\tr}) \le C (1 +\log(p)^2),
\end{gathered}
\]
and such that the reconstruction error for $\cR: \R^p \to L^2(\T^d)$, $\cR(\alpha_1, \dots, \alpha_p) = \sum_{k=1}^p \alpha_k \tr_k$, can be bounded by
\begin{align}
\Err_{\cR}
\le
Cp^{-4/d}.
\end{align}
Furthermore, we have $\Lip(\cR),\, \Lip(\cP) \le 2$, where $\cP: L^2(\T^n) \to \R^p$ denotes the optimal projection \eqref{eq:opt-proj} associated with $\cR$.
\end{proposition}
\subsubsection{Bounds on the approximation error \eqref{eq:approximation}}
For the last two concrete examples, bounds on the approximation error \eqref{eq:approximation} leveraged the fact that the underlying operator $\G$ was \emph{holomorphic} in an appropriate sense. However, it is unclear if the operator $\G$ for the Allen-Cahn equation \eqref{eq:allen-cahn} is holomorphic. In fact, we have only provided that it is Lipschitz continuous.  Hence, the approximator network $\cA \approx \cP \circ \G \circ \cD: \R^m \to \R^p$ approximates a $m$-dimensional Lipschitz function. General neural network approximation results for such mappings \cite{Yarotsky2017} imply that for a fixed $M > 0$, there exists a neural network $\cA$ with $\size(\cA) = \mathcal{O}(p\epsilon^{-1/m})$, such that $\Vert \cA - \cP \circ \G \circ \cD \Vert_{L^\infty([-M,M]^m)} \le \epsilon$.

Recall that the total error \eqref{eq:approxerr} for the DeepONet is given by \eqref{eq:donetbd} with $\alpha =1$ as $\G$ is Lipschitz. In this case from Proposition \ref{prop:ac-encoding-err}, we have that 
\[
\Err_{\cE} \lesssim \exp(-cm),
\]
we require at least $m \sim \log(\epsilon^{-1})$ sensors to achieve an encoding error $\Err_{\cE} \le \epsilon$. Therefore, the general results of \cite{Yarotsky2017} would suggest that the required size of the approximator network $\cA$ scales at best like $\epsilon^{-|\log(\epsilon)|}$, where we note that the exponent $|\log(\epsilon)|\to \infty$ as $\epsilon \to 0$. Hence, by our Definition \ref{def:cod}, such a DeepONet would incur the curse of dimensionality.

Is it possible for us to break this curse of dimensionality for the Allen-Cahn equation? It turns out that the approach of a recent paper \cite{DRM1} might suggest a way around the obstacle discussed above. Following \cite{DRM1}, we will leverage the fact that neural networks can \emph{emulate} conventional numerical methods for approximating a PDE. To this end, we will consider the following finite difference scheme:
\subsubsection{A convergence finite difference scheme for the Allen-Cahn equation.}
In \cite{TangYang2016}, it has been shown that an implicit-explicit finite difference scheme of the following form
\[
\frac{U^{n+1} - U^n}{\Delta t}
=
D_{\Delta x} U^n 
+
f(U^n),
\]
converges to the exact solution as $\Delta t, \Delta h \to 0$. Here $U^n = (U^n_1,\dots, U^n_m)$, $m \sim (\Delta x)^{-d}$ are approximate values of the solution at time $t = t_n$, i.e. $U^n_i \approx u(x_i,t_n)$, with $x_i$, $i=1,\dots, m = (2N+1)^d$ an enumeration of a cartesian grid with grid size $\Delta x$ on $\T^d$. The values of $U^0$ at $t = 0$ are initialized as 
\[
U^0_i := u(x_i).
\]
The evaluation of the nonlinearity is carried out pointwise, $f(U^n) = (f(U^n_1), \dots, f(U^n_m))$. $D_{\Delta x}$ denotes the discrete matrix of the Laplace operator, whose one-dimensional analogue in the presence of periodic boundary conditions is given by 
\[
\Lambda_{\Delta x}
=
\frac{1}{\Delta x^2}
\begin{pmatrix}
-2 &  1     &        &         & 1  \\
1  & -2     & 1      &         &  \\
   & \ddots & \ddots & \ddots  &  \\
   &        & 1      & -2      & 1 \\
1   &        &       & 1      & -2
\end{pmatrix}_{N \times N}
\]
For $d=2$ dimensions, we can write 
\[
D_{\Delta x} = \Lambda_{\Delta x} \otimes I + I \otimes \Lambda_{\Delta x},
\]
where $I$ is the $m\times m$ unit-matrix $\otimes$ denotes the Kronecker product. For $d=3$, we have
\[
D_{\Delta x} = \Lambda_{\Delta x} \otimes I \otimes I + I \otimes \Lambda_{\Delta x} \otimes I + I \otimes I \otimes \Lambda_{\Delta x}.
\]

For our purposes, we simply note that the update rule of the numerical scheme of \cite{TangYang2016} can be written in the form 
\begin{align} \label{eq:allen-cahn-fd}
U^{n+1} 
=
R_{\Delta t, \Delta x} \left( U^n + \Delta t f(U^n) \right),
\quad
U^0_i = u(x_i), \; i=1,\dots, m.
\end{align}
where $R_{\Delta t, \Delta x} = (I - \Delta t D_{\Delta x})^{-1}$ is a $m\times m$ matrix. We furthermore note the following result of \cite{TangYang2016}:
\begin{theorem}[{\cite[Thm. 2.1]{TangYang2016}}]
\label{thm:ty}
Consider the Allen-Cahn problem \eqref{eq:allen-cahn} with periodic boundary conditions. If the initial value is bounded by $1$, i.e. $\max_{x \in \T^d} |u(x)| \le 1$, then the numerical solution of the fully discrete scheme \eqref{eq:allen-cahn-fd} is also bounded by $1$ in the sense that $\Vert U^n \Vert_{\ell^\infty} \le 1$ for all $n>0$, provided that the stepsize satisfies $0<\Delta t \le \frac12$.
\end{theorem}
We also prove in Appendix \ref{app:pf49}, the following convergence result for the scheme \eqref{eq:allen-cahn-fd},
\begin{theorem} \label{thm:ac-approx-err}
Consider the Allen-Cahn problem \eqref{eq:allen-cahn} with periodic boundary conditions. Assume that the solution $v \in C^{(2,4)}([0,T]\times\T^d)$, there exists a constant $C>0$ independent of $v$, $\Delta t$ and $\Delta x$, such that the error $E^n_i := | U^n_i - v(x_i,t_n) |$, $i=1,\dots, m$ is bounded by
\[
\Vert E^n \Vert_{\ell^\infty}
\le 
(\Delta t + \Delta x^2)\exp\left({C t_n \Vert v \Vert_{C^{(2,4)}([0,T] \times \T^d)} }\right).
\]
\end{theorem}
Next, we proceed to bound the approximation error \eqref{eq:approximation}.  To provide an upper bound on the size of the approximator network $\cA$, we next show that the numerical scheme \eqref{eq:allen-cahn-fd} can be efficiently approximated by a suitable neural network. By ``efficient'', we imply that the required size of the neural network $\cA$ increases at most polynomially with the number of sensor points $m$, rather than exponentially. We begin with the following observation, proved in Appendix \ref{app:pf410},

\begin{lemma} \label{lem:nn-nonlinearity}
Let $f(v) = v-v^3$ be the nonlinearity in the Allen-Cahn equation \eqref{eq:allen-cahn}. There exist constants $C,M>0$, such that for any $\epsilon \in (0,1)$, there exists a ReLU neural network $g_\epsilon: \R \to \R$, $\xi \mapsto g_\epsilon(\xi)$ with $\size(g_\epsilon) \le C(1+|\log(\epsilon)|)$, $\depth(g_\epsilon) \le C(1+|\log(\epsilon)|)$, and such that 
\[
\sup_{\eta \in [-1,1]} |f(\eta) - g_\epsilon(\eta)| < \epsilon,
\]
and $\Lip(g_\epsilon) \le M$.
\end{lemma}
Thus, the above lemma shows that there exists a small neural network with ReLU activation function that approximates the nonlinearity in the Allen-Cahn equations to high accuracy. In fact, if we use a smooth, i.e. $C^3$ activation function, one can see from Remark \ref{rem:sact} in Appendix \ref{app:pf410} that an even smaller network (with size that does not need to increase with increasing accuracy) suffices to represent this nonlinearity accurately. However, we stick to ReLU activation functions here for definiteness. 
Given $\epsilon > 0$, let $g_\epsilon: \R \to \R$ be a ReLU neural network as in Lemma \ref{lem:nn-nonlinearity}. It is now clear that if $g_\epsilon(\eta)$ is represented by a (small) neural network, then there exists a larger neural network $\cN$, such that 
\[
\cN(\tilde{U}^0)
=
\tilde{U}^n,
\]
where $\tilde{U}^k$, $k=1,\dots, n$ is determined by the recursion relation 
\begin{align} \label{eq:approx-scheme}
\left\{
\begin{aligned}
\tilde{U}^{k+1}
&=
R_{\Delta x, \Delta t} 
\left(
\tilde{U}^k
+
\Delta t 
g_\epsilon(\tilde{U}^k)
\right)
\\
\tilde{U}^0 &\in \R^m.
\end{aligned}
\right.
\end{align}
More precisely, we have the following lemma, proved in Appendix \ref{app:pf411}, for the \emph{emulation} of the finite difference scheme \eqref{eq:approx-scheme} with a suitable ReLU neural network,
\begin{lemma} \label{lem:approx-scheme}
There exists a constant $C>0$, such that for any $m\in \N$, $\epsilon > 0$, there exists a neural network $\cN$ with $\size(\cN)\le Cn(m^2 + m |\log(\epsilon)|)$, and $\depth(\cN) \le C(1+n|\log\epsilon|)$, such that 
$\cN(\tilde{U}^0) = \tilde{U}^n$, maps any initial data $\tilde{U}^0\in \R^m$ to the solution $\tilde{U}^n$ of the recursion \eqref{eq:approx-scheme}.
\end{lemma}
Finally, we estimate the \emph{error} i.e. difference between the result $U^n$ of the exact update rule for the numerical scheme \eqref{eq:allen-cahn-fd}, and the approximate neural network version \eqref{eq:approx-scheme} $\tilde{U}^n$ in the following Lemma (proved in Appendix \ref{app:pf412}).

\begin{lemma} \label{lem:approx-scheme-err}
Let $U^n$, $\tilde{U}^n$ be obtained by \eqref{eq:allen-cahn-fd} and \eqref{eq:approx-scheme}, respectively, with initial data $U^0 = \tilde{U}^0$ such that $\Vert U^0\Vert_{\ell^\infty} \le 1$, and $\Lip(g_\epsilon) \le M$. Then
\[
\Vert 
U^n
-
\tilde{U}^n
\Vert_{\ell^\infty}
\le 
T e^{MT} \epsilon,
\]
where $T=n \, \Delta t$.
\end{lemma}
After having emulated the finite difference scheme \eqref{eq:allen-cahn-fd} with a ReLU neural network and estimated the error in doing so, we are in a position to state the bounds on the approximation error in the following proposition, proved in Appendix \ref{app:pf413},
\begin{proposition} \label{prop:ac-approx-err}
Let $\mu$ be given as the law of \eqref{eq:mu-allen-cahn}.
Let the encoder/decoder pair $\cE$, $\cD$ be given as in Proposition \ref{prop:ac-encoding-err} for some $m\in \N$. Let $\cR$ be given as in Proposition \ref{prop:ac-rec-err}, for $p \in \N$. There exists a constant $C>0$, independent of $m$ and $p$, such that there exists a neural network $\cA: \R^m \to \R^p$, such that 
\begin{gather*}
\size(\cA) \le C (1+m^{2+2/d} + mp),
\quad
\depth(\cA) \le C(1+m^{2/d} \log(m)),
\end{gather*}
and
\[
\Err_{\cA} \le C m^{-1/d}.
\]
\end{proposition}
\begin{remark} \label{rem:ac-higher-order}
Choosing the number of trunk nets $p \lesssim m^{1+2/d}$, it follows from the previous proposition that an approximation error $\cA$ of order $\epsilon$ can be achieved, with $m\sim \epsilon^{-d}$, i.e. a neural network of size 
\[
\size(\cA) =\mathcal{O}(\epsilon^{-2(d+1)}),
\quad
\depth(\cA) = \mathcal{O}(\epsilon^{-2} |\log(\epsilon)|).
\]
This is polynomial in $\epsilon^{-1}$ (recall that $d=2,3$ is a fixed constant independent of the accuracy $\epsilon$), and hence does not suffer from the curse of dimensionality for the infinite-dimensional approximation problem (cp. Definition \ref{def:cod}). We would nevertheless expect that sharper error estimates could be obtained by basing the neural network emulation result on higher-order finite difference methods. In the present work, our main objective is to show that neural networks can break the curse of dimensionality, rather than attempting to establish optimal complexity bounds.
\end{remark}
\subsubsection{Bounds on the DeepONet approximation error \eqref{eq:approxerr}}
Combining Propositions \ref{prop:ac-encoding-err}, \ref{prop:ac-rec-err}, \ref{prop:ac-approx-err}, we can now state the following theorem

\begin{theorem} \label{thm:ac-error}
Consider the DeepONet approximation problem for the Allen-Cahn equation \eqref{eq:allen-cahn}, where the initial data $u$ is distributed according to a probability measure $\mu \in \P(L^2(\T^d))$ is the law of the random field \eqref{eq:mu-allen-cahn}. There exist constants $C, \, c> 0$, such that for any $m, \, p \in \N$, there exists a DeepONet \eqref{eq:donet} with trunk net $\bm{\tr}$ and branch net $\bm{\br}$, such that 
\begin{gather*}
\size(\bm{\tr}) \le C(1 + p\log(p)^2),
\quad
\depth(\bm{\tr}) \le C(1+\log(p)^2),
\end{gather*}
and
\begin{gather*}
\size(\bm{\br}) \le C(1+m^{2+2/d}+pm), 
\\
\depth(\bm{\br}) \le C(1+ \log(m) \log\log(m)),
\end{gather*}
and such that the DeepONet approximation \eqref{eq:approxerr} is bounded by 
\begin{equation}
\label{eq:ac-err}
\Err 
\le 
C\exp(-c\ell m^{1/d}) + C p^{-4/d} + Cm^{-1/d}.
\end{equation}
\end{theorem}

\begin{remark}
The bound \eqref{eq:ac-err} guarantees that for $\epsilon > 0$, a DeepONet approximation error of $\Err \sim \epsilon$ can be achieved provided that $m \gtrsim \max(\ell^{-d} \log(\epsilon^{-1})^d, \epsilon^{-d})$ and $p \gtrsim \epsilon^{-d/4}$. Assuming that $m^{1/d}\ell \gg 1$ is satisfied, i.e., that the sensors can resolve the typical correlation length scale $\ell$ in the $d$-dimensional domain $D=\T^d$, this can be ensured provided that $m \gtrsim \epsilon^{-d}$ and $p \gtrsim \epsilon^{-d/4}$.
In this case, an error $\Err \lesssim \epsilon$ can thus be achieved with an overall DeepONet size of order 
\begin{equation}\label{eq:ac-comp}
\size(\bm{\tr},\bm{\br}) \lesssim \epsilon^{-d/4} \log(\epsilon^{-1})^2 + \epsilon^{-2(d+1)}
\lesssim \epsilon^{-2(d+1)}.
\end{equation}
For the cases of interest, $d=2,3$, this upper bound on the required size of the DeepONet thus scales as $\size \lesssim \epsilon^{-6}$, and $\size \lesssim \epsilon^{-8}$, respectively. This shows that the required size scales at worst algebraically in $\epsilon^{-1}$, thus \emph{breaking the curse of dimensionality}, as per Definition \ref{def:cod}. As already pointed out in Remark \ref{rem:ac-higher-order}, the explicit exponents $-6$ and $-8$ may be considerably improved if the emulation result were based on a \emph{higher-order order} numerical scheme in place of the low-order finite difference scheme \eqref{eq:allen-cahn-fd}.
\end{remark}
\subsection{Nonlinear hyperbolic PDE: Scalar conservation laws.}
\label{sec:scl}
As a final concrete example in this paper, we will consider this prototypical example of nonlinear hyperbolic PDEs. We remark that the image of the underlying operator $\G$ in the previous three examples consisted of smooth functions. On the other hand, it is well known that solutions of conservation laws can be discontinuous, on account of the formation of shock waves. Thus, one of our objectives in this section is to show that DeepONets can even approximate nonlinear operators, that map into discontinuous functions, efficiently.
\subsubsection{Problem formulation.}
We consider a scalar conservation law on a one-dimensional domain $D \subset \R$: 
\begin{align} \label{eq:conslaw}
\left\{
\begin{aligned}
\partial_t v + \partial_x( f(v) ) = 0, \\
v(t = 0) = u,
\end{aligned}
\right.
\end{align}
with initial data $u$ drawn from some underlying measure $\mu$, and with a given flux function $f\in C^2(\R)$. For simplicity, we will assume that $D = \T = [0,2\pi]$, and periodic boundary conditions for \eqref{eq:conslaw}. Our results are however readily generalized to other boundary conditions and to several space dimensions. Solutions of \eqref{eq:conslaw} are interpreted in the weak sense, imposing the entropy conditions
\[
\partial_t \eta(v) + \partial_x q(v) \le 0, \quad 
\text{(in the distributional sense)}
\]
for all entropy/entropy flux pairs $(\eta,q)$, consisting of a convex $C^1$-function $\eta: \R \to \R$ and the corresponding flux $q: \R \to \R$ with derivative satisfying $q'(u) = f'(u)\eta'(u)$ for all $u\in \R$. 

Under these conditions, it is well known \cite{HRbook} that the solution $v$ of \eqref{eq:conslaw} is unique for any initial data $u \in  L^1(\T) \cap L^\infty(\T)$, and the solution operator $\cS_t$ is contractive as a function $L^1(\T) \to L^1(\T)$ for any $t \in [0,\infty)$:
\begin{align} \label{eq:contractivity}
\Vert 
\cS_t(u) - \cS_t(u') 
\Vert_{L^1(\T)}
\le
\Vert 
u - u'
\Vert_{L^1(\T)},
\quad
\forall \, u,u' \in L^1(\T) \cap L^\infty(\T).
\end{align}
We note that $\cS_t: \mathrm{BV}(\T) \to \mathrm{BV}(\T)$, maps functions of bounded variation to functions of bounded variation, in fact we have
\begin{align} \label{eq:BV}
\Vert \cS_t(u) \Vert_{\mathrm{BV}}
\le
\Vert u \Vert_{\mathrm{BV}}.
\end{align}
Here, we denote $\Vert \slot \Vert_{\mathrm{BV}} = \Vert \slot \Vert_{L^1} + \mathrm{TV}(\slot)$, as the BV-norm with $\mathrm{TV}(w)$ representing the total variation of a function $w$ \cite{HRbook}. Furthermore, $\cS_t(u)$ satisfies the maximum principle, so that
\[
\Vert \cS_t(u) \Vert_{L^\infty(\T)}
\le
\Vert u \Vert_{L^\infty}.
\]

Next, in order to specify the DeepONet approximation problem (cf. Definition \ref{def:data}), we take the nonlinear operator $\G:L^1(D) \to L^1(D)$ as $\G(u) := \cS_T(u)$, mapping the initial data $u$ of the scalar conservation law \eqref{eq:conslaw} to the solution $\cS_T(u) = v(\slot,t=T)$ at the final time $T>0$. Clearly, given the $L^1$-contractivity \eqref{eq:contractivity}, the operator $\G$ is well-defined and Lipschitz continuous. 

Defining the set 
\begin{align}
\BV_M
:=
\set{u \in \BV(\T)}{\Vert u \Vert_{\BV} \le M},
\end{align}
we will consider any measure $\mu$ that satisfies $\mu(\BV_M)=1$, as the underlying measure for the DeepONet approximation problem. 
\subsubsection{DeepONet approximation in the Banach space $L^1(D)$}
So far, we have only considered the DeepONet approximation of operators, defined on the Hilbert spaces. However, given the fact that the solution operator $\cS_t$ and the resulting operator $\G$ are contractive on the Banach space $L^1(D)$, it is very natural to consider the DeepONet approximation problem in this function space. Thus, we need to modify the definition of the DeepONet error \eqref{eq:approxerr} and define its $L^1$-version by,
\begin{align} \label{eq:L1approx}
\Err_{L^1}
=
\int_{L^1(D)} \Vert \G(u) - \cN(u)\Vert_{L^1} \, d\mu(u),
\end{align}
with the DeepONet $\cN$ \eqref{eq:donet} approximating the operator $\cG$. It is of the form,
\begin{equation}
    \label{eq:dnetscl}
    \cN =  \cR \circ \cA \circ \bar{\cE},
\end{equation}
with $\cR: \R^p \to C(D)$ is the usual affine reconstruction defined in \eqref{eq:reconstruction} based on the trunk network $\bm{\tr}$, $\cA: \R^m \to \R^p$ is the approximator neural network used to define the branch network $\bm{\br}$, and we have introduced a generalized encoder $\bar{\cE}: L^1(D) \to \R^m$, which is defined by taking local averages in cells $C_1,\dots, C_m\subset D$:
\begin{align} \label{eq:encoder-avg}
\bar{\cE}(u) 
=
\left(
\fint_{C_1} u(x) \, dx,
\dots,
\fint_{C_m} u(x) \, dx
\right),
\end{align}
with cells $C_1,\dots, C_m$ given by 
\begin{align} \label{eq:cells}
C_j := [x_j - \Delta x/2, x_j + \Delta x/2],
\quad 
(\text{for }j=1,\dots, m),
\end{align}
where $x_1,\dots, x_m$ denote $m$ equidistant sensors on the periodic domain $D = \T$, with $\Delta x = 2\pi / m$. Note that the encoder $\bar{\cE}$ is well-defined (and in fact continuous) for any $u\in L^1(D)$. In particular, this encoding allows us to consider discontinuous initial data $u$, and thus we do not need to assume that the underlying measure $\mu$ is concentrated on $C(D)$. This choice of encoder constitutes a key difference with pointwise encoder $\cE$, considered in section \ref{sec:2} and all the previous examples. 

Given this architecture for the DeepONet \eqref{eq:dnetscl}, we aim to bound the resulting approximation error \eqref{eq:L1approx}. As we are no longer in the Hilbert space setting, we cannot directly appeal to the abstract error estimates of section \ref{sec:3}, nor follow the program outlined at the beginning of this section. Nevertheless, we will still follow a key idea from the last subsection, namely emulating numerical schemes approximating the underlying PDEs with neural networks. To this end, we recall a straightforward adaptation of the results of  a recent paper \cite{DRM1};
\begin{theorem}
\label{thm:LxF}
Let $M>0$ be given. Consider initial data $u\in \BV_M$ for the scalar conservation law \eqref{eq:conslaw} with flux function $f\in C^2(\R)$. Let the operator $\G(u) = \cS_T(u)$ be given by mapping $u\mapsto v(T,\slot)$, where $v$ solves \eqref{eq:conslaw}. There exists a constant $C = C(T,\Vert f\Vert_{C^2}) > 0$, such that for any $m\in \N$, there exists a neural network $\cN: \R^m \to \R^m$ with 
\[
\size(\cN) \le C m^{5/2}, 
\quad
\depth(\cN) \le C m,
\]
such that 
\[
\left\Vert 
\G(u) - \sum_{j=1}^m \cN_j(\bar{\cE}(u)) \, 1_{C_i}(\slot)
\right\Vert_{L^1(D)}
\le
\frac{C}{m^\alpha}.
\]
Here $\alpha>0$ is the convergence rate of the well-known Lax-Friedrichs scheme \cite{HRbook} and $1_{C_j}(\slot)$ denotes the indicator function of the cell $C_i$ (cp. \eqref{eq:cells}). For the neural network $\cN$ we furthermore have 
\[
\Vert \cN(\bar{\cE}(u)) \Vert_{\ell^\infty} \le M, \quad \text{ for all $u\in \BV_M$.}
\]
\end{theorem}
\begin{remark}
The worst-case estimate for the convergence rate $\alpha$ of the Lax-Friedrichs scheme guarantees that $\alpha \ge \frac12$ \cite{HRbook}. However, in practice, the observed convergence rate is often higher, $\alpha \approx 1$.
\end{remark}

The proof of Theorem \ref{thm:LxF} is based on the fact that neural networks can emulate the Lax-Friedrichs difference scheme for \eqref{eq:conslaw}. In fact, the neural network $\cN$ constructed in \cite{DRM1} is an exact form of the Lax-Friedrichs scheme applied to a scalar conservation law \eqref{eq:conslaw}, but with a neural network approximated flux $\hat{f}\approx f$, and satisfying a suitable CFL condition. The bound on $\Vert \cN(\bar{\cE}(u)) \Vert_{\ell^\infty}\le M$ therefore follows from the fact that $\Vert u \Vert_{L^\infty} \le \Vert u \Vert_{\BV} \le M$ and the fact that the Lax-Friedrichs scheme satisfies the maximum principle.

To obtain a DeepONet approximation result from \ref{thm:LxF}, we have the following simple Lemma (proved in Appendix \ref{app:pf414}), on the approximation of characteristic functions by ReLU neural networks,
\begin{lemma} \label{lem:charfun}
There exists a constant $C>0$, such that for any $a,b\in \R$, $a<b$, and $\epsilon > 0$, there exists a ReLU neural network $\chi_{[a,b]}^\epsilon: \R \to \R$, with 
\[
\size(\chi^\epsilon_{[a,b]}) \le C, \quad \depth(\chi^{\epsilon}_{[a,b]}) = 1,
\]
and 
\[
\Vert \chi^\epsilon_{[a,b]} - 1_{[a,b]} \Vert_{L^1(\R)} \le \epsilon.
\]
\end{lemma}

This allows us to prove in Appendix \ref{app:pf415}, our main approximation result for the DeepONet approximation problem \eqref{eq:L1approx} for scalar conservation laws,
\begin{theorem} \label{thm:err-conslaw}
Let $\mu\in \P(L^1(D))$ be a probability measure such that there exists $M>0$ such that $\mu(\BV_M)=1$. Let the underlying operator $\G$ map the initial data $u$ to the solution (at final time) $v(.,T)$ of the scalar conservation law \eqref{eq:conslaw}. Let the encoder $\bar{\cE}$ of the DeepONet \eqref{eq:dnetscl} be given by \eqref{eq:encoder-avg} with $m$ equidistant cells $C_j$ \eqref{eq:cells}. There exists a constant $C>0$, independent of $m$ and $p$, such that for any $m, \, p\in \N$, there exists a DeepONet $(\bm{\tr},\bm{\br})$ with trunk net $\bm{\tr}: \R^p \to C(D)$, and branch net $\bm{\br}: L^1(D) \to \R^p$, of the form $\bm{\br}(u) = \cA(\bar{\cE}(u))$ for a neural network $\cA: \R^m \to \R^p$, with 
\[
\size(\bm{\br}) \le Cm^{5/2},
\quad
\depth(\bm{\br}) \le Cm,
\]
and
\[
\size(\bm{\tr}) \le C p, 
\quad
\depth(\bm{\tr}) = C,
\]
such that 
\begin{align} 
\label{eq:err-conslaw}
\Err_{L^1}
\le
C\max\left(1-\frac{p}{m},0\right) + Cm^{-\alpha},
\end{align}
with the DeepONet approximation error $\Err_{L^1}$ defined in \eqref{eq:L1approx}.  
\end{theorem}
\begin{remark}
To achieve an error $\Err_{L^1}\sim \epsilon$, the estimate \eqref{eq:err-conslaw} provided by Theorem \ref{thm:err-conslaw} shows that it is sufficient that $p \ge (1-\epsilon)m$, i.e. $p\sim m$, and $m \gtrsim \epsilon^{-1/\alpha}$. Thus, a DeepONet approximation error of order $\epsilon$ can be achieved with a DeepONet $(\bm{\br},\bm{\tr})$ of size
\begin{equation}
    \label{eq:scl-comp}
\size(\bm{\br}) \sim \epsilon^{-5/2\alpha}, 
\quad
\size(\bm{\tr}) \sim \epsilon^{-1/\alpha}.
\end{equation}
For the worst-case rate of $\alpha=1/2$, this yields a total DeepONet size of order $\sim \epsilon^{-5}$. For more realistic values of $\alpha \approx 1$, we require a size of order $\sim \epsilon^{-2.5}$. Clearly, this scales polynomially in $\epsilon^{-1}$, and thus DeepONets can efficiently approximate the solution operator for scalar conservation laws \emph{by breaking the curse of dimensionality}, see Definition \ref{def:cod}. 
\end{remark}

\subsubsection{
A lower bound on the DeepONet approximation error
}

Theorem \ref{thm:err-conslaw} shows that with the natural scaling , $p = m$, we have $\Err_{L^1} \le C p^{-\alpha}$, where $\alpha >0$ is the convergence rate of the Lax-Friedrichs scheme, and $p\in \N$ is the output dimension of the branch/trunk net $\bm{\br}, \, \bm{\tr}$. The goal of this section is to present an example of a measure $\mu \in \BV_M$, for which we also have a lower bound of the form $\Err_{L^1} \gtrsim p^{-1}$ and demonstrate that the DeepONet approximation, as described above, \emph{is almost optimal}.  

To this end, we will rely on the lower bound \eqref{eq:donetlbd} for the ($L^2$-based) error $\Err \ge \Err_{\cR} \ge \sqrt{\sum_{k>p} \lambda_k}$, of Theorem \ref{thm:lowerbd}. Recall $\Err$ denotes the DeepONet error \eqref{eq:approxerr} in the $L^2$ norm, $\Err_{\cR}$ denotes the reconstruction error \eqref{eq:decoding} and $\lambda_1 \ge \lambda_2 \ge \dots$ denote the eigenvalues of the covariance operator ${\Gamma}_{\G_\#\mu} = \int (v-\E[v]) \otimes (v-\E[v]) \, d(\G_\#\mu)(v)$, associated with the push-forward measure $\G_\#\mu$. 

To provide a concrete example which exhibits the decay $\Err_{L^1} \lesssim C p^{-1}$ in terms of the output dimension $p$ of the trunk net, we now consider the measure $\mu \in \P(L^2(\T))$ given as the law
\begin{align} \label{eq:burgers-law}
\mu 
=
\mathrm{law}
\set{
-\sin(\slot - \hat{x})
}{
\hat{x} \sim \unif(\T)
},
\end{align}
for the Burgers' equation, i.e., scalar conservation law \eqref{eq:conslaw},
\begin{align} \label{eq:burgers}
\left\{
\begin{aligned}
\partial_t v + \partial_x\left( v^2/2 \right) &= 0,
\\
v(t=0) = u.
\end{aligned}
\right.
\end{align}
Let $v_0(x,t)$ denote the solution of \eqref{eq:burgers} with initial data 
\[
v_0(x,t=0) = u_0(x) := -\sin(x).
\]
The characteristic starting at $x_0$ for this data is given by 
\[
x(t;x_0) = x_0 - \sin(x_0) t.
\]
By the method of characteristics, the solution $v(x,t)$ can be expressed in the form
\[
v(x(t),t) = u_0(x_0),
\]
for any $t\ge 0$ which is sufficiently small such that the mapping $x_0 \mapsto x(t;x_0)$ is one-to-one. Since
\[
\frac{\partial x}{\partial x_0} = 1 - \cos(x_0) t, 
\]
it follows that a classical solution exists for any $t<1$, but the characteristics cross at time $t=1$, corresponding to the formation of a stationary shock wave at $x_0 = 0$. The size of the jump of $v(x,t)$ at $x = 0$ at the particular $t=\pi/2$ is given by 
\[
v(0+,\pi/2) - v(0-,\pi/2)
=
2,
\]
corresponding to the time $t$, at which the characteristics emanating from $x_0 = \pm\pi/2$ reach the stationary shock at the origin. Furthermore, the function $v(x,t)$ is smooth on $\T \setminus \{0\}$. From basic properties of the Fourier coefficients of functions with jump-discontinuities, we can conclude
\begin{lemma} \label{lem:shock-fourier}
The solution $v_t(x) := v(x,t)$ at time $t=\pi/2$ of the Burgers equation \eqref{eq:burgers} with initial data $u_0(x) = -\sin(x)$, has Fourier coefficients with asymptotic decay
\[
\hat{v}_t(k) 
=
\frac{1}{2\pi}
\int_0^{2\pi}
v_t(x) e^{-ikx} \, dx
=
\frac{-i}{\pi k} + o\left(\frac{1}{k}\right), \quad (k\to \infty).
\]
\end{lemma}

Based on this lemma, we can now estimate the spectral decay of the covariance operator $\Gamma_{\G_\#\mu}$ of the push-forward measure $\G_\#\mu$:

\begin{lemma} \label{lem:burgers-spectrum}
Let $\G(u) := v_t$, where $v(x,t) = v_t(x)$ is the solution of the inviscid Burgers equation with initial data $u$, evaluated at time $t= \pi/2$. Let $\mu$ be given as the law \eqref{eq:burgers-law}. Then the eigenfunctions of $\Gamma_{\G_\#\mu}$ is given by the standard Fourier basis $\fb_k$, $k\in \N$, with eigenvalues
\[
\lambda_k = \frac{1}{\pi^2 k^2} + o\left(\frac{1}{k^2}\right), \quad (k \to \infty).
\]
\end{lemma}

The proof of this lemma relies on the observation that $\G_\#\mu$ is a translation-invariant measure on $L^2(\T)$, and hence, the integral kernel representing its covariance operator $\Gamma_{\G_\#\mu}$ is \emph{stationary}. In particular, this implies that the eigenfunctions of $\Gamma_{\G_\#\mu}$ are given by the standard Fourier basis. Furthermore, the asymptotics of the eigenvalues $\lambda_k \propto k^{-2}$, as $k\to \infty$ can in this case be determined explicitly, based on Lemma \ref{lem:shock-fourier}. The details of the argument are provided in appendix \ref{app:burgers-spectrum}.

As a consequence of Lemma \ref{lem:burgers-spectrum}, we can now state the following result:

\begin{theorem} \label{thm:burgers}
Let $\mu \in \P(L^2(\T))$ be given by the law \eqref{eq:burgers-law}. Let $u \mapsto \G(u)$ denote the operator, mapping initial data $u(x)$ to the solution $v_t(x) = v(x,t)$ at time $t = \pi/2$, where $v$ solves the inviscid Burgers equation \eqref{eq:burgers}. Then there exists a universal constant $C>0$ (depending only on $\mu$, but independent of the neural network architecture), such that the DeepONet approximation error $\Err$ for any trunk net of size $p$ is bounded from below by
\[
\Err 
\ge 
\frac{C}{\sqrt{p}}.
\]
\end{theorem}

\begin{proof}
By Theorem \ref{thm:lowerbd}, the DeepONet error $\Err$ satisfies the lower bound
\[
\Err \ge \sqrt{\sum_{k>p} \lambda_k},
\]
where $\lambda_1 \ge \lambda_2 \ge \dots$ denote the ordered (repeated) eigenvalues of the covariance operator $\Gamma_{\G_\#\mu}$, corresponding to a complete orthonormal eigenbasis $\phi_1, \phi_2, \dots$ of $\Gamma_{\G_\#\mu}$. By Lemma \ref{lem:burgers-spectrum}, the asymptotic decay of the eigenvalues can be estimated from below by $\lambda_k \ge \left(\frac{c}{k}\right)^2$, for a suitable constant $c>0$. It follows that
\[
\Err 
\ge  \sqrt{\sum_{k>p} \lambda_k} 
\ge  c \sqrt{\sum_{k>p} \frac{1}{k^2}}
\ge \frac{C}{\sqrt{p}},
\]
for some $C>0$, as claimed.
\end{proof}

\begin{remark}
Note that since $-\sin(x - \hat{x}) =  \sin(\hat{x})\cos(x) - \cos(\hat{x}) \sin(x)$, the probability measure $\mu$ of Theorem \ref{thm:burgers} is supported on a (compact subset of a) \emph{two-dimensional} subspace $\Span(\cos(x), \sin(x)) \subset L^2(\T)$. In particular, the spectrum of $\Gamma_\mu$ decays \emph{at an arbitrarily fast rate}, asymptotically (since almost all eigenvalues are zero), and the encoding error can be made to vanish, $\Err_{\cE} = 0$, with a suitable choice of only two sensors $x_1, x_2$. Nevertheless, Theorem \ref{thm:burgers} shows that the push-forward measure $\G_\#\mu$ under the inviscid Burgers equation is sufficiently complex that the reconstruction error $\Err_{\cR}$ cannot decay faster than $p^{-1/2}$ in the dimension of the reconstruction space $p$. In particular, the fast spectral decay of $\mu$ does not imply a similarly fast spectral decay of $\G_\#\mu$ under the inviscid Burgers dynamics.
\end{remark}

We use the above result to claim that the $L^1$-error $\Err_{L^1}$ satisfies the following lower bound:
\begin{theorem} \label{thm:L1-lower}
Let $\mu \in \P(\BV_M)$ denote the probability measure \eqref{eq:burgers-law}. Let $\Err_{L^1}$ be the DeepONet approximation error given by \eqref{eq:L1approx}. Let $\bar{M}>0$. Then there exists a constant $C = C(\bar{M})>0$, independent of $p$, such that
\[
\Err_{L^1} 
=
\int_{\BV_M} 
\Vert \G(u) - \cN(u) \Vert_{L^1(\T)} 
\, d\mu(u)
\ge \frac{C}{p},
\]
for any DeepONet $\cN = \cR \circ \cA \circ \bar{\cE}$, such that $\sup_{u\in \supp(\mu)} \Vert \cN(u) \Vert_{L^\infty} \le \bar{M}$.
\end{theorem}

\begin{remark}
Although we do not have a proof, it appears reasonable to conjecture that the $L^1$-optimal DeepONet approximations of $\G$ satisfy a uniform bound of the form $\sup_{u\in \supp(\mu)} \Vert \cN(u) \Vert_{L^\infty(\T)} \le \bar{M}$ for some $\bar{M} > 0$, independent of $p$. If this is indeed the case, then Theorem \ref{thm:L1-lower} is valid without the additional assumption on $\sup_{u\in \supp(\mu)} \Vert \cN(u) \Vert_{L^\infty} \le \bar{M}$.

Clearly, the DeepONet constructed in Theorem \ref{thm:err-conslaw} belongs to the class of DeepONets, satsifying a bound $\sup_{u\in \mu} \Vert \cN(u) \Vert_{L^\infty(\T)} \le \bar{M}$ (in fact, with $\bar{M} = M$). Thus, we cannot expect to improve the upper bound \eqref{eq:err-conslaw} to a convergence rate $\alpha > 1$.
\end{remark}

\begin{proof}
Let $X$ denote the support of $\mu$, let $\cR$ be any reconstruction. Following the argument of Theorem \ref{thm:lowerbd}, we have 
\begin{align*}
\sum_{k>p} \lambda_k
\le
(\Err_{\cR})^2 
\le
\int_X \Vert \G(u) - \cN(u) \Vert^2_{L^2(\T)} \, d\mu(u).
\end{align*}
By Theorem \ref{thm:burgers}, we have $\sum_{k>p} \lambda_k \ge Cp^{-1}$, for an absolute constant $C>0$. We now note the following interpolation inequality
\begin{align*}
\Vert \G(u) - \cN(u) \Vert^2_{L^2(\T)}
&\le
\Vert \G(u) - \cN(u) \Vert_{L^1(\T)} \,
\Vert \G(u) - \cN(u) \Vert_{L^\infty(\T)}.
\end{align*}
If $u\in \BV_M$, then $\Vert u \Vert_{L^\infty}\le M$ and hence also $\Vert \G(u) \Vert_{L^\infty(\T)} \le M$. Furthermore, by assumption, we have $\Vert \cN(u) \Vert_{L^\infty(\T)} \le \bar{M}$. Thus, we conclude that 
\[
\frac{C}{p}
\le
(M+\bar{M}) \int_{X} \Vert \G(u) - \cN(u) \Vert_{L^1(\T)} \, d\mu(u),
\]
for any such DeepONet $\cN(u)$. This clearly implies that claimed lower bound.
\end{proof}

\section{On the generalization error for DeepONets}
\label{sec:5}

In section \ref{sec:3} and with the concrete examples in section \ref{sec:4}, we have shown that there exists a DeepONet $\cN$ of the form \eqref{eq:donet} which approximates an underlying operator $\G$ efficiently, i.e., the DeepONet approximation error \eqref{eq:approxerr} can be made small without incurring the curse of dimensionality in terms of the complexity of the DeepONet \eqref{eq:donet}. However, in practice, one needs to \emph{train} the DeepONet \eqref{eq:donet} by using a gradient descent algorithm to find neural network parameters (weights and biases for the trunk and branch nets) that minimize the \emph{loss function} 
\begin{align} \label{eq:loss}
\hL(\cN)
:=
\int_{L^2(D)} \int_{U}
\left|
\G(u)(y)
-
\cN(u)(y)
\right|^2 
\, dy \, d\mu(u),
\end{align}
which is related to the approximation error $\Err$ by $\Err = \sqrt{\hL}$ (cf. \eqref{eq:approxerr}).

However the loss function $\hL$ cannot be computed exactly, and is usually approximated by sampling in both the target space $y \in U$ and the input function space $u \in L^2$. As is standard in deep learning \cite{DLbook}, one can follow \cite{deeponets} and take $N_u$ iid samples $U_1, \dots, U_{N_u} \sim \mu$ (with underlying measure $\mu$ for the DeepONet approximation problem), and for each sample $U_j$ to evaluate $\G(U_j): U \to \R$ at $N_y$ points $Y_j^1, \dots, Y_j^{N_y}$ with corresponding weights $w_j^k > 0$, leading to the following \emph{empirical loss} $\hL_{N_u,N_y} \approx \hL$, for the DeepONet $\cN = \cR \circ \cA \circ \cE$ approximation of $\G$:
\begin{align}\label{eq:empirical-loss}
\hL_{N_u,N_y}(\cN)
:=
\frac{1}{N_u} 
\sum_{j=1}^{N_u} 
\sum_{k=1}^{N_y} 
w_k^j
\left|
\G(U_j)(Y^j_k) - \cN(U_j)(Y^j_k)
\right|^2.
\end{align}
If we denote $\Delta(u,y) := |\G(u)(y) - \cN(u)(y)|^2$, then \eqref{eq:empirical-loss} can be written in the form 
\begin{align}
\hL_{N_u,N_y}(\cN)
&=
\frac{1}{N_u} 
\sum_{j=1}^{N_u} 
I^{N_y}_j(\Delta(U_j,\slot)), \\
I^{N_y}_j(\Delta(U_j,\slot)
&=
\sum_{k=1}^{N_y} 
w_k^j \Delta(U_j,Y^k_j).
\end{align}
So far, we have not specified how the $Y_j^k$, $k=1,\dots, N_y$, $j=1,\dots, N_u$ are to be chosen. As we want the innermost sum to be an approximation
\[
I_j^{N_y}(\Delta(U_j,\slot))
\approx
\int_{U} 
\Delta(U_j,y)
\, 
dy,
\]
we propose two intuitive options:
\begin{enumerate}
\item Choose $Y_j^k$ to be random variables, independent of all $U_{\tilde{j}}$, and drawn iid uniform on $U$, i.e. 
\begin{align} \label{eq:Yrandom}
Y_j^k \sim \unif(U) \text{ iid},
\quad
w_j^k = \frac{|U|}{N_y},
\end{align}
for $j=1,\dots, N_u$, $k=1,\dots, N_y$.
\item Let $y_k\in U$, $w_k > 0$, for $k=1,\dots, N_y$, be the evaluation points and weights of a suitable quadrature rule on $U$. Then choose
\begin{align} \label{eq:Ydeterministic}
Y_j^k = y_k, \quad w_j^k = w_k \quad \text{for } k=1,\dots, N_y.
\end{align}
\end{enumerate}
\subsection{Deterministic choice of $Y^j_k$.}
If the $\G(u)$ and $\cN(u)$ have bounded $k$-th derivatives in $y$, uniformly for all $u \in \supp(\mu)$, then a suitable choice of the quadrature points and weights $y_k$, $w_k$ in \eqref{eq:Ydeterministic} can lead to a quadrature error
\begin{align} \label{eq:quad-error}
\left|
I^{N_y}(\Delta(U_j,\slot))
-
\int_{U} 
\Delta(U_j,y)
\, 
dy
\right|
\le
C\Vert \Delta(U_j,\slot) \Vert_{C^k(U)} \, N_y^{-k/n}.
\end{align}
Thus, for small dimensions (e.g. $n=1,2,3$), and sufficiently regular $\G(u)$, one might expect a deterministic quadrature rule to significantly outperform a random sampling in $y$. Note that if $y \mapsto \Delta(U,y)$ possesses a complex analytic extension to a neighbourhood of $U$ for suitable $U$ and quadrature (e.g. $U$ is a torus, and the quadrature is given by the trapezoidal rule), then the error in \eqref{eq:quad-error} in fact decays \emph{exponentially} in $N_y$. This might be of relevance for problems such as the forced gravity pendulum in section \ref{sec:pendulum} and the elliptic PDE considered in section \ref{sec:elliptic}.

We will be interested in the difference $\hL_{N_u,N_y}(\cN) - \hL(\cN)$ in the following. We denote
\[
\hL_{N_u,\infty}(\cN) 
:=
\frac{1}{N_u} \sum_{j=1}^{N_u} \int_{U} |\G(U_j)(y) - \cN(U_j)(y) |^2 \, dy.
\]
We can decompose the difference $\hL_{N_u,N_y}(\cN) - \hL(\cN)$ as follows:
\begin{align*}
\hL_{N_u,N_y}(\cN) - \hL(\cN)
&=
\left[\hL_{N_u,N_y}(\cN) - \hL_{N_u,\infty}(\cN)\right] 
+
\left[\hL_{N_u,\infty}(\cN) - \hL(\cN)\right].
\end{align*}
Then, for any parametrized network $\cN_\theta  = \sum_{k=1}^p \br_k(u;\theta) \tr_k(y;\theta)$ with parameter $\theta$, we have (for a suitable choice of quadrature points)
\begin{align*}
\left|
\hL_{N_u,N_y}(\cN_\theta) - \hL_{N_u,\infty}(\cN_\theta)
\right|
&\le
\frac{1}{N_u} \sum_{j=1}^{N_u}
\left|
I^{N_y}(\Delta_\theta(U_j,\slot))
-
\int_{U} \Delta_\theta(U_j,y) \, dy
\right|
\\
&\le
\frac{1}{N_u} \sum_{j=1}^{N_u}
C \Vert \Delta_\theta(U_j, \slot) \Vert_{C^k(U)} \, N_y^{-k/n}
\\
&\le
C \left[ 
\Vert \G(u) \Vert_{C^k(U)}^2
+
\sum_{k=1}^p
|\br_k(u;\theta)|^2 \Vert \tr_k(\slot;\theta) \Vert_{C^k(U)}^2 
\right]
\, N_y^{-k/n}.
\end{align*}
Thus, if 
\[
\left\{
\begin{aligned}
\sup_{\theta} \Vert \tr_k(\slot ; \theta) \Vert_{C^k(U)} \le C_0,
\\
\sup_{\theta} | \br_k(\slot ; \theta) |^2 \le \Psi(u),
\\
\Vert \G(u) \Vert_{C^k(U)}^2 \le \Psi(u),
\end{aligned}
\right.
\]
for some integrable $\Psi(u)\ge 0$, such that $\E_\mu[\Psi(u)]< \infty$, then we can estimate
\[
\E\left[
\sup_{\theta}
\left|
\hL_{N_u,N_y}(\cN_\theta) - \hL_{N_u,\infty}(\cN_\theta)
\right|
\right]
\le
C N_y^{-k/n},
\]
for some $C>0$ depending on $k$, the quadrature points, $\Psi$, the upper bound $C_0$ and $\mu$, but independent of $N_y$.
\subsection{Random choice of $U_j$, $Y^j_k$}
However, in addition to the sampling in the target space $Y^j_k\in U$, one also needs to sample from the underlying input function space $U_j\in X \subset L^2(D)$. Random sampling is the only viable option for this infinite dimensional space. For the sake of simplicity of notation and exposition, we will choose $N \in \N$, and $N_u=N$, $N_y=1$ in \eqref{eq:Yrandom}: i.e., choose mutually independent random variables $U_j$, $Y_j$, $j=1,\dots, N$, such that
\begin{align} \label{eq:random-choice}
U_1,\dots, U_N \sim \mu \text{ are iid},  
\quad
Y_1, \dots, Y_N \sim \Unif(U) \text{ are iid}.
\end{align}
Let $\hL_N := \hL_{N_u,1}$, i.e.
\begin{align} \label{eq:empirical}
\hL_N(\cN)
:=
\frac{|U|}{N} \sum_{j=1}^N |\G(U_j)(Y_j) - \cN(U_j)(Y_j)|^2,
\end{align}
be the corresponding empirical loss. We note that the random variables $(U_j,Y_j)$, $j=1,\dots, N$ are iid random variables with joint distribution $(U_j,Y_j) \sim \mu \otimes \unif(U)$.

Fix a DeepONet neural network architecture of the form \eqref{eq:donet} with parameters (weights and biases in the corresponding trunk and branch nets) $\theta \mapsto \cN_\theta$. We assume that the weights and biases are bounded, $\theta \in [-B,B]^{d_\theta}$ for $B>0$, and some large $d_{\theta}\in \N$, representing the number of tuning parameters in the DeepONet $\cN_{\theta}$. Let $\hN_{N} = \cN_{\hat{\theta}}$ denote an optimizer of the empirical loss $\hL_{N}$ among all choices of the weights and biases $\theta$, let $\hN$ be an optimizer of the loss  function $\hL := (\Err)^2$ \eqref{eq:loss}. The quantity 
\begin{equation}
\label{eq:gener}
(\Err_{\mathrm{gen}})^2
=
\hL\left(
\hN_{N}
\right)
- 
\hL\left(
\hN
\right),
\end{equation}
with $\hL\left(
\hN_{N}
\right)$ denoting the empirical loss \eqref{eq:empirical-loss}, is referred to as the \emph{generalization error}. It provides a measure of how far the empirical optimizer $\hN_{N}$ of $\hL_N$ is from being an optimizer of $\hL$. The generalization error has been studied in detail for conventional neural networks defined on \emph{finite-dimensional} spaces. In the present section, we consider the extension of these results to the setting of DeepONets, which are defined on \emph{infinite-dimensional} spaces.

Let us make the following {assumptions}:
\begin{assumption}[Boundedness] \label{ass:boundedness}
We assume that there exists a function $\Psi: L^2(D) \to [0,\infty)$, $(u) \mapsto \Psi(u)$, such that
\[
|\G(u)(y)| \le \Psi(u), 
\quad
\sup_{\theta \in [-B,B]^{d_\theta}} 
|\cN_\theta(u)(y)| \le \Psi(u),
\]
for all $u\in L^2(D)$, $y\in U$, and there exist constants $C, \kappa > 0$, such that
\begin{align} \label{eq:boundedness}
\Psi(u)
\le
C(1 + \Vert u \Vert_{L^2})^\kappa.
\end{align}
\end{assumption}

\begin{assumption}[Lipschitz continuity] \label{ass:lipschitz}
There exists a function $\Phi: L^2(D) \to [0,\infty)$, $u \mapsto \Phi(u)$, such that
\[
|\cN_\theta(u)(y) - \cN_{\theta'}(u)(y)|
\le
\Phi(u)\Vert\theta - \theta'\Vert_{\ell^\infty},
\]
for all $u\in L^2(D)$, $y\in U$, and 
\[
\Phi(u) 
\le 
C(1+\vert u \Vert_{L^2})^\kappa,
\]
for the same constants $C,\kappa > 0$ as in \eqref{eq:boundedness}.
\end{assumption}
Note that both the boundedness and Lipschitz continuity assumptions are satisfied for the concrete examples of operators $\G$ considered in section \ref{sec:4}. 

To simplify the notation for the following estimates, we denote by $Z_j = (U_j,Y_j)$ a (joint) random variable on $L^2(D) \times U$, and by a slight abuse of notation we write $\G(Z_j) = \G(U_j)(Y_j)$, $\cN_\theta(Z_j) = \cN_\theta(U_j)(Y_j)$. Recall that with the random choice \eqref{eq:random-choice}, the $Z_j$ are iid random variables.
We denote
\begin{align}
S_\theta^N
=
\frac{1}{N}
\sum_{j=1}^N 
|\G(Z_j) - \cN_\theta(Z_j)|^2.
\end{align}
We have the following bound on the generalization error \eqref{eq:gener},
\begin{theorem} \label{thm:generalization-err}
Let $\hN$ and $\hN_N$ denote the minimizer of the loss \eqref{eq:loss} and empirical loss \eqref{eq:empirical}, respectively. If the assumptions \ref{ass:boundedness} and \ref{ass:lipschitz} hold, then the generalization error \eqref{eq:gener} is bounded by
\begin{align}
\label{eq:generbd}
\E\left[
\left|
\hL(\hN_N) - \hL(\hN)
\right|
\right]
\le
\frac{C}{\sqrt{N}}
\left(
1 + C d_{\theta} \log(CB\sqrt{N})^{2\kappa+1/2}
\right),
\end{align}
where $C = C(\mu, \Psi, \Phi)$ is a constant independent of $B$, $d_{\theta}$ and $N$ and $\kappa$ is specified in \eqref{eq:boundedness}.
\end{theorem}

The proof of Theorem \ref{thm:generalization-err} relies on a series of technical lemmas and is detailed in Appendix \ref{app:pf5}. 
\begin{remark}
The generalization error bound \eqref{eq:generbd} shows that even if the underlying approximation problem is in infinite dimensions, the DeepONet generalization error \eqref{eq:gener}, at worst, scales (up to a log) with the standard Monte Carlo scaling of $1/\sqrt{N}$ in the number of samples $N$ from the infinite dimensional input space. Thus, one can reduce the generalization error by increasing the number of samples, even in this infinite dimensional setting. In particular, the curse of dimensionality is also overcome for the generalization error. 
\end{remark}
\begin{remark}
A careful observation of the bound \eqref{eq:generbd} on the generalization error \eqref{eq:gener} reveals that the bound depends explicitly on the number of parameters $d_\theta$ of the DeepONet \eqref{eq:donet}. As we have seen in the previous sections, one might need a large number of parameters in order to reduce the approximation error \eqref{eq:approxerr} to the desired tolerance. Thus, this estimate, like all estimates based on covering number and other statistical learning theory techniques \cite{CS1}, applies in the underparametrized regime i.e $N \gg d_{\theta}$. 
\end{remark}
\begin{remark}
The bound \eqref{eq:generbd} can also blow up if the bound on weights $B \to \infty$. However, we note that this blowup is modulated by a log and the weights can indeed be unbounded asymptotically i.e $B \sim e^{N^r}$ for some $r < 1/2$ will still result in a decay of the error with increasing $N$. Given that such an exponential blowup may not occur in practice, it is reasonable to assume that the explicit dependence on the weight bounds may not affect the decay rate of the generalization error. This also holds if the bound on weights $B$, blows up as $B \sim d_{\theta}^{r_{\theta}}$, for some $r_\theta > 0$. Given the $\log$-term in \eqref{eq:generbd}, this blow up of weights will only translate into a $d_{\theta} \log(d_{\theta})$-dependence of the generalization error. As long as one is in the under-parametrized regime, this dependence does not affect the overall decay of the generalization error as $N \to \infty$.
\end{remark}
\section{Discussion}
\label{sec:6}
Operators, mapping infinite-dimensional Banach spaces, arise naturally in the study of differential equations. \emph{Learning} such operators from data using neural networks can be very challenging on account of the underlying infinite-dimensional setup. In this paper, we analyze a neural network architecture termed \emph{DeepONets} for approximating such operators. DeepONets are a recent extension \cite{deeponets} of operator networks, first considered in \cite{ChenChen1995} and have been recently successfully applied in many different contexts \cite{deeponets,dnet2,dnet3,dnet4} and references therein. However, apart from the universal approximation result of \cite{ChenChen1995} and its extension to DeepONets in \cite{deeponets}, very few rigorous results for DeepONets are available currently. In particular, given the underlying infinite-dimensional setup, it is essential to demonstrate that DeepONets can overcome the \emph{curse of dimensionality}, associated with the $\infty$-dimensional input-to-output mapping (see Definition \ref{def:cod}).

Our main aim in this article has been to analyze a form of DeepONets \eqref{eq:donet} and prove rigorous bounds on the error \eqref{eq:approxerr}, incurred by a DeepONet in approximating a nonlinear operator $\G$, with an underlying measure $\mu$ (see Definition \ref{def:data}). To this end, we have presented the following results in the paper:
\begin{itemize}
    \item We extend the universal approximation theorem of \cite{ChenChen1995} to Theorem \ref{thm:uni}, where we show that given any \emph{measurable} operator $\G: X \to Y$, for $X=C(D)$, $Y=L^2(U)$, with respect to an underlying measure $\mu \in \P(X)$, there exists a DeepONet of the form \eqref{eq:donet}, which can approximate it to arbitrary accuracy. In particular, we remove the continuity (of $\G$) and compactness (of subsets of $X$) assumptions of \cite{ChenChen1995} and pave the way for the application of DeepONets to approximate operators that arise in applications of PDEs to fields such as hypersonics \cite{dnet2}.  
    \item We provide an upper bound (Theorem \ref{thm:err-upper-bound}) on the DeepONet error \eqref{eq:approxerr} by decomposing it into three parts, i.e., an encoding error \eqref{eq:encoding} stemming from the encoder $\cE$, an approximation error that arises from the approximator neural network $\cA$ that maps between finite-dimensional spaces and a reconstruction error \eqref{eq:decoding}, corresponding to the trunk net induced affine reconstructor $\cR$ \eqref{eq:reconstruction}. 
    \item In Theorem \ref{thm:rec-opt}, we prove \emph{lower bounds} on the reconstruction error \eqref{eq:decoding} by utilizing \emph{optimal errors} for projections on finite dimensional affine subspaces of separable Hilbert spaces (Theorem \ref{thm:opt-proj}). This allows us in Theorems \ref{thm:err-upper-bound} and \ref{thm:lowerbd} to prove \emph{two-sided bounds} on the DeepONet error \eqref{eq:approxerr}.  In particular, the lower bound is explicitly given in terms of the decay of the eigenvalues of the covariance operator \eqref{eq:cov}, associated with the push-forward measure $\cG_\#{\mu}$ \eqref{eq:donetlbd}. Moreover, this construction also allows us to infer the number of trunk nets $p$ and that these trunk nets should approximate the eigenfunctions of the covariance operator in-order to obtain optimal reconstruction errors. Furthermore, we also provide bounds \eqref{eq:smoothness-NN} on the reconstruction error that leverage the Sobolev regularity of the image of the nonlinear operator $\G$. 
    \item To control the encoding error \eqref{eq:encoding} corresponding to the encoder $\cE$, which is a pointwise evaluation of the input at $m$ \emph{sensor locations}, we construct a \emph{decoder} $\cD$ (approximate inverse of the encoder) \eqref{eq:encode-general}. We show in Theorem \ref{thm:random-enc} that sensors chosen \emph{at random} on the underlying domain $D$ suffice to provide an \emph{almost optimal} (optimal modulo a $\log$) bound on the encoding error. This further highlights the fact that DeepONets allow for a general approximation framework, i.e., no explicit information is needed about the location of sensor points and they can be chosen randomly. 
\item Finally, estimating the approximation error \eqref{eq:approximation} reduces to deriving bounds on a neural network $\cA$ that maps one finite (but possibly very-high) dimensional space to another. Hence, standard neural network approximation results such as from \cite{Yarotsky2017} can be applied. In particular, approximation results for holomorphic maps, such as those derived in \cite{SchwabZech2019,OSZ2019,OSZ2020} are important in this context. 

\end{itemize}
The above results provide a workflow for deriving bounds on DeepONet approximation for general nonlinear operators $\G$ with underlying measures $\mu$, which is outlined at the beginning of Section \ref{sec:4}. We illustrate this program with very general bounded linear operators and with the following four concrete examples of nonlinear operators, each corresponding to a differential equation that serves as a model for a large class of related problems:
\begin{itemize}
\item First, we consider the forced gravity pendulum \eqref{eq:pendulum0}, with the operator $\G$ mapping the forcing term into the solution of the ODE and a parametrized random field defining the underlying measure. We bound the reconstruction error by smoothness of the image of $\G$ and the encoding error decays exponentially in the number of sensors on account of the decay of eigenvalues of the covariance operator associated with the underlying measure. The approximation error is bounded by showing that the operator allows for a complex analytic extension. Combining these ingredients in Theorem \ref{thm:pendulum-error}, we prove an error bound \eqref{eq:pendulum-err} on the DeepONet approximation error. In particular, it is shown in \eqref{eq:pendulum-comp} that the size of the DeepONet (total number of parameters in the trunk and branch nets) only grow sub-algebraically with respect to the error tolerance. 
\item As a second example, we consider the standard elliptic PDE \eqref{eq:random-elliptic} with a variable coefficient $a$, which, for instance, arises in the modeling of groundwater Darcy flow. The nonlinear operator $\G$ in this case maps the coefficient to the solution of the elliptic PDE and the underlying measure is the law of the random field \eqref{eq:law-elliptic}. Again, we utilize the spectral properties of the underlying covariance operator, smoothness of the image of $\G$ and holomorphy of an associated map to prove the bound \eqref{eq:elliptic-err} on the DeepONet approximation error. As in the case of the forced gravity pendulum, we show that the size of the DeepONet only grows \emph{sub-algebraically} with decreasing error tolerance. 
\item As a third example, we consider the Allen-Cahn equation \eqref{eq:allen-cahn} that models phase transitions, as a model for nonlinear parabolic PDEs of the reaction-diffusion type. The operator $\G$ maps the initial data into the solution (at a given time) for the Allen-Cahn equation. In this case, no holomorphic extension of the underlying mapping is available. Nevertheless, we use a novel strategy to \emph{emulate} a convergent finite difference scheme \eqref{eq:allen-cahn-fd} by neural networks and derive an upper bound \eqref{eq:ac-err} on the DeepONet approximation error. In particular, the size of the DeepONet only grows polynomially \eqref{eq:ac-comp} with respect to decreasing error tolerance.   
\item In the final example, we consider a scalar conservation law \eqref{eq:conslaw} as a prototype for nonlinear hyperbolic PDEs. In this case, the operator $\G$ is defined as the mapping between the initial data and the entropy solution of the conservation law at a given time. This example differs from the other three in two crucial respects. First, the underlying solutions are discontinuous and thus a pointwise evaluation based encoder cannot be used and is replaced by local averages \eqref{eq:encoder-avg}. Second, the operator $\G$ is contractive (hence Lipschitz continuous) in $L^1$. Thus, the usual Hilbert space setup, will a priori, will lead to sub-optimal error bounds. Hence, we adapt our theory to an $L^1$-Banach space version and are able to prove the upper bound \eqref{eq:err-conslaw} on the resulting DeepONet approximation error \eqref{eq:L1approx}. This bound also allows us to conclude in \eqref{eq:scl-comp} that the size of the DeepONet only grows polynomially with respect to the error tolerance. Moreover, we also construct an explicit example to show a lower bound in Theorem \ref{thm:L1-lower} on the DeepONet error in this case. This shows that the derived upper bounds are almost optimal for scalar conservation laws.
\end{itemize}
Hence, in all the four concrete examples which cover a large spectrum of nonlinear operators arising in the study of differential equations, \emph{we prove that there exist DeepONets, which break the curse of dimensionality} in approximating the underlying operators. These examples and the underlying abstract theory provide a comprehensive study of the approximation error \eqref{eq:approxerr} for DeepONets. 

Finally, we also study the generalization error \eqref{eq:gener} that arises from replacing the loss function (population risk) \eqref{eq:loss} with its sampled version, the so-called empirical loss (empirical risk) \eqref{eq:empirical-loss} that is used during training. Under very general assumptions on the underlying operator $\G$ and the approximating branch and trunk nets, we apply covering number estimates to prove the bound \eqref{eq:generbd} on the generalization error. In spite of the overall infinite-dimensional setup, this bound shows that the generalization errors decays (up to a log) with the reciprocal of the square-root of the total number of samples in the input function space, thus also overcoming the curse of dimensionality in this respect. 

Thus, the analysis and results of this paper clearly prove that DeepONets can efficiently approximate operators in very general settings that include many examples of PDEs. The analysis also reveals some reasons for why DeepONets can work so well in practical applications and as building blocks for complex multi-physics systems, such as in the DeepM{\&}Mnet architectures introduced in \cite{dnet2,dnet3}. The main reason is the generality and flexibility of DeepONets. In particular \emph{no a priori information about the underlying measure and operator are necessary at an algorithmic level}, apart from being able to sample from the underlying measure. Given our analysis, one can even use a small number of \emph{randomly distributed sensors} to achieve almost optimal encoding error. Similarly, a small number of trunk-nets, with a very general neural network architecture, will be able to learn the eigenfunctions of the underlying covariance operator such that an optimal reconstruction error is attained with the resulting affine reconstructor \eqref{eq:reconstruction}. Finally, the branch nets can be trained simultaneously to minimize the approximation error. 

At this point, we contrast DeepONets with other recently proposed frameworks for operator learning. In particular, we focus on a recent paper \cite{STU3}, where the authors present an operator learning framework based on a principal component analysis (PCA) autoencoder for both the encoding and reconstruction steps. Thus, in that approach, one has to explicitly construct an approximate eigenbasis of the empirical covariance operator for the input measure and its push-forward with respect to the underlying operator. Neural networks are only used to approximate the operator on PCA projected finite-dimensional spaces. In contrast, DeepONets do not require any explicit knowledge of the covariance operator. In fact, our analysis shows that DeepONets \emph{implicitly} and \emph{concurrently} learn a suitable basis in output space along with an approximation of the projected operator. Although, many elements of our analysis overlap with that of \cite{STU3}, we provide significantly more general results, including the alleviation of the curse of dimensionality for DeepONets. Moreover, our analysis can be readily extended to the framework of \cite{STU3} to prove the mitigation of the curse of dimensionality in that context.  

It is also instructive to compare our error bounds with the numerical results of \cite{deeponets,dnet2,dnet3,dnet4}. In particular for the forced gravity pendulum, the authors of \cite{deeponets} considered a Gaussian random field with covariance kernel, similar to \eqref{eq:quadraticexp-kernel}, as the underlying measure and observed an exponential decay of the test error with respect to the number of sensors (see Figure \ref{fig:quadrature} (B) for a simpler example). Indeed, this observation is consistent with both the exponential decay of the encoding error \eqref{eq:enerrg1} and the spectral decay of the overall error \eqref{eq:pendulum-err}, as long as the correlation scale is resolved, i.e. $m \sim 1/\ell$, which is also observed in the numerical experiments of \cite{deeponets}. On the other hand, the decay of the generalization error with respect to the number of training samples, both in examples considered in \cite{deeponets} as well as in figure \ref{fig:quadrature} (C,D) shows a very interesting behavior. For a small number of samples, the training error decays exponentially enabling fast training for DeepONets. Only for a relatively large number of training samples, the generalization error decays algebraically with respect to the number of training samples, at a rate consistent with the error bound \eqref{eq:generbd}. This bi-phasic behavior of the generalization error is certainly not explained by the bound \eqref{eq:generbd} and will be a topic of future work.

The methods and results of this paper can be extended in different directions. We can apply the abstract framework presented in section \ref{sec:3} to other examples of differential equations, for instance the Navier-Stokes equations of fluid dynamics. Although we showed that the curse of dimensionality is broken by DeepONets for all the examples that we consider, it is unclear if our bounds on computational complexity of DeepONets are sharp. We show almost sharpness for scalar conservation laws and given the sub-algebraic decay of DeepONet size, we believe that the results for the pendulum and elliptic PDE are also close to optimal. However, there is certainly room for a sharp estimate for the Allen-Cahn equation. Finally, one can readily extend DeepONets, for instance by endowing them with a recurrent structure, to approximate the whole time-series for a time-parametrized operator, such as the solution operators of time-dependent PDEs. Extending the rigorous results of this paper to cover this recurrent case will also be considered in the future. Another possible avenue of future work is the extension of the approximation results in this paper to the case of multiple nonlinear operators (MNOs), already considered in \cite{backchen}, where the authors prove an universal approximation property, similar to \cite{ChenChen1995} for these operators.

\section*{Acknowledgements}

The research of Samuel Lanthaler and Siddhartha Mishra is partially supported by the European Research Council Consolidator grant ERC-CoG 770880 COMANFLO. George Karniadakis acknowledges partial support from MURI-AFOSR FA9550-20-1-0358: "Learning and Meta-Learning of Partial Differential Equations via Physics-Informed Neural Networks: Theory, Algorithms, and Applications"

\bibliographystyle{agsm}
\bibliography{main}

\clearpage

\appendix
\section{Notation for Standard Fourier basis}  
\label{app:Fourier}
In several instances in this paper, we employ the following ``standard'' real Fourier basis $\{\fb_{{k}}\}_{{k}\in \Z^d}$ in $d$ dimensions: For ${k} = (k_1,\dots, k_d)\in \Z^d$, we define
\begin{align}
\fb_{{k}}
:=
C_{{k}}
\begin{cases}
\cos({k}\cdot {x}),
& (k_1 \ge 0), \\
\sin({k}\cdot {x}),
& (k_1 < 0),
\end{cases}
\end{align}
where the factor $C_k > 0$ ensures that $\fb_{{k}}$ is properly normalized, i.e. that $\Vert \fb_{{k}} \Vert_{L^2(\T^d)} = 1$, or explicitly, 
\[
C_{{k}} = 
\frac{1}{(2\pi)^d}
\begin{cases}
2, & ({k} \ne 0), \\
1, & ({k} = 0).
\end{cases}
\]
We note that the basis $\{\fb_{{k}}\}_{{k}\in \Z^d}$ simply consists of the real and imaginary parts of the complex Fourier basis $\{e^{i{k}\cdot{x}}\}_{{k}\in \Z^d}$. 

On occasion, it will also be convenient to write the standard Fourier basis in the form $\{ \fb_j \}_{j\in \N}$ (indexed by integer $j\in \N$, rather than ${k}\in \Z^d$). In this case, we identify 
\[
\fb_j({x}) := \fb_{\enum(j)}({x}), \quad (j\in \N),
\]
where $\enum: \N \to \Z^d$ is a fixed enumeration of $\Z^d$, with the property that $j \mapsto |\enum(j)|_\infty$ is monotonically increasing, i.e. such that $j\le j'$ implies that $|\enum(j)|_\infty \le |\enum(j')|_\infty$, where
\begin{align}
|{k}|_\infty :=  \max_{\ell=1,\dots, d} |k_\ell|, \quad {k} = (k_1,\dots, k_d)\in \Z^d.
\end{align}

\section{On the definition of Error \eqref{eq:approxerr}}
We need the following lemma in order to conclude that the error \eqref{eq:approxerr} is well-defined on $L^2(D)$, even if the encoder $\cE$ is only well-defined on continuous functions.
\begin{lemma}
\label{lem:app1}
Let $\cE: C(D) \to \R^m$ denote the point-wise encoder $u \mapsto \cE(u) = (u(x_1),\dots, u(x_m))$, for some $x_1,\dots, x_m\in D$. There exists a Borel measurable extension $\bar{\cE}: L^2(D) \to \R^m$, such that $\bar{\cE}(u) = \cE(u)$ for any $u\in C(D)\cap L^2(D)$.
\end{lemma}

\begin{proof}
It suffices to consider the case $m=1$. In this case, we note that $\cE(u) = \limsup_{k\to \infty} \cE_k(u)$ for any $u\in C(D)$, where 
\[
\cE_k: L^2(D) \to \R,
\quad
\cE_k(u)
=
\fint_{B_{1/k}(x_1)} u(y) \, dy,
\]
is continuous for any $k \in \N$. In particular, it follows that for $u\in L^2(D)$, the functional
\[
\tE: L^2(D) \to [-\infty, \infty],
\quad
\tE(u) = \limsup_{k\to \infty} \cE_k(u),
\]
is Borel measurable. We can now define a measurable extension $\bar{\cE}$ of $\cE$ by e.g. setting
\[
\bar{\cE}(u) = 
\begin{cases}
\tE(u), & (\tE(u) \in \R), \\
0, & (\tE(u) = \pm\infty).
\end{cases}
\]
\end{proof}
\section{Proofs of Results in Section \ref{sec:3}}
\label{app:pf3}
\subsection{Proof of Theorem \ref{thm:uni}}
\label{app:pf31}

The proof of the universal approximation theorem will be based on an application of the following well-known version of Lusin's theorem (see e.g. \cite{bogachev2007}, Thm. 7.1.13), which we will state for the special case of probability measures on Polish spaces, below:
\begin{theorem}[Lusin's theorem]
\label{thm:lusin}
Let $X,Y$ be separable and complete metric spaces. Let $\mu\in \cP(X)$ be a probability measure, and let $\G: X \to Y$ be a Borel measurable mapping. Then for any $\epsilon > 0$, there exists a compact set $K \subset X$, such that $\mu(X\setminus K) < \epsilon$, and such that the restriction $\G|_{K}: K \to Y$ is continuous.
\end{theorem}

In addition to Lusin's theorem, the following clipping lemma will be used in the proof of Theorem \ref{thm:uni}:

\begin{lemma}[Clipping lemma] \label{lem:clipping}
Let $\epsilon > 0$, and fix $0<R_1 < R_2$. There exists a ReLU neural network $\gamma: \R^p \to \R^p$, such that 
\begin{align*}
\begin{cases}
\Vert\gamma(x) - x\Vert_{\ell^2} < \epsilon, &\text{if } \Vert x \Vert_{\ell^2} \le R_1, \\
\Vert \gamma(x) \Vert_{\ell^2} \le R_2,  &\forall \, x\in \R^p.
\end{cases}
\end{align*}
\end{lemma}

\begin{proof}
Without loss of generality, we may assume that $\epsilon \le R_2-R_1$. We first note that $x \mapsto \sigma_{R_1}(x) := \min(\max(x,-R_1),R_1)$ maps $\R^p \to [-R_1,R_1]^p$, and $\sigma_{R_1}$ can be represented exactly by a (two-layer) ReLU neural network. Furthermore, for any $x\in [-R_1,R_1]^p$, we have $\sigma_{R_1}(x) = x$. Define a continuous function $\phi: \R^p \to \R^p$ by
\[
\phi(x) :=
\begin{cases}
x, & (\Vert x \Vert_{\ell^2} \le R_1), \\
R_1\dfrac{x}{\Vert x \Vert_{\ell^2}}, &(\Vert x \Vert_{\ell^2} > R_1).
\end{cases}
\]
By the universal approximation theorem, there exists a ReLU network $\tilde{\gamma}: [-R_1,R_1]^p \to \R^p$, such that 
\[
\Vert \tilde{\gamma}(x) - \phi(x) \Vert_{\ell^2} < \epsilon, \quad \forall \, x \in [-R_1,R_1]^p.
\]
Define now $\gamma:\R^p \to \R^p$, by $\gamma(x) := \tilde{\gamma}\circ \sigma_{R_1}(x)$. Then, for $\Vert x\Vert_{\ell^2} \le R_1$, we have
\[
\Vert {\gamma}(x) - x \Vert_{\ell^2}
=
\Vert \tilde{\gamma}(x) - \phi(x) \Vert_{\ell^2} < \epsilon,
\]
and 
\begin{align*}
\sup_{x\in \R^p}
\Vert {\gamma}(x) \Vert_{\ell^2}
&=
\sup_{\xi\in [-R_1,R_1]^p}
\Vert \tilde{\gamma}(\xi) \Vert_{\ell^2}
\\
&\le
\sup_{\xi\in [-R_1,R_1]^p}
\left\{
\Vert \tilde{\gamma}(\xi) - \phi(\xi) \Vert_{\ell^2} + \Vert \phi(\xi) \Vert_{\ell^2}
\right\}
\\
&\le
\epsilon + R_1 \le R_2.
\end{align*}
\end{proof}

Using the above clipping lemma, we can now prove the universal approximation theorem.

\begin{proof}[Proof of Theorem \ref{thm:uni}]
Let $\epsilon > 0$ be given. By assumption $\G: C(D) \to L^2(U)$ is a measurable mapping, $\G\in L^2(\mu)$, where $\mu \in \P(C(D))$. We have to show that there exists a DeepONet $\cN = \cR\circ \cA\circ \cE$, such that 
\[
\Vert \G - \cN \Vert_{L^2(\mu)} < \epsilon.
\]
Given $M>0$, define $\G_M$ by clipping $\G$ at size $M>0$, i.e. define 
\[
\G_M(u)
:=
\begin{cases}
\G(u), & (\Vert \G(u) \Vert_{L^2(U)} \le M), \\
M\dfrac{\, G(u)}{\Vert \G(u) \Vert_{L^2(U)}}, & (\Vert \G(u) \Vert_{L^2(U)} > M).
\end{cases}
\]
so that $\Vert \G_M(u) \Vert_{L^2(U)} \le M$ for all $u\in C(D)$. Then
\begin{align} \label{eq:decmp}
\Vert \G - \cN \Vert_{L^2(\mu)} 
\le 
\Vert \G - \G_M \Vert_{L^2(\mu)} 
+
\Vert \G_M - \cN \Vert_{L^2(\mu)},
\end{align}
and the first term goes to $0$ by the dominated convergence theorem and the fact that $\G_M(u) \to \G(u)$ pointwise, as $M\to \infty$. We can thus choose $M>\epsilon$, such that 
\begin{align} \label{eq:nnclip}
\Vert \G - \G_M \Vert_{L^2(\mu)}  < \epsilon / 3.
\end{align}
It now remains to approximate $\G_M$ by a suitable DeepONet $\cN$, where $\G_M$ is bounded. 

Next, we note that $C(D)$ and $L^2(U)$ are Polish spaces (separable, complete metric spaces). This allows us to invoke Lusin's theorem, Theorem \ref{thm:lusin}, which shows that there exists a compact set $K \subset C(D)$, such that the restriction $\G_M|_{K}$ of $\G_M$ to $K$, 
\[
\G_M|_{K}: K \to L^2(U),
\]
is continuous, and such that $\mu(C(D)\setminus K) < (\epsilon / 9M)^2$. 

Next, fix an orthonormal basis $\phi_1,\phi_2,\phi_3,\dots \subset L^2(U)$ consisting of continuous functions. For $\kappa\in \N$, let $P_\kappa: L^2(U) \to C(U)$ denote the projection onto $\phi_1,\dots, \phi_\kappa$, i.e.
\[
P_\kappa(v) = \sum_{k=1}^\kappa \langle v, \phi_k\rangle \phi_k.
\]
We note that $P_\kappa: L^2(U) \to C(U)$ is continuous for any fixed $\kappa$. Let now $K' := \G_M(K)$ denote the image of the compact set $K$ under $\G_M$. Since $\G_M|_K: K \to L^2(U)$ is continuous, $K'$ is compact as a subset of $L^2(U)$. By the compactness of $K'$, there exists $\kappa\in \N$, such that $\max_{v\in K'} \Vert v - P_\kappa(v) \Vert < \epsilon/6$. We conclude that the composition $\tG := P_\kappa \circ \G_M: K \to C(U)$, with $K\subset C(D)$ a compact subset, is continuous, and 
\begin{gather} \label{eq:strict}
\begin{aligned} 
\max_{u\in K}
 \left\Vert \G_M(u) - \tG(u) \right\Vert_{L^2(U)}
&=
\max_{u\in K}
 \left\Vert \G_M(u) - P_\kappa\circ \G_M(u) \right\Vert_{L^2(U)}^2
\\
&=
\max_{v\in K'} \left\Vert v - P_\kappa(v) \right\Vert_{L^2(U)}
<
\epsilon/6.
\end{aligned}
\end{gather}

We can apply the universal approximation theorem of continuous operators on compact subsets of \cite{ChenChen1995} to the continuous mapping $\tG: K \subset C(D) \to C(U)$ to conclude that there exists an operator network $\tN$ with a single hidden layer in the approximator network and a single hidden layer in the trunk network (and with $\tr_0 \equiv 0$), such that
\begin{align*} 
\sup_{u\in K}
\Vert
\tG(u)
- 
\tN(u)
\Vert_{L^2(U)}
< \epsilon / 12.
\end{align*}
Note that this implies in particular that 
\begin{align*}
\Vert \tN(u) \Vert_{L^2(U)} 
&\le 
\Vert \G_M(u) \Vert_{L^2(U)} 
+
\Vert \G_M(u) - \tG(u) \Vert_{L^2(U)}
+
\Vert \tG(u) - \tN(u) \Vert_{L^2(U)} 
\\
&\le 
M + \epsilon/6 + \epsilon / 6
< 2M,
\end{align*}
for all $u\in K$. We note that after suitably modifying\footnote{If $A\in \R^{p\times p}$ is an invertible matrix, then the transformed trunk and branch nets $\hat{\tau} := A\cdot \tau$ and $\hat{\beta} := A^{-T}\cdot \beta$ represent the same DeepONet, i.e. we have $\sum_{k=1}^p \hat{\beta}_k(u) \hat{\tau}_k(y) = \sum_{k=1}^p {\beta}_k(u) {\tau}_k(y)$ for all $u,y$.} the (linear) output layers of the branch and trunk nets of $\tN$, we can write $\tN$ in the form  
\[
\tN(u)(y) = \sum_{k=1}^p \beta_k(u) \tau_k(y),
\]
with \emph{orthonormal} trunk net functions $\{\tau_1,\dots, \tau_p\}\subset L^2(U)$. In particular, we then have
\[
\Vert \tN(u) \Vert_{L^2} = \Vert \beta(u) \Vert_{\ell^2}, \quad \forall \, u\in C(D),
\]
and
\[
\Vert \beta(u) \Vert_{\ell^2} \le M + \epsilon/3 < 2M, \quad \forall \, u\in K.
\]
Applying Lemma \ref{lem:clipping} with $R_1 = M+\epsilon/3$, $R_2 = 2M$, we conclude that there exists a ReLU neural network $\gamma: \R^p \to \R^p$, such that
\[
\begin{cases}
\Vert\gamma(x) - x\Vert_{\ell^2} < \epsilon/12, &\text{if } \Vert x \Vert_{\ell^2} \le R_1, \\
\Vert \gamma(x) \Vert_{\ell^2} \le 2M,  &\forall \, x\in \R^p.
\end{cases}
\]
We now define a ``clipped'' DeepONet $\cN: C(D) \to L^2(U)$, by
\[
\cN(u) := \sum_{k=1}^p \gamma_k(\beta(u)) \tau_k(y).
\]
Then $\cN(u)$ satisfies
\begin{gather} \label{eq:clipped}
\begin{aligned} 
\max_{u\in K}
\Vert \cN(u) - \tG(u) \Vert_{L^2}
&\le
\max_{u\in K}
\Vert \cN(u) - \tN(u) \Vert_{L^2}
+
\max_{u\in K}
\Vert \tN(u) - \tG(u) \Vert_{L^2}
\\
&=
\max_{u\in K}
\Vert \gamma(\beta(u)) - \beta(u) \Vert_{\ell^2} 
+
\max_{u\in K}
\Vert \tN(u) - \tG(u) \Vert_{L^2}
\\
&\le
\max_{\Vert x \Vert_{\ell^2} \le R_1}
\Vert \gamma(x) - x \Vert_{\ell^2} 
+
\max_{u\in K}
\Vert \tN(u) - \tG(u) \Vert_{L^2}
\\
&\le
\epsilon/12
+
\epsilon/12 
= \epsilon/6,
\end{aligned}
\end{gather}
and $\cN(u)$ is bounded from above by
\[
\Vert \cN(u) \Vert_{L^2}
=
\Vert \gamma(\beta(u)) \Vert_{\ell^2}
\le
2M,
\quad \forall\, u\in C(D).
\]
It follows that for this clipped DeepONet $\cN$, we have
\begin{align}
\Vert \G_M - \cN \Vert_{L^2(\mu)}
&\le 
\Vert \G_M - \cN \Vert_{L^2(\mu; K)} 
+ \Vert \G_M \Vert_{L^2(\mu; X\setminus K)} 
+ \Vert \cN \Vert_{L^2(\mu; X\setminus K)}
\notag
\\
&\le 
\Vert \G_M - \cN \Vert_{L^2(\mu; K)} 
+ 3M \mu(X\setminus K)^{1/2}
\notag
\\
&\le
\Vert \G_M - \cN \Vert_{L^2(\mu; K)} + \epsilon/3,
\label{eq:cpctapprox}
\end{align}
and, using \eqref{eq:strict} and \eqref{eq:clipped}, we find
\begin{gather} \label{eq:cpctclipped} 
\begin{aligned} 
\Vert \G_M - \cN \Vert_{L^2(\mu; K)}
&\le
\max_{u\in K}
\left\{
\Vert \G_M(u) - \tG(u) \Vert_{L^2(U)}
+
\Vert \tG(u) - \cN(u) \Vert_{L^2(U)}
\right\}
\\
&< 
\epsilon/3.
\end{aligned}
\end{gather}
Hence, by \eqref{eq:cpctapprox} and \eqref{eq:cpctclipped}, we have
\begin{align}\label{eq:nnapprox}
\Vert \G_M - \cN \Vert_{L^2(\mu)} < 2\epsilon / 3.
\end{align}
Combining \eqref{eq:nnapprox} and \eqref{eq:nnclip}, we conclude that the DeepONet $\cN$ satisfies
\[
\Vert \G - \cN \Vert_{L^2(\mu)}
< 
2\epsilon/3 + \epsilon /3 = \epsilon.
\]

\end{proof}

\subsection{Proof of Theorem \ref{thm:err-upper-bound}}
\label{app:errdec}
\begin{proof}
The upper estimate follows by a suitable decomposition of the difference. We write
\begin{align*}
\cN - \G = \cR \circ \cA \circ \cE - \G
&=
\left[\cR \circ \cA \circ \cE - \cR \circ \cP \circ \G\right] + 
\left[\cR \circ \cP \circ \G - \G\right]
\\
&=
\left[\cR \circ \cA \circ \cE - \cR \circ \cP \circ \G \circ \cD \circ \cE\right] \\
&\quad + 
\left[\cR \circ \cP \circ \G \circ \cD \circ \cE - \cR \circ \cP \circ \G\right]  \\
&\quad +
\left[\cR \circ \cP \circ \G - \G\right]
\\
&=: 
T_1 + T_2 + T_3.
\end{align*}
We can now estimate the norm of each of the three terms, as follows:
\begin{align*}
\Vert T_1 \Vert_{L^2(\mu)}
&=
\left(
\int_X \Vert \cR \circ \cA \circ \cE - \cR \circ \cP \circ \G \circ \cD \circ \cE \Vert^2_{L^2(U)} \, d\mu
\right)^{1/2}
\\
&\le
\Lip(\cR)
\left(
\int_X \Vert \cA \circ \cE - \cP \circ \G \circ \cD \circ \cE \Vert^2_{\ell^2(\R^p)} \, d\mu
\right)^{1/2}
\\
&=
\Lip(\cR)
\left(
\int_{\R^m} \Vert \cA - \cP \circ \G \circ \cD \Vert^2_{\ell(\R^p)} \, d(\cE_\#\mu)
\right)^{1/2}
\\
&=
\Lip(\cR) \Vert \cA - \cP \circ \G \circ \cD \Vert_{L^2(\cE_\#\mu)}.
\end{align*}
In the second line above, we denote
\[
\Lip(\cR) 
= 
\Lip\left(
\cR: 
(\R^p, \Vert \slot \Vert_{\ell^2(\R^p)}) 
\to 
(L^2(U), \Vert \slot \Vert_{L^2(U)})
\right).
\]

For the second term, we obtain
\begin{align*}
\Vert T_2 \Vert_{L^2(\mu)}
&=
\left(
\int_X 
\Vert 
\cR \circ \cP \circ \G \circ \cD \circ \cE - \cR \circ \cP \circ \G
\Vert_{L^2(U)}^2
\, d\mu
\right)^{1/2}
\\
&\le
\Lip(\cR \circ \cP)
\left(
\int_X 
\Vert 
\G \circ \cD \circ \cE - \G
\Vert_{L^2(U)}^2
\, d\mu
\right)^{1/2}
\\
&\le
\Lip(\cR \circ \cP) \Lip_\alpha(\cG)
\left(
\int_X 
\Vert 
\cD \circ \cE - \Id
\Vert_{X}^{2\alpha}
\, d\mu
\right)^{1/2},
\end{align*}
where 
\begin{align*}
\Lip(\cR\circ \cP) &= \Lip(\cR \circ\cP: (L^2(U), \Vert \slot \Vert_{L^2(U)}) \to (L^2(U), \Vert \slot \Vert_{L^2(U)})), \\
\Lip_\alpha(\G) &= \Lip_\alpha(\G: (A\subset X,\Vert \slot \Vert) \to (L^2(U), \Vert \slot \Vert_{L^2(U)})),
\end{align*}
and $\alpha \in (0,1]$. Since $\alpha \in (0,1]$, we can estimate the last term by Jensen's inequality to obtain:
\begin{align*}
\Vert T_2\Vert_{L^2(\mu)}
&\le
\Lip(\cR \circ \cP) \Lip_\alpha(\cG)
\left(
\int_X 
\Vert 
\cD \circ \cE - \Id
\Vert_{X}^{2\alpha}
\, d\mu
\right)^{1/2}
\\
&\le
\Lip(\cR \circ \cP) \Lip_\alpha(\cG)
\left(
\int_X 
\Vert 
\cD \circ \cE - \Id
\Vert_{X}^{2}
\, d\mu
\right)^{\alpha/2}
\\
&=
\Lip(\cR \circ \cP) \Lip_\alpha(\G) \Vert \cD \circ \cE - \Id \Vert_{L^2(\mu)}^\alpha.
\end{align*}

Finally for the third term, we have (by the definition of the push-forward)
\begin{align*}
\Vert T_3 \Vert_{L^2(\mu)}
&=
\left(
\int_X
\Vert \cR \circ \cP \circ \G - \G \Vert^2_{L^2(U)} \, d\mu
\right)^{1/2}
\\
&=
\left(
\int_{L^2(U)}
\Vert \cR \circ \cP - \Id \Vert^2_{L^2(U)} \, d(\G_\#\mu)
\right)^{1/2}
\\
&= 
\Vert \cR \circ \cP - \Id \Vert_{L^2(\G_\#\mu)}.
\end{align*}
\end{proof}

\subsection{Proof of Theorem \ref{thm:opt-linear}}
\label{app:pf32}
The proof of Theorem \ref{thm:opt-linear} is a consequence of the following series of lemmas. 
\begin{lemma} \label{lem:Eproj-equiv}
Given a separable Hilbert space $H$ and a $p$-dimensional subspace $\hat{V}\subset H$, we have
\begin{align} \label{eq:linear}
\Err_{\Proj}(\hat{V})
=
\int_H \Vert v \Vert^2 \, d\mu(v) 
-
\sum_{i=1}^p
\int_H \left|\langle v, \hat{v}_i \rangle\right|^2 \, d\mu(v),
\end{align}
where $\hat{v}_i$, $i=1,\dots, p$ is any orthonormal basis of $\hat{V}$. In particular, minimizing $\Err_{\Proj}$ is equivalent to maximizing
\begin{align} \label{eq:tomaximize}
\sum_{i=1}^p \int_X \left|\langle v, \hat{v}_i \rangle\right|^2 \, d\mu(v)
=
\sum_{i=1}^p \left[\left|\langle \E[v], \hat{v}_i \rangle \right|^2 +  \langle \hat{v}_i, \Gamma \hat{v}_i \rangle \right],
\end{align}
where
\[
\Gamma = 
\int_X 
(v-\E[v]) \otimes (v-\E[v]) \, d\mu(v).
\]
\end{lemma}

\begin{proof}
Let $\hat{v}_i$, $i=1,\dots,p$ be any orthonormal basis of $\hat{V}$. Then 
\begin{align*}
\inf_{\hat{v} \in \hat{V}} \Vert v - \hat{v} \Vert^2
&=
\left\Vert 
v - \sum_{k=1}^p \langle v, \hat{v}_i \rangle \, \hat{v}_i
\right\Vert^2,
\end{align*}
and since $v-\sum_k  \langle v, \hat{v}_i \rangle \, \hat{v}_i$ is perpendicular to all $\hat{v}_i$, we have
\begin{align*}
\Vert v \Vert^2
&=
\left\Vert 
v - \sum_{k=1}^p \langle v, \hat{v}_i \rangle \, \hat{v}_i
\right\Vert^2
+ 
\left\Vert 
\sum_{k=1}^p \langle v, \hat{v}_i \rangle \, \hat{v}_i
\right\Vert^2
\\
&=
\inf_{\hat{v} \in \hat{V}} \Vert v - \hat{v} \Vert^2
+
\sum_{k=1}^p \left|\langle v, \hat{v}_i \rangle \right|^2.
\end{align*}
This is equivalent to \eqref{eq:linear}. Furthermore, for any $i=1, \dots, p$, we have
\begin{align*}
\int_H |\langle v, \hat{v}_i \rangle |^2 \, d\mu(v)
&=
\int_H |\langle v-\E[v], \hat{v}_i \rangle + \langle \E[v], \hat{v}_i \rangle |^2 \, d\mu(v)
\\
&= 
\int_H 
\left\{
|\langle v-\E[v], \hat{v}_i \rangle|^2 + 2\langle v-\E[v], \hat{v}_i \rangle \langle \E[v],\hat{v}_i\rangle  + |\langle \E[v], \hat{v}_i \rangle |^2
\right\}
\, d\mu(v)
\\
&= 
\int_H
\left\{
|\langle v-\E[v], \hat{v}_i \rangle|^2 + |\langle \E[v], \hat{v}_i \rangle |^2
\right\}
\, d\mu(v)
\\
&=
\langle \hat{v}_i, \Gamma \hat{v}_i \rangle + |\langle \E[v], \hat{v}_i \rangle |^2.
\end{align*}
\end{proof}
\begin{lemma} \label{lem:covop-eigendecomp}
The covariance operator $\Gamma$ of a probability measure $\nu \in \P_2(H)$ is a compact, self-adjoint operator. In particular, there exists a discrete set of (unique) eigenvalues $\overline{\lambda}_1 > \overline{\lambda}_2 > \dots$, and orthogonal subspaces $E_1, E_2, \dots$, such that 
\[
H = \bigoplus_{j=1}^\infty E_j,
\quad
\text{and}
\quad
\Gamma = \sum_{j=1}^\infty \overline{\lambda}_j P_{E_j}.
\]
Here $P_{E}: H \to E$ denotes the orthogonal projection onto the subspace $E\subset H$.
\end{lemma}

\begin{proof}
The proof can e.g. be found in \cite[Appendix E,F]{PinelisMolzon2016}.
\end{proof}

As an immediate corollary, we also obtain

\begin{lemma}
For any $\nu \in \P_2(H)$, the operator 
\[
\overline{\Gamma} 
=
\int_{H} (v\otimes v) \, d\nu(v),
\]
is a self-adjoint, compact operator, and hence it also possesses an eigendecomposition as in Lemma \ref{lem:covop-eigendecomp}.
\end{lemma}

\begin{proof}
We can write $\overline{\Gamma} = \Gamma + \E[v] \otimes \E[v]$, where $\E[v] := \int_H v \, d\nu(v)$ is the expected value under $\nu$. It is now immediate that $\overline{\Gamma}$ is self-adjoint, since both $\Gamma$ and $\E[v]\otimes \E[v]$ are self-adjoint. Furthermore, $\overline{\Gamma}$ is also compact, since $\overline{\Gamma} = \Gamma + \E[v] \otimes \E[v]$ is the sum of a compact operator $\Gamma$ and the finite-rank operator $\E[v]\otimes \E[v]$.
\end{proof}

The next lemma shows that the quantity \eqref{eq:tomaximize} to be maximized can be written in an equivalent form, which only involves the orthogonal projection onto $V_p$.

\begin{lemma} \label{lem:Eproj-equiv2}
Let $\overline{\Gamma} = \Gamma + \E[v]\otimes \E[v]$. Let $\phi_k$, $(k\in \N)$, be an orthonormal basis of eigenvectors of $\overline{\Gamma}$ with corresponding eigenvalues $\lambda_k$. Let $\hat{v}_1, \dots, \hat{v}_p$ be any orthonormal basis of $V_p$. Then 
\begin{align} \label{eq:Eproj-equiv2}
\sum_{j=1}^p \langle \hat{v}_i, \overline{\Gamma} \hat{v}_i \rangle 
=
\sum_{k=1}^\infty \lambda_k \Vert P_{V_p} \phi_k \Vert^2.
\end{align}
Here $P_{V_p}: H \to V_p$ denotes the orthogonal projection onto $V_p$.
\end{lemma}

\begin{proof}
Let $u_1, \dots, u_p \in V_p$ be an orthonormal basis $V_p$. Let $\lambda_1 \ge \lambda_2 \ge \dots$ denote the eigenvalues of $\overline{\Gamma}$ with corresponding orthonormal eigenbasis $\phi_1, \phi_2, \dots$. Then
\[
u_j = \sum_{k=1}^\infty \langle u_j, \phi_k \rangle \phi_k,
\]
and hence
\begin{align*}
\sum_{j=1}^p \left\langle u_j, \overline{\Gamma} u_j\right\rangle
&=
\sum_{j=1}^p \left\langle u_j, \overline{\Gamma}  
\sum_{k=1}^\infty \langle u_j, \phi_k \rangle\phi_k\right\rangle
\\
&=
\sum_{j=1}^p \sum_{k=1}^\infty 
\langle u_j, \phi_k \rangle 
\left\langle u_j, \overline{\Gamma}  \phi_k\right\rangle
\\
&=
\sum_{j=1}^p \sum_{k=1}^\infty 
\lambda_k \left|\langle u_j, \phi_k \rangle \right|^2
\\
&=
 \sum_{k=1}^\infty 
\lambda_k \left( \sum_{j=1}^p \left|\langle u_j, \phi_k \rangle \right|^2 \right)
\\
&=
 \sum_{k=1}^\infty 
\lambda_k \left\Vert P_{V_p} \phi_k \right\Vert^2.
\end{align*}
\end{proof}

\begin{lemma} \label{lem:proj-minimizer}
If $V_p$ is a minimizing subspace, i.e. a subspace which minimizes $\Err_{\Proj}$, then 
\begin{align} \label{eq:proj-minimizer}
\bigoplus_{j=1}^{n-1} E_j
\subset
V_p 
\subset 
\bigoplus_{j=1}^n E_j,
\end{align}
where $n$ is chosen such that 
\[
\dim\Big(
\bigoplus_{j=1}^{n-1} E_j
\Big)
=
\sum_{j=1}^{n-1} \dim(E_j)
< 
p
\le 
\sum_{j=1}^n \dim(E_j)
=
\dim\Big(
\bigoplus_{j=1}^n E_j
\Big)
.
\]
Here, the $E_j$ denote the eigenspaces of
\[
\overline{\Gamma} = \Gamma + \E[v]\otimes\E[v] = \sum_{j=1}^\infty \overline{\lambda}_j P_{E_j},
\]
corresponding to the (distinct) eigenvalues $\overline{\lambda}_1 > \overline{\lambda}_2 > \overline{\lambda}_3 > \dots$ of $\overline{\Gamma}$.
\end{lemma}

\begin{proof}
Let $V_p\subset H$ be a minimizer of $\Err_{\Proj}$ with respect to $\nu$ of dimension $\dim(V_p) = p$. For each eigenspace $E_j$ of $\overline{\Gamma}$, denote 
\[
p_j = \dim(E_j),
\]
and define $p_0 := 0$. Now, we choose an orthonormal eigenbasis $\phi_1, \phi_2, \dots$ of $\overline{\Gamma}$, with corresponding eigenvalues $\lambda_1 \ge \lambda_2 \ge \dots$, such that 
\[
E_j = \Span(\phi_{p_{j-1}+1}, \dots, \phi_{p_j}), \quad (j\in \N).
\]
Note that the eigenvalues thus satisfy
\[
\underbrace{
\lambda_1 = \dots = \lambda_{p_1}
}_{=\overline{\lambda}_1}
>
\underbrace{
\lambda_{p_1 + 1} = \dots = \lambda_{p_2}
}_{=\overline{\lambda}_2}
> 
\dots
\]
By Lemma \ref{lem:Eproj-equiv} and \ref{lem:Eproj-equiv2}, $V_p$ is a maximizer of the mapping
\begin{align*}
\hat{V} 
\mapsto 
\sum_{k=1}^\infty 
\lambda_k \left\Vert P_{\hat{V}} \phi_k \right\Vert^2,
\end{align*}
among all $p$-dimensional subspaces $\hat{V}\subset X$. Given a $p$-dimensional subspace $\hat{V}\subset X$, let $u_1, \dots, u_p$ be an orthonormal basis of $\hat{V}$. We note that
\begin{align*}
 \sum_{k=1}^\infty 
\left\Vert P_{\hat{V}} \phi_k \right\Vert^2
= 
 \sum_{k=1}^\infty 
\sum_{j=1}^p \left|\langle u_j, \phi_k \rangle\right|^2
=
\sum_{j=1}^p \Vert u_j \Vert^2
= p,
\end{align*}
And for all $k\in \N$, we have
\[
0 \le \left\Vert P_{\hat{V}} \phi_k \right\Vert^2 \le 1.
\]
Thus, for any $p$-dimensional subspace, the coefficients
\[
\alpha_k = \alpha_k(\hat{V}) := \Vert P_{\hat{V}} \phi_k \Vert^2, \quad (k \in \N),
\]
belong to the set
\[
\mathcal{A}_p :=
\set{
\alpha = (\alpha_k)_{k\in \N} 
}{
\alpha_k \in [0,1] \, \forall \, k\in \N, \; \sum_{k=1}^\infty \alpha_k = p
}.
\]
And we are interested in the maximizer $\hat{V} = V_p$ of
\[
(\alpha_k(\hat{V}))_{k\in \N} 
\mapsto 
\sum_{k=1}^\infty \lambda_k \alpha_k(\hat{V}).
\]

We now make the following claim:

\begin{claim}
For any $(\alpha_k)_{k\in \N} \in \mathcal{A}_p$, we have
\begin{align} \label{eq:claim-upperbound}
\sum_{k=1}^\infty \lambda_k \alpha_k \le \sum_{k=1}^p \lambda_k,
\end{align}
with equality, if and only if 
\[
\alpha_k
=
\begin{cases}
1, & \left( \text{if } k \le \sum_{j=1}^{n-1} p_j \right), \\
0, & \left( \text{if } k > \sum_{j=1}^n p_j \right),
\end{cases}
\]
where $n$ is chosen such that 
\[
\sum_{j=1}^{n-1} p_j \le p < \sum_{j=1}^n p_j.
\]
\end{claim}
Before proving the above claim, we show that it implies \eqref{eq:proj-minimizer}. Indeed, for the subspace $\hat{V} = \Span(\phi_1, \dots, \phi_p)$, we clearly have
\[
\alpha_k(\hat{V}) = \Vert P_{\hat{V}} \phi_k \Vert^2
=
\begin{cases}
1, & (k\le p), \\
0, & (k > p),
\end{cases}
\]
and hence
\[
\sum_{k=1}^\infty \lambda_k \alpha_k(\hat{V}) = \sum_{k=1}^p \lambda_k,
\]
achieves the upper bound in \eqref{eq:claim-upperbound}. If $V_p$ is another optimizer, then we must thus also have 
\[
\sum_{k=1}^\infty  \lambda_k \alpha_k(V_p) = \sum_{k=1}^p \lambda_k.
\]
The claim then implies that
\[
\Vert P_{V_p} \phi_k \Vert^2 = 
\begin{cases}
1, & \left( \text{if } k \le \sum_{j=1}^{n-1} p_j \right), \\
0, & \left( \text{if } k > \sum_{j=1}^n p_j \right).
\end{cases}
\]
The latter is equivalent to the statement that
\[
\phi_k \in V_p, \; \text{for } k \le \sum_{j=1}^{n-1} p_j,
\quad
\text{and}
\quad
\phi_k \perp V_p, \;\text{for } k > \sum_{j=1}^n p_j,
\]
i.e. that
\[
\bigoplus_{j=1}^{n-1} E_j \subset V_p, 
\quad
\text{and}
\quad 
\bigoplus_{j=n+1}^\infty E_j \perp V_p.
\]
Thus, assuming the claim \eqref{eq:claim-upperbound}, it follows that for any optimal subspace $V_p$, we must have
\[
\bigoplus_{j=1}^{n-1} E_j \subset V_p \subset \bigoplus_{j=1}^{n} E_j.
\]

We finally need to \textbf{prove the claim}: To prove the inequality \eqref{eq:claim-upperbound}, we simply note that
\begin{align*}
\sum_{k=1}^\infty \lambda_k\alpha_k 
-
\sum_{k=1}^p \lambda_k
&=
\sum_{k=1}^\infty \lambda_k\alpha_k 
-
\sum_{k=1}^p \lambda_k \alpha_k 
- 
\sum_{k=1}^p \lambda_k (1-\alpha_k)
\\
&=
\sum_{k>p} \lambda_k \alpha_k - \sum_{k=1}^p \lambda_k (1-\alpha_k).
\end{align*}
Since the sequence of eigenvalues $\lambda_k$ is monotonically decreasing and $\alpha_k \ge 0$, $1-\alpha_k \ge 0$, we have
\[
\sum_{k>p} \lambda_k \alpha_k \le \sum_{k>p} \lambda_p \alpha_k,
\quad
\text{and}
\quad
\sum_{k=1}^p \lambda_k (1-\alpha_k)
\ge 
\sum_{k=1}^p \lambda_p (1-\alpha_k),
\]
and hence
\begin{align*}
\sum_{k>p} \lambda_k \alpha_k - \sum_{k=1}^p \lambda_k (1-\alpha_k)
&\le 
\sum_{k>p} \lambda_p \alpha_k - \sum_{k=1}^p \lambda_p (1-\alpha_k)
\\
&=
\sum_{k>p} \lambda_p \alpha_k - p \lambda_p  + \sum_{k=1}^p \lambda_p\alpha_k
\\
&=
\lambda_p \left(\sum_{k=1}^\infty \alpha_k - p\right)
\\
&= 0.
\end{align*}
The last line follows from the fact that $\sum_{k=1}^\infty \alpha_k = p$, for any $(\alpha_k)_{k\in \N} \in \mathcal{A}_p$. We thus conclude that
\[
\sum_{k=1}^\infty \lambda_k \alpha_k \le \sum_{k=1}^p \lambda_k,
\]
for all $(\alpha_k)_{k\in \N} \in \mathcal{A}_p$. Furthermore, if $\sum_{k=1}^\infty \lambda_k \alpha_k = \sum_{k=1}^p \lambda_k$, then in all the above estimates, we must have equality. In particular, we must have
\[
\sum_{k>p} \lambda_k \alpha_k = \sum_{k>p} \lambda_p \alpha_k,
\quad
\sum_{k=1}^p \lambda_k (1-\alpha_k) = \sum_{k=1}^p \lambda_p (1-\alpha_k).
\]
The first equality is only possible, if $\alpha_k = 0$, for any $k\in \N$, such that $\lambda_k < \lambda_p$, i.e. we must have
\[
\alpha_k = 0,
\quad \text{if } 
k > \sum_{j=1}^n p_j.
\]
The second equality is only possible, if $\alpha_k = 1$, for any $k\in \N$, such that $\lambda_k > \lambda_p$, i.e. we must have
\[
\alpha_k = 1,
\quad \text{if }
k \le \sum_{j=1}^{n-1} p_j.
\]
This concludes the proof.
\end{proof}

\begin{remark} \label{rem:proj-minimizer}
It is also straight-forward to check that if $\hat{V}\subset H$ is a $p$-dimensional subspace, such that 
\[
\bigoplus_{k=1}^{n-1} E_k 
\subset 
\hat{V}
\subset 
\bigoplus_{k=1}^{n} E_k,
\]
where
\[
\sum_{k=1}^{n-1} \dim(E_k) < p \le \sum_{k=1}^n \dim(E_k),
\]
then 
\[
\sum_{k = 1}^\infty \lambda_k \Vert P_{\hat{V}} \phi_k \Vert^2 = \sum_{k=1}^p \lambda_k,
\]
i.e. $\hat{V}$ is optimal for $\Err_\Proj(\hat{V};\nu)$, in this case.
\end{remark}

From the last lemma, we immediately have the following corollary.

\begin{corollary} \label{cor:proj-minimizer}
For any $n\in \N$, the minimizing subspace $V_p$ of $\Err_{\Proj}(V_p;\nu)$ of dimension
\[
p = \sum_{k=1}^n \dim(E_k),
\]
is unique.
\end{corollary}

\begin{proof}
By Lemma \ref{lem:proj-minimizer}, if $V_p$ is any minimizer of $\Err_{\Proj}$, then 
\[
V_p \subset \bigoplus_{k=1}^n E_k.
\]
Since $\dim(V_p) = p = \sum_{k=1}^n \dim(E_k)$, by assumption, it follows that $V_p = \bigoplus_{k=1}^n E_k$ is uniquely determined by the eigenspaces $E_1, \dots, E_n$.
\end{proof}

\subsubsection{Proof of Theorem \ref{thm:opt-linear}}

\begin{proof}
The existence and characterization of optimal subspaces $\hat{V}$ is a consequence of Lemma \ref{lem:proj-minimizer} and Remark \ref{rem:proj-minimizer}. Uniqueness of the optimal subspace $V_p\subset H$ for $p_n = \sum_{k=1}^n \dim(E_k)$ is proved in Corollary \ref{cor:proj-minimizer}. Finally, the identity for the projection error $\Err_{\proj}(\hat{V}) = \sum_{k>p} \lambda_k$ for the optimal subspace $V_p = \Span(\phi_1,\dots, \phi_p)$, with $\phi_k$ denote the eigenfunctions of the uncentered covariance operator $\bar{\Gamma}$ corresponding to decreasing eigenvalues $\lambda_1\ge \lambda_2 \ge \dots$, follows from 
\begin{align*}
\Err_{\proj}(\hat{V}) 
&\explain{=}{\eqref{eq:linear}} 
\int_H \Vert v \Vert^2 \, d\mu(u) - \sum_{i=1}^p \int_H |\langle v, \phi_i \rangle |^2 \, d\mu(v)
\\
&=
\Tr(\bar{\Gamma}) - \sum_{i=1}^p \langle \phi_i, \bar{\Gamma} \phi_i \rangle 
\\
&\explain{=}{\eqref{eq:Eproj-equiv2}}
\sum_{k=1}^\infty \lambda_k - \sum_{k=1}^\infty \lambda_k \Vert P_{V_p} \phi_k \Vert^2
\\
&=
\sum_{k=1}^\infty \lambda_k - \sum_{k=1}^p \lambda_k
=
\sum_{k>p} \lambda_k.
\end{align*}
\end{proof}

\subsubsection{Proof of Theorem \ref{thm:opt-proj}}
\label{app:pf-opt-proj}

\begin{proof}
Since $V_0$ is a finite-dimensional affine subspace, the existence and uniqueness of
\[
\hat{v}_0 \in \argmin_{\hat{v}\in \hat{V}_0} \Vert \E[v] - \hat{v} \Vert^2,
\]
is straight-forward. We also note that $\E[v] - \hat{v}_0 \perp \hat{V}$. We can now write
\begin{align*}
\Err_{\Proj}(\hat{V}_0)
&=
\int_H 
\inf_{\hat{v}\in \hat{V}_0}
\Vert 
{v} - \hat{v}
\Vert^2
\, 
d\nu(v)
\\
&=
\int_H 
\inf_{\hat{v}\in \hat{V}}
\Vert 
{v} - \hat{v}_0 - \hat{v}
\Vert^2
\, 
d\nu(v)
\\
&=
\int_H 
\inf_{\hat{v}\in \hat{V}}
\Vert 
[{v}- \E[v] - \hat{v}] -  [\E[v] - \hat{v}_0]
\Vert^2
\, 
d\nu(v)
\\
&=
\int_H 
\inf_{\hat{v}\in \hat{V}}
\left\{
\Vert 
\E[v] - \hat{v}_0
\Vert^2
-
2\langle
{v}- \E[v] - \hat{v}
,
\E[v] - \hat{v}_0
\rangle
\right. \\
&\hspace{2cm} \left.
+
\Vert 
{v}- \E[v] - \hat{v}
\Vert^2
\right\}
\, 
d\nu(v)
\end{align*}
Since $\E[v] - \hat{v}_0 \perp \hat{v}$ for all $\hat{v} \in \hat{V}$, we have
\[
2\langle
{v}- \E[v] - \hat{v}
,
\E[v] - \hat{v}_0
\rangle
=
2\langle
{v}- \E[v]
,
\E[v] - \hat{v}_0
\rangle,
\]
and thus
\begin{align*}
\Err_{\Proj}(\hat{V}_0)
&=
\int_H 
\inf_{\hat{v}\in \hat{V}}
\left\{
\Vert 
\E[v] - \hat{v}_0
\Vert^2
-
2\langle
{v}- \E[v]
,
\E[v] - \hat{v}_0
\rangle
\right. \\
&\hspace{2cm} \left.
+
\Vert 
{v}- \E[v] - \hat{v}
\Vert^2
\right\}
\, 
d\nu(v)
\\
&=
\int_H 
\Vert 
\E[v] - \hat{v}_0
\Vert^2
\, d\nu(v)
\\
&\quad
-
2\int_H
\langle
{v}- \E[v]
,
\E[v] - \hat{v}_0
\rangle
\, d\nu(v)
 \\
&\quad
+
\int_H
\inf_{\hat{v}\in \hat{V}}
\Vert 
{v}- \E[v] - \hat{v}
\Vert^2
\, 
d\nu(v).
\end{align*}
The first term is independent of $v$, the second term averages to zero. Hence, we finally obtain
\[
\Err_{\Proj}(\hat{V}_0)
=
\Vert \E[v] - \hat{v}_0 \Vert^2
+
\int_H
\inf_{\hat{v}\in \hat{V}}
\Vert 
{v}- \E[v] - \hat{v}
\Vert^2
\, 
d\nu(v).
\]
Note now if $\hat{V}_0$ is an affine subspace with associated vector space $\hat{V}$ for which the first term is not equal to $0$, then the affine subspace $\hat{V}_0' := \E[v] + \hat{V}$ satisfies
\begin{align*}
\Err_{\Proj}(\hat{V}_0') 
&=
\int_H \inf_{\hat{v}\in \hat{V}} \Vert v - \E[v] - \hat{v}\Vert^2 \, d\nu(v)
\\
&<
\Vert \E[v] - \hat{v}_0 \Vert^2
+
\int_H \inf_{\hat{v}\in \hat{V}} \Vert v - \E[v] - \hat{v}\Vert^2 \, d\nu(v)
\\
&=
\Err_{\Proj}(\hat{V}_0).
\end{align*}
Thus, if $\hat{V}_0$ is a minimizer among affine subspaces, then we must have $\hat{v}_0 = \E[v]$. Next, define a measure $\overline{\nu} \in \P(H)$ by
\[
\int_H \Phi(v) \, d\overline{\nu}(v) := \int_H \Phi(v - \E[v]) \, d\nu(v),
\quad
\forall \, \Phi \in L^\infty(H).
\]
Then $\int_H v \, d\overline{\nu}(v) = 0$, and
\[
\int_H \inf_{\hat{v}\in \hat{V}} \Vert v - \E[v] - \hat{v}\Vert^2 \, d\nu(v)
=
\int_H \inf_{\hat{v}\in \hat{V}} \Vert v - \hat{v}\Vert^2 \, d\overline{\nu}(v),
\]
is the projection error $\Err_{\Proj}(\hat{V};\overline{\nu})$ of the $p$-dimensional vector space $\hat{V}$ with respect to the measure $\overline{\nu}$. Since $\hat{V}_0$ is a minimizer, we must have that $\hat{V}$ is a minimizer of $\Err_{\Proj}(\hat{V})$ among $p$-dimensional subspaces. By Lemma \ref{lem:proj-minimizer}, it follows that
\[
\bigoplus_{j=1}^{n-1} E_j \subset \hat{V} \subset \bigoplus_{j=1}^n E_j,
\]
where $E_j$ denote the eigenspaces of the covariance operator of $\overline{\nu}$, given by 
\[
\int_H v\otimes v \, d\overline{\nu}(v)
=
\int_H (v-\E[v])\otimes (v-\E[v]) \, d\nu(v)
=
\Gamma,
\]
and corresponding to distinct eigenvalues $\overline{\lambda}_1 > \overline{\lambda}_2 > \dots$ of ${\Gamma}$, as claimed. This concludes the proof of this theorem.
\end{proof}

\subsection{Proof of Lemma \ref{lem:opt-proj}}
\label{app:pf34}
\begin{proof}
Let $V = \Span(\tr_1,\dots, \tr_p)$. We first note that since we have a direct sum $L^2(U) = V^\perp \oplus V$, we can decompose any $u \in L^2(U)$ uniquely as 
\[
u
=
u^\perp 
+
\sum_{k=1}^p \alpha_k \tr_k,
\]
for coefficients $\alpha_1,\dots, \alpha_p\in \R$. Let $\tr^\ast_1,\dots, \tr^\ast_p \in \Span(\tr_1,\dots, \tr_p)$ denote the dual basis of $\tr_1,\dots, \tr_p$. Taking the inner product of the last identity with $\tr_\ell^\ast$, it follows that
\[
\langle \tr_\ell^\ast, u \rangle
=
\langle \tr_\ell^\ast, u^\perp \rangle
+
\sum_{k=1}^p \alpha_k \langle \tr_\ell^\ast, \tr_k \rangle
=
\alpha_\ell,
\]
where we have used that $u^\perp \perp V \ni \tr_\ell^\ast$, and $\langle \tr_\ell^\ast, \tr_k\rangle = \delta_{k\ell}$. Applying this identity, we now note that
\begin{align*}
\cR \circ \cP(u)
&=
\tr_0 
+ 
\sum_{k=1}^p
[\cP(u)]_k \tr_{k}
=
\tr_0^\perp
+ 
\sum_{k=1}^p
\big(
[\cP(u)]_k + \langle \tr_k^\ast, \tr_0 \rangle 
\big)  \tr_{k},
\end{align*}
and 
\begin{align*}
u = u^\perp + \sum_{k=1}^p \langle \tr_k^\ast, u\rangle \tr_k.
\end{align*}
Hence
\begin{align*}
\cR \circ \cP(u) - u 
&=
\tr_0^\perp - u^\perp
+
\sum_{k=1}^p
\big(
[\cP(u)]_k + \langle \tr_k^\ast, \tr_0 - u \rangle
\big) \tr_k.
\end{align*}
The norm of the last term is clearly minimized if the sum over $k=1,\dots, p$ vanishes. This is the case, provided that
\[
[\cP(u)]_k =  \langle \tr_k^\ast, u - \tr_0 \rangle,
\quad
(k=1,\dots, p).
\]
The claim follows.
\end{proof}
\subsection{Proof of Proposition \ref{prop:err-rec-approx}}
\label{app:pf35}
\begin{proof}
Denote $\hat{V}_0 := \im(\cR)$ the affine image of $\cR = \cR_{\bm{\tr}}$. Taking into account Theorem \ref{thm:opt-proj}, we obtain
\begin{align*}
(\Err_{\cR})^2
&=
\int_{Y} \Vert \cR \circ \cP - \Id \Vert^2_{L^2_y} \, d\nu
\\
&=
\inf_{\hat{v} \in \hat{V}_0} 
\Vert \E_\nu[v] - \hat{v} \Vert_{L^2_y}^2 
+ 
\sum_{k=1}^\infty \lambda_k \Vert \cR \circ \cP \phi_k - \phi_k \Vert^2_{L^2_y}
\\
&\le 
\Vert \E_\nu[v] - \tr_0 \Vert_{L^2_y}^2 
+ 
\sum_{k\le p} \lambda_k \Vert \tr_k - \phi_k \Vert^2_{L^2_y}
+ 
\sum_{k>p} \lambda_k \Vert \cR \circ \cP \phi_k - \phi_k \Vert^2_{L^2_y}
\\
&= 
\Vert \hat{\tr}_0 - \tr_0 \Vert_{L^2_y}^2 
+ 
\sum_{k=1}^p \lambda_k \Vert \tr_k - \hat{\tr}_k \Vert^2_{L^2_y}
+ 
\sum_{k>p} \lambda_k \Vert \cR \circ \cP \phi_k - \phi_k \Vert^2_{L^2_y}
\\
&\le 
\Vert \hat{\tr}_0 - \tr_0 \Vert_{L^2_y}^2 
+ 
\left(\sum_{k=1}^\infty \lambda_k \right) \sup_{k=1,\dots, p}\Vert \tr_k - \hat{\tr}_k \Vert^2_{L^2_y}
+ 
\sum_{k>p} \lambda_k
\\
&\le
\left[1 + \Tr(\Gamma_\nu)\right] \sup_{k=0,1,\dots, p}\Vert \tr_k - \hat{\tr}_k \Vert^2_{L^2_y}
+
\sum_{k>p} \lambda_k.
\end{align*}
Taking square roots of both sides, the claimed estimate \eqref{eq:err-rec-approx} now follows from the trivial inequality $\sqrt{a+b} \le \sqrt{a} + \sqrt{b}$ for $a,b\ge 0$.
\end{proof}
\subsection{Proof of Lemma \ref{lem:err-rec-comp}}
\label{app:pf306}
\begin{proof}
By the triangle inequality, we have
\begin{align*}
\Err_{\cR} 
&\le 
\Err_{\tR}
+
\left(
\int_{L^2}
\Vert \cR\circ \cP(u) - \tR\circ\tP(u) \Vert^2_{L^2}
\, d\nu(u)
\right)^{1/2}.
\end{align*}
Note that by the assumed orthonormality of the $\tilde{\tr}_k$, the projection can be written as follows (cp. \eqref{eq:RP-opt})
\[
\tR \circ \tP(u) 
=
\sum_{k=1}^p \langle \tilde{\tr}_k, u\rangle \tilde{\tr}_k.
\]
Furthermore, in terms of the dual basis $\tr_1^\ast, \dots, \tr_p^\ast \in Y$, we have 
\[
\cR \circ \cP(u) 
=
\sum_{k=1}^p \langle \tr_k^\ast, u\rangle {\tr}_k.
\]
In terms of this expansion we can again use the triangle inequality to obtain
\begin{gather} \label{eq:triangle}
\begin{aligned}
\left(
\int_{L^2}
\right.
&\Vert \cR\circ \cP(u) - \tR\circ\tP(u) \Vert^2_{L^2}
\left.
\, d\nu(u)
\right)^{1/2}
\\
&\le
\sum_{k=1}^p
\left(
\int_{L^2}
\Vert \langle \tr_k^\ast, u\rangle \tr_k - \langle \tilde{\tr}_k, u\rangle \tilde{\tr}_k \Vert^2_{L^2}
\, d\nu(u)
\right)^{1/2}
\end{aligned}
\end{gather}
Since 
\begin{gather} \label{eq:est-yktk}
\begin{aligned} 
\Vert \langle \tr_k^\ast, u \rangle \tr_k - \langle \tilde{\tr}_k, u \rangle \tilde{\tr}_k \Vert_{L^2}
&\le
\Vert \langle \tr_k^\ast, u \rangle (\tr_k - \tilde{\tr}_k) \Vert_{L^2}
+\Vert \langle \tr_k^\ast-\tilde{\tr}_k, u \rangle \tilde{\tr}_k \Vert_{L^2}
\\
&\le 
\Vert \tr_k^\ast \Vert_{L^2} \Vert u \Vert_{L^2} \Vert \tr_k - \tilde{\tr}_k \Vert_{L^2}
+
\Vert \tr_k^\ast - \tilde{\tr}_k \Vert_{L^2} \Vert u \Vert_{L^2} \Vert \tilde{\tr}_k \Vert_{L^2},
\end{aligned}
\end{gather}
we next wish to establish that under the assumptions of this lemma, we can bound
\[
\Vert \tr_k^\ast - \tilde{\tr}_k \Vert_{L^2}
\le
C \sqrt{p} \max_{j=1,\dots, p}\Vert \tr_j - \tilde{\tr}_j \Vert_{L^2},
\]
for some absolute constant $C>0$, independent of $k$ and $p$. To see this, we note that for any $k,j=1,\dots, p$, we have $\langle \tr_k^\ast, \tr_j\rangle = \delta_{kj}$, and hence
\[
\langle (\tr_k^\ast - \tilde{\tr}_k), \tilde{\tr}_j \rangle
=
\langle \tr_k^\ast , \tilde{\tr}_j \rangle - \delta_{kj}
=
\langle \tr_k^\ast , \tilde{\tr}_j \rangle - \langle \tr_k^\ast, \tr_j \rangle
=
\langle \tr_k^\ast , (\tilde{\tr}_j - \tr_j) \rangle.
\]
If $u\in Y$ is arbitrary, then we obtain from the above identity
\begin{align}
|\langle (\tr_k^\ast - \tilde{\tr}_k), u\rangle|
&=
\left|\sum_{j=1}^p \langle \tilde{\tr}_j, u\rangle \langle (\tr_k^\ast - \tilde{\tr}_k), \tilde{\tr}_j\rangle\right|
\notag\\
&=
\left|\sum_{j=1}^p \langle \tilde{\tr}_j, u\rangle \langle \tr_k^\ast, \tilde{\tr}_j - \tr_j\rangle\right|
\notag\\
&\le
\left(\sum_{j=1}^p |\langle \tilde{\tr}_j, u\rangle|^2 \right)^{1/2} 
\left(\sum_{j=1}^p |\langle \tr_k^\ast, \tilde{\tr}_j - \tr_j\rangle|^2 \right)^{1/2}
\notag\\
&\le
\Vert u \Vert_{L^2} \, \Vert \tr_k^\ast \Vert_{L^2} \, \sqrt{p} \max_{j=1,\dots, p} \Vert \tr_j - \tilde{\tr}_j \Vert_{L^2}.
\label{eq:ymtt}
\end{align}
By assumption, we have that 
\[
\sqrt{p} \max_{j=1,\dots, p} \Vert \tr_j - \tilde{\tr}_j \Vert_{L^2}
\le \epsilon \le \frac12
.
\]
It follows that for any $k=1,\dots, p$, we have
\[
\Vert \tr_k^\ast - \tilde{\tr}_k \Vert_{L^2}
=
\sup_{\Vert u \Vert \le 1} |\langle \tr_k^\ast - \tilde{\tr}_k, u \rangle|
\le
\epsilon\, \Vert \tr_k^\ast \Vert_{L^2} 
\le
\epsilon\, \Vert \tr_k^\ast - \tilde{\tr}_k \Vert_{L^2} +\epsilon\, \Vert \tilde{\tr}_k \Vert_{L^2},
\]
and hence
\[
\Vert \tr_k^\ast - \tilde{\tr}_k \Vert_{L^2}
\le
\frac{\epsilon}{1-\epsilon} \Vert \tilde{\tr}_k \Vert_{L^2}
=
\frac{\epsilon}{1-\epsilon}
\le
2\epsilon.
\]
Note also that this implies 
\begin{align} \label{eq:ynorm}
\Vert \tr_k^\ast \Vert_{L^2}
\le 
\Vert \tr_k^\ast - \tilde{\tr}_k\Vert_{L^2} + \Vert \tilde{\tr}_k \Vert_{L^2}
\le
2\epsilon + 1 \le 2.
\end{align}
It now follows from \eqref{eq:est-yktk}, that 
\begin{align*}
\Vert \langle \tr_k^\ast, u \rangle \tr_k - \langle \tilde{\tr}_k, u \rangle \tilde{\tr}_k \Vert_{L^2}
&\le
2\Vert u \Vert_{L^2}
\left(
(\sqrt{p} + 1) \max_{j=1,\dots, p} \Vert \tr_j - \tilde{\tr}_j \Vert_{L^2}.
\right)
\\
&\le
4\Vert u \Vert_{L^2} \sqrt{p} \, \max_{j=1,\dots, p} \Vert \tr_j - \tilde{\tr}_j \Vert_{L^2}.
\end{align*}
Substitution in \eqref{eq:triangle}, finally yields
\[
\Err_{\cR} 
\le
\Err_{\tR}
+
4 \left(\int_{L^2} \Vert u \Vert^2_{L^2} \, d\nu(u)\right)^{1/2} \, p^{3/2} \, \max_{j=1,\dots, p} \Vert \tr_j - \tilde{\tr}_j \Vert_{L^2}.
\]
Letting 
\[
C := 4 \left(\int_{L^2} \Vert u \Vert^2_{L^2} \, d\nu(u)\right)^{1/2}, 
\]
we conclude that 
\[
\Err_{\cR} 
\le
\Err_{\tR}
+
Cp^{3/2} \, \max_{j=1,\dots, p} \Vert \tr_j - \tilde{\tr}_j \Vert_{L^2}
\le
\Err_{\tR}
+
C\epsilon,
\]
as claimed.

To prove the Lipschitz bound on $\cP: L^2(U) \to \R^p$, let $u,u' \in L^2(U)$ be given, and denote $w = u-u'$. Then
\begin{align*}
\Vert \cP (u) - \cP(u') \Vert_{\ell^2}
&=
\left(\sum_{k=1}^p |\langle \tr_k^\ast, w \rangle|^2 \right)^{1/2}
\\
&\le 
\left(\sum_{k=1}^p |\langle \tr_k^\ast-\tilde{\tr}_k, w \rangle|^2 \right)^{1/2}
+
\left(\sum_{k=1}^p |\langle \tilde{\tr}_k, w \rangle|^2 \right)^{1/2}
\end{align*}
By \eqref{eq:ymtt} and \eqref{eq:ynorm}, we can bound each term in the first sum by
\[
|\langle \tr_k^\ast-\tilde{\tr}_k, w \rangle|
\le
2 \sqrt{p} \max_{j=1,\dots, p} \Vert \tr_j - \tilde{\tr}_j\Vert_{L^2} \, \Vert u \Vert_{L^2}.
\]
Hence,
\[
\left(\sum_{k=1}^p |\langle \tr_k^\ast-\tilde{\tr}_k, w \rangle|^2 \right)^{1/2}
\le 
2 p \max_{j=1,\dots, p} \Vert \tr_j - \tilde{\tr}_j\Vert_{L^2} \Vert u \Vert_{L^2}.
\]
As the $\tilde{\tr}_k$ are orthonormal by assumption, the second sum can be estimated simply by
\[
\left(\sum_{k=1}^p |\langle \tilde{\tr}_k, w \rangle|^2 \right)^{1/2}
\le 
\Vert w \Vert_{L^2}.
\]
Recalling that $w = u-u'$, it follows that
\begin{align*}
\Vert \cP(u)-\cP(u') \Vert_{\ell^2}
&\le 
\left(
1 + 2 p \max_{j=1,\dots, p} \Vert \tr_j - \tilde{\tr}_j \Vert_{L^2}
\right)
\,
\Vert u - u'\Vert_{L^2}.
\\
&\le 
\left(
1 + 2\epsilon
\right)
\,
\Vert u - u'\Vert_{L^2}
\le 
2\Vert u - u'\Vert_{L^2}.
\end{align*}
As $u,u'\in L^2(U)$ were arbitrary, the claimed estimate for $\Lip(\cP)$ follows. Similarly, we can estimate for the reconstruction:
\begin{align*}
\Vert \cR(\alpha) - \cR(\alpha') \Vert_{L^2}
&=
\left\Vert
\sum_{k=1}^p (\alpha_k - \alpha_k') \tr_k
\right\Vert_{L^2}
\\
&\le
\left\Vert
\sum_{k=1}^p (\alpha_k - \alpha_k') \tilde{\tr}_k 
\right\Vert_{L^2}
+
\left\Vert
\sum_{k=1}^p (\alpha_k - \alpha_k') (\tr_k - \tilde{\tr}_k)
\right\Vert_{L^2}
\\
&\le
\Vert \alpha - \alpha' \Vert_{\ell^2}
+
\sum_{k=1}^p |\alpha_k - \alpha_k'| \Vert \tr_k - \tilde{\tr_k}\Vert_{L^2}
\\
&\le
\left(
1
+
 p^{1/2} \max_{k=1,\dots, p}\Vert \tr_k - \tilde{\tr_k}\Vert_{L^2}
 \right)
 \Vert \alpha - \alpha' \Vert_{\ell^2}
 \\
 &\le
\left(
1
+
 \epsilon
 \right)
 \Vert \alpha - \alpha' \Vert_{\ell^2} 
 \le
 2\Vert \alpha - \alpha' \Vert_{\ell^2}.
\end{align*}
Thus, $\Lip(\cR)\le 2$.
\end{proof}

\subsection{Proof of Lemma \ref{lem:Fourier-NN}}
\label{app:pf36}
\begin{proof}
We note that each element in the (real) trigonometric basis $\fb_1,\dots, \fb_p$ can be expressed in the form 
\[
\fb_j({x}) 
=
\cos(\kappa\cdot {x}),
\quad
\text{or}
\quad
\fb_j({x}) = \sin(\kappa\cdot {x}),
\]
for $\kappa = \kappa(j)\in \Z^n$ with $|{k}|_\infty \le N$, where $N$ is chosen as the smallest natural number such that $p \le (2N+1)^n$. It follows from 
\[
|\kappa\cdot {x}|
\le n |\kappa|_{\infty} |{x}|_{\infty}
\le 
2\pi n N, 
\]
that if $\cC_\epsilon, \cS_\epsilon: \R \to \R$ are neural networks, such that 
\begin{align} \label{eq:cos-global}
\sup_{\xi \in [0,2\pi n N]} | \cC_\epsilon(\xi) - \cos(\xi) |, \; |\cS_\epsilon(\xi) - \sin(\xi) | \le \epsilon,
\end{align}
then the map
\[
{x} \mapsto ((\kappa(j)\cdot {x}))_{j=1,\dots, p}
\mapsto (\cC_\epsilon(\kappa(j)\cdot {x}), \cS_\epsilon(\kappa(j)\cdot {x}))_{j=1,\dots, p}
\]
can be represented by a neural network $\cT$ with a size bounded by
\[
\begin{gathered}
\size(\cT) = \mathcal{O}( p \, \size(\cC_\epsilon) + p\, \size(\cS_\epsilon)), \\
\depth(\cN) = \mathcal{O}\left(\max\left(\depth(\cC_\epsilon),\depth(\cS_\epsilon)\right)\right).
\end{gathered}
\]
By {\cite[Theorem III.9]{EPGB2020}}, there exist $\cC_\epsilon$, $\cS_\epsilon$, satisfying 
\eqref{eq:cos-global}, with 
\[
\begin{gathered}
\size(\cC_\epsilon), \size(\cS_\epsilon) = \mathcal{O}(\log(\epsilon^{-1})^2 + \log(N)), \\
\depth(\cC_\epsilon), \depth(\cS_\epsilon) = \mathcal{O}(\log(\epsilon^{-1})^2 + \log(N)).
\end{gathered}
\]
Finally, we note that $\log(N) \sim 1/n\log(p) = \mathcal{O}(\log(p))$. We thus conclude that for any $\epsilon > 0$, there exists a neural network $\cT_\epsilon = (\cT_{\epsilon,1},\dots, \cT_{\epsilon,p})$, such that 
\[
\Vert \cT_{\epsilon,j} - \fb_j \Vert_{L^\infty([0,2\pi]^n)} \le \epsilon,
\]
for all $j=1,\dots, p$, and 
\[
\begin{gathered}
\size(\cT_\epsilon) = \mathcal{O}( p \, (\log(p) + \log(\epsilon^{-1})^2), \\
\depth(\cT_\epsilon) = \mathcal{O}(\log(p) + \log(\epsilon^{-1})^2).
\end{gathered}
\]
To satisfy the estimate \eqref{eq:Fourier-NN}, we set $\bm{\tr} = \cT_{\epsilon/p^{3/2}}$, for which we have 
\[
\begin{gathered}
\size(\bm{\tr}) = \mathcal{O}( p \, \log(\epsilon^{-1}p)^2), \\
\depth(\bm{\tr}) = \mathcal{O}(\log(\epsilon^{-1}p)^2).
\end{gathered}
\]
and
\[
\max_{j=1,\dots, p} 
\Vert \tr_{j} - \fb_j \Vert_{L^\infty([0,2\pi]^n)} \le \epsilon.
\]
\end{proof}
\subsection{Proof of Proposition \ref{prop:cexample}}
\label{app:pf307}
\begin{proof}
Proposition \ref{prop:cexample} follows form the following two claims:

\begin{claim}
If $\mu$ is a non-degenerate Gaussian, then there exists a Lipschitz continuous mapping $P: X \to [0,1]$, such that $P_\#\mu = dx$ is the uniform measure on $[0,1]$.
\end{claim}

\begin{claim}
If $Y$ is an infinite-dimensional Hilbert space with orthonormal basis $\{e_k\}_{k\in \N}$, and $\gamma_k\ge 0$, $k\in\N$, a sequence, such that 
\[
\sum_{k=1}^\infty k^2 \gamma_k < \infty,
\]
then there exists a Lipschitz continuous mapping $G: [0,1] \to Y$, such that the covariance operator of $G_\# dx$ is given by
\[
\Gamma_{G_\# dx}
=
\sum_{k=1}^\infty \gamma_k (e_k \otimes e_k).
\]
\end{claim}

The sought-after map $\G$ can then be defined as $\G = G \circ P: X \to Y$, where $P: X \to [0,1]$ and $G: [0,1]\to Y$ are defined as in the above claims.

To prove the \textbf{first claim}, we simply note that since $\mu$ is a non-degenerate Gaussian measure, there exists a one-dimensional projection $L: X \to \R$, such that $L_\#\mu = \Normal(m,\sigma^2)$ is a non-degenerate Gaussian with mean $m\in \R$, and variance $\sigma^2 > 0$. Upon performing a translation by $\mu$ and scaling by $1/\sigma$, we obtain an affine mapping $\tilde{L}: X \to \R$, such that $\tilde{L}_\#\mu = \Normal(0,1)$. We finally note that the error function 
\[
\mathrm{erf}(x) := \frac{1}{\sqrt{2\pi}}\int_{-\infty}^x e^{-u^2/2} \, du,
\]
maps the standard Gaussian distribution to the uniform distribution on $[0,1]$. And hence, we have $(\mathrm{erf}\circ \tilde{L})_\#\mu = dx$ on $[0,1]$.

To prove the \textbf{second claim}, we show that the function
\[
G: [0,1] \to Y,
\quad
G(x) := \sum_{k=1}^\infty \sqrt{2\gamma_k} \cos(2\pi kx) e_k,
\]
possesses the desired properties: We first note that 
\[
\Vert \partial_x G(x) \Vert_Y^2
=
(2\pi)^2 \sum_{k=1}^\infty 2\gamma_k k^2 |\sin(2\pi kx)|^2 
\le
(2\pi)^2 \sum_{k=1}^\infty 2\gamma_k k^2 < +\infty,
\]
is uniformly bounded in $x$ by assumption on $\gamma_k$. In particular, it follows that $x\mapsto G(x)$ is Lipschitz continuous. Furthermore, we have 
\begin{align*}
\int_Y u\otimes u \, d(G_\#dx)
&=
\sum_{k,k'=1}^\infty
2\sqrt{\gamma_k \gamma_{k'}}
\int_0^1 \cos(2\pi kx) \cos(2\pi k'x) (e_k\otimes e_{k'}) \, dx
\\
&=
\sum_{k=1}^\infty \gamma_k (e_k \otimes e_k).
\end{align*}
\end{proof}
\subsection{Proof of Proposition \ref{prop:lin}}
\label{app:pf37a}
\begin{proof}
Let $v_1, \dots, v_p \in X$ denote pairwise orthogonal eigenvectors corresponding to the first $p$ eigenvalues $\lambda_1, \dots, \lambda_p$ of the covariance operator $\Gamma_{\mu}$. 
It follows from Theorem \ref{thm:opt-proj}, that the $p$-dimensional affine subspace $V_0$, given by $V_0 = v_0 + V$, where $v_0 = \E_{\mu}[u]$ and $V = \Span(v_1,\dots, v_p)$, satisfies
\[
\int_X \inf_{v\in V_0} \Vert v - u \Vert^2 \, d\mu(u)
=
\sum_{j>p} \lambda_j.
\]
Let $W_0 := \G(V_0)$. Since $\G: X \to Y$ is linear, $W_0$ is an affine subspace of dimension at most $p$. If $\dim(W_0) < p$, we extend $W_0$ to an affine subspace $W \subset Y$, such that $W_0\subset W$ and $\dim(W) = p$, else set $W := W_0$. Then, irrespective of the choice of the extension, we have
\begin{align*}
\int_Y \inf_{w\in W} \Vert w - u \Vert_Y^2 \, d(\G_\#\mu(u))
&\le
\int_Y \inf_{w\in W_0} \Vert w - u \Vert_Y^2 \, d(\G_\#\mu(u))
\\
&=
\int_X \inf_{w \in W_0} \Vert w - \G(u) \Vert_Y^2 \, d\mu(u)
\\
&=
\int_X \inf_{v \in V_0} \Vert \G(v) - \G(u) \Vert_Y^2 \, d\mu(u)
\\
&\le
\int_X \inf_{v \in V_0} \Vert \G \Vert^2 \Vert v - u \Vert_Y^2 \, d\mu(u)
\\
&=
\Vert \G \Vert^2 \, \sum_{j > p} \lambda_j,
\end{align*}
as claimed. Finally, we note that if $(\cR,\cP)$ are chosen such that $\cR\circ \cP: Y\to Y$ is the orthogonal projection onto $W$, then 
\[
(\Err_{\cR} )^2
=
\int_Y \Vert \cR\circ \cP(u) - u \Vert^2 \, d(\G_\#\mu(u))
=
\int_Y \inf_{w\in W} \Vert w - u \Vert^2 \, d(\G_\#\mu(u)).
\]
\end{proof}
\subsection{Proof of Theorem \ref{thm:enlb}}
\label{app:pf37}
\begin{proof}
Since $\cD: \R^m \to X$ is assumed to be linear, there exist $\psi_1, \dots, \psi_m \in X$, such that 
\[
\cD(u_1, \dots, u_m) = \sum_{k=1}^m u_k \psi_k.
\]
Let $U := \Span(\psi_1, \dots, \psi_m) \subset X$. Since $U$ is a subspace of $X$ of dimension at most $m$, we have
\begin{align*}
\int_X 
\Vert \cD \circ \cE - \Id \Vert_{L^2_x}^2 
\, 
d\mu
&\ge 
\int_X 
\inf_{\hat{u} \in U} \Vert u - \hat{u} \Vert_{L^2_x}^2 
\, 
d\mu(u)
\\
&\ge 
\inf_{
{\substack{\hat{U} \subset X; \\ \dim(U) = m}}
}
\int_X 
\inf_{\hat{u} \in \hat{U}} \Vert u - \hat{u} \Vert_{L^2_x}^2
\, 
d\mu(u).
\end{align*}
By Theorem \ref{thm:opt-linear}, the infimum on the last line is attained for the subspace $\hat{U}$ spanned by the orthonormal eigenfunctions $\phi_1, \dots, \phi_m$ of the uncentered covariance operator $\Gamma_\mu$ of $\mu$. For this space, we have
\begin{align*}
\int_X 
\inf_{\hat{u} \in \hat{U}} \Vert u - \hat{u} \Vert_{L^2_x}^2
\, 
d\mu(u)
&=
\int_X \Vert u \Vert^2 \, d\mu(u)
-
\sum_{k=1}^m
\int_X 
|\langle u, \phi_k \rangle|^2
\, 
d\mu(u)
\\
&=
\Tr(\Gamma)
-
\sum_{k=1}^m
\langle \phi_k, \Gamma \phi_k \rangle
\\
&=
\sum_{k=1}^\infty \lambda_k
-
\sum_{k=1}^m \lambda_k
\\
&=
\sum_{k>m} \lambda_k.
\end{align*}
Thus, we conclude that 
\[
\int_X 
\Vert \cD \circ \cE - \Id \Vert_{L^2_x}^2 
\, 
d\mu
\ge 
\sum_{k>m} \lambda_k.
\]
\end{proof}

\subsection{Proof of Lemma \ref{lem:aliasing-general}}
\label{app:pf38}
\begin{proof}
Since $\cD \circ \cE$ is the identity on $\Span(\phi_1, \dots, \phi_m)$, we have 
\[
\cD \circ \cE - \Id 
=
\cD \circ \cE\circ P_m^\perp - P_m^\perp.
\]
Furthermore, since $\im(\cD)\subset \Span(\phi_1,\dots, \phi_m)$, we also have $\cD \circ \cE(P_m^\perp u) \perp  P_m^\perp u$, for all $u\in L^2_x$. Hence
\begin{align*}
(\Err_\cE)^2 
&=
\int_X \Vert \cD \circ \cE - \Id \Vert_{L^2_x}^2 \, d\mu(u)
\\
&=
\int_X \Vert \cD \circ \cE(P_m^\perp u) - P_m^\perp u \Vert_{L^2_x}^2 \, d\mu(u)
\\
&=
\underbrace{
\int_X \Vert \cD \circ \cE(P_m^\perp u) \Vert^2_{L^2_x} \, d\mu(u)
}_{
\displaystyle
=(\Err_{\mathrm{aliasing}})^2
}
+
\underbrace{
\int_X \Vert P_m^\perp u \Vert_{L^2_x}^2 \, d\mu(u)
}_{
\displaystyle
=(\Err_{\perp})^2
}
.
\end{align*}
The aliasing error 
\[
(\Err_{\mathrm{aliasing}})^2
=
\int_X \Vert \cD \circ \cE (P^\perp_m u) \Vert_{L^2_x}^2 \, d\mu(u),
\]
can be re-written as follows\footnote{we assume $\int_X u \, d\mu(u) = 0$, but this is not essential here}: Let $\phi_1, \phi_2, \dots$ denote the eigenfunctions of the covariance operator of $\mu$, $\Gamma \phi_k = \lambda_k \phi_k$, where $\lambda_1 \ge \lambda_2 \ge \dots$. First, we note that there exist random variables $Z_\ell$ (not necessarily iid), with
\[
\E[Z_\ell Z_{\ell'}] = \delta_{\ell\ell'},
\]
such that the random variable 
\[
\sum_{\ell=1}^\infty \sqrt{\lambda_k} \, Z_k \phi_k 
\sim
\mu
\]
is distributed according to $\mu$. Then, we have
\begin{align*}
(\Err_\perp)^2
&=
\E\left[
\left\Vert
\sum_{\ell=1}^\infty \sqrt{\lambda_\ell} Z_\ell P^\perp_m\phi_\ell
\right\Vert^2
\right]
=
\E\left[
\left\Vert
\sum_{\ell>m} \sqrt{\lambda_\ell} Z_\ell \phi_\ell
\right\Vert^2
\right]
\\
&=
\sum_{\ell,\ell'>m}
\sqrt{\lambda_\ell \lambda_{\ell'}}
\E\left[
Z_\ell Z_{\ell'} 
\right]
\langle \phi_\ell,\phi_{\ell'}\rangle
=
\sum_{\ell ,\ell'>m}
\sqrt{\lambda_\ell \lambda_{\ell'}}
\,\delta_{\ell,\ell'}
\\
&=
\sum_{\ell>m}
\lambda_\ell.
\end{align*}
With $\Phi_M := (\phi_i(X_j)) \in \R^{m\times M}$ a matrix such that $\det(\Phi_M \Phi_M^T) \ne 0$, the aliasing error is given by
\begin{align*}
\int_X \Vert \cD \circ \cE (P^\perp_m u) \Vert_{L^2_x}^2 \, d\mu(u)
&=
\E\left[
\left\Vert 
\sum_{k=1}^m \phi_k \sum_{j=1}^m [\Phi^{\dagger}]_{kj} \sum_{\ell>m} \sqrt{\lambda_\ell} \, Z_\ell \phi_\ell(X_j)
\right\Vert^2
\right]
\\
&=
\sum_{k=1}^m
\Vert \phi_k \Vert^2
\sum_{j,j'=1}^m
[\Phi^{\dagger}]_{kj}[\Phi^{\dagger}]_{kj'} \quad \times 
\\
&\qquad 
\sum_{\ell,\ell'>m} \sqrt{\lambda_\ell \lambda_{\ell'}} \,\E\left[Z_\ell Z_{\ell'}\right]
\phi_\ell(X_j) \phi_{\ell'}(X_j)
\\
&=
\sum_{k=1}^m
\Vert \phi_k \Vert^2
\sum_{j,j'=1}^m
[\Phi^{\dagger}]_{kj}[\Phi^{\dagger}]_{kj'} 
\sum_{\ell>m} \lambda_{\ell} \phi_\ell(X_j) \phi_{\ell}(X_j)
\\
&=
\sum_{\ell>m} \lambda_{\ell}
\sum_{k=1}^m
\Bigg(
\sum_{j=1}^m
[\Phi^{\dagger}]_{kj}
 \phi_\ell(X_j) 
 \Bigg)^2
\end{align*}
Denote $\phi_\ell(\bm{X}) := (\phi_\ell(X_1),\dots, \phi_\ell(X_m))$. Then we can write the last line equivalently in the form 
\[
(\Err_{\mathrm{aliasing}})^2 
=
\int_X \Vert \cD \circ \cE (P^\perp_m u) \Vert_{L^2_x}^2 \, d\mu(u)
=
\sum_{\ell>m} \lambda_{\ell} \Vert \Phi^{\dagger} \phi_\ell(\bm{X}) \Vert^2_{\ell^2},
\]
as claimed.
\end{proof}
\subsection{Proof of Lemma \ref{lem:singb}}
\label{app:pf39}
\begin{proof}

The minimal singular value of $A_M = \frac{|D|}M \Phi_M \Phi_M^T$ is given by 
\begin{align*}
\sigma_{\mathrm{min}}
\left(
A_M
\right)
&=
\inf_{\Vert v \Vert_{\ell^2}=1}
\langle v, A_M v\rangle
\\
&=
\inf_{\Vert v \Vert_{\ell^2}=1}
\left[
\langle v, \bm{1} v\rangle
-
\langle v, (\bm{1} - A_M) v\rangle
\right]
\\
&=
1
-
\sup_{\Vert v \Vert_{\ell^2}=1}
\langle v, (\bm{1} - A_M) v\rangle
\\
&=
1 - \sigma_{\mathrm{max}}\left( \bm{1}-A_M \right).
\end{align*}
The maximal singular value (=spectral norm) of $\bm{1}-A_M \in \R^{m\times m}$ can be estimated from above by 
\[
\sigma_{\mathrm{max}}\left( \bm{1}-A_M \right)
\le
\Vert
\bm{1}-A_M
\Vert_{\mathrm{F}},
\]
where we denote, for any matrix $A = (a_{ij})$, by $\Vert A \Vert_{\mathrm{F}}$ the Frobenius norm, 
\[
\Vert A \Vert_{\mathrm{F}} = 
\left(
\sum_{i,j=1}^m |a_{ij}|^2
\right)^{1/2}.
\]
Note that for all $i,j \in \{1,\dots, m\}$, we have that the $(i,j)$ entry of $\bm{1}-A_M$ is given by
\[
\delta_{ij} - \frac{|D|}M \sum_{k=1}^m \phi_i(X_k) \phi_j(X_k).
\]
With $Y_{k} = - |D|\phi_i(X_k) \phi_j(X_k)$, this can be written in the form 
\[
\frac{1}M \sum_{k=1}^m Y_k - \E[Y_k],
\]
where the $Y_k$ are iid random variables bounded by $|D|\omega_m^2$. It follows from Hoeffding's inequality that for any $\delta > 0$, we have
\[
\Prob\left[
\left|
\frac{1}M \sum_{k=1}^m Y_k - \E[Y_k]
\right|
\ge \delta
\right]
\le 
2 \exp\left(-\frac{2M^2 \delta^2}{|D|^2\omega_m^4} \right).
\]
Equivalently, for any $i,j\in \{1,\dots, m\}$ we have:
\[
\Prob 
\left[
|
\left[ 
\bm{1} - A_M
\right]_{ij}
|
\ge \delta
\right]
\le
2
\exp\left(
-\frac{2M^2 \delta^2}{|D|^2\omega_m^4}
\right).
\]
It now follows that
\begin{align*}
\Prob\left[
\Vert \bm{1} - A_M \Vert_F
\ge \delta
\right]
&=
\Prob\left[
\sum_{i,j=1}^m | (\bm{1} - A_M)_{ij} |^2
\ge \delta^2
\right]
\\
&\le
\Prob\left[
\max_{i,j=1,\dots,m} | (\bm{1} - A_M)_{ij} |
\ge \frac{\delta}{m}
\right]
\\
&\le
m^2 \max_{i,j=1,\dots,m}
\Prob\left[
 | (\bm{1} - A_M)_{ij} |
\ge \frac{\delta}{m}
\right]
\\
&\le
2m^2 \exp
\left(
-2
\left(
\frac{M\delta}{|D|\omega_m^2 m} 
\right)^2
\right)
\end{align*}
Choosing $\delta = 1/\sqrt{2}$, we have
\[
\Prob
\left[
\sigma_{\mathrm{max}}\left(
\bm{1}-A_M
\right)
\ge 1/\sqrt{2}
\right]
\le
2m^2 \exp
\left(
-
\left(
\frac{M}{|D|\omega_m^2 m} 
\right)^2
\right),
\]
and thus, from $\sigma_{\mathrm{min}}(A_M) = 1 - \sigma_{\mathrm{max}}(\bm{1}-A_M)$, also
\[
\Prob
\left[
\sigma_{\mathrm{min}}\left(
A_M
\right)
< 1 - 1/\sqrt{2}
\right]
\le
2m^2 \exp
\left(
-
\left(
\frac{M}{|D|\omega_m^2 m} 
\right)^2
\right).
\]

\end{proof}

\subsection{Proof of Theorem \ref{thm:random-enc}}
\label{app:pf301}
\begin{proof}
Let 
\[
C
=
\frac{\sqrt{2} \, \max(|D|,1)}{\sqrt{2}-1} \sup_{\ell \in \N} (1+\Vert \phi_\ell \Vert_{L^\infty}^2).
\] 
By assumption, we have $C < \infty$, and we also note that $C\ge 1$. Let $(\Omega, \P)$ be the probability space $\Omega = \prod_{\ell=1}^\infty D$, with probability measure $\P = \prod_{\ell=1}^\infty \unif(D)$, such that the iid random variables $X_1, X_2,\dots$ are given by projection onto the corresponding factor
\[
X_\ell: \Omega \to D,
\quad 
X_\ell(\omega) = \omega_\ell, 
\] 
where $\omega = (\omega_1,\omega_2, \dots, ) \in \Omega$.

By Lemma \ref{lem:aliasing-gen}, if we choose $M(m) = \lceil C \kappa m\log(m) \rceil$, with $\kappa = 4$, then we have
\begin{align} \label{eq:interm-1}
\Prob\left[
\set{
\omega 
}{
\Err_{\cE}(X_1(\omega), \dots, X_M(\omega)) \le C \textstyle \sum_{\ell > m} \lambda_\ell
}
\right]
\ge 
1 - 2m^{-2}.
\end{align}
where $\Err_{\cE} = \Err_{\cE}(X_1,\dots, X_M)$ is the random encoding error based on the sensors $X_1,\dots, X_M$. We note that asymptotically as $m\to \infty$, we have
\[
M
\sim
C \kappa m \log(m)
\gtrsim C \kappa m,
\]
and hence $\log(M) \gtrsim \log(C\kappa) + \log(m) \ge \log(m)$, where we used that $C > 1$ and $\kappa = 4$, so that $\log(C\kappa) \ge 0$ in the last estimate. In particular, this implies that
\[
M/\log(M)
\lesssim
\left[
C \kappa m \log(m)
\right]
/
\log(m)
=
C \kappa m.
\]
It follows that, by possibly enlarging the constant $C>1$, we have
\[
\sum_{\ell > M/C\log(M)} \lambda_\ell
\ge 
\sum_{\ell > m} \lambda_\ell,
\]
and using also \eqref{eq:interm-1}, we conclude that for sufficiently large $C$, we have
\[
\Prob\left[
\set{
\omega
}{
\Err_{\cE}(X_1,\dots, X_M) \le C \textstyle\sum_{\ell > M/C\log(M)} \lambda_\ell
}
\right]
\ge 1 - C \frac{\log(M)^2}{M^2},
\]
for all $M\in \N$. In particular, the probability that 
\[
\Err_{\cE}(X_1,\dots, X_M) > C \sum_{\ell > M/C\log(M)} \lambda_\ell,
\]
for infinitely many $M$, can be bounded by 
\begin{align*}
\Prob
&\Big[
\Err_{\cE}
(X_1,\dots, X_M)
> C\textstyle\sum_{\ell > M/C\log(M)} \lambda_\ell
\text{ infinitely often}
\Big]
\\
&\le 
\limsup_{M_0\to \infty}
\Prob\left[
\bigcup_{M > M_0}
\set{
\omega
}{
\Err_{\cE}(X_1(\omega),\dots, X_M(\omega)) > C \sum_{\ell > M/C\log(M)} \lambda_\ell
}
\right]
\\
&\le 
\limsup_{M_0\to \infty}
\sum_{M>M_0}
\Prob\left[
\set{
\omega
}{
\Err_{\cE}(X_1(\omega),\dots, X_M(\omega)) > C \sum_{\ell > M/C\log(M)} \lambda_\ell
}
\right]
\\
&\le 
C \limsup_{M_0\to \infty}
\sum_{M>M_0} \frac{\log(M)^2}{M^2}
\\
&= 0.
\end{align*}
Thus, for almost all $\omega \in \Omega$, we have
\[
\Err_{\cE}(X_1(\omega), \dots, X_M(\omega))
\le 
C \sum_{\ell > M/C\log(M)} \lambda_\ell,
\]
for all sufficiently large $M>0$.
\end{proof}

\subsection{Proof of Lemma \ref{lem:encoding-quadraticexp}}
\label{app:pf310}
\begin{proof}
The starting point is the claim that if the covariance operator $\Gamma$ be given by \eqref{eq:quadraticexp-kernel}, the eigenfunctions and eigenvalues $(\phi_k, \lambda_k)$ of $\Gamma$ are given by 
\begin{align} \label{eq:EVquadraticexp1}
\phi_k(x) = e^{-ikx}, \quad \lambda_k = \sqrt{2\pi} \ell \, e^{-(\ell k)^2/2}, \quad (k\in \Z).
\end{align}

To see this, we note that  
\begin{align*}
\int_0^{2\pi} 
k_p(x,x') \phi_k(x') 
\, dx'
&=
\sum_{h \in 2\pi \Z}
\int_0^{2\pi} 
e^{-(x-x'-h)^2/(2\ell^2)} e^{-ikx'} 
\, dx'
\\
&=
\int_{-\infty}^\infty
e^{-(x-x')^2/(2\ell^2)} e^{-ikx'} 
\, dx'
\\
&= 
\sqrt{2\pi} \ell \, e^{-(\ell k)^2} \, e^{-ikx},
\end{align*}
where we used that the Fourier transform on the real line $\R$ 
\[
\mathcal{F}[u](k) = \int_{-\infty}^\infty u(x') e^{-ikx'} \, dx',
\]
satisfies
\[
\mathcal{F}[u(\slot - x)][k] = \mathcal{F}[u](k) \, e^{-ikx},
\]
and we recall that for $\alpha > 0$, the Fourier transform of a Gaussian is
\[
\mathcal{F}[\exp(-\alpha x^2)](k) = \sqrt{\frac{\pi}{\alpha}} \exp(-k^2/4\alpha).
\]

For simplicity, assume $m = 2 K + 1$ for $K\in \mathbb{N}$. Recall that  the decoder $\cD$ for encoder $\cE$ is the discrete Fourier transform \eqref{eq:fdecode1}, it is straightforward to check that $\cE \circ \cD = \Id$, so $(\cE,\cD)$ is an admissible encoder/decoder pair, and by the definition \eqref{eq:encoding} of $\Err_\cE$:
\[
(\Err_{\cE})^2 \le \int_X \left\Vert \cD \circ \cE - \Id \right\Vert_{L^2_x}^2 \, d\mu(u).
\]
Let $P_K: L^2_x \to L^2_x$ denote the orthogonal projection onto $\Span(e^{ikx}; \, |k|\le K)$, and denote by $P_K^\perp: L^2_x \to L^2_x$ the orthogonal projection onto the orthogonal complement, so that $\Id = P_K + P_K^\perp$. 
We note that $\cD \circ \cE$ is \emph{linear} and we have $(\cD \circ \cE)(e^{ikx}) = e^{ikx}$ for all $|k|\le K$ (where $m = 2K + 1$). From this, it follows that 
\begin{align*}
\cD \circ \cE - \Id
&=
\left[\cD \circ \cE \circ P_K - P_K\right]
+
\left[\cD \circ \cE \circ P_K^\perp - P_K^\perp\right]
\\
&=
\cD \circ \cE \circ P_K^\perp - P_K^\perp
\\
&=
P_K \circ \cD \circ \cE \circ P_K^\perp - P_K^\perp
,
\end{align*}
where we used that $\cD = P_K \circ \cD$. The two terms on the last line are obviously perpendicular to each other. Hence
\begin{align} \label{eq:quadraticexp-errordecomp}
\Vert \cD \circ \cE - \Id \Vert_{L^2_x}^2
=
\Vert \cD \circ \cE \circ P_K^\perp \Vert_{L^2_x}^2
+
\Vert P_K^\perp \Vert_{L^2_x}^2.
\end{align}
We note that the non-vanishing of the first term in \ref{eq:quadraticexp-errordecomp} is due to \emph{aliasing}, i.e. the fact that for any point on the grid $x_j$, $j=1, \dots, m$, we have
\[
e^{ikx_j} 
=
e^{i(k+qm)x_j},
\quad \text{for all } q \in \Z.
\]
Therefore, the higher-order modes $e^{ikx}$ for $|k| > K$ map under $\cE$ to 
\[
\cE(e^{ikx}) = \cE(e^{i k_0 x}),
\]
where $k_0\in \{-K,\dots, K\}$ is the unique value such that there exists $q\in \Z$ with $k_0 = k + qm$. 

As the functions $x \mapsto e^{ikx}$, $k\in \Z$, are the eigenfunctions of the covariance operator of $\mu$. Let $\lambda_k$, $k\in \Z$, denote the corresponding eigenvalues. By the Karhunen-Loeve expansion for the Gaussian measure $\mu$, we can now write
\begin{align*}
\int_{X} 
\Vert \cD \circ \cE - \Id \Vert_{L^2_x}^2
\, d\mu
&=
\E\left[
\Vert \cD \circ \cE(X) - X \Vert_{L^2_x}^2
\right]
\\
&=
\E\left[
\Vert \cD \circ \cE \left( P_K^\perp X \right) 
\Vert_{L^2_x}^2
\right]
+
\E\left[
\Vert P_K^\perp (X) 
\Vert_{L^2_x}^2
\right]
,
\end{align*}
where 
\[
X
=
\sum_{k\in \Z} 
\sqrt{\lambda_k} X_k e^{ikx},
\]
and the $X_k \sim \Normal(0,1)$ are iid Gaussian random variables with unit variance. Due to aliasing, we have
\[
\cD \circ \cE\left(
P_K^\perp X
\right)
=
\sum_{k_0 = -K}^K
\left(
\sum_{q\in \Z \setminus \{0\}} \sqrt{\lambda_{k_0 + qm}} X_{k_0 + qm}
\right) e^{ik_0 x}.
\]
And 
\begin{align*}
\E\left[
\Vert \cD \circ \cE \left( P_K^\perp X \right) 
\Vert_{L^2_x}^2
\right]
&=
\E\left[
\sum_{k_0 = -K}^K
\left(
\sum_{q\in \Z \setminus \{0\}} \sqrt{\lambda_{k_0 + qm}} X_{k_0 + qm}
\right)^2
\right]
\\
&=
\sum_{k_0 = -K}^K
\E\left[
\sum_{q,q'\in \Z \setminus \{0\}} \sqrt{\lambda_{k_0 + qm}}\sqrt{\lambda_{k_0 + q'm}} X_{k_0 + qm}  X_{k_0 + q'm}
\right]
\\
&=
\sum_{k_0 = -K}^K
\sum_{q,q'\in \Z \setminus \{0\}} \sqrt{\lambda_{k_0 + qm}}\sqrt{\lambda_{k_0 + q'm}} 
\E\left[
X_{k_0 + qm}  X_{k_0 + q'm}
\right].
\end{align*}
As the $X_k$ are iid with zero mean and unit variance, we have
\[
\E[X_k X_{k'}] = \delta_{k,k'} = \begin{cases} 1, & k=k' \\ 0, & k\ne k' \end{cases}.
\]
Thus,
\begin{align*}
\E\left[
\Vert \cD \circ \cE \left( P_K^\perp X \right) 
\Vert_{L^2_x}^2
\right]
&=
\sum_{k_0 = -K}^K
\sum_{q,q'\in \Z \setminus \{0\}} \sqrt{\lambda_{k_0 + qm}}\sqrt{\lambda_{k_0 + q'm}} \delta_{q,q'}
\\
&=
\sum_{k_0 = -K}^K
\sum_{q\in \Z \setminus \{0\}} \lambda_{k_0 + qm}
\\
&=
\sum_{|k| > K}
\lambda_{k}.
\end{align*}
On the other hand, it is easy to see that
\[
\E\left[
\Vert  P_K^\perp X 
\Vert_{L^2_x}^2
\right]
=
\sum_{|k| > K} \lambda_k.
\]
We thus conclude that
\begin{align*}
\int_X \Vert \cD \circ \cE - \Id \Vert_{L^2_x}^2 \, d\mu
&=
\E\left[
\Vert \cD \circ \cE \left( P_K^\perp X \right) 
\Vert_{L^2_x}^2
\right]
+
\E\left[
\Vert P_K^\perp (X) 
\Vert_{L^2_x}^2
\right]
\\
&=
2\sum_{|k|>K} \lambda_k.
\end{align*}
Recalling that by \eqref{eq:EVquadraticexp1}, we have
\[
\lambda_k = \sqrt{2\pi} \ell \, e^{-(\ell k)^2/2},
\]
we finally obtain
\begin{align*}
2\sum_{|k|>K} \lambda_k
&=
2\sqrt{2\pi}  \sum_{|k|>K} \ell \, e^{-(\ell k)^2/2}
\\
&\le 
2\pi\sqrt{2} \frac{2}{\sqrt{\pi}} \int_K^{\infty} e^{-(\ell x)^2/2} \, \ell \, dx
\\
&=
4\pi \left[ \frac{2}{\sqrt{\pi}} \int_K^{\infty} e^{-(\ell x/\sqrt{2})^2} \, d\left( \ell x / \sqrt{2}\right) \right]
\\
&=
4\pi \, \erfc\left( \frac{\ell K}{\sqrt{2}} \right).
\end{align*}

\end{proof}

\subsection{Proof of Lemma \ref{lem:311}}
\label{app:pf311}

\begin{proof}
By Lemma \ref{lem:encoding-quadraticexp}, the eigenfunctions of the covariance operator are $\phi_k = \fb_k$, the standard Fourier basis. By Theorem \ref{thm:random-enc}, there exists a constant $C\ge 1$, depending on $|D|$ and $\sup_{k\in \N} \Vert \phi_k \Vert_{L^\infty} \le 2(2\pi)^{-d}$, such that with probability $1$ in the random sensors $X_1,X_2,X_3,\dots \in D$, we have for almost all $M\in \N$:
\[
\Err_{\cE} 
\le 
C \sqrt{\sum_{k> M/C\log(M)} \lambda_k}.
\]
By Lemma \ref{lem:encoding-quadraticexp} (cp. \eqref{eq:enerrg}, \eqref{eq:enerrg1}), we have
\[
\sqrt{\sum_{k> M/C\log(M)} \lambda_k}
\lesssim
\exp\left(-\tilde{\gamma} \frac{M^2\ell^2}{C^2\log(M)^2} \right), \quad \forall \tilde{\gamma} < \frac 1{16}.
\]
The implied constant here only depends on the value of $\tilde{\gamma}$ and on $|D|$. Thus, choosing e.g. $\tilde{\gamma} = 1/20$ and noting that $|D|=2\pi$ is a fixed constant, we conclude that 
\[
\Err_{\cE}
\le
C \exp\left(-\gamma \frac{M^2\ell^2}{\log(M)^2} \right),
\]
for $\gamma = 1/(20 C^2)$, where the constant $C>0$ is independent of $\ell$ and $M$.
\end{proof}

\subsection{Proof of Proposition \ref{prop:ps-spectral}}
\label{app:pf312}
\begin{proof}
Fix $u = u(\slot; Y)$ for some $Y \in [-1,1]^\cJ$. Let $\hat{Y} = \hat{Y}(\cE(u))$ be given by \eqref{eq:Yhat}. Let $\tilde{Y}$ be given by \eqref{eq:Ytilde}, such that
\[
\hat{Y}_j = \shrink(\tilde{Y}_j), \quad
\forall \, j\in \cJ = \Z^d.
\]
We note that since $Y_j \in [-1,1]$ for all $j\in \cJ$, we have
\begin{align*}
\left|
Y_j - \hat{Y}_j
\right|
=
\left|
Y_j - \shrink(\tilde{Y}_j)
\right|
\le
\left|
Y_j - \tilde{Y}_j
\right|,
\quad
\forall \, j\in \cJ = \Z^d.
\end{align*}
But, by definition of $\tilde{Y}_j$, we have that
\[
u \mapsto 
u(\slot; \tilde{Y}(\cE(u)))
=
\sum_{j\in \Z^d} \tilde{Y}_j(\cE(u)) \alpha_j \fb_j
=
\sum_{j\in \Z^d} \hat{u}_j \fb_j,
\]
is the pseudo-spectral Fourier projection of $u$ onto $\{\fb_j\}_{j\in \cK_N}$. In particular, it follows from standard estimates for the pseudo-spectral projection that for $u \in H^s(\T^d)$, with $s > n/2$, we have
\[
\Big
\Vert 
u(\slot; Y) - u(\slot;\tilde{Y}(\cE(u))
\Big\Vert_{L^2}
\le 
C
\Vert u \Vert_{H^s} \, N^{-s},
\]
for some $ C = C(s) > 0$,
and hence also
\begin{align*}
\left\Vert u(\slot; Y) - u(\slot;\hat{Y}) \right\Vert_{L^2}^2
&=
\sum_{j\in \Z^d} \left|Y_j - \hat{Y}_j\right|^2 \alpha_j^2 
\\
&\le 
\sum_{j\in \Z^d} \left|Y_j - \tilde{Y}_j\right|^2 \alpha_j^2
\\
&=
\Big
\Vert 
u(\slot; Y) - u(\slot;\tilde{Y}
\Big\Vert_{L^2}^2
\\
&\le 
C
\Vert u \Vert_{H^s}
\, 
N^{-s}.
\end{align*}
\end{proof}

\subsection{Proof of Theorem \ref{thm:holomorphic-encoding}}
\label{app:pf313}
\begin{proof}
We note that since $m = (2N+1)^d$, we have  $N \gtrsim m^{1/d}$. Thus, it suffices to show that $\Err_{\cE} \le C \exp(-c N)$ for some constants $c,C$ independent of $N$. We now note that if $Y = (Y_j)_{j\in \cJ}$ are distributed according to the probability measure $\rho \in \cP([-1,1]^\cJ)$, and if $\hat{Y} = \hat{Y}(\cE(u)$ and $\tilde{Y} = \tilde{Y}(\cE(u))$ are given by \eqref{eq:Yhat}, \eqref{eq:Ytilde}, respectively, then
\begin{align*}
(\Err_{\cE})^2
&=
\int_{L^2(D)}
\Vert \cD \circ \cE(u) - u \Vert_{L^2}^2 
\, d\mu(u)
\\
&=
\int_{L^2(D)}
\Vert u(\slot; \hat{Y}) - u(\slot; Y) \Vert_{L^2}^2 
\, d\rho(Y)
\end{align*}
As in the proof of Proposition \ref{prop:ps-spectral} in appendix \ref{app:pf312}, we see that -- due to the exponential decay $\lesssim \exp(-|k|\ell)$ of the Fourier coefficients of $u(\slot; Y)$ -- there exist $C,c>0$, independent of $N$ and $Y\in [-1,1]^\cJ$, such that the last term can be estimated by
\begin{align*}
&\le 
\int_{L^2(D)} 
\left[C \exp(-c\ell N)\right]^2\, d\rho(Y)
\\
&=
\left[C \exp(-c\ell N)\right]^2.
\end{align*}
This shows that $\Err_{\cE} \le C \exp(-c\ell N)$. And we recall that $N \lesssim m^{1/d}$, by definition of $m = (2N+1)^d$.

Finally, if 
\[
\cF(\bm{y}) = \G(u(\slot; \bm{y})),
\]
then for any $\bm{u} = (u_{i})_{i\in \cI_N}$, we have by the definition of $\cD(\bm{u}) = u(\slot; \hat{Y}(\bm{u}))$ (cp. \eqref{eq:holomorphic-decoder})
\[
\G \circ \cD( \bm{u} )
=
\G\left( u(\slot, \hat{Y}(\bm{u}))\right)
=
\cF(\hat{Y}(\bm{u})).
\]
\end{proof}

\subsection{Proof of Corollary \ref{cor:holomorphic-NN}}
\label{app:pf314}
\begin{proof}
Let $\hN$ be a neural network satisfying the estimates of Proposition \ref{prop:be-expansion}, with $N:= m$. Then 
\[
\cN := \cP \circ \hN 
=
\sum_{\bm{\nu}\in \Lambda_{N}}
\underbrace{(\cP c_{\bm{\nu}})}_{\in \R^p} \hN_{\bm{\nu}}(y_{\kappa(1)}, \dots, y_{\kappa(N)})
\]
is a linear combination of the neural network output of $\hN$, with coefficients in $\R^p$. In particular, by adding a linear output layer of size $\mathcal{O}(pN)$ to the network $\hN$, we can represent $\cN$ as a neural network with 
\[
\size(\cN) \le \size(\hN) + pN,
\quad
\depth(\cN) \le \depth(\hN) + 1.
\]
The claimed bounds on the size of $\cN$ thus readily follow from the corresponding bounds for $\hN$. Furthermore, we have for any $\bm{y} \in [-1,1]^\cJ$:
\begin{align*}
\left\Vert
\cP \circ \cF(\bm{y})
-
\cN(y_{\kappa(1)}, \dots, y_{\kappa(N)}
\right\Vert_{\ell^2}
&=
\left\Vert
\cP \circ \cF(\bm{y})
-
\cP \circ \hN(y_{\kappa(1)}, \dots, y_{\kappa(N)}
\right\Vert_{\ell^2}
\\
&\le 
\Vert \cP \Vert
\left\Vert
\cF(\bm{y})
-
\hN(y_{\kappa(1)}, \dots, y_{\kappa(N)}
\right\Vert_{V}.
\end{align*}
The claimed error estimate thus follows from the error estimate in Proposition \ref{prop:be-expansion}.
\end{proof}
\subsection{Proof of Theorem \ref{thm:holomorphic-approx}}
\label{app:pf315}
\begin{proof}
The proof of this theorem requires the following simple Lemma on Neural Network Calculus,
\begin{lemma} \label{lem:composition-NN}
Let $\cN_1: \R^{n_0} \to \R^{n_1}$, $\cN_2: \R^{n_1} \to \R^{n_2}$ be two neural networks. Then the composition $\cN = \cN_2 \circ \cN_1: \R^{n_0} \to \R^{n_2}$ can be represented by a neural network $\cN$ with 
\[
\size(\cN) = \size(\cN_1) + \size(\cN_2),
\quad
\depth(\cN) = \depth(\cN_1) + \depth(\cN_2).
\]
\end{lemma}
\subsubsection{Proof of the Theorem \ref{thm:holomorphic-approx}}
From the network size bounds of Corollary \ref{cor:holomorphic-NN}, $\cN$ is a neural network of size 
\[
\begin{gathered}
\size(\cN) \le C(1+mp\log(m)\log\log(m))), 
\\
\depth(\cN) \le C(1+\log(m) \log\log(m))).
\end{gathered}
\]
with a constant $C>0$, independent of $m$, $p$. By Lemma \ref{lem:Yhat-NN}, the map $\bm{u}\mapsto \hat{Y}_{\kappa}(\bm{u})$ can be represented by a neural network of size
\[
\size(\hat{Y}_{\kappa}) \le C(1+m\log(m)), 
\quad
\depth(\hat{Y}_{\kappa}) \le C(1+\log(m)),
\]
for a constant $C>0$, independent of $m$. Since the definition of $\hat{Y}_{\kappa}$ does not involve the projection $\cP$ at all, the constant is also independent of $p$. By the composition Lemma \ref{lem:composition-NN} it follows that $\bm{u} \mapsto \cA(\bm{u}) = (\cN \circ \hat{Y}_\kappa)(\bm{u})$ can be represented by a neural network with
\[
\size(\cA) = \size(\cN) + \size(\hat{Y}_{\kappa}), 
\, 
\depth(\cA) = \depth(\cN) + \depth(\hat{Y}_{\kappa}).
\]
We also note that the following estimate for the approximation error $\Err_{\cA}$: 
\begin{align*}
(\Err_{\cA})^2
&=
\int_{L^2(D)}
\Vert 
\cP \circ \cG \circ \cD (\bm{u}) 
-
\cA(\bm{u})
\Vert^2_{\ell^2} \, d(\cE_\#\mu)(\bm{u})
\\
&\le
\sup_{\bm{u} \in \supp(\mu)}
\Vert 
\cP \circ \cG \circ \cD (\bm{u}) 
-
\cA(\bm{u})
\Vert^2_{\ell^2}
\\
&=
\sup_{\bm{u} \in \supp(\mu)}
\Vert 
\cP \circ \cF (\hat{Y}(\bm{u})) 
-
\cN(\hat{Y}(\bm{u}))
\Vert^2_{\ell^2}
\end{align*}
By Corollary \ref{cor:holomorphic-NN} we can further estimate the last term by
\begin{align*}
&\le
\sup_{\bm{y}\in [-1,1]^\cJ}
\Vert \cP\circ\cF(\bm{y}) - \cN(y_{\kappa(1)}, \dots, y_{\kappa(m)})\Vert_{\ell^2}^2
\\
&\le
C \Vert \cP \Vert \,  m^{-2s},
\quad
\text{where } s := \frac1q - 1 > 0,
\end{align*}
provided that $\alpha_k \in \ell^q(\Z^d)$ with $q\in (0,1)$. But by the exponential decay assumption \eqref{eq:law-decay}, we have $\alpha_k \in \ell^q(\Z^d)$ for any $q\in (0,1)$. The claim follows.
\end{proof}

\section{DeepONet approximation of linear operators}
\label{app:deeponet-linear}
In this section, we will illustrate the ability of a DeepONet to approximate bounded linear operators efficienty. A simple numerical example of the DeepONet approximation of a linear functional $\G: C([0,1]^2) \to \R$ is presented in Figure \ref{fig:quadrature}, below. 

\begin{figure}[H]

\begin{subfigure}{.45\textwidth}
\includegraphics[width=\textwidth]{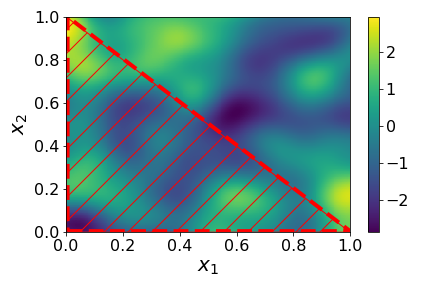}
\caption{}
\label{fig:1a}
\end{subfigure}
\begin{subfigure}{.45\textwidth}
\includegraphics[width=\textwidth]{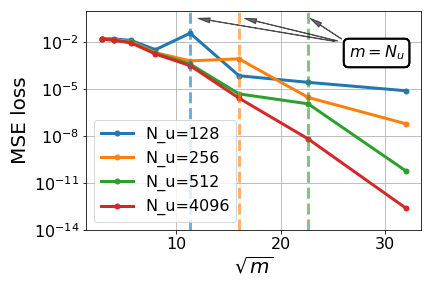}
\caption{}
\label{fig:1b}
\end{subfigure}
\\
\begin{subfigure}{.45\textwidth}
\includegraphics[width=\textwidth]{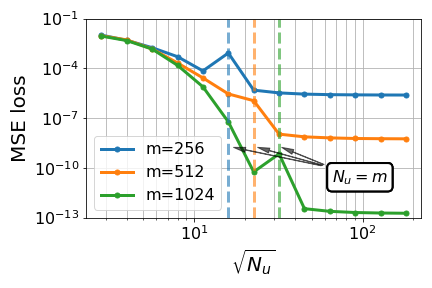}
\caption{}
\label{fig:1c}
\end{subfigure}
\begin{subfigure}{.45\textwidth}
\includegraphics[width=\textwidth]{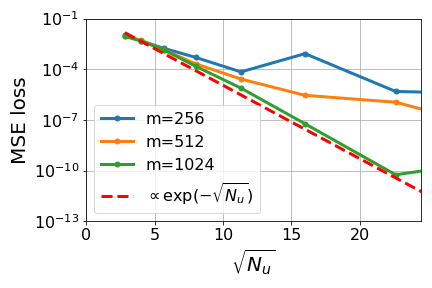}
\caption{}
\label{fig:1d}
\end{subfigure}

\caption{Illustration of a DeepONet (\eqref{eq:bnet1}-\eqref{eq:donet1}, with $p=1$ and $\sigma(x)=x$) approximating the operator $\G: C([0,1]^2) \to \R$, defined as $\G(u) := \int_{\Delta} u(x) \, dx$, with integration domain $\Delta$ (shown in red in (A)), with underlying measure being the law of a Gaussian random field with covariance kernel $k(x,y) = \exp(-|x-y|^2/2\ell^2)$, with $\ell=0.1$. $m$ sensor points $x_1,\dots, x_m$ are drawn at random from the uniform distribution on $[0,1]^2$ and the DeepONet is trained by minimizing the mean square loss function, with respect to $N_u$ samples, drawn from the underlying measure. (A) Illustration of a typical sample drawn from the Gaussian random field. (B) Convergence of the test mean squared error (computed with respect to $10'240$ test samples) as $m\to \infty$, for different numbers of training samples $N_u$. We observe a clear exponential decay of error wrt $m$ as well as ``resonances'' at $m = N_u$ (double descent). (C) Convergence of the test MSE as $N_u \to \infty$, for different numbers of random sensor points $x_1,\dots, x_m$. (D) semilog plot of the MSE loss as $N_u\to \infty$ where one observes that the test error decays as $\exp(-\sqrt{N_u})$ when $N_u \leq m$. Note that similar behavior of the error is seen for more complicated operators in \cite{deeponets}.}
\label{fig:quadrature}
\end{figure}

Of particular interest is the observed exponential decay in the DeepONet approximation error as the number of sensors $m \to \infty$, which has also been observed for other (linear and even non-linear) problems in \cite{deeponets}. Can the theoretical framework developed in the present work explain such behaviour? To answer this question, a general error estimate for the DeepONet approximation of linear operators is derived in Theorem \ref{thm:deeponet-linear}. The estimate is then applied to a prototypical elliptic PDE (cf. example \ref{ex:linear}). In the following, we consider the following setup:

\begin{setup} \label{setup:linear}
We consider data for the DeepONet approximation problem $\mu$, $\G$ (cp. Definition \ref{def:data}), where 
\begin{itemize}
\item $\G: L^2(D) \to L^2(U)$ is a bounded linear mapping,
\item $\mu\in \P_2(L^2(D))$ is a probability measure with mean $0$, and with uniformly bounded eigenfunctions of the covariance operator $\Gamma_\mu$,
\item for $m\in \N$, the sensors $x_1,\dots, x_m \sim \Unif(D)$ are drawn iid random, 
\item for $p\in \N$, we denote by $\hat{\tr}_k$, $k=0,\dots, p$, the optimal choice for an affine reconstruction $\cR_{\hat{\bm{\tr}}} = \cR_{\mathrm{opt}}$ for the push-forward measure $\G_\#\mu$, as in Theorem \ref{thm:rec-opt}; i.e., in the present case, $\hat{\tr}_0 \equiv 0$, and $\hat{\tr}_k$, $k=1,\dots, p$, are the first $p$ eigenfunctions of the covariance operator $\Gamma_{\G_\#\mu}$.
\end{itemize}
\end{setup}

Under the above assumptions, we then have

\begin{theorem} \label{thm:deeponet-linear}
Consider the setup \ref{setup:linear}. Let $m,p\in \N$ denote the number of sensors and the output dimension of the branch/trunk nets $(\bm{\br},\bm{\tr})$, respectively. Let $\bm{\tr}$ be a trunk-net approximation of $\hat{\bm{\tr}}$, such that the associated reconstruction $\cR = \cR_{\bm{\tr}}$ and projection $\cP$ satisfy $\Lip(\cR),\, \Lip(\cR\circ \cP) \le 2$.  For $m\in \N$, we define the (random) encoder $\cE: \, u \mapsto (u(x_1),\dots, u(x_m))$. Then, with probability $1$ in the choice of the random sensor points, there exists a constant $C>0$, depending only on $\mu$ and the measure of the domain $|D|$, such that for any $m,p\in \N$ there exists a shallow ReLU approximator net $\cA: \R^m \to \R^p$ with 
\[
\size({\cA}) \le 2(2+m)p,
\quad
\depth(\cA) \le 1,
\]
and such that the DeepONet $\cN = \cR_{\bm{\tr}} \circ \cA \circ \cE$, with branch net $\bm{\br} = \cA\circ \cE$ and trunk net $\bm{\tr}$ satisfies the following asymptotic DeepONet approximation error estimate
\begin{align}\label{eq:lin-bound}
\Err 
\le
C \sqrt{1+\Vert \G\Vert^2} 
\left\{
\max_{k=0,\dots, p} \Vert \tr_k - \hat{\tr}_k \Vert_{L^2(U)}
+
\sqrt{
\sum_{\ell > p} \lambda_\ell
}
+
\sqrt{
\sum_{\ell > \frac{m}{C\log(m)}} \lambda_\ell
}
\right\},
\end{align}
for almost all $m$, as $m\to \infty$.
\end{theorem}

\begin{proof}
By the DeepONet error decomposition of Theorem \ref{thm:err-upper-bound} and the assumed Lipschitz bounds $\Lip(\cR),\, \Lip(\cR \circ \cP) \le 2$, we have
\[
\Err \le 2\Vert \G \Vert  \Err_{\cE} + 2 \Err_{\cA} + \Err_{\cR}.
\]
We first observe that for any choice of the (affine) encoder/decoder and reconstruction/projection pairs $(\cE,\cD)$, $(\cR,\cP)$, and for a linear mapping $\G$ there exists an \emph{exact, affine approximator} $\cA: \R^m \to \R^p$, such that $\cA(\bm{u}) = \cP \circ \G \circ \cD(\bm{u})$ for all $\bm{u} \in \R^m$. Furthermore, $\cA$ can be represented by a shallow ReLU neural net of the claimed size, on account of the fact that $Ax + b = \sigma(Ax+b) - \sigma(-(Ax+b))$ has an exact representation for the ReLU activation function $\sigma(x) = \max(x,0)$. Thus, the approximation error $\Err_{\cA}$ can be made to vanish in this case, $\Err_{\cA} = 0$.

Under the assumptions of this theorem, the random encoding error has been estimated in Theorem \ref{thm:random-enc}, where it is shown that there exists a constant $C\ge 1$, depending only on the uniform upper bound of the eigenfunctions of $\Gamma_\mu$, and on $|D|$, such that (for almost all $m\in \N$):
\begin{align} \label{eq:Ee}
\Err_{\cE} \le C \sqrt{\sum_{\ell > m/C\log(m)} \lambda_\ell},
\end{align}
as $m\to \infty$, with probability $1$ in the iid random sensors $x_1,x_2,\dots\sim\unif(D)$. Here, $C = C(|D|,\mu)$ is a constant.

Finally, we estimate the reconstruction error $\Err_{\cR}$: By Proposition \ref{prop:err-rec-approx}, we have
\[
\Err_{\cR} 
\le
\sqrt{1+\Tr(\Gamma_{\G_\#\mu})}
\max_{k=0,\dots, p} \Vert \hat{\tr}_k - \tr_k \Vert_{L^2(U)}
+
\sqrt{
\sum_{k>p} \lambda_{k}^{\G_\#\mu}
},
\]
where $\lambda^{\G_\#\mu}_1\ge \lambda^{\G_\#\mu}_2 \ge \dots$ denote the eigenvalues of the covariance operator $\Gamma_{\G_\#\mu}$ of the push-forward measure $\G_\#\mu$. By Proposition \ref{prop:lin}, we can estimate
\[
\sum_{k>p} \lambda_k^{\G_\#\mu}
\le
\Vert \G\Vert^2 \sum_{k>p} \lambda_k,
\]
in terms of the eigenvalues $\lambda_1\ge \lambda_2 \ge \dots$, of the covariance operator $\Gamma_\mu$ of $\mu$, and in particular $\Tr(\Gamma_{\G_\#\mu}) \le\Vert \G\Vert^2 \Tr(\Gamma_\mu) \le C \Vert \G\Vert^2$, where $C = C(\mu)$ depends only on $\mu$. It thus, follows that there exists a constant $C = C(\mu)$, such that
\begin{align}\label{eq:Er}
\Err_{\cR}
\le
C \sqrt{1+\Vert \G\Vert^2}
\left\{
\max_{k=0,\dots, p} \Vert \hat{\tr}_k - \tr_k \Vert_{L^2(U)}
+
\sqrt{\sum_{\ell>p} \lambda_\ell}
\right\}.
\end{align}
Combining \eqref{eq:Ee} and \eqref{eq:Er} with the error decomposition, and taking into account that $\Err_{\cA} = 0$, yields the desired result.
\end{proof}

\begin{remark}
From the proof of Theorem \ref{thm:deeponet-linear}, it is clear that to achieve the error bound \eqref{eq:lin-bound}, we can replace the shallow ReLU approximator net $\cA$, by a single-layer affine approximator $\cA: \R^m \to \R^p$, $\cA(\bm{u}) := A\cdot \bm{u} + b$, where $A \in \R^{p\times m}$, $b\in \R^m$. This clearly corresponds to the use of the linear activation function $\sigma(x) = x$ in the branch net. In contrast, the activation function in the trunk net must be kept non-linear to be able to approximate the optimal trunk net $\hat{\bm{\tr}}$.
\end{remark}

\begin{example} \label{ex:linear}
To illustrate Theorem \ref{thm:deeponet-linear}, we consider the following example of a linear operator $\G: L^2(\T^d) \to L^2(\T^d)$, $f\mapsto v$, mapping the source term $f$ to the solution of the PDE $\Delta v = f - \fint_{\T^d} f \, dx$, with periodic boundary conditions and imposing $\fint_{\T^d} v \, dx = 0$. It is well-known that this operator is bounded; in fact, we have $\Vert \G \Vert \le 1$. We fix the initial measure $\mu \in \P(L^2(\T^d))$ as a Gaussian random field, with Karhunen-Loeve expansion 
\[
f = \sum_{k\in \Z^d} \alpha_{k} X_k \fb_k,
\]
where $|\alpha_k|\le \exp(-\ell |k|)$ have exponential decay with typical length scale $\ell > 0$, $\{\fb_k\}_{k\in \Z^d}$ is the standard Fourier basis on $\T^d$, and the $X_k\sim \cN(0,1)$ are iid Gaussian random variables. In this case, the eigenfunctions of the associated covariance operator $\Gamma_\mu$ of $\mu$ are given by the $\fb_{k}$, with corresponding eigenvalues $\alpha_k$, $k\in \Z^d$. In particular, it follows that there exists constants $C, c>0$, depending only on $d$ and $\ell$, such that the last two terms in \eqref{eq:lin-bound} can be bounded from above by 
\[
\le C\exp\left(-c\, p^{1/d}\right) + 
C \exp\left(-\frac{c\,m^{1/d}}{\log(m)^{1/d}}\right).
\]
In particular, Theorem \ref{thm:deeponet-linear} implies that for any fixed $\sigma > 0$, and with $p \sim \log(\epsilon^{-1})^d$, $m \sim \log(\epsilon^{-1})^{d(1+\sigma)}$, we can achieve an overall DeepONet approximation error 
\[
\Err \lesssim \epsilon + \max_{j=0,\dots, p} \Vert \fb_j - \tr_j \Vert_{L^2(U)},
\]
where we have taken into account that the Fourier basis $\fb_j$ is an eigenbasis also for the push-forward measure $\G_\#\mu$. Furthermore, it follows from Lemma \ref{lem:Fourier-NN}, that for any fixed $\sigma > 0$, there exists a trunk net $\bm{\tr}$ with asymptotic size (as $\epsilon \to 0$)
\begin{align} \label{eq:lin-tr}
\size(\bm{\tr}) \lesssim \log(\epsilon^{-1})^{d+2+\sigma},
\quad
\depth(\bm{\tr}) \lesssim \log(\epsilon^{-1})^{2+\sigma},
\end{align}
and such that $\max_{k=0,\dots, p} \Vert \fb_j - \tr_j\Vert_{L^2(U)} < \epsilon$. Thus, for the present example, an overall DeepONet error $\Err\lesssim \epsilon$ can be achieved with a DeepONet $(\bm{\br},\bm{\tr})$ with trunk net $\bm{\tr}$ satisfying the size bounds \eqref{eq:lin-tr}, and a branch net $\bm{\br}$ of size
\begin{align} \label{eq:lin-br}
\size(\bm{\br}) \lesssim \log(\epsilon^{-1})^{2d+\sigma},
\quad
\depth(\bm{\br}) \lesssim 1,
\end{align}
for any fixed $\sigma > 0$. The implied constants here depend on the length scale $\ell>0$, the dimension $d$ and the additional parameter $\sigma > 0$, which was introduced to avoid the appearance of multiple logarithms. 
\end{example}

\section{Proofs of Results in Section \ref{sec:4}}
\subsection{Proof of Lemma \ref{lem:pend1}}
\label{app:pf41}
\begin{proof}
Let $v,v'$ solve \eqref{eq:pendulum} with forcing $u,u'$, respectively. Then we can write
\begin{align}\label{eq:difference}
\frac{d(v-v')}{dt}
=
G(v,v') (v-v') + (u-u'),
\end{align}
where
\[
G(v,v') = \int_{0}^{1} g^{(1)}(s v + (1-s) v') \, ds,
\]
so that $|G(v,v')| \le \Vert g^{(1)} \Vert_{L^\infty}$ for all $v,v'$. It follows readily from \eqref{eq:difference} that 
\[
\frac{d}{dt} |v-v'|^2 
\le
C |v-v'|^2 + |u-u'|^2,
\]
for some $C>0$ depending only on $\Vert g^{(1)} \Vert_{L^\infty}$. Gronwall's inequality then implies that 
\[
|v-v'|^2(t) 
\le
\int_0^t |u-u'|^2 \, ds \, e^{Ct}
\le
\Vert u -u'\Vert^2_{L^2([0,T])} \, e^{CT}.
\]
The claim follows by integration over $t\in [0,T]$.
\end{proof}
\subsection{Proof of Lemma \ref{lem:pend2}}
\label{app:pf42}
\begin{proof}
We first show that there exists a constant $C>0$, depending only on the final time $T$, such that 
\[
\Vert v \Vert_{L^\infty}
\le 
C \Vert u \Vert_{L^2}.
\]
To this end, we simply note that integrating \eqref{eq:pendulum} from $0$ to $T$ and taking into account that $v(0) = 0$, we have
\begin{align*}
|v(t)|
&\le
\int_0^t |g(v(s))| \, ds + \int_0^t |u(s)| \, ds
\\
&\le
\Vert g^{(1)} \Vert_{L^\infty} \int_0^t |v(s)| \, ds + \sqrt{T} \Vert u \Vert_{L^2([0,T])}
\\
&\le
C_1 \int_0^t |v(s)| \, ds + \sqrt{T} \Vert u \Vert_{L^2([0,T])}.
\end{align*}
Gronwall's inequality implies that
\[
|v(t)| \le \sqrt{T} \Vert u \Vert_{L^2([0,T])} e^{C_1 T}, \quad \forall \, t \in [0,T].
\]
Thus, $\Vert v \Vert_{L^\infty([0,T])} \le C \Vert u \Vert_{L^2([0,T]}$, where $C = \sqrt{T} \exp(C_1 T)$ depends only on $T$. For general $k\in \N$, we take $k$ derivatives of \eqref{eq:pendulum} to find:
\begin{align} \label{eq:higher-deriv}
\frac{d}{dt} v^{(k)} 
=
g^{(k)}(v) \, v^{(k)} + P_k\left(\{g^{(\ell)}\}_{\ell=1}^{k-1}, \{v^{\ell}\}_{\ell=1}^{k-1}\right) + u^{(k)}.
\end{align}
Here, $P_k = P_k\left(\{g^{(\ell)}\}_{\ell=1}^{k-1},\{v^{\ell}\}_{\ell=1}^{k-1}\right)$ is a polynomial of the following form
\[
P_k\left(\{g^{(\ell)}\}_{\ell=1}^{k-1}, \{v^{\ell}\}_{\ell=1}^{k-1}\right)
= 
\sum_{\ell=1}^{k-1} g^{(\ell)}(v)
\left(
\sum_{\bm{\gamma}} C_{\ell,\bm{\gamma}}  \prod_{j=1}^{N_{\bm{\gamma}}} v^{(\gamma_j)}
\right).
\]
The sum in the parentheses is over the (finite) set of $\bm{\gamma} = (\gamma_1, \dots, \gamma_{N_{\bm{\gamma}}})$, $\gamma_j \in \N$, satisfying
\[
\sum_{j=1}^{N_{\bm{\gamma}}}
\gamma_{j} = k,
\quad
1 \le \gamma_{j} < k, \quad \forall \, j=1, \dots, N_{\bm{\gamma}}.
\]
The coefficients $C_{\ell,\bm{\gamma}}$ are combinatorial coefficients that depend only on $k$, and which can in principle be determined for any given $k$. We will prove the claimed estimate on $\Vert v^{(k)} \Vert_{L^\infty}$ by induction on $k=1,2, \dots$. We will first estimate the size of the polynomial $P_k$: To this end, note that for $k=1$, the sum defining $P_k$ is necessarily empty and thus, we have $P_k \equiv 0$. For $k>1$, we assume that the claimed inequality for $\Vert v^{(k)} \Vert_{L^\infty}$ has already been proven to hold for derivatives $v^{(\gamma_j)}$ of order $\gamma_j \le k-1$. In this case, we can estimate
\[
|P_k|
\le
\sum_{\ell=1}^{k-1} \Vert g^{(\ell)}\Vert_{L^\infty} 
\left(
\sum_{\bm{\gamma}} |C_{\ell,\bm{\gamma}}|  \prod_{j=1}^{N_{\bm{\gamma}}} \Vert v^{(\gamma_j)} \Vert_{L^\infty}
\right),
\]
and 
\begin{align*}
\prod_{j=1}^{N_{\bm{\gamma}}} \Vert v^{(\gamma_j)} \Vert_{L^\infty}
&\le
\prod_{j=1}^{N_{\bm{\gamma}}} A_{\gamma_j} \left(1 + \Vert u \Vert_{H^{\gamma_j}}\right)^{\gamma_j}
\\
&\le
\prod_{j=1}^{N_{\bm{\gamma}}} A_{\gamma_j}  \left(1 + \Vert u \Vert_{H^{k}}\right)^{\gamma_j}
\\
&=
\left(\prod_{j=1}^{N_{\bm{\gamma}}} A_{\gamma_j} \right) \left( 1 + \Vert u \Vert_{H^{k}}\right)^{k}.
\end{align*}
Thus, taking into account that the derivatives $g^{(\ell)}(v)$ of order $\ell \le k-1$ are assumed to be uniformly bounded, $|g^{(\ell)}(v)| \le C_k$, we can now estimate
\[
|P_k| \le C (1 + \Vert u \Vert_{H^k}^k), \quad (k>1),
\]
where $C = C(k,g)>0$ is a constant depending only on $k$ and $g$. As pointed out above, this inequality for $P_k$ holds trivially also for $k=1$, since $P_k \equiv 0$, in this case. Inserting the above estimate in \eqref{eq:higher-deriv}, and integrating over $[0,t]$, we obtain
\begin{align*}
|v^{(k)}(t)|
&\le
\Vert g^{(k)} \Vert_{L^\infty} \int_0^t |v^{(k)}(s)|\, ds
+
C (1 + \Vert u \Vert_{H^k})^k
+
\int_0^T |u^{(k)}(s)| \, ds
\\
&\le
C_k  \int_0^t |v^{(k)}(s)|\, ds
+
C (1 + \Vert u \Vert_{H^k})^k
+
\sqrt{T} \Vert u \Vert_{H^{k}}.
\end{align*}
Increasing the constant $C$ if necessary, we can absorb the last term in the second term and conclude that there exists a constant $C = C(k,g,T) > 0$, such that 
\[
|v^{(k)}(t)|
\le
C_k \int_0^t |v^{(k)}(s)|\, ds
+
C (1 + \Vert u \Vert_{H^k})^k.
\]
Gronwall's inequality now yields
\[
|v^{(k)}(t)|
\le
Ce^{C_kT} (1 + \Vert u \Vert_{H^k})^k,
\]
for all $t \in [0,T]$. Hence, for $A_k = Ce^{C_kT}$, we have
\[
\Vert v^{(k)} \Vert_{L^\infty([0,T])}
\le
A_k (1 + \Vert u \Vert_{H^k})^k,
\]
where $A_k$ is independent of $u$.
\end{proof}

\subsection{Proof of Lemma \ref{lem:Legendre-rec}}
\label{app:pf43}

\begin{proof}
The Legendre polynomials form an orthonormal basis of $L^2([0,T])$. For $p\in \N$, let $P_p: L^2([0,T]) \to L^2([0,T])$ denote the orthogonal projection onto the span of the first $p$ Legendre polynomials, $\Span(\tilde{\tr}_1,\dots, \tilde{\tr}_p)$. By \cite[Theorem 2.3]{CanutoQuarteroni}, for any $k\in \N$, there exists a constant $C = C(T,k)>0$, such that
\[
\Vert u - P_p u \Vert_{L^2([0,T])} 
\le
C p^{-k} \Vert u \Vert_{H^k([0,T])}.
\]
Thus, if $\cR = \cR_{\tilde{\bm{\tr}}}$ is the reconstruction with trunk net $\tilde{\bm{\tr}} = (0,\tilde{\tr}_1, \dots, \tilde{\tr}_p)$, and if $\cP$ denotes the corresponding optimal projection \eqref{eq:opt-proj}, then $P_p = \cR \circ \cP$, and 
\begin{align*}
\Err_{\cR}
&=
\left(
\int_{X} 
\Vert \cR \circ \cP(v) - v \Vert_{L^2([0,T])}^2 
\, d(\G_\#\mu)(v)
\right)^{1/2}
\\
&\le
\frac{C}{p^k}
\left(
\int_{X} 
\Vert v \Vert_{H^k([0,T])}^2 
\, d(\G_\#\mu)(v)
\right)^{1/2}
\\
&=
\frac{C}{p^k}
\left(
\int_{X} 
\Vert \G(u) \Vert_{H^k([0,T])}^2 
\, d\mu(u)
\right)^{1/2}
\end{align*}
By Lemma \ref{lem:pend2}, we can furthermore estimate $\Vert v \Vert_{H^k} \le C(1+\Vert u \Vert_{H^k})^k$, for some constant $C>0$, depending on $T$ and $k$. Hence, we find
\begin{align*}
\Err_{\cR} 
&\le
\frac{C}{p^k}
\left(
\int_{X} 
\left( 1 + \Vert u \Vert_{H^k([0,T])}\right)^{2k} 
\, d\mu(u)
\right)^{1/2},
\end{align*}
where $C = C(k,T)>0$ is a constant independent of $p$.
\end{proof}

\subsection{Proof of Lemma \ref{lem:pendulum-impart}}
\label{app:pf44}
\begin{proof}
We have by \eqref{eq:pendulum-cplx-estimate}
\[
|v_i(t)| \le \alpha + \int_0^t |v_i(s)| \beta(s) \, ds,
\]
where $\alpha = \Vert \Im(u)\Vert_{L^\infty} T$, and $\beta(s) = \gamma (1+\exp(|v_i(s)|))$. By Gronwall's inequality, it follows that
\begin{align*}
|v_i(t)|
&\le
\alpha \exp\left(\int_0^t \beta(s)\, ds\right)
\\
&=
\Vert \Im(u) \Vert_{L^\infty} T \exp\left(\int_0^t \gamma \left(1+e^{|v_i(s)|}\right) \, ds\right).
\end{align*}
Let 
\[
\delta := \frac{1}{2T\exp(\gamma(1+e)T)}.
\]
We claim that if $\Vert \Im(u) \Vert_{L^\infty} \le \delta$, then $|v_i(t)| < 1$ for all $t\in [0,T]$. Suppose this was not the case. If there exists $t\in [0,T]$ such that $|v_i(t)| \ge 1$, then the set
\[
B := 
\set{
t \in [0,T]
}{
|v_i(t)| \ge 1
}
\]
is nonempty. Let $t_0$ be given by
\[
t_0 
=
\inf B.
\]
By the continuity of $t\mapsto v_i(t)$, $B$ is a closed set. In particular, this implies that $t_0 \in B$. Since $v_i(0) = 0$, we must have $t_0 > 0$, and $|v_i(t)|<1$ for all $t\in [0,t_0)$.
But then, 
\begin{align*}
|v_i(t_0)|
&\le
\Vert \Im(u)\Vert_{L^\infty} T \exp\left( \int_0^{t_0}  \gamma\left(1+e^{|v_i(s)|}\right) \, ds\right)
\\
&\le
\Vert \Im(u)\Vert_{L^\infty} \, T \exp(\gamma(1+e)T)
\\
&\le
\delta \, T \exp(\gamma(1+e)T)
\\
&\le \frac{1}{2} < 1,
\end{align*}
leads to a contradiction to the assumption that $|v_i(t_0)| \ge 1$. Thus, we conclude that the set $B$ must in fact be empty, i.e. that
\[
|v_i(s)| < 1, \quad \text{for all } s \in [0,T].
\]
\end{proof}
\subsection{Proof of Lemma \ref{lem:pendulum-cplx-existence}}
\label{app:pf45}
\begin{proof}
We argue by contradiction. Suppose the claim is not true. Then there exists $u: [0,T]\to \C$, $u\in L^\infty([0,T])$, and $0 < T_1 < T$ such that the solution\footnote{We note that short-time existence follows from the fact that the right-hand side is locally Lipschitz continuous in $v$, so that $T_1 > 0$.} of \eqref{eq:pendulum-cplx} is defined on $[0,T_1)$, but $\lim_{t\nearrow T_1} |v(t)| = \infty$. Since the right-hand side of \eqref{eq:pendulum-cplx} is uniformly Lipschitz continuous in $v_r$, this can only be the case, if $\lim_{t\nearrow T_1} |v_i(t)| = \infty$. In particular, there exists $T_0 < T_1$, such that $|v_i(T_0)| > 1$. But then, this would imply the existence of a solution of \eqref{eq:pendulum-cplx}, which is defined on $[0,T_0]$, for which $\sup_{s\in [0,T]} |\Im(u(s))| \le \delta$, and such that we have $\sup_{s\in [0,T_0]} |v_i(s)| \ge |v_i(T_0)| > 1$. This is clearly in contradiction with Lemma \ref{lem:pendulum-impart}. By contradiction, it thus follows that we must have $T_0 = T$, and hence the solution of \eqref{eq:pendulum-cplx} exists on $[0,T]$ for any $u\in L^\infty([0,T])$, with $\Vert \Im(u) \Vert_{L^\infty([0,T])} \le \delta$. Furthermore, it follows from Lemma \ref{lem:pendulum-impart} that $\sup_{s \in [0,T]} |v_i(s)|\le 1$ in this case.
\end{proof}

\subsection{Proof of Lemma \ref{lem:pendulum-differentiable}}
\label{app:pf46}
\begin{proof}
Fix $z_k \in E_{\rho_k}$ for $k\ne j$. For $z_j \in E_{\rho_j}$, and by slight abuse of notation, let us denote the forcing by $u(t,z_j) = u(t,\bm{z})$, and the corresponding solution of the pendulum equations \eqref{eq:pendulum-cplx} by $v(t,z_j) = \cF(\bm{z})(t)$, where $\bm{z} = (z_\ell)_{\ell \in \N}$. Then the claim is that $z_j \mapsto v(t,z_j)$ is complex-differentiable in $z_j \in E_{\rho_j}$. We note that $v(t,z_j)$ is the unique fixed point of
\[
v(t,z_j)
=
\int_0^t
G(s,z_j,v(s,z_j)) 
\, ds,
\]
where $G(t,z_j,v):= g(v) + u(t,z_j)$.  To show boundedness, we note that by \eqref{eq:pendulum-adm}, the assumed $(\bm{b},\delta)$-admissibility of $\bm{\rho}$ and Lemma \ref{lem:pendulum-cplx-existence}, the solution $t \mapsto v(t,z_j)$ exists for any $z_j\in E_{\rho_j}$, and that $\sup_{t\in [0,T]}|\Im(v(t,z_j))|\le 1$. But then, also the real part $v_r(t) := \Re(v(t,z_j))$ is bounded, because from \eqref{eq:pendulum-cplx}, we conclude that
\begin{align*}
\left|\frac{dv_r(t)}{dt}\right|
&\le
|g_r(v_r(t),v_i(t))| + |\Re(U(t))|
\\
&\le
|v_r(t)| + k \cosh(|v_i(t)|) + |\Re(U(t))|
\\
&\le
|v_r(t)| + k \cosh(1) + \Vert\Re(U)\Vert_{L^\infty([0,T])},
\end{align*}
which implies by Gronwall's inequality that
\[
|v_r(t)| \le \left( k \cosh(1) + \Vert\Re(U)\Vert_{L^\infty([0,T])}\right) Te^{T},
\]
for all $t\in [0,T]$. We furthermore note that 
\begin{align*}
\Vert \Re(U) \Vert_{L^\infty([0,T])}
&\le 
\sum_{k=1}^\infty
|\Re(z_k)| \alpha_k \Vert \psi_k \Vert_{L^\infty}
=
\sum_{k=1}^\infty
|\Re(z_k)| b_k
\\
&\le
\sum_{k=1}^\infty
\rho_k b_k
=
\sum_{k=1}^\infty
b_k
+
\sum_{k=1}^\infty
(\rho_k-1) b_k
\\
&\le 
\Vert \bm{b} \Vert_{\ell^1(\N)} + \delta < \infty,
\end{align*}
is uniformly bounded for all $\bm{z} \in E_{\bm{\rho}}$ for any admissible $\bm{\rho}$. This shows that 
\begin{align*}
\sup_{\bm{z}\in E_{\bm{\rho}}} \Vert v \Vert_{L^2_{\C}([0,T])}
&\le
\sup_{\bm{z}\in E_{\bm{\rho}}} \sqrt{T} \Vert v \Vert_{L^\infty([0,T])} 
\\
&\le
\sup_{\bm{z}\in E_{\bm{\rho}}} \sqrt{T} \Vert v_r \Vert_{L^\infty([0,T])}
+
\sup_{\bm{z}\in E_{\bm{\rho}}} \sqrt{T}\Vert v_i \Vert_{L^\infty([0,T])}
\\
&\le
T\sqrt{T}e^T\left(\gamma\cosh(1) + \Vert \bm{b} \Vert_{\ell^1(\N)} + \delta \right)
+
\sqrt{T}
\\
&=: C,
\end{align*}
is uniformly bounded for all $(\bm{b},\e)$-admissible $\bm{\rho}$.

Finally, we prove the holomorphy of $z_j \mapsto v(t,z_j)$: By definition, $z_j \mapsto u(t,z_j)$ is an affine function. Hence $z_j\mapsto G(t,z_j,v)$ is also an affine function of $z_j$, and in particular differentiable in $z_j$. Since $v(t,z_j)$ exists and is bounded for all $z_j\in E_{\rho_j}$, it then follows from the general theory of parametric ODEs that $\partial_{z_j} v(t,z_j)$ exists and that
\[
\partial_{z_j} v(t,z_j)
=
\int_0^t
\left\{
\partial_{z_j} G(s,z_j,v(s,z_j)) 
+
\partial_{v} G(s,z_j,v(s,z_j)) \partial_{z_j} v(s,z_j)
\right\}
\, ds.
\]
This implies that $z_j \mapsto v(t,z_j)$ is a holomorphic mapping.
\end{proof}
\subsection{Proof of Lemma \ref{lem:elliptic1}}
\label{app:pf47}
\begin{proof}
The difference $w = u - u'$ is a solution of the equation
\begin{align*}
\nabla \cdot \left( a \nabla w \right)
&= \nabla \cdot \left( a \nabla u \right) - \nabla \cdot \left( a \nabla u' \right)
\\
&= f - \nabla \cdot \left( a \nabla u' \right)
\\
&= \nabla \cdot \left( a' \nabla u' \right) - \nabla \cdot \left( a \nabla u' \right)
\\
&=
\nabla \cdot \left( (a'-a) \nabla u' \right).
\end{align*}
By elliptic theory, we thus have
\begin{align*}
\Vert u - u '\Vert_{L^2(D)}
&= \Vert w \Vert_{L^2(D)}
\le \Vert \nabla \cdot \left( (a'-a) \nabla u' \right) \Vert_{H^{-1}(D)}
\\
&\le \Vert (a'-a) \nabla u' \Vert_{L^2(D)}
\le \Vert a'-a \Vert_{L^\infty(D)} \Vert \nabla u' \Vert_{L^2(D)}
\\
&\le \Vert a'-a \Vert_{L^\infty(D)} \Vert u' \Vert_{H^1_0(D)}
\le \Vert a'-a \Vert_{L^\infty(D)} C\Vert f \Vert_{L^2(D)}.
\end{align*}
\end{proof}
\subsection{Proof of Lemma \ref{lem:elliptic2}}
\label{app:pf48}
\begin{proof}
It is well-known that if $u$ is a solution of \eqref{eq:random-elliptic}, with smooth coefficient $a(x)$ and right-hand side $f\in H^{k}$, then $u\in H^{k+1}$. The main point of this lemma is the explicit dependence on the norm of $a$, which will be required to estimate the reconstruction error. Let $\bm{k} = (k_1,\dots, k_m)\in \N_0^m$ denote any multi-index. Then by differentiation of \eqref{eq:random-elliptic}, we find
\begin{align} \label{eq:leibnitz}
\nabla \cdot \left(a(x) \nabla \partial_x^{\bm{k}} u \right)
=
-\nabla \cdot \left(
\sum_{\bm{\ell} < \bm{k}} 
{\bm{k} \choose \bm{\ell}}
\partial_x^{\bm{k-\ell}} a(x) \nabla \partial_x^{\bm{\ell}} u 
\right)
+
\partial_x^{\bm{k}} f.
\end{align}
where $\ell$ runs over all indices $\bm{\ell} = (\ell_1,\dots, \ell_m)\in \N_0$, such that $\ell_j \le k_j$ for all $j=1,\dots, m$, with a strict inequality for at least one $j$. We also write 
\[
{\bm{k} \choose \bm{\ell}}
:=
\prod_{j=1}^m {k_j \choose \ell_j}.
\]
Integrating \eqref{eq:leibnitz} against $\partial_x^{\bm{k}} u$, with $k := |\bm{k}| = k_1 + \dots + k_m$, it follows that
\begin{align*}
\lambda \Vert \nabla\partial_x^{\bm{k}} u\Vert^2_{L^2_x}
&\le
\int_{\T^n} a(x) |\nabla \partial_x^{\bm{k}}(x)|^2 \, dx
\\
&\le
C_k \sum_{\bm{\ell}<\bm{k}}
\int_{\T^n} \left|\partial_x^{\bm{k-\ell}}a(x) \right| \left|\nabla\partial_x^{\bm{\ell}} u(x)\right| \left|\nabla\partial_x^{\bm{k}} u(x)\right| \, dx
\\
&\qquad
+
\int_{\T^n} |\partial_x^{\bm{k}} f| |\partial_x^{\bm{k}} u|\, dx
\\
&\le
C_k \Vert a \Vert_{C^k} \Vert u \Vert_{H^{k}} \Vert \nabla\partial_x^{\bm{k}} u \Vert_{L^2}
+
\Vert \partial_x^{\bm{k}} f \Vert_{L^2} \Vert u \Vert_{H^k}.
\end{align*}
Using the inequality $ab \le 2^{-1} \epsilon a^2 + (2\epsilon)^{-1} b^2$ for $a,b,\epsilon > 0$, we find
\begin{align*}
\lambda \Vert \nabla\partial_x^{\bm{k}} u\Vert^2_{L^2_x}
&\le
\frac{C_k^2\Vert a \Vert_{C^k}^2}{2\lambda}\Vert u \Vert_{H^k}^2 + \frac{\lambda}2 \Vert \nabla \partial_x^{\bm{k}} u \Vert_{L^2}^2 + \frac{\lambda}{2} \Vert \partial_x^{\bm{k}} f \Vert_{L^2}^2 + \frac{1}{2\lambda}\Vert u \Vert_{H^k}^2,
\end{align*}
and hence
\[
\Vert \nabla\partial_x^{\bm{k}} u\Vert^2_{L^2_x}
\le 
(C_k^2\Vert a \Vert_{C^k}^2+1)\frac{\Vert u \Vert_{H^k}^2}{\lambda^2} + \Vert \partial_x^{\bm{k}} f \Vert_{L^2}^2.
\]
Summing the last estimate over all $|\bm{k}|=k$, and increasing the constant $C_k$, if necessary, we obtain
\[
\Vert u\Vert^2_{H^{k+1}}
\le 
C_k(\Vert a \Vert_{C^k}^2+1)\Vert u \Vert_{H^k}^2 + \Vert f \Vert_{H^k}^2,
\]
where the new constant $C_k = C_k(k,\lambda)$ now depends on both $k$ and $\lambda$. Repeating the same argument for any $\ell \le k$, we also find
\begin{align*}
\Vert u\Vert^2_{H^{\ell+1}}
\le 
C_\ell(\Vert a \Vert_{C^\ell}^2+1)\Vert u \Vert_{H^\ell}^2 + \Vert f \Vert_{H^\ell}^2 \\
\le
C' (\Vert a \Vert_{C^k}^2+1)\Vert u \Vert_{H^\ell}^2 + \Vert f \Vert_{H^k}^2, 
\end{align*}
with $C' := \max_{\ell \le k} C_\ell$ a fixed constant depending only on $k$ and $\lambda$. Writing the last inequality in the form
\[
\Vert u\Vert^2_{H^{\ell+1}} \le A \Vert u \Vert^2_{H^\ell} + \Vert f \Vert_{L^2}^2,
\]
it follows by induction on $\ell=0,\dots, k$, that
\[
\Vert u \Vert^2_{H^{k+1}}
\le
A^{k} \Vert u \Vert_{H^1}^2
+
\Vert f \Vert_{L^2}^2 \sum_{\ell=0}^{k-1} A^\ell.
\]
For $k=0$, we have the well-known bound $\Vert u \Vert_{H^1}^2 \le C'' \Vert f\Vert_{L^2}^2$ with $C'' = C''(\lambda)$. We may wlog assume $C''\ge 1$. In particular, we then conclude that 
\[
\Vert u \Vert^2_{H^{k+1}}
\le
C''
\Vert f \Vert_{L^2}^2
\sum_{\ell=0}^{k} A^\ell,
\]
where $A := C'(\Vert a \Vert_{C^k} + 1)$. Finally, for fixed $k$, we note that since $C' = C'(k, \lambda)$, $C''= C''(\lambda)$, there exists a constant $C$ depending only on $k$ and $\lambda$, but independent of $\Vert a \Vert_{C^k}$, such that 
\[
C'' \sum_{\ell=0}^k A^\ell
=
C'' \, C' \sum_{\ell=0}^k (1+ \Vert a \Vert_{C^k}^2)^\ell
\le
C (1 + \Vert a \Vert^{2k}_{C^k}).
\]
For such $C = C(k,\lambda) > 0$, we conclude that
\[
\Vert u \Vert^2_{H^{k+1}} 
\le
C \Vert f \Vert_{H^k}^2 (1+\Vert a\Vert_{C^k}^{2k}),
\]
as claimed.
\end{proof}
\subsection{Proof of Theorem \ref{thm:ac-regularity}}

To prove Theorem \ref{thm:ac-regularity}, we recall some notions and results from the theory of second order parabolic equations (cp. \cite{lieberman1996}). First, let $U = [0,T]\times \T^d$ denote the domain of the solution of the parabolic equation \eqref{eq:allen-cahn}. We denote by (cp. \cite[p. 46]{lieberman1996})
\[
\cC^{k,\alpha}(U)
:=
\set{
v \in C(U)
}{
|v|_{(k,\alpha)} < \infty
},
\quad 
k\in \N, \; \alpha \in (0,1),
\]
the parabolic H\"older space on $U$, where 
\[
|v|_{(k,\alpha)}
:= 
\sum_{\beta + 2j \le k}
\sup_{U} |\partial_x^\beta \partial^j_t v|
+
[v]_{(k,\alpha)},
\]
and
\[
[v]_{(k,\alpha)}
:=
\sum_{\beta + 2j = k} [\partial^\beta_x \partial^j_t v]_{\alpha},
\]
and where the parabolic H\"older semi-norm $[w]_{\alpha}$ of a function $w$ is defined by
\[
[w]_{\alpha}
:=
\sup_{(t,x)\in U}
\sup_{(t',x')\ne (t,x)}
\frac{|w(t,x) - w(t',x')|}{\left(|x-x'| + |t-t'|^{1/2} \right)^\alpha}.
\]

We then have the following Schauder estimate:
\begin{lemma}[Schauder estimate] 
\label{lem:schauder}
Let $\alpha \in (0,1)$. For any $k\in \N$, there exists a constant $C>0$, such that if $v$ is a solution to 
\[
\left\{
\begin{aligned}
\partial_t v - \Delta v &= f,
\\
v(t=0) &= 0,
\end{aligned}
\right.
\]
with $f\in \cC^{(k,\alpha)}(U)$, then 
\[
| v |_{(k+2,\alpha)}
\le
C
| f |_{(k,\alpha)} + \Vert u\Vert_{C^{k+2,\alpha}(\T^d)}.
\]
\end{lemma}
\begin{proof}
For $k=0$, this follows e.g. from \cite[Theorem 4.28]{lieberman1996}. For $k>0$, we note that the coefficients in equation $\partial_t v - \Delta v = f$ are constant, and hence we can apply the base-case to partial derivatives of the equation (more precisely, finite-difference approximations thereof, and take the limit).
\end{proof}

We also note the following strong $L^p$ estimate for parabolic equations:
\begin{lemma} \label{lem:strongLp}
Let $p\in [2,\infty)$. There exists a constant $C>0$, such that if $v \in L^\infty([0,T]\times \T^d)$ is a weak solution of $\partial_t v - \Delta v = f$, for $f\in L^\infty(\T^d)$, and with initial data $v(t=0) = u$, then
\[
\Vert v \Vert_{W^{1,p}([0,T]\times \T^d)}
\le
C\left(
\Vert f \Vert_{L^p} 
+ 
\Vert u \Vert_{W^{2,p}(\T^d)}
\right).
\]
\end{lemma}

\begin{proof}
Theorem 7.32 of \cite{lieberman1996} provides a sharper estimate, from which the claim readily follows.
\end{proof}

\begin{corollary} \label{cor:strongLp}
Let $d \in \{2,3\}$. There exists a constant $C>0$, and $\alpha \in (0,1)$, such that if $f\in L^\infty$, and $u \in C^1([0,T]\times\T^d)$, and if $v$ solves $\partial_t v - \Delta v = f$, with initial data $v(t=0) = u$, then
\[
\Vert v \Vert_{(0,\alpha)}
\le
C
\left(
\Vert f \Vert_{L^\infty}
+
\Vert u \Vert_{C^2(\T^d)}
\right).
\]
\end{corollary}
\begin{proof}
This follows directly from Lemma \ref{lem:strongLp} and the fact that by Sobolev embedding, we have $W^{1,p} \embeds C^\alpha$, for $\alpha \le 1 - (d+1)/p$ for $p > d+1$.
\end{proof}

\begin{proof}[Proof of Theorem \ref{thm:ac-regularity}]
The claim follows from a bootstrap argument: From the a priori estimate of Theorem \ref{thm:ac-boundedness} for the Allen-Cahn equation, we know that for initial data $\Vert u \Vert_{L^\infty} \le 1$, we have $\Vert v(t) \Vert_{L^\infty}\le 1$ for all $t\in [0,T]$. In particular, it follows that $|f| = |f(v)| \le C$ is bounded in $L^\infty$. Thus, $v$ solves 
\begin{gather} \label{eq:bootstrap}
\left\{
\begin{aligned}
\partial_t v - \Delta v &= f, 
\\
v(t=0) &= u,
\end{aligned}
\right.
\end{gather}
with source term $f\in L^\infty$. By the strong $L^p$ estimate (cp. Corollary \ref{cor:strongLp}), and the assumed smoothness of $u\in C^{4,\alpha}$, it follows that $v\in \cC^{\alpha}([0,T]\times \T^d)$ for some $\alpha \in (0,1)$. In turn this implies the following chain of improved regularity based on the Schauder estimate of Lemma \ref{lem:schauder}, and using also the fact that $u \in C^{k,\alpha}(\T^d)$ for all $k\le 4$:
\begin{align*}
v \in \cC^{\alpha}(U)
&\implies f(v) \in \cC^{\alpha} (U)
\implies v \in \cC^{2,\alpha}(U)
\\
&\implies f(v) \in \cC^{2,\alpha} (U)
\implies v \in \cC^{4,\alpha}(U).
\end{align*}
Thus, we conclude that if $u\in C^{4,\alpha}(\T^d)$, then we must have $v \in \cC^{4,\alpha}(U)$. Furthermore, it follows from the estimates of Corollary \ref{cor:strongLp}, Lemma \ref{lem:schauder}, that there in fact exists a constant $\sigma = \sigma(\Vert u \Vert_{C^{4,\alpha}}) > 0$, depending on the $C^{4,\alpha}$-norm of the initial data $u$, such that
\[
| v |_{(4,\alpha)}
\le
\sigma(\Vert u \Vert_{C^{4,\alpha}(\T^d)}).
\]
Clearly, we have $\Vert v \Vert_{C^{(4,2)}(U)} \le |v|_{(4,\alpha)}$ for any $\alpha > 0$. The claimed estimate thus follows.
\end{proof}

\subsection{Proof of Corollary \ref{cor:ac-continuity}}
\begin{proof}
Since we have $\Vert v(t)\Vert_{L^\infty(\T^d)},\, \Vert v'(t) \Vert_{L^\infty(\T^d)}$ for all $t \in [0,T]$, the difference $w = v-v'$, solves the following equation
\[
\partial_t w = \Delta w + F(v,v') w,
\]
where $F(a,b) := 1-a^2 - ab - b^2$. It follows from the uniform boundedness of $v,v'$ that 
$|F(v,v')| \le 4$ is uniformly bounded. Multiplying the equation by $w$ and integrating over $x$, we obtain
\[
\frac{d}{dt} \int_{\T^d} w^2 \, dx
\le
8 \int_{\T^d} w^2 \, dx.
\]
And hence, by Gronwall's inequality, we must have $ \Vert w(t) \Vert_{L^2}^2 \le \Vert w(0) \Vert_{L^2}^2 e^{8t}$ for all $t\in [0,T]$. We conclude that
\[
\Vert v(T) - v'(T) \Vert_{L^2(\T^d)}
\le
e^{4T}
\Vert u - u' \Vert_{L^2(\T^d)}.
\]
\end{proof}

\subsection{Proof of Theorem \ref{thm:ac-approx-err}}
\label{app:pf49}
\begin{proof}
Our main observation is that the local truncation error for the scheme \eqref{eq:allen-cahn-fd}, given by
\[
T^n_j := 
(I-\Delta t D_{\Delta x}) v(t_n,x_j) - (v(t_n,x_j) + \Delta t \left(v(t_{n-1},x_j) - v(t_{n-1},x_j)^3\right),
\]
has a Taylor expansion 
\begin{align*}
T^n_j 
&:= 
\Delta t
\left(
\frac{\partial v(t_{n-1},x_j)}{\partial t}
- 
\Delta v(t_{n-1},x_j)
-(v(t_{n-1},x_j) - v(t_{n-1},x_j)^3)
\right) \\
&\qquad + \Delta t \, R(\Delta t ,\Delta x),
\end{align*}
where similar to \cite{TangYang2016}, the remainder term $R(\Delta t, \Delta x)$ can be estimated by 
\begin{align*}
|R(\Delta t, \Delta x)|
&\le
C
\left(
\Delta t\, \left\Vert \frac{\partial v}{\partial t}\right\Vert_{L^\infty} 
+ 
\Delta t\, \left\Vert \frac{\partial^2 v}{\partial t^2}\right\Vert_{L^\infty}
+
\Delta x^2 \, \sum_{k=1}^d \left\Vert \frac{\partial^4 v}{\partial x_k^4}\right\Vert_{L^\infty}
\right)
\\
&\le
C \left(\Delta t + \Delta x^2\right) \Vert v \Vert_{C^{(2,4)}([0,T]\times \T^d)},
\end{align*}
where $\Vert v \Vert_{C^{(2,4)}([0,T]\times \T^d)}$ is defined by \eqref{eq:mixed-derivatives}.
We note that in contrast to our estimate, the remainder term was bounded in \cite{TangYang2016} by the larger norm $\Vert v \Vert_{C^2([0,T];C^4(\T^d))}$. In view of the available a priori estimates for parabolic equations, the parabolic norm $\Vert v \Vert_{C^{(2,4)}([0,T]\times \T^d)}$ appears better adapted to the problem, and provides a less restrictive convergence result for the scheme. The remainder of the proof is the same as in \cite[Theorem 4.1]{TangYang2016}. 
\end{proof}

\subsection{Proof of Lemma \ref{lem:nn-nonlinearity}}
\label{app:pf410}
\begin{proof}
By a result of Yarotsky \cite[Prop. 2, 3]{Yarotsky2017}, there exists a constant $C' > 0$, such that for any $\epsilon > 0$, there exists a ReLU network $\tilde{\times}: [-2,2] \to \R$, such that $\size(\tilde{\times}) \le C'(|\log(\epsilon)| + 1)$, and 
\[
\sup_{\xi,\eta\in [-2,2]} |\tilde{\times}(\xi,\eta) - \xi \eta| < \epsilon/2.
\]
In fact, the mapping constructed in \cite{Yarotsky2017} is based on the identity $\xi \eta = \frac12 \left((\xi+\eta)^2 - \xi^2 - \eta^2\right)$, and finding a suitable neural network approximation $f_m(x) \approx x^2$, of the form \cite[paragraph above eq. (3)]{Yarotsky2017} $f_m(x) = x - \sum_{s=1}^m 2^{-2s} g_s(x)$, with $m = \mathcal{O}(|\log(\epsilon)|)$ and with $g_s$ a $s$-fold iteration of the following sawtooch function $g(x)$:
\[
g(x) =
\begin{cases}
2x, &x < \frac12, \\
2(1-x), &x \ge \frac12,
\end{cases}
\quad
g_s(x) = \underbrace{(g\circ \dots \circ g)}_{s \text{ times}}(x).
\]
It is then immediate that $\Lip(g_s) \le \Lip(g)^s \le 2^s$, and hence $\Lip(f_m) \le 1 + \sum_{s=1}^m 2^{-2s} \Lip(g_s) \le 1 + \sum_{s=1}^m 2^{-s} \le 2$. It is then readily seen that there exists a constant $M>0$, independent of $\epsilon$, such that 
\[
\Lip(\tilde{\times}: [-2,2]^2 \to ) \le M.
\]
in addition to the approximation property
\[
\sup_{\xi,\eta \in [-1,1]^2}
|
\tilde{\times}(\xi,\eta) - \xi \eta 
|
< \epsilon / 2.
\]

In particular, with the mapping constructed above, we then have $\sup_{\eta\in [-1,1]} |\tilde{\times}(\eta,\eta) - \eta^2| < \epsilon$, and from the assumption that $\epsilon < 1$, it follows that $\tilde{\times}(\eta,\eta)-1 \in [-2,2]$ for all $\eta \in [-1,1]$. Observing that $\xi^3 - \xi = \xi(\xi^2 - 1)$, we find for any $\eta \in [-1,1]$
\begin{align*}
\left|
\tilde{\times}\left(\eta,(\tilde{\times}(\eta,\eta) - 1)\right)
-
(\eta^3 - \eta)
\right|
&\le
\left|
\tilde{\times}\left(\eta,(\tilde{\times}(\eta,\eta) - 1)\right)
-
\eta (\tilde{\times}(\eta,\eta) - 1)
\right|
\\
&\quad
+
\left|
\eta (\tilde{\times}(\eta,\eta) - 1)
-
\eta(\eta^2 - 1)
\right|
\\
&\le
\sup_{\xi \in [-2,2]}
\left|
\tilde{\times}\left(\eta,\xi\right)
-
\eta \xi
\right|
\\
&\quad
+
|\eta|
\left|
\tilde{\times}(\eta,\eta) 
-
\eta^2
\right|
\\
&\le
2\sup_{\xi,\eta\in[-1,1]} \left|
\tilde{\times}\left(\eta,\xi\right)
-
\eta \xi
\right|
\\
&<
\epsilon.
\end{align*}
To finish the proof, we note that if $(\xi,\eta) \mapsto \tilde{\times}(\xi,\eta)$ is represented by a ReLU neural network of size $\le C'(|\log(\epsilon)| + 1)$, then the function
\[
g_\epsilon(\eta) := 
\max\left(-1,
\min\left(1,
\tilde{\times}\left(\eta,(\tilde{\times}(\eta,\eta) - 1)\right)
\right)
\right)
\]
can be represented by a ReLU neural network of size $\le C(|\log(\epsilon)| + 1)$, where $C$ is a constant multiple of $C'$, and we have $g_\epsilon(\eta) \in [-1,1]$ for all $\eta \in \R$. Furthermore, since $\eta^3 - \eta \in [-1,1]$ for all $\eta \in [-1,1]$, it also follows from the above estimate that 
\[
\sup_{\eta \in [-1,1]} |g_\epsilon(\eta) - (\eta^3-\eta)| \le \epsilon.
\]
\end{proof}
\begin{remark}
\label{rem:sact}
If we consider neural networks $g$ with \emph{any smooth} (e.g. $\sigma \in C^3(\R)$) non-linear activation function $\sigma: \R \to \R$, then the previous approximation result can be considerably improved \cite[Prop. 3.4]{Pinkus1999}: Indeed, by assumption on $\sigma$ there exists a point $x_0\in \R$, such that $\sigma''(x_0)\ne 0$. Then for $h>0$, $\eta \in [-1,1]$, we have by Taylor expansion
\[
\frac{\sigma(x_0+\eta h)-2\sigma(x_0)+\sigma(x_0-\eta h)}{\sigma''(x_0) h^2}
= \eta^2 + \mathcal{O}(h),
\]
as $h\to 0$, uniformly in $\eta \in [-1,1]$. In particular, it follows that for smooth $\sigma$ there exists a neural network architecture \emph{of fixed size}, such that and for any $h>0$, there is neural network $g: [-1,1] \to \R$, $\eta \mapsto g(\eta)$, such that $|g(\eta) - \eta^2| \le Ch$ for all $\eta \in [-1,1]$. As a consequence, using the representation \cite{Yarotsky2017}
\[
xy = \frac12( (x+y)^2 - x^2 - y^2 ),
\]
it follows that there exists a constant $C>0$, such that for any $\epsilon > 0$, there exists a neural network $\times_\epsilon: [-1,1]^2 \to \R$ of size $\size(\times_\epsilon) \le C$, such that 
\[
\sup_{x,y\in [-1,1]} |\times_\epsilon(x,y) - xy| < \epsilon.
\]
But then, arguing as in the proof of Lemma \ref{lem:nn-nonlinearity}, we find that for any $\epsilon > 0$, the function $g_\epsilon(x) = \times_\epsilon(\times_\epsilon(x,x),x) - x$ can be represented by a neural network, with a neural network size, $\size(g_\epsilon) \le C'$, which is uniformly bounded in $\epsilon>0$, and such that $\sup_{\eta\in [-1,1]} |g_\epsilon(\eta) - (\eta^3 - \eta)| < \epsilon$.
\end{remark}
\subsection{Proof of Lemma \ref{lem:approx-scheme}}
\label{app:pf411}
\begin{proof}
By Lemma \ref{lem:nn-nonlinearity}, the non-linearity in \eqref{eq:approx-scheme} can be represented by a neural network with size bounded by
\[
\begin{gathered}
\size(g_\epsilon) \le C(1+|\log(\epsilon)|),
\quad
\depth(g_\epsilon) \le C(1+|\log(\epsilon)|).
\end{gathered}
\]
It follows that the mapping
\[
\left(
\tilde{U}^k_1,\dots, \tilde{U}^k_m
\right)
\mapsto 
\left(
g_\epsilon\left(\tilde{U}^k_1\right),\dots, g_\epsilon\left(\tilde{U}^k_m\right)
\right)
=:
G_\epsilon(\tilde{U}^k),
\]
can be represented by a neural network $G_\epsilon: \R^m \to \R^m$, with 
\[
\begin{gathered}
\size(G_\epsilon) = \mathcal{O}(m|\log(\epsilon)|),
\quad
\depth(G_\epsilon) = \mathcal{O}(|\log(\epsilon)|).
\end{gathered}
\]
Since the identity mapping $\tilde{U}^k \mapsto \tilde{U}^k$ can be represented by a ReLU network with $\size = \mathcal{O}(m)$, $\depth = \mathcal{O}(1)$, there exists a neural network with $\size = \mathcal{O}(m|\log(\epsilon)|)$, $\depth = \mathcal{O}(|\log\epsilon|)$, which represents
\[
\tilde{U}^k
\mapsto 
\tilde{U}^k
+
\Delta t 
G_\epsilon(\tilde{U}^k).
\]
Finally, we note that 
\[
\tilde{U}^k
+
\Delta t 
G_\epsilon(\tilde{U}^k)
\mapsto 
R_{\Delta x, \Delta t} (\tilde{U}^k
+
\Delta t 
G_\epsilon(\tilde{U}^k)),
\]
is simply a matrix-vector multiplication of a \emph{fixed} matrix $R_{\Delta x, \Delta t} = (I - \Delta t \, D_{\Delta x})^{-1} \in \R^{m\times m}$ (independent of $\tilde{U}^k$), with a vector in $\R^m$. This operation can be represented by a ReLU neural network layer with $\size = \mathcal{O}(m^2)$, $\depth = \mathcal{O}(1)$. We conclude that the composition
\[
\tilde{U}^k
\mapsto \tilde{U}^k + \Delta t \, G_\epsilon(\tilde{U}^k)
\mapsto R_{\Delta x, \Delta t} \left(\tilde{U}^k + \Delta t \, G_\epsilon(\tilde{U}^k)\right)
= \tilde{U}^{k+1},
\]
can be represented by a neural network $\hN: \R^m \to \R^m$ with 
\begin{gather*}
\size(\hN) = \mathcal{O}(m^2 + m |\log\epsilon|), 
\quad
\depth(\hN) = \mathcal{O}(|\log\epsilon|).
\end{gather*}

Since $\tilde{U}^0 \mapsto \tilde{U}^1 \mapsto \dots \mapsto \tilde{U}^n$ involves the composition of $n$ such steps, we conclude that $\tilde{U}^0 \mapsto \tilde{U}^n$ can be represented by a neural network $\cN = \hN \circ \hN \circ \dots \circ \hN: \R^m \to \R^m$ ($n$ iterations) with
\begin{gather*}
\size(\cN) = \mathcal{O}(n(m^2 + m |\log\epsilon|)), \quad
\depth(\cN) = \mathcal{O}(n|\log(\epsilon)|).
\end{gather*}

\end{proof}

\subsection{Proof of Lemma \ref{lem:approx-scheme-err}}
\label{app:pf412}
\begin{proof}
We note that
\begin{align*}
U^{k+1} - \tilde{U}^{k+1}
&=
R_{\Delta x, \Delta t} \left(
U^k - \tilde{U}^k + \Delta t \left[ g_\epsilon (U^k) - g_\epsilon(\tilde{U}^k) \right]
\right)
\\
&\quad 
+
\Delta t R_{\Delta x, \Delta t} \left(
 \left[ g_\epsilon (U^k) - f({U}^k) \right]
\right).
\end{align*}
We note that $\Vert U^k\Vert_{\ell^\infty} \le 1$ for all $k>0$, by Theorem \ref{thm:ty}. Furthermore, it has been shown in \cite{TangYang2016} that $\Vert R_{\Delta x, \Delta t} \Vert_{\ell^\infty \to \ell^\infty} \le 1$. It follows that
\[
\left\Vert
R_{\Delta x, \Delta t} \left(
 \left[ g_\epsilon (U^k) - f({U}^k) \right]
\right)
\right\Vert_{\ell^\infty}
\le
\sup_{\eta \in [-1,1]} |g_\epsilon(\eta) - f(\eta)|
\le \epsilon.
\]
Denote now $E^k = \Vert U^k - \tilde{U}^k\Vert_{\ell^\infty}$, for $k=0,\dots, n$. Then 
\begin{align*}
E^{k+1}
&\le
\Vert R_{\Delta t, \Delta x} \Vert_{\ell^\infty\to \ell^\infty}
\left(
E^k + \Delta t \, \Lip(g_\epsilon)  E^k
\right)
+
\Delta t \, \epsilon
\\
&\le
\left(1 + \Delta t \, \Lip(g_\epsilon) \right) E^k
+
\Delta t \, \epsilon.
\end{align*}
Summing over $k=1,\dots,\ell$, and taking into account that $E^0=0$, we find for any $\ell \in \{1,\dots, n\}$ that
\begin{align*}
E^\ell
&=
E^0
+
\sum_{k=0}^{\ell-1} [E^{k+1} - E^{k}]
\\
&\le
T \epsilon
+
\Delta t
\, 
\sum_{k=1}^{\ell-1} \Lip(g_\epsilon) E^k.
\end{align*}
By the Gronwall inequality, it follows that
\[
E^n \le \epsilon\, T \exp(\Lip(g_\epsilon) T).
\]
The result follows from $\Lip(g_\epsilon) \le M$.
\end{proof}
\subsection{Proof of Proposition \ref{prop:ac-approx-err}}
\label{app:pf413}
\begin{proof}
By definition, the approximator neural network $\cA$ should provide an approximation $\cA \approx \cP \circ \G \circ \cD$. Our goal is to construct $\cA$ based on the neural network approximation \eqref{eq:approx-scheme} of the convergent numerical scheme \eqref{eq:allen-cahn-fd}. To this end, we first note that by Lemmas \ref{lem:approx-scheme}, \ref{lem:approx-scheme-err}, there exists a constant $C = C(\sup_{u\in \supp(\mu)}\Vert u \Vert_{C^{(2,4)}}, T) >0$, independent of $m$, $n$ and $\epsilon$, and a neural network $\cN$ with 
\[
\size(\cN) \le C(1+n(m^2 + m|\log(\epsilon)|)),
\quad
\depth(\cN) \le C(1+n|\log(\epsilon)|),
\]
such that for any $u \in \supp(\mu)$, we have 
\[
\max_{j=1,\dots, m}
|
\cN_j(\bm{u})
-
v(x_j,T)
|
\le
C (\epsilon + \Delta x^2 + T/n),
\]
where $\bm{u} = (u(x_1),\dots, u(x_m)) = \cE(u)$.
We note that for the present choice of the $x_j$ as nodes on an equidistant grid, we have $\Delta x \sim m^{-2/d}$. Thus, we choose $n = \lceil T m^{2/d} \rceil$, $\epsilon = m^{-2/d}$ to conclude that there exists a neural network $\cN: \R^m \to \R^m$, such that 
\begin{align} \label{eq:size-hN}
\size(\cN) \le C(1+m^{2+2/d}),
\quad
\depth(\cN) \le C(1+m^{2/d}\log(m)),
\end{align}
and 
\begin{align} \label{eq:err-hN}
\max_{j=1,\dots, m}
|
\cN_j(\bm{u})
-
v(x_j,T)
|
\le
C m^{-2/d}.
\end{align}
Given the grid points $x_j$, we note that we can define a linear mapping 
\[
\cL: \R^m \to C(D), 
\quad
\bm{v} = (v_1,\dots, v_m)
\mapsto \cL(\bm{v})(x),
\]
by employing local linear interpolation of the values $v_j$ on the mesh $x_j$, i.e. such that $\cL(\bm{v})(x_j) = v_j$. Given the projection operator $\cP$, we define a linear mapping $\hL: \R^m\to \R^p$ by $\hL(\bm{v}) := \cP \circ \cL(\bm{v})$. With these definitions, we now set $\cA := \hL \circ \hN = \cP \circ \cL \circ \hN$. We note that since $\hL = \cP \circ \cL$ is a linear mapping $\R^m \to \R^p$, and since $\hN$ is a neural network with size bounded by \eqref{eq:size-hN}, we can represent $\cA$ as a neural network with
\begin{gather*}
\size(\cA)
= \mathcal{O}(m^{2+2/d} + mp), 
\\
\depth(\cA)
=\mathcal{O}(m^{2/d}\log(m) + 1) = \mathcal{O}(m^{2/d}\log(m)).
\end{gather*}
Furthermore, for this choice of $\cA$, we have 
\begin{align*}
(\Err_{\cA})^2
&=
\int_{\R^m} \Vert \cA(\bm{u}) - \cP \circ \cG \circ \cD(\bm{u}) \Vert^2_{\ell^2} \, d(\cE_\#\mu)(\bm{u})
\\
&=
\int_{\supp(\mu)} \Vert \cA\circ \cE(u) - \cP \circ \cG \circ \cD\circ \cE(u) \Vert^2_{\ell^2} \, d\mu(u)
\\
&=
\int_{\supp(\mu)} 
\Vert 
\cP \circ \cL \circ \hN \circ \cE(u) - \cP \circ \cG \circ \cD\circ \cE(u) 
\Vert^2_{\ell^2} \, d\mu(u)
\end{align*}

Furthermore, we can estimate the integrand for any $u\in \supp(\mu)$:
\begin{align*}
\Vert 
\cP \circ \cL \circ \hN \circ \cE(u) &- \cP \circ \cG \circ \cD\circ \cE(u) 
\Vert_{\ell^2} 
\\
&\le
\Vert \cP \Vert_{L^2(U) \to \ell^2(\R^p)}
\Vert 
\cL \circ \hN \circ \cE(u) - \cG \circ \cD\circ \cE(u) 
\Vert_{L^2(U)} 
\\
&\le 
\Vert 
\cL \circ \hN \circ \cE(u) - \cG \circ \cD\circ \cE(u) 
\Vert_{L^2(U)}.
\end{align*}
If we denote by $x \mapsto v(x,T) = \G(u)$ the solution at $t = T$ of the Allen-Cahn equation \eqref{eq:allen-cahn} with initial data $u$, and $\bm{u}=(u(x_1),\dots, u(x_m)) = \cE(u)$, then we have 
\begin{align}
\Vert 
\cL \circ \hN \circ \cE(u) &- \cG \circ \cD\circ \cE(u) 
\Vert_{L^2(U)}
\notag
\\
&\le
\Vert 
\cL \circ \hN (\bm{u}) - \cG(u) 
\Vert_{L^2(U)}
+
\Vert 
\cG(u) - \cG\circ \cD\circ \cE(u) 
\Vert_{L^2(U)}
\notag
\\
&\le
C
\Vert 
\cL \circ \hN (\bm{u}) - \cG(u) 
\Vert_{L^\infty(U)}
+
C
\Vert 
\cG(u) - \cG\circ \cD\circ \cE(u) 
\Vert_{L^\infty(U)}
. \label{eq:dec}
\end{align}
To estimate the first term, we note that since $\cL \circ \hN(\bm{u})$ and $v(\slot, T) = \G(u)$ are Lipschitz continuous functions with Lipschitz constant that can be bounded by $\Vert u \Vert_{C^{4,\alpha}}$ for a fixed $\alpha \in (0,1)$, as in Theorem \ref{thm:ac-regularity}, we have
\begin{align*}
\Vert 
\cL \circ \hN (\bm{u}) - \cG(u) 
\Vert_{L^\infty(U)}
&\le
C(\Vert u \Vert_{C^{4,\alpha}}) \Delta x 
\\
&\qquad + \max_{j=1,\dots, m} |\cL\circ\hN(\bm{u})(x_j) - \cG(u)(x_j)|
\\
&=C(\Vert u \Vert_{C^{4,\alpha}}) \Delta x + \max_{j=1,\dots, m} |\hN_j(\bm{u}) - v(x_j,T)|
\\
&\le
C m^{-1/d}.
\end{align*}
In the last step, we used the fact that $\Delta x \lesssim m^{-1/d}$ and \eqref{eq:err-hN}. 

To estimate the other second term in \eqref{eq:dec}, we note that the numerical scheme \eqref{eq:allen-cahn-fd} applied to $u$ and $\cD \circ \cE(u)$ starts from the same discrete initial data $\cE(u) = (u(x_1),\dots, u(x_m))$, since $\cE\circ \cD = \Id$, by assumption. It then follows from the error estimate of Theorem \ref{thm:ac-approx-err}, and the fact that the Lipschitz constant of $\G(u)$ and $\G \circ \cD\circ \cE(u)$ are bounded in terms of $\Vert u \Vert_{C^{4,\alpha}}$, that 
\begin{align*}
\Vert \G(u) - \G\circ \cD \circ \cE(u) \Vert_{L^\infty}
&\le
C m^{-1/d} + \max_{j=1,\dots,m} | \G(u)(x_j) - \G\circ \cD \circ \cE(u)(x_j) |
\\
&\le
C m^{-1/d} + C \Delta x^2
\le
C m^{-1/d}.
\end{align*}
The constant $C=C(u)>0$ depends on $u$ only through $\Vert u \Vert_{C^{4,\alpha}}$. In particular, there exists a constant $C'>0$, such that $C(\Vert u \Vert_{C^{4,\alpha}} \le C'$ for all $u\in \supp(\mu)$. Combining the above estimates, we can now estimate
\[
\Err_{\cA}
\le
C' m^{-1/d}.
\]
To conclude, we have shown that there exists an approximator network $\cA: \R^m \to \R^p$, for $p=m$, such that 
\begin{gather*}
\size(\cA)
= \mathcal{O}(m^{2+2/d} + mp),
\quad
\depth(\cA)
= \mathcal{O}(m^{2/d}\log(m)),
\end{gather*}
and $\Err_{\cA} = \mathcal{O}(m^{-1/d})$.
\end{proof}

\subsection{Proof of Lemma \ref{lem:charfun}}
\label{app:pf414}
\begin{proof}
Let $\sigma(x) = \max(x,0)$ be the ReLU activation function. We assume wlog that $\epsilon < b-a$ (otherwise decrease $\epsilon$). We note that
\[
\chi^\epsilon_{[a,b]}(x)
:=
\frac{\sigma(x-a) - \sigma(x-a-\epsilon/2) + \sigma( x - b) - \sigma(x-b-\epsilon/2)}{\epsilon/2},
\]
is continuous,  satisfies
\[
\chi^\epsilon_{[a,b]}(x)
=
\begin{cases}
0, & x\notin [a,b], \\
1, & x\in [a+\epsilon/2,b-\epsilon/2], \\
\end{cases}
\]
and is linear on $[a,a+\epsilon/2]$ and $[b-\epsilon/2,b]$. In particular, it follows that
\[
\Vert \chi^\epsilon_{[a,b]} - 1_{[a,b]}\Vert_{L^1(\R)}
\le
|[a,a+\epsilon/2]| + |[b-\epsilon/2,b]| = \epsilon.
\]
Furthermore, $\chi^\epsilon_{[a,b]}$ is represented by the same neural network architecture for any choice of $a,b,\epsilon$.
\end{proof}

\subsection{Proof of Theorem \ref{thm:err-conslaw}}
\label{app:pf415}
\begin{proof}
By Theorem \ref{thm:LxF}, there exists a constant $C>0$, and a neural network $\cN: \R^m \to \R^m$, with 
\begin{align} \label{eq:L1-branch0}
\size(\cN) \le C m^{-5/2}, \quad
\depth(\cN) \le C m,
\end{align}
such that for any $u\in \BV_M$, we have
\[
\left\Vert 
\G(u) - \sum_{j=1}^m \cN_j(\bar{\cE}(u)) \, 1_{C_j}(\slot)
\right\Vert_{L^1(D)}
\le
\frac{C}{m^\alpha}.
\]
We now define $\cA: \R^m \to \R^p$ by $\cA(\bm{u}) := (\cN_1(\bm{u}), \dots, \cN_p(\bm{u}))$, where we formally set $\cN_j \equiv 0$, if $j>m$. For each $j=1,\dots, m$, let $\tr_j(y) := \chi^{\epsilon}_{C_j}$, where $\chi^{\epsilon}_{C_j}$ is a neural network approximation of $1_{C_j}$ as in Lemma \ref{lem:charfun}, i.e. such that 
\begin{align} \label{eq:L1-trunk0}
\size(\tr_j) \le C, 
\quad
\depth(\tr_j) = 1,
\end{align}
such that 
\[
\Vert \tr_j - 1_{C_j} \Vert_{L^1} \le \epsilon, \text{for $j=1,\dots, m$.}
\]
For $j>m$, we define $\tr_j \equiv 0$, corresponding to an approximation of $C_j := \emptyset$ for $j>m$. Note that we clearly obtain from \eqref{eq:L1-trunk0} for the trunk net $\bm{\tr} = (\tr_1,\dots, \tr_p)$:
\begin{align} \label{eq:L1-trunk}
\size(\bm{\tr}) \le Cp, 
\quad
\depth(\bm{\tr}) = 1.
\end{align}
Similarly, we define for $j=1,\dots, p$ the branch net $\bm{\br} = (\br_1,\dots, \br_p)$ by:
\[
\br_j(u) := 
\begin{cases}
\cN_j(\bar{\cE}(u)), & (j\le m) \\
0, &(j>m).
\end{cases}
\]
By \eqref{eq:L1-branch0}, we note that
\begin{align} \label{eq:L1-branch}
\size(\bm{\br}) \le C m^{-5/2}, \quad
\depth(\bm{\br}) \le C m,
\end{align}

Then, clearly we have 
\begin{align*}
\left\Vert 
\G(u) - \sum_{j=1}^p \br_j(u) \tr_j
\right\Vert_{L^1}
&\le
\left\Vert 
\G(u) - \sum_{j=1}^m \cN_j(\bar{\cE}(u)) 1_{C_j}
\right\Vert_{L^1}
\\
&\qquad 
+
\sum_{j=1}^{p}
|\cN_j(\bar{\cE}(u))|
\left\Vert 
 1_{C_j} - \tr_j 
\right\Vert_{L^1}
\\
&\qquad
+
\sum_{j=p+1}^m
|\cN_j(\bar{\cE}(u))|
\left\Vert 
 1_{C_j}
\right\Vert_{L^1}
\end{align*}
The first term can be estimated by $Cm^{-\alpha}$. Each term in the sum for $j=1,\dots, p$ can be estimated by 
\[
\Vert 1_{C_j} - \tr_j \Vert_{L^1}
=
\left.
\begin{cases}
\Vert 1_{C_j} - \tr_j \Vert_{L^1}, & (j=1,\dots, m), \\
0, & (j > m)
\end{cases}
\right\}
\le \epsilon,
\]
by choice of the $\tr_j$. We also note that $|\cN_j(\bar{\cE}(u))|\le M$ for all $u\in \BV_M$ (cp. Theorem \ref{thm:LxF}). Hence, we have 
\[
\sum_{j=1}^{p}
|\cN_j(\bar{\cE}(u))|
\left\Vert 
 1_{C_j} - \tr_j 
\right\Vert_{L^1}
\le
Mp \epsilon.
\]
Finally, the last sum over $j=p+1,\dots, m$, if non-empty, can be estimated by
\begin{align*}
\sum_{j=p+1}^m
|\cN_j(\bar{\cE}(u))|
\left\Vert 
 1_{C_j}
\right\Vert_{L^1}
&\le
\sum_{j=p+1}^m \frac{2\pi M}{m}
\\
&=
\max(m-p,0) \frac{2\pi M}{m}
\\
&=
2\pi M \, \max\left(1-\frac pm,0\right).
\end{align*}
In particular, we conclude that there exists a constant $C>0$, independent of $p$ and $m$, such that for any $\epsilon > 0$, there exists a DeepONet $(\bm{\br},\bm{\tr})$ with size bounded by \eqref{eq:L1-branch} and \eqref{eq:L1-trunk}, such that 
\[
\left\Vert 
\G(u) - \sum_{j=1}^p \br_j(u) \tr_j
\right\Vert_{L^1}
\le
Cm^{-\alpha} + C p \epsilon + C \max\left( 1 - \frac{p}{m}, 0\right),
\quad
\forall u \in \BV_M.
\]
Since $\epsilon >0$ is arbitrary, we may set $\epsilon = m^{-\alpha} p^{-1}$, and absorb the second term in the first. Integrating against $\mu$ with $\mu(\BV_M)=1$, we find that 
\[
\int_{L^1(D)}
\left\Vert 
\G(u) - \cR \circ \cA \circ \bar{\cE}(u)
\right\Vert_{L^1}
\, d\mu(u)
\le
C m^{-\alpha} + C \max\left(1-\frac{p}{m},0\right),
\]
where we recall that, by definition, we have
\[
\cR \circ \cA \circ \bar{\cE}(u) \equiv \sum_{j=1}^p \br_j(u) \tr_j.
\]
The claimed estimate on $\Err_{L^1}$ thus follows with trunk and branch net complexity bounds \eqref{eq:L1-trunk} and \eqref{eq:L1-branch}.
\end{proof}

\subsection{Proof of Lemma \ref{lem:burgers-spectrum}} \label{app:burgers-spectrum}

\begin{proof}
We first note that the covariance operator $\Gamma_{\G_\#\mu}$ can be represented in the form 
\[
(\Gamma_{\G_\#\mu}u)(x)
=
\int_0^{2\pi} k(x,x') u(x') \, dx',
\]
where 
\begin{align*}
k(x,x') 
&= 
\int_{L^2(\T)} u(x)u(x') \, d(\G_\#\mu)(u)
\\
&=
\int_{L^2(\T)} \G(u)(x)\G(u)(x') \, d\mu(u).
\end{align*}
Next, we note that for any functional $\cF \in L^1(\mu)$, we have
\[
\int_{L^2(\T)} \cF(u) \, d\mu(u)
=
\frac{1}{2\pi}\int_0^{2\pi} \cF(u(\slot - \hat{x})) \, d\hat{x},
\]
and the solution with initial data $u_0(x-\hat{x})$ at $t=\pi/2$ is given by $\cG(u(\slot - \hat{x})) = v_t(x-\hat{x})$. It follows that 
\begin{align*}
\int_{L^2(\T)} \G(u)(x)\G(u)(x') \, d\mu(u)
&=
\frac{1}{2\pi}\int_0^{2\pi} \G(u(\slot - \hat{x}))(x)\G(u(\slot - \hat{x}))(x') \, d\hat{x}
\\
&=
\frac{1}{2\pi}\int_{0}^{2\pi} v_t(x-\hat{x}) v_t(x'-\hat{x}) \, d\hat{x}.
\end{align*}
By a change of variables, we thus find
\begin{align*}
k(x,x') 
&=
\frac{1}{2\pi}\int_{0}^{2\pi} v_t(x-\hat{x}) v_t(x'-\hat{x}) \, d\hat{x}
\\
&=
\frac{1}{2\pi}\int_{0}^{2\pi} v_t(x-x'+\xi) v_t(\xi) \, d\xi
\\
&=
g(x-x'),
\end{align*}
where
\[
g(x) := \frac{1}{2\pi} \int_0^{2\pi} v_t(x+\xi) v_t(\xi) \, d\xi, 
\]
is written as a convolution. In particular, $k(x,x') =g(x-x')$ is a \emph{stationary} kernel. From the stationarity of $k(x,x')$, it follows that the eigenfunctions of $k(x,x')$ are given by the Fourier basis $\{\fb_k\}_{k\in \Z}$, with corresponding eigenvalues
\[
\lambda_k = (2\pi) \hat{g}(k),
\]
where $\hat{g}(k)$ denotes the $k$-th Fourier coefficient of $g$. Finally, we note that 
\[
\lambda_k 
=
(2\pi)\hat{g}(k)
= |\hat{v}_t(k)|^2 = \frac{1}{\pi^2 k^2} + o\left(\frac{1}{k^2}\right).
\]
\end{proof}

\section{Proof of Theorem \ref{thm:generalization-err}}
\label{app:pf5}

We note that, since $\hN$ is a minimizer of $\hL$, and $\hN_{N}$ is a minimizer of $\hL_{N}$, we have the following well-known bound:
\begin{align*}
\left|\hL\left(
\hN_{N}
\right)
-
\hL\left(
\hN
\right)
\right|
&=
\hL\left(
\hN_{N}
\right)
-
\hL\left(
\hN
\right)
\\
&\le 
\hL\left(
\hN_{N}
\right)
-
\hL_{N}\left(
\hN_{N}
\right)
\\
&\qquad
+
\hL_{N}\left(
\hN
\right)
-
\hL\left(
\hN
\right)
\\
&\le
2\sup_{\theta} 
\left|
\hL_{N}\left(
\cN_\theta
\right)
-
\hL\left(
\cN_\theta
\right)
\right|,
\end{align*}
where the supremum is taken over all admissible $\theta \in [-B,B]^{d_\theta}$. Starting from this bound, the proof of Theorem \ref{thm:generalization-err} relies on the following lemmas, which follow very closely the argument in \cite[Chapter 5.3]{WeltiThesis} (see also \cite{CS1,berner2020analysis}).
\begin{lemma} \label{lem:gen-lip}
Under assumptions \ref{ass:boundedness} and \ref{ass:lipschitz}, we have
\[
\left|
S^N_\theta - S^N_{\theta'}
\right|
\le
\frac{4}{N}
\left(
\sum_{j=1}^N \Psi(Z_j) \Phi(Z_j)
\right)
\Vert\theta - \theta'\Vert_{\ell^\infty}.
\]
\end{lemma}

\begin{proof}
We have 
\begin{align*}
|S^N_\theta - S^N_{\theta'}|
&\le
\frac1N \sum_{j=1}^N 
\left| 
|\G(Z_j) - \cN_\theta(Z_j)|^2
-
|\G(Z_j) - \cN_{\theta'}(Z_j)|^2
\right|
\\
&\le 
\frac1N \sum_{j=1}^N 
\left( 
2|\G(Z_j)| + |\cN_\theta(Z_j)| + |\cN_{\theta'}(Z_j)|
\right)
|\cN_{\theta}(Z_j) - \cN_{\theta'}(Z_j)|
\\
&\le
\frac1N \sum_{j=1}^N 4 |\Psi(Z_j)| |\Phi(Z_j)| \Vert \theta - \theta'\Vert_{\ell^\infty}
\\
&=
\frac4N \left(\sum_{j=1}^N |\Psi(Z_j)| |\Phi(Z_j)| \right) \Vert \theta - \theta'\Vert_{\ell^\infty},
\end{align*}
as claimed.
\end{proof}

\begin{lemma} \label{lem:gen-sup}
If $\theta_1, \dots, \theta_K$ are such that for all $\theta \in [-B,B]^d$, there exists $j$ with $\Vert\theta - \theta_j\Vert_{\ell^\infty} \le \epsilon$, then 
\begin{align*}
\E\left[
\sup_{\theta\in [-B,B]^d}
\left|
S^N_\theta - \E[S^N_\theta]
\right|^p
\right]^{1/p}
\le
8\epsilon \E\left[\left|\Psi \Phi \right|^p\right]^{1/p}
+
\E\left[
\max_{j=1,\dots, K} 
\left|S^N_{\theta_j} - \E[S^N_{\theta_j}]\right|^p
\right]^{1/p}.
\end{align*}
\end{lemma}

\begin{proof}
Fix $\epsilon > 0$. Define a mapping $j: [-B,B]^d \to \N$, by $j(\theta) = \min \set{j\in \{1,\dots, K\}}{|\theta - \theta_j| \le \epsilon}$. Then, we have 
\begin{align*}
\E\left[
\sup_{\theta\in [-B,B]^d} |S^N_\theta - \E[S^N_\theta]|^p
\right]^{1/p}
&\le
\E\left[
\Bigg(
\sup_{\theta\in [-B,B]^d} 
\left|S^N_\theta - S^N_{\theta_{j(\theta)}} \right|
+
\left|S^N_{j(\theta)} - \E[S^N_{j(\theta)}]\right|
\right.
\\
&\qquad\quad
\left. 
+
\left|\E[S^N_{\theta_{j(\theta)}}] - \E[S^N_\theta]\right|
\Bigg)^p
\right]^{1/p}
\\
&\le 
\E\left[
\Bigg(
\max_{j=1,\dots, K} 
\left|S^N_{j(\theta)} - \E[S^N_{j(\theta)}]\right|
\right.
\\
&\qquad\quad
\left.
+
\frac{8\epsilon }{N}
\left(
\sum_{j=1}^N |\Psi(Z_j)||\Phi(Z_j)|
\right)
\Bigg)^p
\right]^{1/p}
\\
&\le 
\E\left[
\max_{j=1,\dots, K} 
\left|S^N_{\theta_j} - \E[S^N_{\theta_j}]\right|^p
\right]^{1/p}
\\
&\qquad \quad
+\frac{8\epsilon}{N} \sum_{j=1}^N \E\left[|\Psi(Z_j) \Phi(Z_j)|^p\right]^{1/p}
\\
&=
\E\left[
\max_{j=1,\dots, K} 
\left|S^N_{\theta_j} - \E[S^N_{\theta_j}]\right|^p
\right]^{1/p}
\\
&\quad\qquad
+
8\epsilon \E\left[|\Psi(Z_1) \Phi(Z_1)|^p\right]^{1/p},
\end{align*}
as claimed.
\end{proof}

\begin{lemma} \label{lem:gen-Kest}
Let $K\in \N$ and let $\theta_1,\dots, \theta_K \in [-B,B]^{d_\theta}$ be given. Then, for any $p\ge 1$, we have
\begin{align*}
\E\left[
\max_{j=1,\dots, K} 
\left|
S^N_{\theta_j} - \E[S^N_{\theta_j}]
\right|^p
\right]^{1/p}
\le
K^{1/p} 
\max_{j=1,\dots, K} 
\E\left[
\left|
S^N_{\theta_j} 
-
\E[S^N_{\theta_j}]
\right|^p
\right]^{1/p}.
\end{align*}
\end{lemma}

\begin{proof}
This follows readily from the fact that for any measurable $X_1,\dots, X_K$, we have
\begin{align*}
\E
\left[\max_{j=1,\dots, K} |X_j|^p \right]
&\le
\E\left[ \sum_{j=1}^K |X_j|^p \right]
=
\sum_{j=1}^K \E[|X_j|^p]
\le 
K \max_{j=1,\dots, K} \E[|X_j|^p].
\end{align*}
\end{proof}

\begin{lemma} \label{lem:gen-MCest}
Let $2 \le p < \infty$. For any $\theta \in [-B,B]^{d_\theta}$, we have
\begin{align*}
\E\left[
\left|
S^N_\theta - \E[S^N_\theta]
\right|^p
\right]^{1/p}
\le
\frac{16 \sqrt{p-1} \E\left[|\Psi|^{2p}\right]^{1/p}}{\sqrt{N}}.
\end{align*}
\end{lemma}

\begin{proof}
It follows from \cite[Corollary 5.18]{WeltiThesis}, with $S^N_\theta = \frac{1}{N} \sum_{j=1}^N X_j$, $X_j := |\G(Z_j) - \cN_\theta(Z_j)|^2$, that 
\[
\E\left[
|S^N_\theta - \E[S^N_\theta]|^p
\right]^{1/p}
\le 
\frac{2\sqrt{p-1}}{\sqrt{N}}
\left[
\max_{j=1,\dots, N} \E[|X_j - \E[X_j]|^p]^{1/p}.
\right]
\]
But by the boundedness assumption \ref{ass:boundedness}, we have
\[
X_j
\le
2|\G(Z_j)|^2 + 2|\cN_\theta(Z_j)|^2
\le
4 |\Psi(Z_j)|^2.
\]
Hence
\begin{align*}
\E[|X_j - \E[X_j]|^{p}]^{1/p} 
&\le
8\E[|\Psi(Z_j)|^{2p}]^{1/p} 
=
8\E[|\Psi|^{2p}]^{1/p},
\end{align*}
where the last equality follows from the fact that the $Z_j$ are iid. Hence
\[
\E\left[
\left|
S^N_\theta - \E\left[S^N_\theta\right]
\right|^p
\right]^{1/p}
\le
\frac{16\sqrt{p-1} \E\left[|\Psi|^{2p}\right]^{1/p}}{\sqrt{N}}.
\]
\end{proof}

\begin{lemma} \label{lem:gen-est-unbalanced}
Let $\epsilon > 0$ be given. Let $p\ge 2$. Then we have
\begin{align*}
\E\left[
\sup_{\theta \in [-B,B]^d}
\left|
S^N_\theta - \E[S^N_\theta]
\right|^p
\right]^{1/p}
\le
16\Vert \Psi \Vert_{L^{2p}} 
\left(
\epsilon \Vert \Phi\Vert_{L^{2p}}
+
\frac{\sqrt{p} \, \Vert \Psi \Vert_{L^{2p}} 
}{\sqrt{N}}
K(\epsilon)^{1/p}
\right),
\end{align*}
where $K(\epsilon)$ denotes the $\epsilon$-covering number of $[-B,B]^{d_\theta}$.
\end{lemma}

\begin{proof}
Denote $K := K(\epsilon)$ the covering number of $[-B,B]^{d_\theta}$. Then, by the definition of a covering number, there exist $\theta_1,\dots, \theta_K$, such that for any $\theta \in [-B,B]^{d_\theta}$, there exists $j\in \{1,\dots, K\}$, such that $|\theta - \theta_j|\le \epsilon$. By Lemma \ref{lem:gen-sup}, we have
\begin{align*}
\E\left[
\sup_{\theta \in [-B,B]^{d_\theta}} 
\left|
S^N_\theta - \E\left[S^N_\theta\right]
\right|^p
\right]^{1/p}
\le
8\epsilon \E[\left|\Psi \Phi \right|^p]^{1/p}
+
\E\left[
\max_{j=1,\dots, K}
\left|
S^N_{\theta_j} - \E\left[
S^N_{\theta_j}
\right]
\right|^p
\right]^{1/p}.
\end{align*}
We estimate the first term by
\[
8\epsilon \E[\left|\Psi \Phi \right|^p]^{1/p}
\le
8\epsilon \E[\left|\Psi\right|^{2p}]^{1/2p} \E[\left|\Phi \right|^{2p}]^{1/2p}
= 8\epsilon \Vert \Psi \Vert_{L^{2p}} \Vert \Phi \Vert_{L^{2p}}.
\]
By Lemma \ref{lem:gen-Kest} and Lemma \ref{lem:gen-MCest}, we can estimate the last term 
\begin{align*}
\E\left[
\max_{j=1,\dots, K}
\left|
S^N_{\theta_j} - \E\left[
S^N_{\theta_j}
\right]
\right|^p
\right]^{1/p}
&\le
K^{1/p} \max_{j=1,\dots, K} \E\left[
\left|
S^N_{\theta_j}
-
\E\left[
S^N_{\theta_j}
\right]
\right|^p
\right]^{1/p}
\\
&\le
\frac{16 K^{1/p} \sqrt{p}\, \E\left[|\Psi|^{2p}\right]^{1/p}}{\sqrt{N}}
\\
&=
\frac{16 K^{1/p} \sqrt{p}\, \Vert \Psi \Vert^2_{L^{2p}}}{\sqrt{N}}.
\end{align*}
Substitution of these upper bounds now yields
\[
\E\left[
\max_{j=1,\dots, K}
\left|
S^N_{\theta_j} - \E\left[
S^N_{\theta_j}
\right]
\right|^p
\right]^{1/p}
\le
8 \epsilon \Vert \Psi \Vert_{L^{2p}} \Vert \Phi \Vert_{L^{2p}}
+
\frac{16 K^{1/p} \sqrt{p}\, \Vert \Psi \Vert^2_{L^{2p}}}{\sqrt{N}}.
\]
The claimed bound follows.
\end{proof}

We also remark the following well-known fact:
\begin{lemma} \label{lem:gen-covering}
The covering number of $[-B,B]^d$ satisfies
\begin{align*}
K(\epsilon)
\le
\left(
\frac{CB}{\epsilon}
\right)^{d},
\end{align*}
for some constant $C>0$, independent of $\epsilon$, $B$ and $d$.
\end{lemma}

\begin{proof}
For a proof, see e.g. \cite[Lemma 5.11]{WeltiThesis}.
\end{proof}
\noindent 
Often, one can prove a bound of the form 
\[
\G(u) \le
\Psi(u,y) \le 
C \left(1+\Vert u \Vert_{L^2_x} \right)^\kappa,
\]
and for Gaussian measures $\mu$, we note that there exists $\alpha > 0$, such that 
\[
\int_{L^2_x} e^{\alpha \Vert u \Vert_{L^2_x}^2} \, d\mu(u) < \infty.
\]
We next want to derive some estimates on $\Vert \Psi \Vert_{L^p}$, as a function of $p$.
\begin{lemma} \label{lem:log-bound}
Let $p\ge 1$, $\alpha > 0$. The mapping 
\begin{align*}
[0,\infty) \to \R,
\quad
x \mapsto p \log(1+x) - \alpha x^2,
\end{align*}
satisfies the upper bound
\[
p\log(1+x) - \alpha x^2 \le p \log\left( 1 + \frac{p}{2\alpha} \right).
\]
\end{lemma}

\begin{lemma} \label{lem:kappa-int}
If $A = \int_{L^2_x} \exp(\alpha \Vert u \Vert^2_{L^2_x}) \, d\mu(u) < \infty$, then 
\[
\left(
\int_{L^2_x}
\left(
1 + \Vert u \Vert_{L^2_x}
\right)^{\kappa p}
\, d\mu(u)
\right)^{1/p}
\le 
A \left(
1 + \frac{\kappa p}{2\alpha}
\right)^{\kappa}.
\]
\end{lemma}
\begin{proof}
We have
\begin{align*}
(1+\Vert u \Vert_{L^2})^{\kappa p}
&=
\exp\left(
\kappa p \log(1+\Vert u \Vert_{L^2})
\right)
\\
&=
\exp\left(
\kappa p \log(1+\Vert u \Vert_{L^2}) - \alpha \Vert u \Vert^2
\right)
e^{\alpha \Vert u \Vert^2}
\\
&\le
\exp\left(
\sup_{x \in [0,\infty)}
\kappa p \log(1+x) - \alpha x^2
\right)
e^{\alpha \Vert u \Vert^2}.
\end{align*}
From Lemma \ref{lem:log-bound}, it follows that
\[
(1+\Vert u \Vert)^{\kappa p}
\le
\left(1 + \frac{\kappa p}{2\alpha} \right)^{\kappa p} e^{\alpha \Vert u \Vert^2}.
\]
Thus, we conclude that 
\begin{align*}
\left(
\int_{L^2} (1+\Vert u \Vert_{L^2})^{\kappa p}
\right)^{1/p}
\le
A^{1/p} \left(1 + \frac{\kappa p}{2\alpha} \right)^{\kappa}
\le
A \left(1 + \frac{\kappa p}{2\alpha} \right)^{\kappa},
\end{align*}
for all $p\ge 1$, where in the last step, we have used that 
\[
A = \int_{L^2} \exp(\alpha \Vert u \Vert^2_{L^2}) \, d\mu(u) \ge 1,
\] 
for any $\alpha >0$.
\end{proof}

\begin{proof}[Proof of Theorem \ref{thm:generalization-err}]
Note that 
\[
\left|
\hL(\hN_N) - \hL(\hN)
\right|
\le 
2
\sup_{\theta \in [-B,B]^{d_\theta}}
\left|
\hL_N(\cN_\theta)
-
\hL(\cN_\theta)
\right|.
\]
We now claim that
\begin{align} \label{eq:gen-claim}
\E\left[
\sup_{\theta \in [-B,B]^{d_\theta}}
\left|
\hL_N(\cN_\theta)
-
\hL(\cN_\theta)
\right|
\right]
\le
\frac{C}{\sqrt{N}}
\left(
1 + d_\theta \log(CB\sqrt{N})
\right)^{2\kappa+1/2},
\end{align}
for $C = C(\alpha, \kappa, \Psi, \Phi)$, from which the claimed bound on the generalization error follows. To prove the claimed inequality \eqref{eq:gen-claim}, we note that 
\[
\hL_N(\cN_\theta) - \hL(\cN_\theta) 
=
S^N_\theta - \E[S^N_\theta].
\]
By Lemma \ref{lem:gen-est-unbalanced} and \ref{lem:gen-covering}, we have for any $p\ge 2$ and $\epsilon > 0$:
\[
\E\left[
\sup_{\theta \in [-B,B]^{d_\theta}}
\left|
S^N_\theta - \E\left[ S^N_\theta \right]
\right|^p
\right]^{1/p}
\le
16 
\Vert \Psi \Vert_{L^{2p}}
\left(
\epsilon \Vert \Phi \Vert_{L^{2p}}
+
\left(\frac{CB}{\epsilon}\right)^{d_{\theta}/p}
\frac{\sqrt{p}\, \Vert \Psi \Vert_{L^{2p}}}{\sqrt{N}}
\right).
\]
By assumption on $\Psi, \Phi$, there exist constants $C>0$, $\kappa>0$, such that 
\begin{align} \label{eq:upper-bd}
|\Psi(u,y)|, \, |\Phi(u,y)| 
\le
C (1+\Vert u \Vert_{L^2})^\kappa.
\end{align}
By Lemma \ref{lem:kappa-int}, we can thus estimate
\[
\Vert \Psi \Vert_{L^{2p}}, \, \Vert \Phi \Vert_{L^{2p}}
\le 
C \left(
1 + \gamma \kappa p
\right)^{\kappa},
\]
for constants $C,\gamma>0$, depending only the measure $\mu$ and the constant $C$ appearing in the upper bound \eqref{eq:upper-bd}. In particular, we have
\[
\E\left[
\sup_{\theta \in [-B,B]^{d_\theta}}
\left|
S^N_\theta - \E\left[ S^N_\theta \right]
\right|^p
\right]^{1/p}
\le
16 
C^2 \left(
1 + \gamma \kappa p
\right)^{2\kappa}
\left(
\epsilon
+
\left(\frac{CB}{\epsilon}\right)^{d_{\theta}/p}
\frac{\sqrt{p}}{\sqrt{N}}
\right),
\]
for some constants $C, \gamma>0$, independent of $\kappa$, $\mu$, $B$, $d_\theta$, $N$, $\epsilon >0$ and $p\ge 2$.

We now choose $\epsilon = \frac{1}{\sqrt{N}}$, so that 
\[
\epsilon
+
\left(\frac{CB}{\epsilon}\right)^{d_{\theta}/p}
\frac{\sqrt{p}}{\sqrt{N}}
=
\frac{1}{\sqrt{N}} 
\left( 
1 + 
\left(CB\sqrt{N}\right)^{d_{\theta}/p} \sqrt{p}
\right).
\]
Next, let $p = d_\theta \log(CB\sqrt{N})$. We may wlog assume that $p \ge 2$ (otherwise, increase the constant $C$). Then, 
\[
\left(CB\sqrt{N}\right)^{d_{\theta}/p} \sqrt{p}
=
\exp
\left(\frac{\log(CB\sqrt{N}) d_{\theta}}{p}
\right) \sqrt{p}
=
e \, \sqrt{d_\theta \log(CB \sqrt{N})}, 
\]
and thus we conclude that 
\[
\epsilon
+
\left(\frac{CB}{\epsilon}\right)^{d_{\theta}/p}
\frac{\sqrt{p}}{\sqrt{N}}
\le
\frac{1}{\sqrt{N}}
\left(
1 + e\, \sqrt{ d_{\theta} \log(CB\sqrt{N})}.
\right).
\]
On the other hand, we have
\[
\left(
1 + \gamma \kappa p
\right)^{2\kappa}
=
\left(
1 + \gamma \kappa d_{\theta} \log(CB\sqrt{N})
\right)^{2\kappa}.
\]
Increasing the constant $C>0$, if necessary, we can further estimate
\[
\left(
1 + \gamma \kappa d_{\theta} \log(CB\sqrt{N})
\right)^{2\kappa}
\left(
1 + e\, \sqrt{ d_{\theta} \log(CB\sqrt{N})}.
\right)
\le
C\left(
1 + d_{\theta} \log(CB\sqrt{N})
\right)^{2\kappa + 1/2},
\]
where $C> 0$ depends on $\kappa$, $\gamma$, $\mu$ and the constant appearing in \eqref{eq:upper-bd}, but is independent of $d_\theta$, $B$ and $N$. We can express this dependence in the form $C = C(\mu, \Psi, \Phi)>0$, as the constants $\kappa$ and $\gamma$ depend on the Gaussian tail of $\mu$ and the upper bound on $\Psi$, $\Phi$.

To conclude, we have shown that there exists a constant $C = C(\mu, \Psi, \Phi) > 0$, such that for any $d_{\theta}$, $B$ and $N$, we have
\begin{align*}
\E\left[
\sup_{\theta\in [-B,B]^{d_\theta}} \left|\hL_N(\cN_\theta) - \hL(\cN_\theta) \right|
\right]
&=
\E\left[
\sup_{\theta\in [-B,B]^{d_\theta}} \left|S^N_\theta - \E\left[S^N_\theta\right] \right|
\right]
\\
&\le 
\E\left[
\sup_{\theta\in [-B,B]^{d_\theta}} \left|S^N_\theta - \E\left[S^N_\theta\right] \right|^p
\right]^{1/p}
\\
&\le 
\frac{C}{\sqrt{N}} \left( 1 + d_\theta\log(CB\sqrt{N})\right)^{2\kappa + 1/2}.
\end{align*}
This is the claimed inequality \eqref{eq:upper-bd}.
\end{proof}

\end{document}